\documentclass[10pt,reqno]{elsarticle}

\usepackage{times,amsmath,amsthm,amssymb,amsfonts,epsfig,color,graphicx,enumitem,accents,booktabs}
\usepackage{rotating,multirow,color,url,mathrsfs,mathrsfs,bm}
\usepackage{arydshln}
\usepackage{bbm}
\usepackage{esint}
\usepackage{gensymb}
\usepackage{epstopdf}
\usepackage{comment}
\usepackage{leftidx}
\usepackage{mathtools}
\usepackage{subfig}
\usepackage{float}
\usepackage{tcolorbox}
\newtheorem{theorem}{Theorem}[section]
\newtheorem{corollary}[theorem]{Corollary}
\newtheorem{lemma}[theorem]{Lemma}
\newtheorem{proposition}[theorem]{Proposition}
\newtheorem{conjecture}[theorem]{Conjecture}
\newtheorem{problem}[theorem]{Problem}

\theoremstyle{definition}

\newtheorem{remark}[theorem]{Remark}
\newtheorem{example}[theorem]{Example}

\numberwithin{equation}{section}
\numberwithin{figure}{section}
\numberwithin{table}{section}

\newcommand{\e}{{{el}}}

\newcommand{\R}{\mathbb{R}}

\newcommand{\Rd}{\R^2}

\newcommand{\Comp}{{\mathcal{C}}}

\newcommand{\ov}{\overline}

\newcommand{\Z}{\mathcal{Z}}
\newcommand{\Vmin}{{\mathcal{V}_\mathrm{min}}}
\newcommand{\NN}{{N}}
\newcommand{\h}{{[-H,H]}}
\newcommand{\Oh}{{\hat{\Omega}_H}}

\newcommand{\symtens}{{\overset{s}{\otimes}}}

\newcommand{\mbf}[1]{{\mathbf{#1}}}

\newcommand{\opnorm}[1]{{\left\vert\kern-0.25ex\left\vert\kern-0.25ex\left\vert #1 
		\right\vert\kern-0.25ex\right\vert\kern-0.25ex\right\vert}}

\newcommand{\Ob}{{\overline{\Omega}}}
\newcommand{\bO}{{\partial\Omega}}

\newcommand{\argu}{{\,\cdot\,}}

\newcommand{\Mes}{{\mathcal{M}}}

\newcommand{\Sdd}{{\mathrm{T}^2}}
\newcommand{\Sddp}{{\mathrm{T}^{2}_{\!+}}}

\newcommand{\DIV}{{\mathrm{Div}}}
\newcommand{\hDIV}{{\hat{\mathrm{D}}\mathrm{iv}}}
\newcommand{\dive}{{\mathrm{div}}}

\newcommand{\sig}{{\sigma}}

\newcommand{\tr}{{\mathrm{Tr}}}
\newcommand{\htr}{{\mathrm{Tr}}}

\newcommand{\Cmin}{{\mathcal{C}_{\mathrm{min}}}}

\newcommand{\pairing}[1]{{\left \langle #1 \right \rangle}}
\newcommand{\norm}[1]{\Arrowvert #1 \Arrowvert}
\newcommand{\abs}[1]{{\left \lvert #1 \right \rvert}}

\newcommand{\eps}{\varepsilon}
\newcommand{\Rb}{\overline{\mathbb{R}}}

\newcommand{\D}{\mathcal{D}}

\newcommand{\V}{{\mathcal{V}}}
\newcommand{\Ha}{\mathcal{H}}
\newcommand{\IM}{\mathrm{Im}}

\newcommand{\mres}{\mathbin{\vrule height 1.6ex depth 0pt width
		0.13ex\vrule height 0.13ex depth 0pt width 1.3ex}}

\makeatletter
\def\ps@pprintTitle{%
	\let\@oddhead\@empty
	\let\@evenhead\@empty
	\def\@oddfoot{}%
	\let\@evenfoot\@oddfoot}
\makeatother

\begin{document}

	\title{Optimal vault problem -- form finding through 2D convex program}

	
	\author{Karol Bo{\l}botowski}
	
	\address{Department of Structural Mechanics and Computer Aided Engineering, Faculty of Civil Engineering, Warsaw University of Technology, 16 Armii Ludowej Street, 00-637 Warsaw, \linebreak
	College of Inter-Faculty Individual Studies in Mathematics and Natural Sciences, University of Warsaw, 2C Stefana Banacha St., 02-097 Warsaw
	}

	\begin{abstract}
		This work puts forward a form finding problem of designing a least-volume vault that is a surface structure spanning over a plane region, which via pure compression transfers a vertically tracking load to the supporting boundary. Through a duality scheme, developed recently for the design of pre-stressed membranes, the optimal vault problem is reduced to a pair of mutually dual convex problems $(\mathcal{P})$,\,$(\mathcal{P}^*)$ formulated on the 2D reference region. The vault constructed upon solutions of those problems is proved to be both of minimum volume and minimum compliance; analytical examples of optimal vaults are given. Through a measure-theoretic approach, thus found optimal vaults are proved to solve the Prager problem of designing a 3D structure that by compression carries a transmissible load. The ground structure method applied to the convex problems furnishes a pair of discrete, conic quadratic programs $(\mathcal{P}_X)$,\,$(\mathcal{P}_X^*)$ leading  to optimal design of grid-shells. By adopting the member-adding adaptive technique this pair is efficiently tackled  numerically, which is demonstrated on a number of examples where highly precise grid-shell approximations of optimal vaults are found.
	\end{abstract}
	

	\maketitle


\section{Introduction}
\label{sec:introdction}

A long standing engineering  problem is the one of designing a surface structure that in a pure membrane state efficiently transfers a given load to the boundary. This involves both continuous shells or bar frameworks that lie on a single surface. The structures from the latter class are often termed \textit{grid-shells}. Eliminating bending allows to significantly reduce thickness of the structural elements thus generating material savings. Back in the days design of such flexureless structures -- known as \textit{form finding} --  required intuition of crafted engineers and architects. Over time the process evolved being aided by computer methods. For instance, in \cite{day1965} the equilibrated configuration of grid-shells was established via the \textit{dynamic relaxation method with kinetic damping}. An additional constraint may concern the sign of the stress: the pure compression or pure tension state can be imposed thus arriving at design of vaulted masonry structures or, respectively, hanging roofs and cable-nets. To this aim the \textit{force density method} was developed for grid-shells in \cite{schek1974}, see also the works on self-supporting structures in \cite{block2007}, \cite{vouga2012}. Its generalization to design of continuous shells undergoing pure tensile stresses was developed e.g. in \cite{bletzinger1999}, \cite{nguyen2020} where a fully non-linear shell theory is employed. Another topic is optimization of such surface forms, for instance in terms of volume or elastic compliance, which can be combined with the aforementioned form finding methods, see e.g. \cite{bletzinger2005}, \cite{richardson2013}, \cite{dzierzanowski2020b}.

This paper focuses on a very specific form finding problem. For a bounded domain $\Omega \subset \Rd$ contained in a horizontal plane we will be designing a structure that shall lie on a single surface $\mathcal{S}_z$ being a graph of function $z : \Omega \to \R$ that is zero on the boundary $\bO$. The \textit{elevation functon} $z$ is a design parameter itself. The structure is free of bending and is capable of withstanding compressive stresses only while it transfers a vertical load $f$ whose intensity is given with respect to the horizontal domain $\Omega$ and its vertical position is not fixed: the load tracks the surface $\mathcal{S}_z$. The design objective is to minimize the structure's volume while maintaining the principal stresses in the regime $-k_0 \leq \sigma_{\mathrm{I}},\sigma_{\mathrm{II}}\leq 0$ for a yield stress $k_0>0$. No constraints are imposed on the structural topology of the design: the material may occupy any subset of the surface $\mathcal{S}_z$ and  it can be either continuously spread over 2D patches or distributed along 1D elements in the form of struts or arches. The optimal design problem thus formulated will be termed \textit{the minimum-volume vault problem}. A similar problem can be posed when assuming that the structure is subject to pure tension, in which case one may speak of the minimum-volume hanging roof. 

The plane variant of the problem put forward, where a curve $\mathrm{C}_z$ constituting a least-volume masonry arch is being sought, can be solved analytically as it is directly linked to \textit{the funicular problem}, cf. \cite{czubacki2019}. With bending dismissed the arch's equilibrium enforces that the elevation function $z:[a,b] \to \R$ satisfies equation $H  z'' = f$ where $H>0$ is a horizontal thrust force. Owing to the boundary conditions $z(a)=z(b) = 0$ the elevation $z$ is determined up to multiplicative constant $1/H$. Then, assuming that the arch is fully stressed, one can find the optimal thrust $H$ that minimizes the volume. In contrast, the shape of surface forms which carry the load $f$ via compression is clearly non-unique, which puts the vault optimization in the class of more involved structural design problems.

To avoid the geometrically difficult challenge of finding an optimal surface $\mathcal{S}_z$ a relaxed formulation may be proposed: any three dimensional structure in compression is admissible while the vertical position of the load remains to be optimally determined. With the structural volume being minimized such formulation resembles the Michell problem \cite{michell1904} (see also \cite{strang1983,bouchitte2008,lewinski2019a}) up to enforcing the stress sign and adjusting the load. This relaxed problem was first investigated in \cite{rozvany1982} where, to honour the memory of William Prager who initiated this line of research, it was referred to as \textit{Prager problem} whereas the optimal structures themselves -- as \textit{Prager structures}. To the knowledge of the present author, however, designing of spatial Prager structures has never been formulated as a clear mathematical problem, instead the authors of \cite{rozvany1982} depart from the \textit{Prager-Shield optimality criteria}. This work will address this issue whilst relaxing the problem even further by employing the concept of \textit{transmissible loads} introduced in \cite{fuchs2000} (cf. also \cite{chiandussi2009}): the load $f$ given with respect to $\Omega$ can be arbitrarily distributed along vertical lines, continuously or discretely. The hope is that, despite this extra freedom, at each point $x \in \Omega$ there is one optimal load position $z = z(x)$ and that the Prager structure is essentially a vault for it concentrates on a single surface $\mathcal{S}_z$. Such a conjecture is driven by the planar case where Prager structure is indeed a single funicular, cf. the formal proof in \cite{rozvany1983}. At this point it must be stressed that the constraint on the stress sign is crucial -- in \cite{darwich2010} the authors showed that once tension and compression is allowed the 2D optimal design under transmissible load (that is uniformly distributed in the horizontal direction) is not a parabolic arch and instead small lobe-like 2D tension-compression regions occur in vicinity of the supports furnishing smaller volume. To the present author's knowledge so far analytical examples of Prager structures have only been found for axisymmetric or "quasi-axisymmetric" boundary and loading conditions, cf. \cite[Sections 6.8,\,6.9]{rozvany1982}.

The most efficient numerical technique for tackling the Michell problem is \textit{the ground structure method}, its idea was first given in \cite{dorn1964} and further developed in \cite{gilbert2003}, \cite{sokol2010}, \cite{zegard2014}. For a finite nodal grid, populating the design region $\Omega$, a dense highly redundant truss is built by connecting each pair of nodes by a bar; then the volume minimization by member sizing becomes a linear programming (LP) problem. For very fine nodal grids one arrives at a very precise truss approximations of Michell structures. Recently the ground structure approach was successfully brought to grillage optimization, cf. \cite{bolbotowski2018}. Adopting the concept of transmissible loads to the 3D ground structure method allows to stay within the LP framework thus furnishing a natural tool for numerical prediction of Prager structures. First simulations may be already found in \cite{gilbert2005}, however both compression and tension were allowed therein; in \cite{jiang2018} the solution in pure compression was provoked by imposing a very large cost for the elements in tension. In both cases the 3D trusses obtained seem to oscillate about a single surface thus approximating a grid-shell solution. Similar numerical experiments were also presented in \cite{lewinski2019b}: the truss in Fig. \ref{fig:Prager_structure} was obtained by Tomasz Sok{\'o}{\l} and is an approximation of a Prager structure for a square domain and uniformly distributed vertical load.

\begin{figure}[h]
	\centering
	\includegraphics*[trim={0cm 0.5cm 0cm 4.5cm},clip,width=0.45\textwidth]{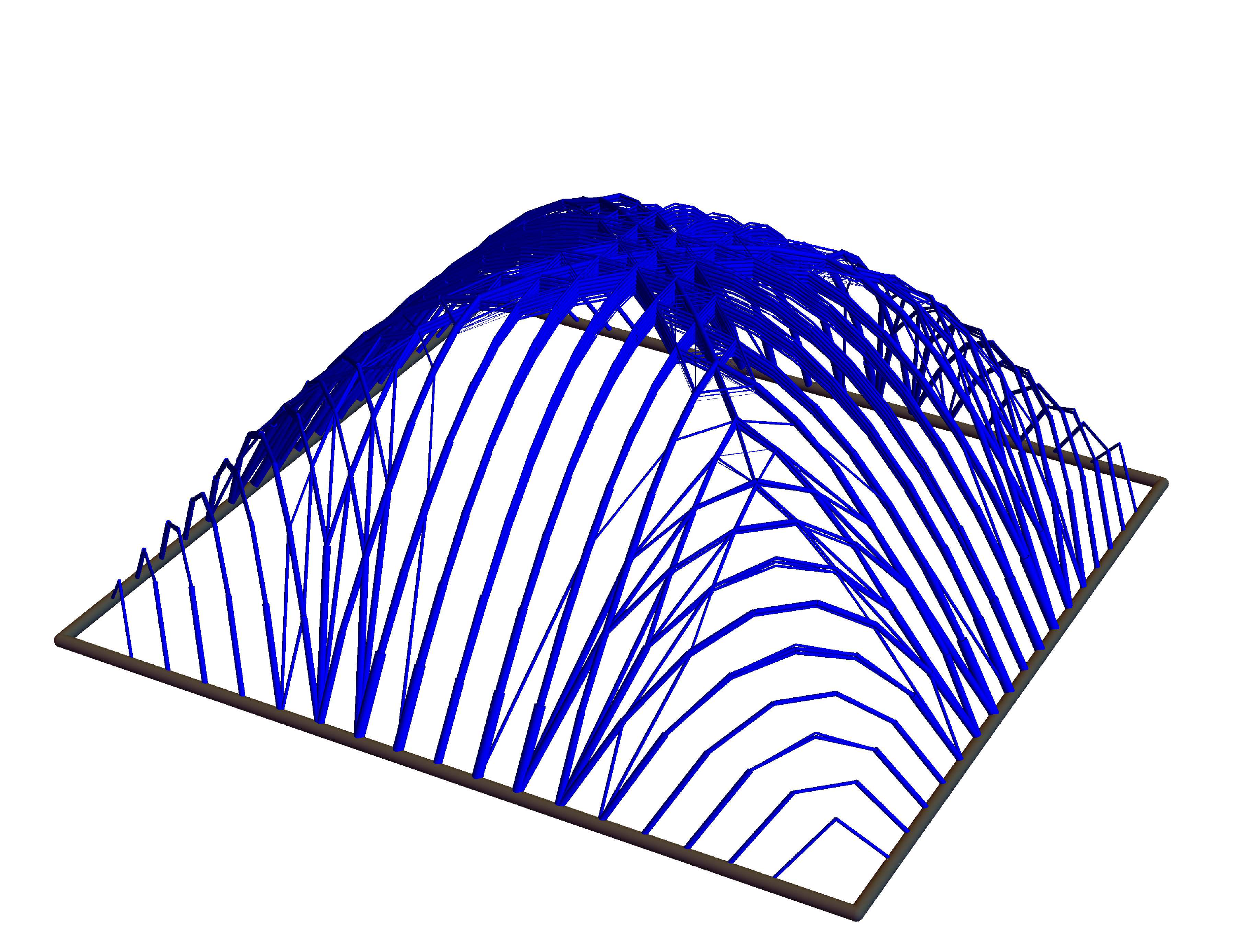}
	\caption{Prager structure for a square domain and uniformly distributed load -- a vault-like numerical prediction through the 3D ground structure approach, courtesy of Tomasz Sok{\'o}{\l}.}
	\label{fig:Prager_structure}       
\end{figure}

One of the main goals of the present paper is to tackle the optimal vault problem from theoretical point of view, therefore we will be bound to work with an explicit surface $\mathcal{S}_z$. A natural step is to reduce the problem to the base plane region $\Omega$, which potentially paves a way to a numerical method that is more efficient than full 3D methods. As an example of such a 2D approach to form finding problem one may give the work \cite{rozvany1979} where the \textit{optimal arch-grid problem} was introduced. It may be treated as a first step of generalizing the optimal funicular problem to 3D case: the tracking load $f$ is carried by two families of arches whose projections onto the base plane are respectively parallel to two prefixed orthogonal directions. The main result is that, for a configuration that minimizes the total volume, the arches lie on a single surface $\mathcal{S}_z$. The Prager-Rozvany arch-grid problem was recently revisited in \cite{czubacki2019} and \cite{dzierzanowski2020}, see also \cite{lewinski2019b} and \cite[Chapter 6]{lewinski2019a}. In the spirit of theory of Michell structures those papers put the arch-grid problem as a pair of mutually dual convex problems in 2D domain $\Omega$: in the primal form we seek a field $q$ guaranteeing vertical equilibrium $-\DIV\,q = f$ and minimizing a very specific convex integral functional; in the dual problem the virtual work $\int_\Omega w \,f$ is being maximized with respect to scalar virtual displacement field $w$ that satisfies an upper bound condition for the mean square slope along each of the two fixed directions. Such duality-based approach towards the form finding problem inspired the present author to chase after a generalization of the arch-grid problem towards freeing the direction of arches at each point of the structure, which eventually led to formulation of the optimal vault problem herein presented.

Amongst other methods of form finding formulated in the reference plane domain $\Omega$ one can find \cite{block2007} where a concept of \textit{thrust network} is developed -- a horizontal network of forces in the in-plane equilibrium is generated based on the idea of reciprocal figures and then suitably erected to constitute a surface form in pure compression. The idea inspired a chain of numerical techniques for construction of \textit{self-supporting} vaulted structures, cf. \cite{fraternali2010} or \cite{vouga2012} to name a few. In those works a continuous description may be found as well where the thrust is represented by a plane negative semi-definite tensor field $\sigma_- \! \preceq 0$. Such thrust field ought to be treated as a projection of compression stress field in a vault lying on a surface $\mathcal{S}_z$, which is in equilibrium with the tracking load $f$ if and only if
\begin{equation}
\label{eq:feas_vault}
\DIV \, \sigma_- \! = 0, \quad -\dive \bigl( \sigma_- \nabla{z} \bigr) = f \qquad   \text{in} \quad \Omega.
\end{equation}
This line of research on the self-supporting structures, however, did not involve optimization, e.g. in terms of volume or compliance.

Equilibrium equations \eqref{eq:feas_vault} can be used as a staring point in the optimal vault problem: any statically admissible vault may be represented as a pair of fields $(z,\sigma_-)$ defined in the reference domain $\Omega$ and satisfying \eqref{eq:feas_vault}, cf. Chapter 11.2 in \cite{green1968}. With the help of differential geometry we can derive a formula for the fully stressed vault's volume $V = V(z,\sigma_-)$ that turns out to be non-convex therefore ruling out the direct method of calculus of variations in showing existence of a solution. Moreover, tackling a non-convex optimization problem numerically puts us at risk of ending up at a local minimum.

The core idea of the present contribution is to establish a non-trivial link between the optimal vault problem and the recently proposed problem of plane pre-stressed membrane of minimum compliance, cf. \cite{bouchitte2020}. Classically, a membrane is characterized by a pre-stress parameter $p>0$. It can be generalized to a non-homogeneous anisotropic pre-stress being a positive semi-definite tensor field $\sigma \succeq 0$ that is bound to satisfy the in-plane equilibrium $\DIV\,\sigma =0$ in $\Omega$. Subject to a pressure load $f$, the membrane undergoes a deflection $w = w_\sigma:\Omega \to \R$ determined by the out-of-plane equilibrium equation $-\dive\bigl( \sigma\, \nabla w_{\sigma} \bigr) = f$ that is a generalization of the renown Possion equation $-p \,\triangle w = f$. Then, the pair $(z,\sigma_-) = (-w_{\sigma},-\sigma)$ furnishes a feasible vault in pure compression, i.e. it satisfies equilibrium equations \eqref{eq:feas_vault}. This concept of form finding is not new and can be traced back to Antoni Gaud\'{i}'s hanging chain models, see \cite{bergos1999} or the inspired numerical \textit{potential energy method} in \cite{jiang2018}. A deeper relation that is promoted in this work concerns optimally designed plane membrane: with membrane's compliance defined as $\Comp_M = \Comp_M(\sigma) = \frac{1}{2}\int_{\Omega} w_\sigma\, f$ there holds an implication
\begin{equation*}
\Comp_M(\bar{\sigma}) = \min_{\sigma \succeq 0}\left\{ \Comp_M(\sigma) \, \left \vert \ \DIV \, \sigma =0, \   \int_{\Omega} \tr\,\sigma \leq \V_0  \right.  \right\}  \ \ \  \Rightarrow \ \ \  (z,\sigma_-)=\left(-\frac{1}{\alpha}\,w_{\bar{\sigma}},-\alpha\,\bar{\sigma} \right) \ \text{ yields an optimal vault,}
\end{equation*}
where $\alpha$ is a suitably chosen positive number. This connection inspired the work \cite{bouchitte2020} co-written by the present author where the optimal pre-stressed membrane problem was investigated from the mathematical point of view. Unlike the optimal vault problem, the membrane problem is convex and existence of solutions was established. With the use of general duality tools the design problem was reformulated as a pair of mutually dual convex programs: the primal problem $(\mathcal{P})$ on the space of Radon measures $\sigma$ and $q$, the latter representing a vector field of transverse force; the dual problem $(\mathcal{P}^*)$ where virtual displacement fields are sought -- vectorial in-plane displacement $u$ and scalar deflection function \nolinebreak $w$. This pair is reminiscent of the one emerging in Michell theory.

We shall now summarize the content of the present paper. After a brief review of Michell theory and gathering some tools of convex analysis in Section \ref{sec:review}, we devote Section \ref{sec:plastic_design} to proving that a least-volume vault may be recast from solutions of convex problems $(\mathcal{P}), (\mathcal{P}^*)$ that are proposed directly based on \cite{bouchitte2020}. We depart from the precise formulation of the optimal form finding problem employing the classical membrane shell theory in the differential geometry setting, cf. \cite{green1968}, \cite{ciarlet2000}. For convenience and clarity in the derivations and proofs the structure will be assumed to work in pure tension rather than compression; in due course a simple alteration by change of signs will be proposed to recover the vault in compression (notwithstanding this, the term \textit{vault} will be used in both contexts). We shall extensively build upon the duality theory between problems $(\mathcal{P})$,\,$(\mathcal{P}^*)$ developed in \cite{bouchitte2020}. The recipe for the least-volume vault is given in Theorem \ref{thm:recovring_MV_dome}: the optimal elevation function reads $z = \frac{1}{2}\, w$ where $w$ solves $(\mathcal{P}^*)$ while the vault's stress field is obtained by unprojecting $\sigma$ being a solution of $(\mathcal{P})$. 
Using the derived optimality conditions an analytical solution of an axisymmetric vault is given.

While Section \ref{sec:plastic_design} deals with the so called "plastic design setting" of the optimal vault problem where no elastic compatibility is enforced on the deformation, Section \ref{sec:elastic_design} puts forward a problem of designing vaults of minimum elastic compliance that proves to be equivalent to the former problem. Aside from the elevation function $z$ one seeks an optimal elastic material distribution $\rho$.
The optimal elastic vault is then once again constructed based on solutions of convex problems $(\mathcal{P}), (\mathcal{P}^*)$, namely for an optimal elevation $z = \frac{1}{2} \, w$ the density $\rho$ follows directly from $\sigma$. An astonishing relation is discovered and carefully examined: the optimal elevation function $z$ and the elastic vertical displacement function $w_\e$ coincide up to a multiplicative constant. 

The purpose of Section \ref{sec:3D} is twofold. First, it is to cover the more general design setting where vaults with lower dimensional stiffeners are admissible. Secondly, it  establishes a desirable link to the three dimensional Prager problem for which we first propose a rigorous mathematical formulation. To that aim the more modern tools of measure theory \cite{evans1992} and measure-tangential calculus \cite{bouchitte1997} will be engaged. The objective functional in $(\mathcal{P})$ is of linear growth therefore, similarly as in Michell problem, its solutions $\sigma$ and $q$ in general lie in the space of tensor and, respectively, vector valued measures --  objects that simultaneously describe continuously distributed stress as well as one dimensional stress channels. Inspired by the pioneering works on optimal design in the measure-theoretic setting -- see \cite{bouchitte2001} on the mass optimization or \cite{bouchitte2008} on Michell problem -- we pose the Prager problem of designing a structure in three dimensional space modelled by a 3D tensor valued measure $\hat{\sigma}$ potentially combining 3D, 2D and 1D structural elements. The main theoretical result of the paper is stated in Theorem \ref{thm:optimal_3D_structure}: once more we utilize solutions of problems $(\mathcal{P})$,\,$(\mathcal{P}^*)$ to construct the measure $\hat{\sigma}$ that solves the Prager problem and concentrates on a single surface $\mathcal{S}_z$, which proves that optimally designed vaults are indeed Prager structures, possibly consisting of 2D and 1D elements. 

Every optimal design problem considered in this work reduces to the pair of mutually dual convex problems $(\mathcal{P})$,\,$(\mathcal{P}^*)$. The natural strategy is thus to develop the numerical method around a discretized variant of this very pair. In Section \ref{sec:discrete}, upon mathematical justification, the ground structure method is adopted for this purpose. In the emerging discrete problems $(\mathcal{P}_X)$,\,$(\mathcal{P}^*_X)$ the truss spanned by the nodal grid $X$ is essentially a pre-stressed system of strings that reacts out of plane due to the discretized load $\mbf{f}$. In the primal problem $(\mathcal{P}_X)$ we search for vectors $\mbf{s}$ and $\mbf{q}$ of, respectively, pre-stressing and transverse member forces. In the dual problem $(\mathcal{P}_X^*)$ the virtual displacement vectors $\mbf{u}$ (in-plane) and $\mbf{w}$ (out-of-plane) are linked by a convex quadratic constraint -- the pair is therefore not a linear program as in the case of truss optimization problem. Drawing upon the achievements in finite dimensional convex optimization, cf. \cite{ben2001}, we convert $(\mathcal{P}_X)$, $(\mathcal{P}^*_X)$ to a pair of mutually dual conic quadratic programming problems that may be tackled by powerful interior point methods, cf. \cite{andersen2003}. Based on solutions of problems $(\mathcal{P}_X)$,\,$(\mathcal{P}_X^*)$ a grid-shell approximation of the optimal vault may be constructed: it is extracted from the 3D truss obtained through elevating the nodes $X$ of the plane ground structure by $\mbf{z} = \frac{1}{2} \,\mbf{w}$. From Theorems \ref{thm:recovering_MV_grid-shell}, \ref{thm:recovering_MC_grid-shell} we find that, with thicknesses of the grid-shell's bars induced by $\mbf{s}$, we arrive at an optimal design amongst all grid-shells spanned over the plane ground structure -- both in the plastic and the elastic setting.

Section \ref{sec:numerics} presents the numerical simulations. The adaptive \textit{member adding approach} developed in \cite{gilbert2003} for the LP truss optimization problem (cf. also \cite{sokol2015} and \cite{he2019}) was successfully converted to the conic program $(\mathcal{P}_X)$,\,$(\mathcal{P}_X^*)$, thereby making it possible to solve large scale problems with 2D ground structures counting up to several billion potential members while using a laptop computer. The high resolution in the reference domain $\Omega$ transfers directly to high resolution on the surface $\mathcal{S}_z$ which is very difficult to match by 3D ground structure methods that require discretization in the third, vertical direction. The numerical scheme is demonstrated through a number of examples of load conditions for a square domain $\Omega$; a non-convex cross-shaped design domain is considered as well. A study of the obtained numerical solutions is given. 

In Section \ref{sec:variations} we discuss possible variations of the optimal vault problem. First, the kinematical support of the designed vault is relaxed to be any closed set $\Gamma \subset \Ob$ and not necessarily the boundary $\bO$; in particular $\Gamma$ may be a finite set that simulates point supports. Next we consider the case when the supporting boundary is not horizontal and may be obtained by elevation of the plane boundary $\bO$ instead. Finally, by constraining principal directions of $\sigma$ in $(\mathcal{P})$, we put forward a link between  the optimal vault problem and the optimal arch-grid problem that motivated this very work in the first place. The paper concludes with Sections \ref{sec:final_remarks} where several open problems are outlined.

\vspace{0.5cm}

\noindent\textbf{Notation:} The set of non-negative reals will be denoted by $\R_+$, the symbols $\R^d$, $\mathrm{T}^d$ will be used for $d$-dimensional vectors and second order symmetric tensors/matrices, whilst $\mathrm{T}_+^d$ will stand for the positive semi-definite matrices. For vectors $a,b \in \R^d$ and matrices $A,B \in \mathrm{T}^d$ by $a \cdot b$ and $A:B$ we shall understand the standard scalar products, while $\abs{a}$ will be the Euclidean norm. For two symmetric matrices $A,B \in \mathrm{T}^d$ by writing $A \preceq B$ we shall compare the two induced quadratic forms, i.e. $A \preceq B \ \Leftrightarrow \ A:(\tau\otimes \tau) \leq B:(\tau\otimes \tau)$ for any $\tau \in \R^d$. Next, for a bounded domain $\Omega \subset \R^d$ (open and connected set), $C^k(\Ob;V)$ will be the space of $V$-valued functions (e.g. $V = \R^d$) that are continuously $k$-differentiable up to the boundary; $L^p(\Omega;V)$ will stand for the Lebesgue space of $p$-integrable $V$-valued functions. Finally, by $\Mes(\Ob;V)$ we shall denote the space of Radon $V$-valued measures, in particular $\Mes_+(\Ob) = \Mes(\Ob;\R_+)$ is the set of positive measures containing e.g. Lebesgue measure $\mathcal{L}^2$, Dirac delta measure $\delta_{x_0}$ at point $x_0$ or  $1$-dimensional Hausdorff measure on a rectifiable curve $\Ha^1 \mres \mathrm{C}$, where by symbol $\mu \mres B$ we understand restriction of measure $\mu$ to a Borel set $B$. Integrals over $\Ob$ written as $\int_\Ob f$ will be understood twofold: either with respect to Lebesgue measure once $f$ is a function in $L^1(\Omega;\R)$ or, more generally, with respect to Radon measure $f \in \Mes(\Ob;\R)$. The divergence operator shall be applied in the distributional sense with respect to some open set, for instance $\Omega$, i.e. for a vector field $q$ (integrable function or a measure) $\Lambda = \dive\, q$ is a distribution such that $\Lambda(\varphi) = -\int_{\Omega} \nabla \varphi \cdot q$ for any smooth test function $\varphi$ with compact support in $\Omega$. In case of tensor valued field $\sigma$ we shall use the upper case symbol $\DIV \,\sigma$ for distinction. Throughout the text the hat symbol $\hat{\argu}$ will be used to stress that the object is "three dimensional": either $\hat{f}$ will be a function/measure defined on $\R^3$, or it will admit values in $\R^3$ (alternatively in  $\mathrm{T}^3$), or both.

\section{A short review: convex analysis and theory of plane Michell structures}
\label{sec:review}

\subsection{Basic tools of convex analysis employed in the paper}
\label{ssec:convex_analysis}

For convenience of the reader we recall the basic notions of convex analysis that are employed in this paper; we give \cite{rockafellar1970convex} as a reference. For any set $B$ in a linear space $Y$ the indicator function will be denoted by $\mathbbm{I}_B:Y \to \Rb$ where $\Rb = \R \cup \{-\infty,\infty\}$ is the extended real line, namely $\mathbbm{I}_B(y) = 0$ for $y \in B$ and $\mathbbm{I}_B(y) = \infty$ whenever $y \notin B$. Once $Y$ is a normed space we may assign its dual space $Y^*$ (equivalent to the space $Y$ itself once $Y$ is finite dimensional) along with a duality pairing $\pairing{y;y^*}$ for $y\in Y,\ y^*\in Y^*$ (e.g. typically $\pairing{y;y^*} = y \cdot y^*$ when $Y = \R^n \equiv Y^*$). Then, any extended real convex function $j:Y \to \Rb$ enjoys its convex conjugate $j^*:Y^* \to \Rb$ given by the formula $j^*(y^*) = \sup_{y\in Y} \bigl\{ \pairing{y;y^*} - j(y) \bigr\}$. A particular class of convex functions is the one of \textit{gauges} i.e. convex functions $g:Y\to \Rb$ which are non-negative and positively 1-homogeneous; we say that a gauge $g$ is closed when in addition it is lower semi-continuous. To every closed gauge $g$ there corresponds a unique closed convex set $K \subset Y$ containing the origin such that $g(y) = g_K(y) := \inf \bigl\{ t\geq 0 \, \big\vert \, y\in t K \}$ and, vice versa, each such set $K$ induces a closed gauge $g_K$.
For each gauge function $g:Y \to \Rb$ we define its polar function $g^0:Y^* \to \Rb$ via $g^0(y^*):= \sup_{y\in Y} \bigl\{ \pairing{y;y^*} \, \big\vert\, g(y) \leq 1  \bigr\}$ which is a closed gauge. For a closed set $K$ we note that $g_K$ is the unique positively 1-homogeneous function that gives the equivalence $g_K(y)\leq 1 \ \Leftrightarrow \ y\in K$ hence follows an equality that will be of particular importance in this work:
\begin{equation}
	\label{eq:polar}
	g^0_K(y^*) = \sup_{y\in Y} \Big\{ \pairing{y;y^*} \, \big\vert\, y \in K  \Big\}.
\end{equation}
The function $g_K^0$ is the so called \textit{support function} of the convex set $K$.
Finally, for any gauge $g$ the function $j=j(y) = \frac{1}{2}\bigl(g(y) \bigr)^2$ yields $j^*(y^*) = \frac{1}{2} \bigl(g^0(y^*) \bigr)^2$.

\subsection{Theory of plane Michell structures and an interpretation of the link to least volume trusses}
\label{ssec:Michell}

For a plane and convex domain $\Omega \subset \Rd$ let there be given a load being a vector-valued measure $F \in \Mes(\Ob;\Rd)$ (with the use of measures we may account for body forces as well as point forces, also concentrated on the boundary $\bO$) and a closed subset of the boundary $\bO_0 \subset \bO$ where the potential structure may be kinematically fixed.
The problem of finding the plane Michell structure is inextricably linked to the closed gauge  $\gamma: \Sdd \to \R_+$ that happens to be the spectral norm:
\begin{equation}
	\label{eq:spectral_norm}
	\gamma(\eps) = \norm{\eps}_\mathrm{spec} = \max \Big\{\abs{\lambda_1(\eps)},\abs{\lambda_2(\eps)} \Big\} = \sup_{\tau \in \Rd, \ \abs{\tau}\leq 1} \abs{\eps:(\tau \otimes \tau)},
\end{equation}
where $\lambda_i(A)$ for a symmetric matrix $A$ is its $i$-th eigenvalue. Upon setting the duality pairing $\pairing{\eps;\sigma}$ as the scalar product $\eps :\sigma$ the polar gauge $\gamma^0:\Sdd \to \R_+$ can be computed:
\begin{equation*}
	\gamma^0(\sigma) = \abs{\lambda_1(\sigma)}+\abs{\lambda_2(\sigma)}.
\end{equation*}
The renowned Michell problem reduces to a pair of mutually dual variational problems (see e.g. \cite{lewinski2019a} or \cite{bouchitte2008}):
\begin{align}
\label{eq:Michell_min}
\mathcal{V}_\mathrm{Michell} &=\min\limits_{\sig \in \Mes(\Ob;\Sdd)} \biggl\{ \int_\Ob \gamma^0(\sigma) \ \biggl\vert \,  -\DIV\, \sigma = F \ \text{ in } \R^2\backslash \bO_0 \biggr\}\\
\label{eq:Michell_max}
&= \sup_{u \in C^1(\Ob;\Rd)} \biggl\{ \int_\Ob u\cdot F \ \biggl\vert \, u = 0 \text{ on } \bO_0, \  \gamma\bigl( e(u) \bigr) \leq 1 \text{ point-wise in } \Ob   \biggr\},
\end{align}
where $e(u)$ stands for the symmetric part of the gradient $\nabla u$, i.e. $$e(u) = \frac{1}{2}\biggl(\nabla u + (\nabla u)^\top\biggr).$$
The minimization problem \eqref{eq:Michell_min} has an objective with the integrand of linear growth thus the existence of solution is established (cf. \cite{bouchitte2008}) within the space of tensor-valued measures $\sigma \in \Mes(\Ob;\Sdd)$ which encompasses lower dimensional stress paths such as curved cables of finite cross section area (see Section \ref{ssec:regularity} for more insight). The equilibrium equation  $-\DIV\, \sigma = F$ is intended in the sense of distributions in the open set $\R^2\backslash \bO_0$ thus incorporating a natural boundary condition of the type $"\sigma \nu = p"$ on $\bO \backslash \bO_0$.

The stress field $\sigma \in \Mes(\Ob;\Sdd)$ and the smooth displacement function $u \in C^1(\Ob;\Rd)$ solve the problems \eqref{eq:Michell_min} and \eqref{eq:Michell_max}, respectively, if and only if the optimality conditions are met: 
\begin{align}
\label{eq:opt_cond_Michell}
\begin{cases}
(i)& u=0 \ \text{ on $\bO_0$}, \quad \gamma\bigl( e(u) \bigr) \leq 1 \text{ everywhere in } \Ob;\\
(ii) &   -\DIV\, \sigma = F \quad \text{ in } \R^2\backslash \bO_0; \\
(iii) &  e(u):\sigma = \gamma^0(\sigma) \quad \text{ and } \quad \gamma\bigl( e(u) \bigr) = 1 \quad \sigma\text{-a.e. in } \Ob.
\end{cases}
\end{align}
Naturally the first two conditions are the admissibility constraints for $u$ and $\sigma$ hence the essence lies in the point-wise condition $(iii)$ that is often referred to as the \textit{extremality condition}. The condition $\gamma\bigl( e(u) \bigr) = 1$ must be satisfied only $\sigma$ almost everywhere $(\sigma\text{-a.e.})$, namely only at points of non-zero stress or, effectively, where the material occurs.

In contrast to the minimization problem \eqref{eq:Michell_min} the displacement-based maximization problem \eqref{eq:Michell_max} does not have a solution in general, i.e. one that is of class $u\in C^1(\Ob;\Rd)$. By virtue of Lemma 2.1 in \cite{bouchitte2008} the point-wise constraint $\gamma\bigl(e(u) \bigr) \leq 1$ may be equivalently put as a two-point condition:
\begin{equation}
	\label{eq:two_point_Michell}
	-\abs{y-x} \leq \bigl(u(y) - u(x) \bigr)\cdot \frac{y-x}{\abs{y-x}} \leq \abs{y-x} \qquad \forall (x,y) \in \Ob \times \Ob
\end{equation}
which does not require differentiability of $u$, merely its continuity; it should be recalled that $\Omega$ is here assumed to be convex. Existence of solution $u\in C(\Ob;\Rd)$ in the problem \eqref{eq:Michell_max} can be readily proved, see \cite{bouchitte2008}.

One has to carefully note that the possibility of rewriting the kinematic constraint in the two-point manner \eqref{eq:two_point_Michell} strongly relies on the very particular form of the gauge $\gamma$ being the spectral norm.
This structure of the constraint affects the choice of the numerical approach towards Michell problem or, more accurately, towards the maximization problem \eqref{eq:Michell_max}: although the problem is \textit{a priori} posed as continuous the natural alternative for the finite element method is a discrete approach where the vector function $u$ is determined only on a finite subset $X$ of $\Ob$. Upon fixing a Cartesian basis $(e_1, e_2)$ of the plane $\Rd$ such function $u$ is represented by a pair of column vectors $\mbf{u}_1,\mbf{u}_2 \in \R^n$ where $n = \#(X \backslash \bO_0)$ ($\#(A)$ is the cardinality of a set $A$). With the load $F$ suitably discretized to $\mbf{F}_1,\mbf{F}_2 \in \R^n$ the discrete maximization problem may be put forward:
\begin{align}
	\label{eq:truss_max}
	\max_{\mbf{u}_1, \mbf{u}_2 \in \R^n} \biggl\{ \mbf{F}_1^\top\! \mbf{u}_1 + \mbf{F}_2^\top\! \mbf{u}_2 \ \biggl\vert \, - \mbf{l} \leq \mbf{B}_1\, \mbf{u}_1 + \mbf{B}_2\, \mbf{u}_2 \leq \mbf{l}    \biggr\}
\end{align}
where the constraints are the matrix-vector version of the two-point constraint \eqref{eq:two_point_Michell} written for the pairs $(x,y) \in X \times X$; matrices $\mbf{B}_1, \mbf{B}_2$ are the geometric matrices consisting of directional cosines of segments $[x,y]$ while vector $\mbf{l}$ stores lengths of these segments, cf. Section \ref{ssec:discretication} for more details. The linear programming (LP) proposed above attains its dual (or in fact primal) formulation
\begin{equation}
	\label{eq:truss_min}
	\min_{\mbf{s} \in \R^m} \left\{ \sum_{k=1}^m l_k \abs{s_k}  \ \biggl\vert \,  \mbf{B}_1^\top \mbf{s} = \mbf{F}_1,\ \mbf{B}_2^\top\, \mbf{s} = \mbf{F}_2 \right\}.
\end{equation}
In the minimization problem \eqref{eq:truss_min} we recognize a reduced version of the minimum volume truss problem where we search for the axial tensile/compressive forces $s_k$ in bars interconnecting all the nodes of the grid $X$. This dense universe of bars is known in the literature as the \textit{ground structure} while the pair of LP problems \eqref{eq:truss_max}, \eqref{eq:truss_min} has now long been used to find very precise numerical predictions of Michells structures, see \cite{dorn1964,gilbert2003,sokol2015, zegard2014}.

\section{Vaults of the least volume and the link to a 2D convex problem}
\label{sec:plastic_design}

\subsection{Formulation of the form finding problem in the plastic setting}
\label{ssec:problem_formulation_plastic}

Over a bounded domain $\Omega$ with a "smooth" boundary (see Remark \ref{rem:bO_smoothness} below) lying within a horizontal base plane $\Rd$ we shall design a vault being a surface $\mathcal{S}_z = \bigl\{\hat{x}=\bigl(x,z(x)\bigr)\, \big\vert \, x\in \Omega \bigr\} \subset \R^3$ that is pinned on the boundary $\bO$ or, more precisely, $\partial \mathcal{S}_z  = \bO \times \{0\}$,  cf. Fig. \ref{fig:form_finding_problem}. We are in fact seeking the \textit{elevation function} $z \in C^1(\Ob;\R)$ such that $z = 0$ on $\bO$; the function $z$ is allowed to change sign, i.e. the surface may lie above ($z \geq 0$) and below $(z \leq 0)$ the base plane. 

\begin{figure}[h]
	\centering
	\includegraphics*[trim={-1cm 0.5cm 1cm 4.2cm},clip,width=0.9\textwidth]{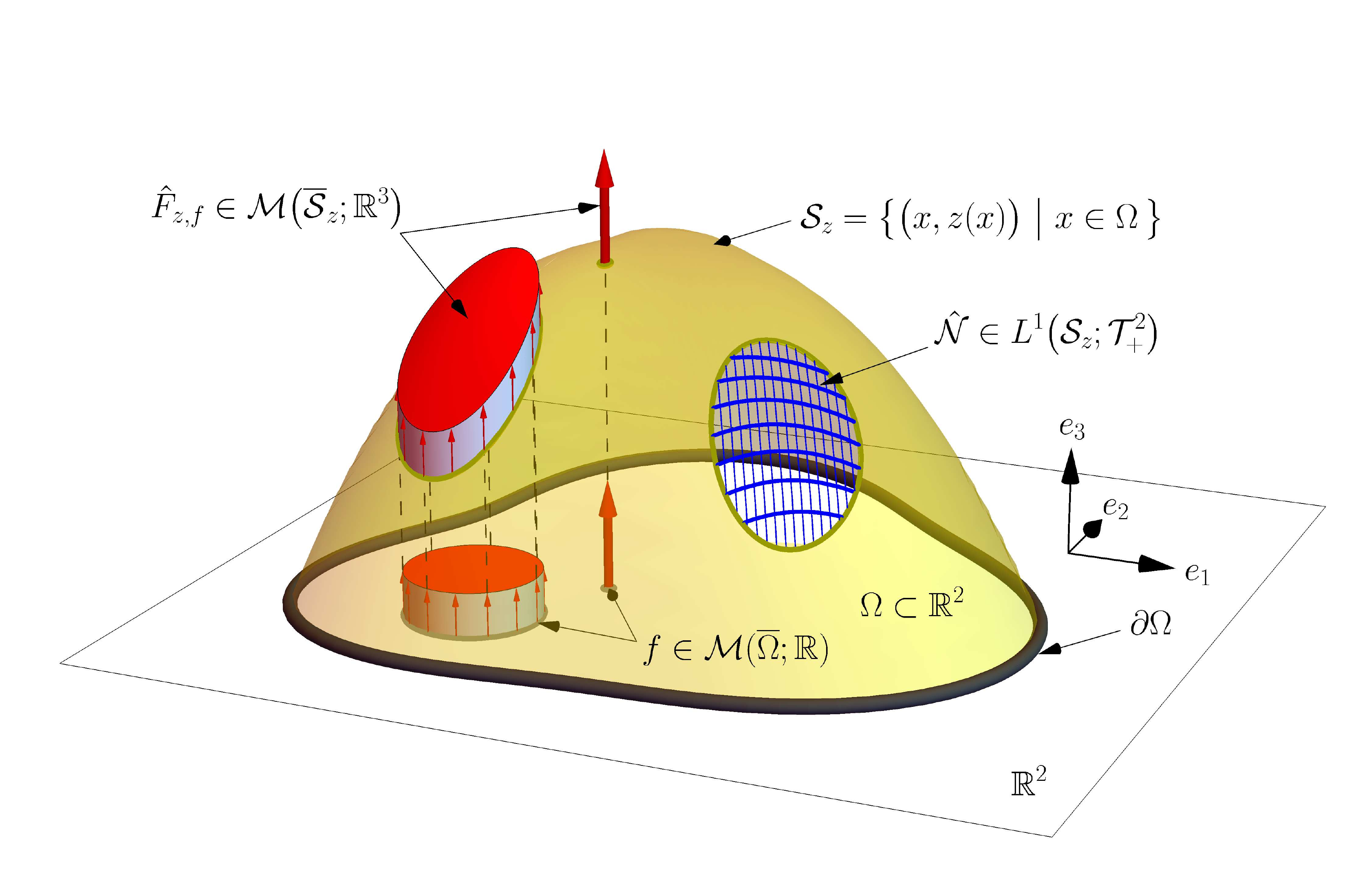}
	\caption{The setting of the form finding problem: optimal design of a "vault in tension".}
	\label{fig:form_finding_problem}       
\end{figure}

Upon the surface there will act a load that \textit{vertically tracks} the surface $\mathcal{S}_z$ being a design variable. The intensity of the load is given with respect to the base plane, i.e. "per unit area of $\Omega$", therefore the tracking load can be represented as a signed Radon measure $f \in \Mes(\Ob;\R)$ (positive means acting upwards) where, for a given function $z$, the actual load $\hat{F}=\hat{F}_{z,f} \in \Mes(\ov{\mathcal{S}}_z;\R^3)$ is the one that for any open subdomain $\omega \subset \Omega$ gives $\int_{\mathcal{S}_z(\omega)} \hat{F}_{z,f} = \int _\omega f e_3$ where $\mathcal{S}_z(\omega) = \mathcal{S}_z \cap (\omega \times \R)$ and $e_3$ is the unit vertical vector, see the visualization in Fig. \ref{fig:form_finding_problem}.  Note that since $f$ is a measure it encompasses loads distributed over 2D sets, "knife loads" (distributed along curves) and point forces.

The vault will be designed as a membrane shell that is capable of withstanding stresses of one sign only: we can choose between tension or compression. In the context of a vault, that is typically loaded gravitationally, i.e. with $f \leq 0$, it is more natural to consider the "compression setting" of the problem. On the other hand it will be far more convenient and clear from the mathematical point of view once the tension setting is chosen and we shall do exactly that -- the stress resultants will be thus membrane forces $\hat{\mathcal{N}}$ that point-wise on $\mathcal{S}_z$ admit values in $\mathcal{T}_+^{\,2}$ being the set of symmetric positive semi-definite tensors. In Section \ref{ssec:compression} it will be showed that one may switch between the optimal designs in tension and compression through a simple change of signs; by exploiting this fact the numerical solutions will be given for the compression setting.

Apart from the surface $\mathcal{S}_z$ the second design variable shall be the material distribution of the vault. In this aspect we draw upon the theory of Michell structures: assuming a perfectly plastic material of the yield stress $k_0$ we are designing a fibrous structure by aligning the fibres along the principal lines of $\hat{\mathcal{N}}$ so that the fibres are fully stressed. The elementary volume of such structure evaluated point-wise on $\mathcal{S}_z$ reads $dV = \frac{1}{k_0} \bigl(\vert\lambda_{1}(\hat{\mathcal{N}}) \vert+\vert\lambda_{2}(\hat{\mathcal{N}})\vert \bigr)\, d\Ha^2$ which, considering positivity of $\hat{\mathcal{N}}$, reduces to $dV = \frac{1}{k_0} \tr\,\hat{\mathcal{N}}\, d\Ha^2$. This reasoning paves a way to posing the problem of \textit{Minimum Volume Vault} in the plastic setting:
\begin{equation*}\tag*{$(\mathrm{MVV})$}
	V_\mathrm{min}=\inf_{\substack{z \in C^1\!(\Ob;\R) \\ \hat{\mathcal{N}} \in L^1\!(\mathcal{S}_z;\mathcal{T}_+^{\,2})}} \left\{ \frac{1}{k_0}\int_{\mathcal{S}_z} \tr\,\hat{\mathcal{N}}\, d \Ha^2 \ \biggl\vert \, z = 0 \text{ on } \bO, \ \hat{\mathcal{N}} \in \Sigma\bigl(\mathcal{S}_z,\hat{F}_{z,f}\bigr) \right\} 
\end{equation*}
where by $\Sigma\bigl(\mathcal{S}_z,\hat{F}_{z,f}\bigr)$ we understand the set of \textit{statically admissible membrane forces fields}, namely the fields $\hat{\mathcal{N}}$ that satisfy the equilibrium equations of the membrane shell. It is fair to deem the formulation of $(\mathrm{MVV})$ rather quick, in particular deriving the volume functional requires a more careful usage of differential geometry. Nonetheless, it was the author's intention that this rather tedious step is skipped here, which may be justified by the equivalence between $(\mathrm{MVV})$ and the Prager problem $(\mathrm{MVPS}_H)$ that will be posed in Section \ref{sec:3D} using the full 3D framework, i.e. without differential geometry.

The form finding problem $(\mathrm{MVV})$ is posed in a formal way: the field $\hat{\mathcal{N}}:\mathcal{S}_z \to \mathcal{T}_+^{\,2}$ is an intrinsic tensor field (which is emphasized by the calligraphic font both in $\hat{\mathcal{N}}$ and $\mathcal{T}_+^{\,2}$) and the trace $\tr\,\hat{\mathcal{N}}$ must be understood intrinsically as well. In order to proceed we shall introduce a parametrization of the surface $\mathcal{S}_z$ whereas here the only natural choice is the following:
\begin{equation}
	\label{eq:parametrization}
	\Omega \ni x \quad \mapsto \quad \hat{x} = \hat{z}(x) := \bigl(x, z(x) \bigr) \in \Omega \times \R,
\end{equation}
by which we understand, in a more classical format, as $x=(x_1,x_2) \mapsto \hat{x} = \bigl(x_1,x_2, z(x_1,x_2) \bigr)$ once $x = x_1\, e_1 + x_2 \, e_2$. The Cartesian parametrization of $\Omega$, however, will be immaterial to us (until Section \ref{sec:discrete} when the discrete setting will be covered) since we shall stick to absolute notation for plane vectors in $\Rd$ and matrices in \nolinebreak $\Sdd$.

A reformulation of the classical equations of membrane shell theory that is specifically tailored for the parametrization \eqref{eq:parametrization} may be found in Chapter 11.2 of \cite{green1968} which we will follow in the remainder of this section. The matrix $G_z$ of covariant components of the metric tensor on the surface $\mathcal{S}_z$ (with respect to par. \eqref{eq:parametrization}) can be computed along with the Jacobian $J_z$ of the parametrization:
\begin{equation*}
	G_z = \mathrm{I} + \nabla z \otimes \nabla z, \qquad J_z = \sqrt{\mathrm{det}\,G_z} = \sqrt{1 + \vert\nabla z \vert^2},
\end{equation*}
where $\mathrm{I}$ is the $2 \times 2$ identity matrix. Henceforward $\eta: \Omega \to \Sddp$ will be the field of contravariant components of the membrane force tensor $\hat{\mathcal{N}}$. Point-wise there holds an equality
\begin{equation*}
	\tr\, \hat{\mathcal{N}}\!(\hat{x}) = G_z(x) : \eta(x)
\end{equation*}
where $\hat{x} = \hat{z}(x) \in \R^3$. Owing to the transformed equilibrium equations (11.2.5-6) in \cite{green1968} we find that $\hat{\mathcal{N}} \in \Sigma(\mathcal{S}_z, \hat{F}_{z,f})$ if and only if there exists a matrix field $\sigma:\Omega \mapsto \Sddp$ such that:
\begin{equation*}
	\sigma = J_z\, \eta, \qquad \DIV \, \sigma = 0, \quad -\dive \bigl( \sigma \nabla{z} \bigr) = f,
\end{equation*}
where the two last equations are intended in the sense of distributions in $\Omega$, i.e. there is no natural boundary condition involved. With the use of the change of variable formula we may readily rewrite the form finding problem $(\mathrm{MVV})$ in its parametrized setting:
\begin{equation*}\tag*{$(\mathrm{MVV}_\Omega)$}
	\Vmin=\inf_{\substack{z \in C^1\!(\Ob;\R) \\ \eta \in L^1\!(\Omega;\Sddp)}} \left\{ \int_{\Omega} \bigl(G_z : \eta \bigr) J_z \, \left\vert \, z = 0 \text{ on } \bO,
	\begin{array}{c}
	\DIV \, \sigma = 0,\\
	-\dive \bigl( \sigma \nabla{z} \bigr) = f 
	\end{array}
	 \text{ in } \Omega, \ \sigma = J_z \, \eta \right.\right\}
\end{equation*}
where we have put $\Vmin = k_0 \, V_\mathrm{min}$ to eliminate $k_0$. The next natural step is to utilize the very simple relation $\sigma = J_z \, \eta$ to dispose of the variable $\eta$ and arrive at:
\begin{equation}
	\label{eq:MVFD_sigma}
	\Vmin=\inf_{\substack{z \in C^1\!(\Ob;\R) \\ \sigma \in L^1\!(\Omega;\Sddp)}} \left\{ \int_{\Omega} \Big(\tr \, \sigma + \sigma : \bigl(\nabla z \otimes \nabla z \bigr)\Big)  \, \left\vert \, z = 0 \text{ on } \bO,
	\begin{array}{c}
	\DIV \, \sigma = 0,\\
	-\dive \bigl( \sigma \nabla{z} \bigr) = f 
	\end{array}
	\text{ in } \Omega \right.
	\right\},
\end{equation}
where point-wise  $\tr\, \sigma = \mathrm{I}:\sigma$ is the classical trace of a matrix $\sigma$.

The last formulation \eqref{eq:MVFD_sigma} is much more transparent than its predecessors, although the underlying mathematical structure is still far from clear, e.g. there is no convexity with respect to the pair $(z,\sigma)$ which rules out classical variational methods of examining the existence of solution.  The core idea of this paper connects our form finding problem to a convex problem that emerged spontaneously in the recent analysis of the optimal pre-stressed membrane problem, see \cite{bouchitte2020}.

\subsection{The underlying 2D convex problems $(\mathcal{P})$,\,$(\mathcal{P}^*)$}
\label{ssec:2D_convex_program}

We start the passage with a remark: apart from the boundary condition $z=0$ on $\bO$ the elevation function $z$ enters the formulation \eqref{eq:MVFD_sigma} only by means of gradient $\nabla z$. It is thus clear that if in place of $\nabla z$ we put a vector function $\zeta \in C(\Ob;\Rd)$ as an independent design variable we would obtain a non-greater value of the infimum  (not every function $\zeta$ is a gradient of some function $z$). Taking a step further we may relax the continuity condition and choose from all Borel measurable functions $\zeta$. Next we may change variables by introducing vector valued field $q = \sigma \zeta $ while, at the formal level, we have $\sigma:(\zeta \otimes\zeta) = (\sigma \zeta) \cdot \zeta = \bigl(\sigma\, (\sigma^{-1} q)\bigr) \cdot (\sigma^{-1} q) = q \cdot (\sigma^{-1} q)$. Finally, we may further relax the regularity conditions on $\sigma, q$ by allowing that they are Radon measures thus arriving at the problem:
\begin{tcolorbox}
\vspace{0.17cm}
\begin{equation*}\tag*{$(\mathcal{P})$}
\Z:=\inf_{\substack{\sigma \in \Mes(\Ob;\Sddp) \\ q \in \Mes(\Ob;\Rd)}} \left\{ \int_{\Ob} \Big( \tr \, \sigma + (\sigma^{-1}q)\cdot q  \Big) \, \left\vert \,
\begin{aligned}
\DIV \, \sigma &= 0\\
-\dive\, q &= f 
\end{aligned}
\ \text{ in } \Omega \right. \right\}
\end{equation*}
\end{tcolorbox}
It should be stressed that the integral $\int_{\Ob} \,(\sigma^{-1}q)\cdot q$ is intended formally: first, $\sigma$ is merely positive semi-definite and, secondly, $\sigma,q$ are measures. Below, however, we will show that this integral is after all meaningful in the sense of theory of convex functionals on measures. It is remarkable that the transformed equilibrium equations $\DIV \, \sigma = 0, \ -\dive\, q = f$ are reminiscent of the plate theory where $\sigma$ are the in-plane membrane forces while $q$ is the transverse shear force. 

The problem $(\mathcal{P})$ above has been obtained by relaxing the problem \eqref{eq:MVFD_sigma} therefore we have
\begin{equation}
	\label{eq:Z_leq_Vmin}
	\mathcal{Z} \leq \Vmin.
\end{equation}
The important question concerns the opposite inequality or more accurately: assuming that problem $(\mathcal{P})$ attains a solution $(\sigma,q)$ does there exist a function $z \in C^1(\Ob;\R)$ with $z=0$ on $\bO$ such that equality $q = \sigma \nabla z$ holds? We shall next address this issue by employing a duality argument developed in \cite{bouchitte2020}.

\subsubsection{Duality framework}
\label{sssec:duality}

After the work \cite{bouchitte2020} we put forward the problem that will turn out to be the dual of $(\mathcal{P})$:
\begin{tcolorbox}
\vspace{0.3cm}
\begin{equation*}\tag*{$(\mathcal{P}^*)$}
	\sup_{\substack{u \in C^1\!(\Ob;\Rd) \\ w \in C^1\!(\Ob;\R)}} \left\{\, \int_\Ob w\, f  \, \left\vert \, u=0, \, w=0 \text{ on } \bO, \ \ \frac{1}{4}\, \nabla w \otimes \nabla w +e(u) \preceq  \mathrm{I} \ \text{ in } \Ob \right. \right\}
\end{equation*}
\end{tcolorbox}
\noindent where $\mathrm{I}$ is the $2 \times 2$ identity matrix.
\begin{remark}
	\label{rem:different_factor}
	Two minor differences with respect to the paper \cite{bouchitte2020} should be emphasized: the roles of symbols $u$ and $w$ are interchanged (due to a different context) while the quotient $\frac{1}{2}$ therein is replaced by $\frac{1}{4}$ here, which is only a matter of rescaling the solution $w$.
\end{remark}

In order to show duality between the problems $(\mathcal{P})$ and $(\mathcal{P}^*)$ we shall first rewrite the point-wise constraint in $(\mathcal{P}^*)$ to reveal that the set of admissible pairs $(u,w)$ is convex. Since the operation $(u,w) \mapsto \bigl(e(u),\nabla w\bigr)$ is linear this issue amounts to showing that the set
\begin{equation*}
\mathscr{C} := \left\{ (\eps,\vartheta)\in \Sdd \times \Rd  \ \left\vert \ \frac{1}{4} \,\vartheta \otimes \vartheta + \eps \preceq  \mathrm{I} \right.  \right\}
\end{equation*}
is convex in $\Sdd \times \Rd$; henceforward the symbols $\eps \in \Sdd$ and $\vartheta \in \Rd$ should be associated with $e(u)$ and, respectively, $\nabla w$ at a point. An auxiliary function as follows will prove useful:
\begin{equation}
	\label{eq:c}
	c(\eps,\vartheta) := \sup\limits_{\tau \in \Rd, \ \abs{\tau} \leq 1} \left\{ \Big(\frac{1}{4} \,\vartheta \otimes \vartheta + \eps \Big):(\tau \otimes \tau)  \right\}
	= \sup\limits_{\tau \in \Rd, \ \abs{\tau} \leq 1} \left\{ \frac{1}{4} (\vartheta \cdot \tau)^2 + \eps:(\tau\otimes\tau)   \right\};
\end{equation}
assuming that $\lambda_1(\argu)$ is always the biggest eigenvalue of a matrix an alternative formula may be given: $c(\eps,\vartheta) = \max\left\{\lambda_1\left(\frac{1}{4} \,\vartheta \otimes \vartheta + \eps \right) ,0\right\}$. For a fixed vector $\tau \in \Rd$ it is straightforward to check that the function $(\eps,\vartheta) \mapsto \frac{1}{4} (\vartheta \cdot \tau)^2 + \eps:(\tau\otimes\tau)$ is convex and continuous. Then, owing to the latter formula in \eqref{eq:c} where $c$ is a point-wise supremum of convex continuous functions, convexity and lower semi-continuity of $c$ may be inferred, cf. \cite{rockafellar1970convex}. Almost by definition there holds $\mathscr{C} = \bigl\{ (\eps,\vartheta)\in \Sdd \times \Rd  \ \big\vert \ c(\eps,\vartheta) \leq 1  \bigr\}$ hence the convexity and closedness of the set $\mathscr{C}$. It should be noted that $\mathscr{C}$ is unbounded as it contains half-lines $\bigl\{t\,(\eps,0) \,\big\vert\, t\geq 0 \bigr\}$ whenever $\eps$ is negative semi-definite (the function $c$ is zero on such half-lines).

With $\mathscr{C}$ being convex and closed we may assign to it its gauge function $g_\mathscr{C}:\Sdd \times \Rd \to \Rb$ being the convex l.s.c. positively 1-homogeneous function that satisfies $g_\mathscr{C}(\eps,\vartheta) \leq 1 \ \Leftrightarrow (\eps,\vartheta) \in \mathscr{C}$ (note that $c$ is not a gauge since it is not 1-homogeneous). The point-wise constraint in $(\mathcal{P}^*)$ gets to be rewritten as $g_\mathscr{C}\bigl(e(u),\nabla w \bigr) \leq 1$ in $\Ob$ and we discover certain analogy to the Michell's maximization problem \eqref{eq:Michell_max} where the constraint reads $\gamma\bigl(e(u) \bigr) \leq 1$ and $\gamma:\Sdd \to \Rb$ is a gauge as well. In order to examine duality between $(\mathcal{P})$ and $(\mathcal{P}^*)$ the polar gauge $g_\mathscr{C}^0$ ought to be computed. To that aim we choose a natural pairing between the spaces $Y = \Sdd \times \Rd \equiv Y^*$, namely for $(\eps,\vartheta) \in Y$ and $(\sigma,q)\in Y^*$ we set $\pairing{(\eps,\vartheta)\,;(\sigma,q)} := \eps:\sigma + \vartheta \cdot q$. From \eqref{eq:polar} follows the formula
\begin{equation*}
	g_\mathscr{C}^0(\sigma,q) = \sup\limits_{\eps \in \Sdd, \ \vartheta\in \Rd} \biggl\{ \eps:\sigma + \vartheta\cdot q \ \bigg\vert \  \frac{1}{4} \,\vartheta \otimes \vartheta + \eps \preceq  \mathrm{I} \biggr\}.
\end{equation*}
First we observe that for $\sigma$ that is not positive definite there exists $\eps_- \preceq 0$ such that $\eps : \sigma >0$ and since $\bigl\{t\,(\eps_-,0) \,\big\vert\, t\geq 0 \bigr\} \subset \mathscr{C}$ we have $g_\mathscr{C}^0(\sigma,q) = \infty$ for any $q$. Then, assuming $\sigma \in \Sddp$, for a fixed $\vartheta$ the maximization with respect to $\eps$ is solved by choosing $\eps = \mathrm{I} -  \frac{1}{4} \,\vartheta \otimes \vartheta$ being the biggest matrix (with respect to relation $A \preceq B$) that satisfies the constraint. A non-constrained problem in $\vartheta$ emerges: $g_\mathscr{C}^0(\sigma,q) = \tr\,\sigma + \sup_{\vartheta \in \Rd} \bigl\{\vartheta \cdot q - \frac{1}{4}\sigma:(\vartheta \otimes \vartheta)  \bigr\}$. Two scenarios may occur: either $q \notin \IM(\sigma)$ and then again $g_\mathscr{C}^0(\sigma,q) = \infty$, or $q \in \IM(\sigma)$ and then the maximum is reached for any $\vartheta$ such that $q = \frac{1}{2} \sigma \vartheta$. Since for each such $\vartheta$ the product $\vartheta \cdot q$ is invariant the following formula is meaningful:
\begin{equation}
	\label{eq:g_C^0}
	g_\mathscr{C}^0(\sigma,q) = \begin{cases} \tr\, \sigma + (\sigma^{-1}q)\cdot q  & \text{if $\sigma\in \Sddp$ and  }q \in  \IM(\sigma),\\
	\infty & \text{otherwise}.
	\end{cases}
\end{equation}
We thus find that the functional being minimized in $(\mathcal{P})$ is exactly $\int_{\Ob} g_\mathscr{C}^0(\sigma,q)$ and, although the product $(\sigma^{-1}q)\cdot q$ seems vague for measures $\sigma, q$, the integral functional itself is well defined as $g_\mathscr{C}^0$ is a closed gauge, see \cite{bouchitte1988}. 

In order to see the duality relation between the problems $(\mathcal{P})$ and $(\mathcal{P}^*)$ we dispose of the equilibrium constraints in $(\mathcal{P})$ arriving at an $\inf$-$\sup$ problem for a suitable Lagrangian $\mathscr{L} = \mathscr{L}\bigl((\sigma,q);(u,w) \bigr)$:
\begin{equation*}
	\Z=\inf \mathcal{P} = \inf_{\substack{\sigma \in \Mes(\Ob;\Sddp) \\ q \in \Mes(\Ob;\Rd)}} \sup_{\substack{(u,w) \in C^1\!(\Ob;\R^3)\\(u,w)=0 \text{ on }\bO }} \left\{ \int_{\Ob} g_\mathscr{C}^0(\sigma,q) + \left(-\int_{\Ob} e(u):\sigma\right) + \left(-\int_\Ob \nabla w \cdot q + \int_{\Ob} w \, f \right)  \right\}.
\end{equation*}
Without dwelling on the question whether the order of $\inf$ and $\sup$ can be interchanged the inequality as below may always be given (along with rearranging the terms):
\begin{equation}
	\label{eq:weak_duality}
	\Z=\inf \mathcal{P}\geq \sup_{\substack{(u,w) \in C^1\!(\Ob;\R^3)\\(u,w)=0 \text{ on }\bO }} \inf_{\substack{\sigma \in \Mes(\Ob;\Sddp) \\ q \in \Mes(\Ob;\Rd)}} 
\left\{ \int_{\Ob} w \, f + \int_\Ob \biggl( - \Big\langle\bigl(e(u),\nabla w\bigr); \bigl(\sigma,q\bigr) \Big \rangle + g_\mathscr{C}^0(\sigma,q)\biggr)  \right\} = \sup \mathcal{P}^*.
\end{equation}
An explanation of the last equality above is in order: for fixed $(u,w)$ we must show that the infimum above equals $-\infty$ if there is a point $\bar{x} \in \Ob$ such that $\frac{1}{4} \nabla w(\bar{x}) \otimes \nabla w(\bar{x}) + e(u)(\bar{x}) \not\preceq \mathrm{I}$ or equivalently $g_\mathscr{C}\bigl(e(u)(\bar{x}),\nabla w(\bar{x}) \bigr) >1$. In that case, since $\bigl(g_\mathscr{C}^0\bigr)^{\!0} = g_\mathscr{C}$, there exist $(\bar{\sigma},\bar{q}) \in \Sddp \times \Rd$ such that: $g_\mathscr{C}^0(\bar{\sigma},\bar{q}) \leq 1$ and $\pairing{\bigl(e(u)(\bar{x}),\nabla w(\bar{x})\bigr); \bigl(\bar\sigma,\bar{q}\bigr)} > 1$. Then, for Dirac delta measures  $\sigma = \bar{\sigma}\,  \delta_{\bar{x}} \in \Mes(\Ob;\Sddp)$, \ $q= \bar{q}\,  \delta_{\bar{x}} \in \Mes(\Ob;\Rd)$ the second integral in \eqref{eq:weak_duality} attains negative value. Since the infimum is taken with respect to a cone $\Mes(\Ob;\Sddp) \times \Mes(\Ob;\Rd)$ its value must be $-\infty$. Once $\frac{1}{4} \nabla w \otimes \nabla w + e(u) \preceq \mathrm{I}$ in whole $\Ob$ we immediately see that the infimum is reached for $(\sigma,q) =(0,0)$.

The inequality $\inf \mathcal{P} \geq \sup \mathcal{P}^*$, sometimes referred to as the \textit{weak duality} result, was proved above to convince the reader of the duality relation between the two problems $(\mathcal{P})$ and $(\mathcal{P}^*)$. For the strong result we refer to \cite[Theorem 3.18]{bouchitte2020} where through more advanced duality tools we find:
\begin{theorem}
	\label{thm:duality}
	The problems $(\mathcal{P})$ and $(\mathcal{P}^*)$ are mutually dual convex problems, moreover
	\begin{equation*}
	\Z = \min \mathcal{P} = \sup \mathcal{P}^*<\infty,
	\end{equation*}
	where the problem $\mathcal{P}$ always attains a solution $(\sigma,q) \in \Mes(\Ob;\Sddp) \times \Mes(\Ob;\Rd)$.
\end{theorem}
In contrast to problem $(\mathcal{P})$ the problem $(\mathcal{P}^*)$ fails to have a solution in general, i.e. solution of the class $C^1$. Section \ref{ssec:regularity} discusses the strategy of relaxing the differentiability condition so that solution $(u,w)$ exists after all.

\begin{remark}
	\label{rem:bO_smoothness}
	The reader may find that throughout the paper \cite{bouchitte2020} convexity of $\Omega$ is assumed whereas in the present work this restriction is dropped. In fact, all the results from \cite{bouchitte2020} that are herein exploited do not rely on the convexity assumption. Proposition \ref{prop:regularitu_u_w} below, where convexity of $\Omega$ is explicitly imposed, is displayed for illustrative purposes only, i.e. it is not used in any of the proofs herein (moreover, Proposition \ref{prop:regularitu_u_w} may be easily extended to \textit{star-shaped} domains $\Omega$). In the end, none of the results in the present paper explicitly require any regularity of the boundary $\bO$. One should keep in mind, however, that assumptions on smoothness of functions  $(u,w)$ solving $(\mathcal{P}^*)$ (see e.g. Theorem \ref{thm:opt_cond} or Theorem \ref{thm:optimal_3D_structure} below) are difficult to meet if $\bO$ is not at least Lipschitz regular.
\end{remark}

\subsubsection{Optimality conditions for the pair of problems $(\mathcal{P})$ and $(\mathcal{P}^*)$}
\label{sssec:opt_cond}

So far we have passed from the original formulation $(\mathrm{MVV})$ to a pair of mutually dual problems $(\mathcal{P})$ and $(\mathcal{P}^*)$. Up till now the connection is not clear as it merely relies on the inequality \eqref{eq:Z_leq_Vmin}. To take the next step we shall derive the optimality conditions for the pair $(\mathcal{P})$,\,$(\mathcal{P}^*)$ resembling the conditions \eqref{eq:opt_cond_Michell} for Michell problem. Once more we directly draw upon the work \cite{bouchitte2020} and we repeat Theorem 4.1 therein (in fact we present the less general version, see Section \ref{ssec:3D_to_2D} below):

\begin{theorem}
\label{thm:opt_cond}
The pairs $(\sigma,q) \in L^1(\Omega;\Sddp) \times L^1(\Omega;\Rd)$ and $(u,w) \in C^1(\Ob;\Rd) \times C^1(\Ob;\R)$ are solutions of problems $(\mathcal{P})$ and $(\mathcal{P}^*)$, respectively, if and only if the following optimality conditions are met:
\begin{align}
\label{eq:opt_cond}
\begin{cases}
(i)& u=0, \, w =0 \ \text{ on $\bO$}, \qquad \frac{1}{4}\, \nabla w \otimes \nabla w + e(u) \preceq  \mathrm{I} \quad \text{everywhere in } \Ob;\\
(ii) &  \DIV\, \sigma = 0 \ \text{ in } \Omega, \quad -\dive \, q = f \ \text{ in } \Omega; \\
(iii) & \bigl( \frac{1}{4}\, \nabla w \otimes \nabla w + e(u) \bigr):\sigma = \tr\,\sigma \quad \text{ and }\quad \lambda_1\bigl( \frac{1}{4}\, \nabla w \otimes \nabla w + e(u) \bigr) = 1 \quad \sigma\text{-a.e. in } \Ob;\\
(iv) &   q = \frac{1}{2}\, \sigma\, \nabla w.
\end{cases}
\end{align}
where $\lambda_1(\argu)$ stands for the biggest eigenvalue of the matrix.
\end{theorem}
\begin{proof}
Conditions (i),\,(ii) are the admissibility conditions in problems $(\mathcal{P}^*)$,\,$(\mathcal{P})$ respectively hence they shall be assumed as true in the remainder of the proof. Then, owing to Theorem \ref{thm:duality} the pairs $(\sigma,q)$ and $(u,w)$ are solutions if and only if the global extremality condition is met: $\int_\Ob w f =\int_{\Ob} g_\mathscr{C}^0(\sigma,q)$. The distributional equilibrium equations $\DIV\,\sigma =0$, $-\dive q = f$ in $\Omega$ by definition imply that $\int_\Ob\, e(u):\sigma = 0$ and $\int_\Ob \nabla w \cdot q = \int_\Ob w \, f$ therefore the global condition may be rewritten:
\begin{equation*}
	\int_\Ob \pairing{\bigl(e(u),\nabla w\bigr); \bigl(\sigma,q\bigr)} = \int_{\Ob} g_\mathscr{C}^0(\sigma,q).
\end{equation*}
Due to (i) there holds $g_\mathscr{C}\bigl(e(u),\nabla w\bigr) \leq 1$, which implies that point-wise $\pairing{\bigl(e(u),\nabla w\bigr); \bigl(\sigma,q\bigr)} \leq g_\mathscr{C}^0(\sigma,q)$. Hence the equality above holds true if and only if $\pairing{\bigl(e(u),\nabla w\bigr); \bigl(\sigma,q\bigr)} = g_\mathscr{C}^0(\sigma,q)$ almost everywhere. Acknowledging \eqref{eq:g_C^0} and the fact that $\int_{\Ob} g_\mathscr{C}^0(\sigma,q)<\infty$ at optimality,
we find that $(\sigma,q)$ and $(u,w)$ are optimal if and only if there exists a vector function $\zeta$ such that $q= \sigma \zeta$ and equality $e(u):\sigma+ \nabla w \cdot (\sigma \zeta) = \tr\,\sigma + \sigma: (\zeta \otimes \zeta)$ holds. Careful computations give an equivalent form of this equation:
\begin{equation*}
	\left(\frac{1}{4}\, \nabla w \otimes \nabla w +e(u) \right) : \sigma = \tr \, \sigma + \sigma:\left(\frac{1}{2} \nabla{w} - \zeta\right) \otimes \left(\frac{1}{2} \nabla{w} - \zeta\right).
\end{equation*}
Due to condition (i) the LHS above is not greater than $\tr\,\sigma$ while the RHS is not smaller than $\tr\,\sigma$ and equality holds if and only if the two conditions are met: $\sigma\, \bigl(\frac{1}{2} \nabla{w} - \zeta\bigr) = 0$ and $\bigl(\frac{1}{4}\, \nabla w \otimes \nabla w +e(u) \bigr) : \sigma = \tr \, \sigma$. The first condition says that $\frac{1}{2}\sigma \, \nabla w = \sigma\, \zeta = q$ which is exactly the condition (iv) while the second condition together with (i) implies that indeed the maximum eigenvalue of $\frac{1}{4}\, \nabla w \otimes \nabla w +e(u)$ is one whenever $\sigma$ is non-zero.
\end{proof}

\subsection{Recasting the optimal vault}
\label{ssec:recovering_domes}

In comparison with optimality conditions \eqref{eq:opt_cond_Michell} for the Michell problem, in the system of conditions \eqref{eq:opt_cond} tailored for the pair $(\mathcal{P})$,\,$(\mathcal{P}^*)$ we spot one extra condition of a new nature: $q = \frac{1}{2}\, \sigma\, \nabla w$. This answers the key question posed below inequality \eqref{eq:Z_leq_Vmin}: for a solution $(\sigma,q)$ of $(\mathcal{P})$ does there exist an elevation function $z$ such that $q = \sigma\, \nabla z$? As long as $\sigma$ and $q$ are functions in $L^1$, which we shall assume throughout this subsection, the condition $q = \frac{1}{2}\, \sigma\, \nabla w$ immediately paves a way of retrieving solution of the minimum volume vault problem:  
\begin{theorem}[\textbf{Construction of the least-volume vault in the 'continuous' case}]
	\label{thm:recovring_MV_dome}
	Assume that $(\sigma,q) \in L^1(\Omega;\Sddp) \times L^1(\Omega;\Rd)$ and $(u,w) \in C^1(\Ob;\R^2) \times C^1(\Ob;\R)$ are solutions of problems $(\mathcal{P})$ and $(\mathcal{P}^*)$ respectively. Then the pair
	\begin{equation*}
		z = \frac{1}{2}\, w,\qquad \eta = \frac{1}{J_{z}} \, \sig
	\end{equation*} 
	solves the problem $(\mathrm{MVV}_\Omega)$ with $\Vmin = \Z$.
\end{theorem}
\begin{proof}
It suffices to show that the pair $({z},\sig)$ solves the simplified variant \eqref{eq:MVFD_sigma}. By virtue of Theorem \ref{thm:opt_cond} there holds $q = \frac{1}{2}\,\sigma\, \nabla w$ thus $q = \sigma\, \nabla {z}$ and as consequence $\sigma:(\nabla z \otimes \nabla z) = (\sigma^{-1} q)\cdot q$. By acknowledging: inequality \eqref{eq:Z_leq_Vmin}, equations $\DIV\,\sigma =0$, $\ -\dive \bigl(\sigma \nabla z\bigr) = -\dive\, q = f$ and optimality of $(\sigma,q)$ a chain of inequalities follows:
\begin{equation}
	\label{eq:Z_leq_V_leq_Z}
	\Z \leq \Vmin \leq  \int_\Ob \Big(\tr \, \sigma + \sigma : \bigl(\nabla z \otimes \nabla z \bigr) \Big) = \int_\Ob \Big(\tr \, \sigma + \bigl(\sigma^{-1} q\bigr)\cdot q \Big) = \int_\Ob g^0_{\mathscr{C}}(\sigma,q) = \Z
\end{equation}
being, as a result, a chain of equalities, which proves the assertion.
\end{proof}

The optimal field $\eta$, point-wise being the matrix of contravariant components of the tensor field $\hat{\mathcal{N}}$, proves convenient in computations yet is impractical in terms of physical interpretation. One way of remedying this it to embed our membrane shell into 3D space as a lower dimensional structure in the measure-theoretic manner: the membrane force field is a tensor valued measure $\hat{\sigma} = \hat{S} \, \Ha^2 \mres \mathcal{S}_z \in \Mes(\R^3;\mathrm{T}^3_+)$ that charges the two dimensional surface $\mathcal{S}_z$ while $\hat{S}\in L^1(\mathcal{S}_z;\mathrm{T}^3_+)$ is the function satisfying
\begin{equation}
	\label{eq:unprojection}
	\hat{S}\bigl(z(\argu)\bigr) = \begin{bmatrix}
	\,\mathrm{I}  &  \nabla z\, 
	\end{bmatrix}^\top\!
	\eta\, \begin{bmatrix}
	\,\mathrm{I}  &  \nabla z\, 
	\end{bmatrix}.
\end{equation}
where by $\bigl[ \,\mathrm{I}  \ \  \nabla z\, \bigr]$ we understand a $2 \times 3$ matrix composed of the $2 \times 2$ identity matrix $\mathrm{I}$ and the column 2D vector $\nabla z$. The proof that the stress field $\hat{\sigma}$ is in equilibrium with the tracking load $f$, namely that $-\hDIV \,\hat{\sigma} = \hat{F}_{f,z}$ in the sense of distributions in $\R^3\backslash (\partial\Omega \times \{0\})$, is postponed to Section \ref{ssec:3D_to_2D} where a broader setting shall be considered.

\begin{example}[\textbf{Analytical solution of the least-volume axisymmetric vault}]
\label{ex:axisymmetric}
For a radius $R>0$ a disk domain $\Omega = \bigl\{x \in \Rd \, \big\vert\, \abs{x} < R \bigr\}$ is considered. The load per unit area of $\Omega$ is uniformly distributed, i.e. $f = p = \mathrm{const}$ ($f = p\, \mathcal{L}^2 \mres \Omega$ as a measure). Following Theorem \ref{thm:recovring_MV_dome} the pair of problems $(\mathcal{P})$ and $(\mathcal{P}^*)$ must be solved. The strategy is to guess the fields $(\sigma,q)$, $(u,w)$ and then to show that optimality conditions \eqref{eq:opt_cond} are met. By making use of the axial symmetry we give the candidate fields expressed in polar coordinates for which $(e_r, e_\varphi)$ is the standard basis: $\sigma = \sigma_{rr}\, e_r\otimes e_r$, $\ q = q_r \,e_r$, $\ u = u_r\, e_r$, $\ w = \tilde{w}$ \ where
\begin{equation}
	\label{eq:axisym_sol}
	\sigma_{rr}(r) = \frac{p R^2}{2\sqrt{5}}\, \frac{1}{r}, \qquad q_r(r) = - \,\frac{p\, r}{2}, \qquad u_r(r) = r - \frac{r^5}{R^4}, \qquad \tilde{w}(r) = \frac{2\sqrt{5}}{3} \frac{R^3 - r^3}{R^2}. 
\end{equation}
From the well established formulas one obtains: $\DIV\, \sigma = (\sigma'_{rr}+ \sigma_{rr}/r)\,  e_r,\ $ $-\dive\,q = -(q'_r + q_r/r)$ and $e(u) = u'_r\, e_r \otimes \nolinebreak e_r + u_r/r \,e_\varphi \otimes e_\varphi,\ $ $\nabla w =  \tilde{w}'\, e_r$. Optimality condition (i) follows immediately since $\DIV \,\sigma = 0$ and $-\dive\,q = p = f$ (including the origin). Since $\tilde{w}' = -2\sqrt{5}\,r^2/R^2$ it is also straightforward that $q =\frac{1}{2}\, \sigma\, \nabla w$ hence the condition (iv). To verify (i) and (iii) we compute $\frac{1}{4} \nabla w \otimes \nabla w + e(u) = e_r \otimes e_r + \bigl(1- (r/R)^4\bigr)\, e_\varphi \otimes e_\varphi$. We have shown that the pairs $(\sigma,q)$, $(u,w)$ given in \eqref{eq:axisym_sol} solve the system \eqref{eq:opt_cond} which renders them optimal for $(\mathcal{P})$ and $(\mathcal{P}^*)$, respectively, while
\begin{equation*}
	\Vmin = \Z = \int_\Ob \Big(\tr \, \sigma + \bigl(\sigma^{-1} q\bigr)\cdot q \Big) =\int_{\Ob} w \, f = \frac{2\pi}{\sqrt{5}} \, p R^3. 
\end{equation*}
The reader is referred to \cite[Section 4.3]{bouchitte2020} for solutions in the case of an arbitrary axisymmetric load $f$.

We observe that the field $\sigma$ is unbounded in vicinity of the origin, although it is integrable; it is a rank-one field and can be decomposed into continuum of bars in tension that coincide with diameters of the disc $\Omega$, cf. the disintegration formula \eqref{eq:fibrous_sigma_q} in Section \ref{ssec:regularity}. In this spirit $\sigma$ was illustrated in Fig. \ref{fig:axisymmetric}(a) -- this type of fibrous structure regularly appears as a part of Michell structures, cf. the Michell's bicycle half-wheel in Chapter 4.5.3 of \cite{lewinski2019a}.
\begin{figure}[h]
	\centering
	\subfloat[]{\includegraphics*[trim={-0cm -1.8cm -0cm -0cm},clip,width=0.22\textwidth]{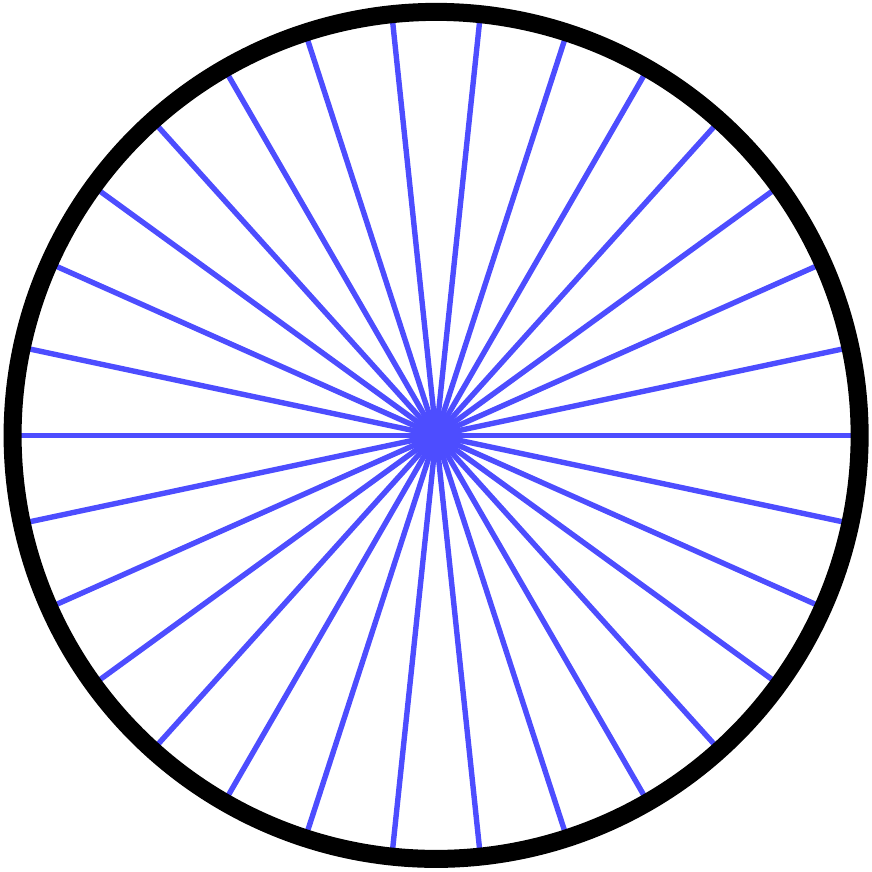}}\hspace{0.4cm}
	\subfloat[]{\includegraphics*[trim={4cm 1cm 3.cm 4cm},clip,width=0.33\textwidth]{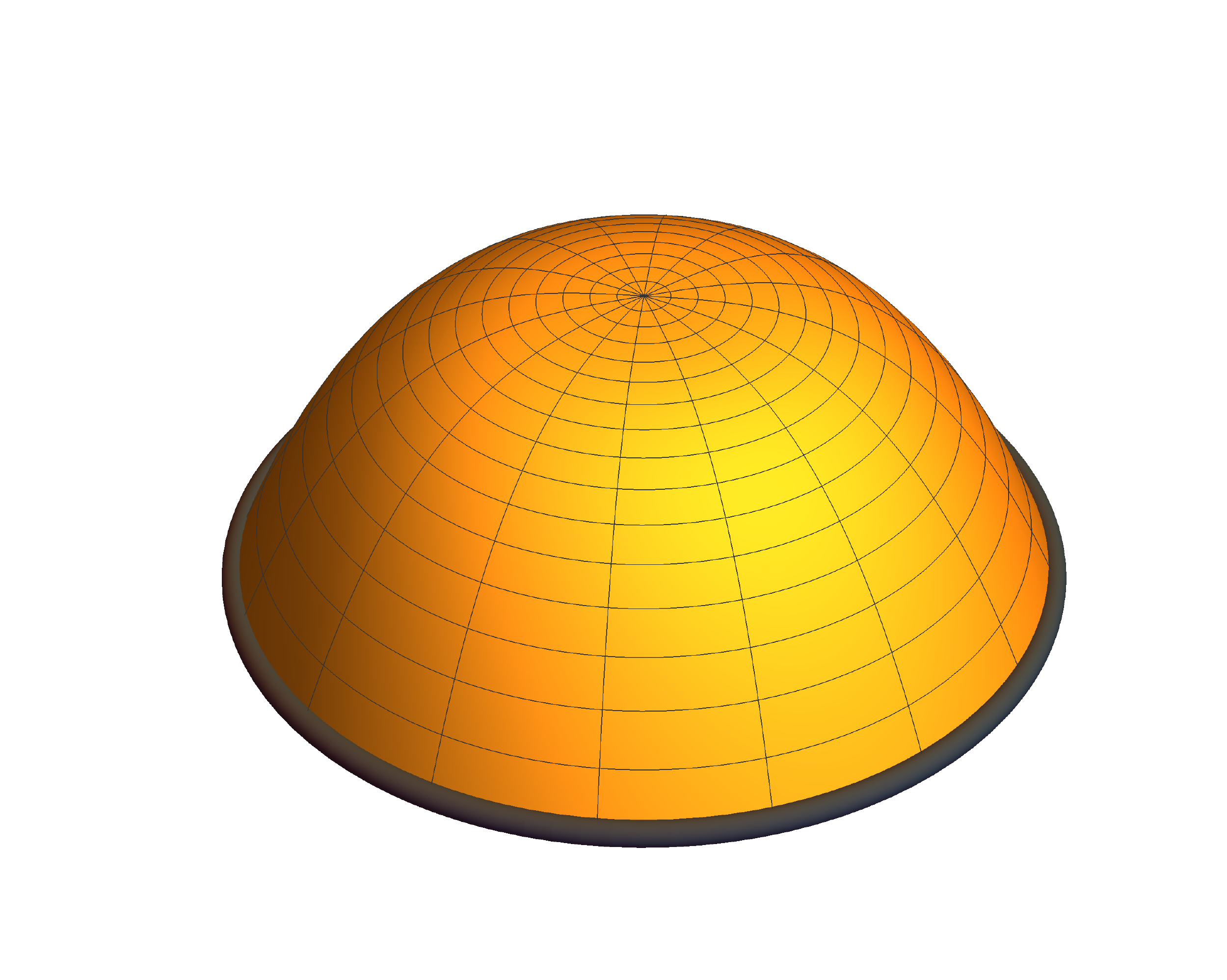}}\hspace{0.3cm}
	\subfloat[]{\includegraphics*[trim={6cm 2.5cm 4.5cm 5cm},clip,width=0.33\textwidth]{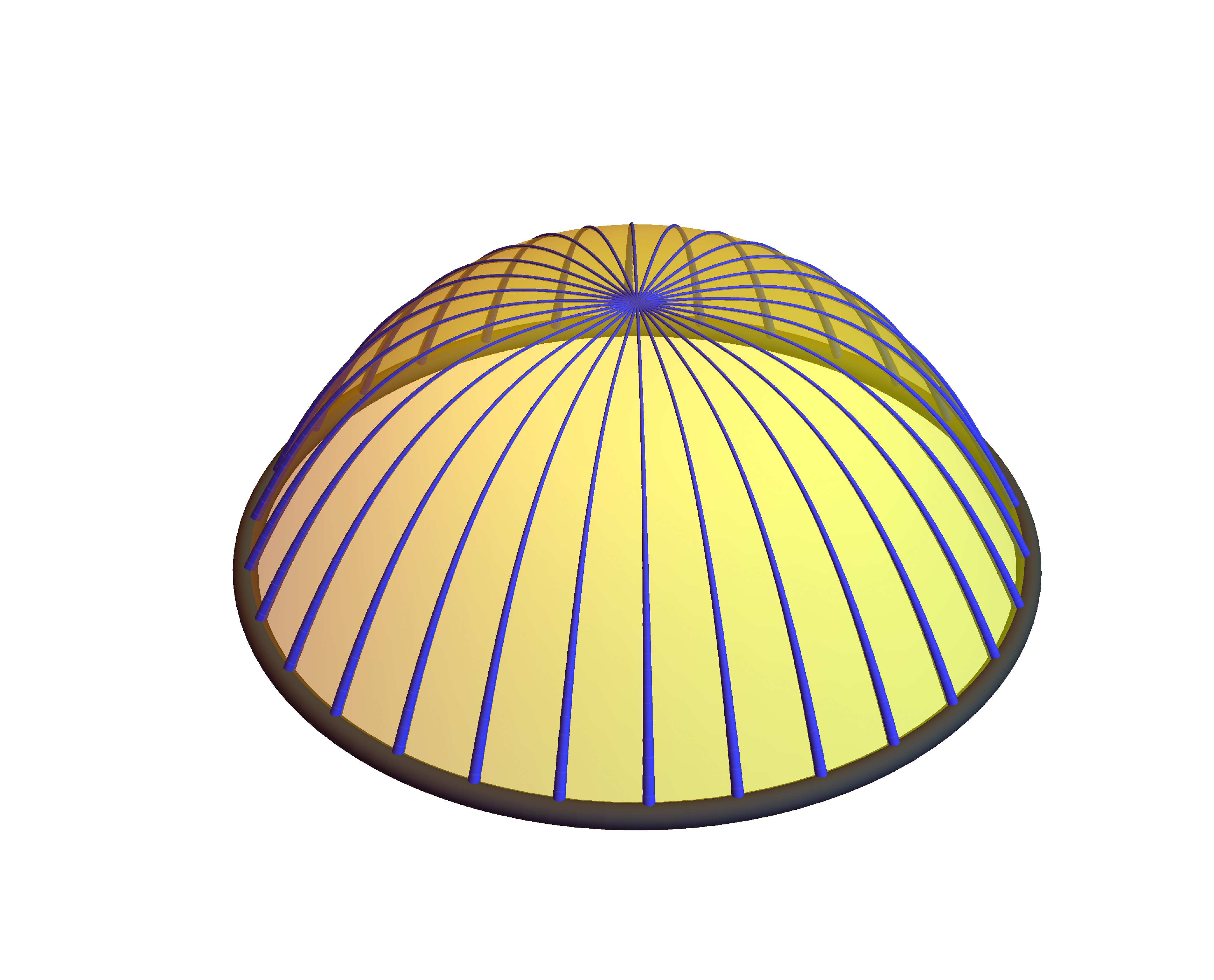}}\\
	\caption{The problem of minimum volume vault for a load uniformly distributed over a disk: (a) the field $\sigma$ solving problem $(\mathcal{P})$; (b) optimal elevation function $z = \frac{1}{2}\,w$ where $w$ solves problem $(\mathcal{P}^*)$; (c) the optimal vault composed of uncountable number of cables.}
	\label{fig:axisymmetric}
\end{figure}

Theorem \ref{thm:recovring_MV_dome} combined with formula \eqref{eq:unprojection} paves the way of unprojecting the field $\sigma$ onto the surface $\mathcal{S}_z$ where $z = z(r) = \frac{1}{2} \,w(r)$. Simple computations furnish the 3D stress field on the vault:
\begin{equation*}
	\hat{S}\bigl(\hat{z}(r) \bigr) = \frac{1}{J_z(r)}\frac{p R^2}{\sqrt{5} \,r} \Big(e_r \otimes e_r +2\, z'(r)\, e_r\,\symtens\, e_3 + (z'(r))^2 \, e_3 \otimes e_3   \Big)
\end{equation*}
where $J_z(r) =\sqrt{1+\vert\frac1{2}\nabla w(r)\vert^2}= \sqrt{1+(z'(r))^2}$, \  $e_3$ is the vertical unit vector and $\symtens$ stands for the symmetrical part of the tensor product. The field $\hat{S}$ is rank-one and to see this we find a formula for a unit vector being tangent to $\mathcal{S}_z$ and horizontally radial: $\hat{\tau} = \hat{\tau}(r) = \bigl(e_r + z'(r)\, e_3 \bigr)/\sqrt{1+(z'(r))^2}$ hence
\begin{equation*}
	\hat{S}\bigl(z(r) \bigr) = \frac{1+(z'(r))^2}{J_z(r)}\frac{p R^2}{\sqrt{5} \,r} \, \hat{\tau}(r) \otimes \hat{\tau}(r) = \sqrt{1+(z'(r))^2} \, \frac{p R^2}{\sqrt{5} \,r}  \, \hat{\tau}(r) \otimes \hat{\tau}(r) = \sqrt{1+5\left(r/R\right)^4} \, \frac{p R^2}{\sqrt{5} \,r}  \, \hat{\tau}(r) \otimes \hat{\tau}(r).
\end{equation*}
By analogy we find that the membrane force field in the vault decomposes to continuum of cables in tension and of radially varying thickness, cf. the 3D visualization in Fig. \ref{fig:axisymmetric}(c).
\end{example}

\section{Vaults of the least compliance -- the elastic design problem}
\label{sec:elastic_design}

\subsection{Formulation of the form finding problem in the elastic setting}
\label{ssec:problem_formulation_elastic}

In this section we shall once more tackle the optimal design problem of a vault being a membrane shell that is capable of withstanding tension only, although this time distribution of an elastic material will be sought while minimizing the elastic compliance. Consequently the data $\Omega, f$ will remain unchanged and for an elevation function $z \in C^1(\Ob;R)$, with $z =0$ on $\bO$, $\mathcal{S}_z$ will be a surface spread over $\Omega$ and parametrized by \eqref{eq:parametrization}. \textit{A priori} the material distribution is represented by a function $\hat{\rho}:\mathcal{S}_z \to \R_+$, however, unlike in Section \ref{ssec:problem_formulation_plastic}, we will pose the optimization problem in its parametrized version on $\Omega$ straightaway, therefore the elastic material distribution shall be identified by the function $\rho \in L^1(\Omega;\R_+)$; the relation is clear: $\rho = \hat{\rho} \circ \hat{z}$, where $\circ$ stands for composition of functions.

The underlying elastic material will be modelled through the Michell-like energy potential that for a strain $\eps \in \Sdd$ in a plane plate reads $j(\eps) = \frac{E_0}{2} \bigl(\gamma(\eps)\bigr)^2$ where $E_0$ is an \textit{a priori} fixed Young modulus and $\gamma$ is the spectral norm, cf. \eqref{eq:spectral_norm}. This somewhat artificial choice ought to reflect the fibrous character of the vault and may be justified two fold. On the one hand, in \cite{bendsoe1993} and \cite{allaire1993} it was  noted that the Michell-like potential occurs to be the integrand in the suitably relaxed minimum compliance problem of the highly porous structure made of an isotropic material of given volume: the Michell's energy density is obtained by performing a passage to the zero limit with the volume fraction of the isotropic material to be optimally distributed. One may expect that a similar asymptotic analysis could be performed for membrane shells here. Another explanation goes through the problem of the \textit{Free Material Design}, cf. Remark \ref{rem:FMD} below for more details.

In theory of Michell structures it is well established that the virtual displacement $u:\Ob \to \Rd$ being a solution of the problem \eqref{eq:Michell_max} is, after rescaling by a multiplicative constant, the displacement field in the optimally designed elastic structure. One could thus hope that solutions $u:\Ob \to \Rd,\ w: \Ob \to \R$ \ to  $(\mathcal{P}^*)$ relate to deformation of optimally designed elastic vault. We will show that this is indeed the case.

In the sequel of this section $u\in C^1(\Ob;\Rd)$ and $w \in C^1(\Ob;\R)$ will stand for horizontal and, respectively, vertical displacement functions of the vault $\mathcal{S}_z$. More precisely for each $x \in \Omega$ the 3D displacement vector at a point $\hat{z}(x) =\bigl(x,z(x) \bigr) \in \mathcal{S}_z$ shall read $\hat{v} = \hat{v}\bigl(\hat{z}(x)\bigr) = u(x) + w(x)\,e_3$ where we agree that the in-plane vector $u(x) \in \Rd$ is naturally embedded into $\R^3$ space; $e_3$ is the vertical unit vector according to Fig. \ref{fig:form_finding_problem}. The vault $\mathcal{S}_z$ shall be kinematically fixed on the boundary $\bO \times \{0\}$ which amounts to enforcing boundary conditions $u=0,\ w=0$ on $\bO$. After formula (11.2.14) in \cite{green1968} displacements $u,w$ generate the linearized strain tensor field $\hat{\mathcal{E}}:\mathcal{S}_z \to \mathcal{T}^{\,2}$ whose covariant components (with respect to parametrization \eqref{eq:parametrization}) \ $\xi:\Omega \to \Sdd$ are expressed via the linear differential operator $\mathcal{A}_z$:  
\begin{equation*}
	\xi = \mathcal{A}_z(u,w) := e(u) + \nabla z\, \symtens\, \nabla w
\end{equation*}
where $\symtens$ stands for the symmetric part of the tensor product. Since $\xi$ is merely the matrix of covariant components the Michell-like potential $j(\argu) = \frac{E_0}{2} \bigl(\gamma(\argu)\bigr)^2$ must be properly adjusted:
\begin{equation*}
	j_z\bigl(x, \xi \bigr) = \frac{E_0}{2} \bigl( \gamma_z(x,\xi ) \bigr)^2, \qquad \gamma_z\bigl(x,\xi \bigr) = \sup_{\tau \in \Rd, \, \tau \neq 0}\ \abs{\frac{\xi : (\tau \otimes \tau)}{G_z(x) : (\tau \otimes \tau)}};
\end{equation*}
further we shall suppress the dependence on $x$ of the two functions and we will shortly write: $j_z=j_z(\xi),\ \gamma_z = \gamma_z(\xi)$.
The function $\gamma_z$ is a closed gauge and one can observe that it actually measures the maximal (with respect to direction $\tau$) absolute value of the physical strain, exactly as $\gamma$ in the case of Cartesian coordinates on the plane.  

For the next step, to the elastic formulation we must introduce a constraint on the stress field: the membrane shell cannot withstand compression which may be interpreted as perfect immunity to local bucking. A neat way to do so leads through the constitutive law, namely  starting from the base elastic potential $j_z$ we propose
\begin{equation}
	\label{eq:j_+}
	j_{z,+} = \left( j^*_z  + \mathbbm{I}_{\Sddp}\right)^*
\end{equation}
where $\mathbbm{I}_\Sddp$ is the indicator functions of the set $\Sddp$ (the reader is referred to Section \ref{ssec:convex_analysis} that summarizes some basic concepts of convex analysis). The idea \eqref{eq:j_+} was already employed in \cite[Example 6.1]{bolbotowski2020a}. it also views the formulation of elasticity of masonry structures proposed in \cite{giaquinta1985} from perspective of convex analysis. Since $\Sddp$ is a closed convex cone in $\Sdd$ it is straightforward (see Corollary 15.3.1 in \cite{rockafellar1970convex}) that
\begin{equation}
	\label{eq:j_gamma}
	j_{z,+}\bigl(\xi \bigr) = \frac{E_0}{2} \bigl( \gamma_{z,+}(\xi ) \bigr)^2, \qquad j^*_{z,+}\bigl(\eta \bigr) = \frac{1}{2 E_0} \bigl( \gamma^0_{z,+}(\eta) \bigr)^2
\end{equation}
for some mutually polar closed gauges $\gamma_{z,+},\ \gamma_{z,+}^0$; moreover one has $j^*_{z,+}(\eta) = j^*_{z}(\eta) +  \mathbbm{I}_{\Sddp}(\eta)$ and as a consequence $\gamma^0_{z,+}(\eta) = \infty$ for any $\eta \notin \Sddp$ while for $\eta \in \Sddp$ it holds that
\begin{equation}
	\label{eq:gamma_z_+}
	\gamma^0_{z,+}(\eta) = \gamma^0_z(\eta) = \sup_{\xi \in \Sdd} \biggl\{ \xi : \eta \ \biggl\vert \ -G_z \preceq \xi \preceq G_z  \biggr\} = G_z : \eta;
\end{equation}
we recall the dependence on $x$, i.e. $\gamma^0_{z,+}(\eta) = \gamma^0_{z,+}(x,\eta)$ and $G_z =G_z(x)$.
Above, the constraint in the supremum is equivalent to $\gamma_z(\xi) \leq 1$, while for $\eta \in \Sddp$ the maximization problem is solved for $\xi = G_z$ being the greatest matrix that satisfies the constraint. The closed gauge $\gamma_{z,+}$ can be readily computed as $(\gamma_{z,+}^0)^0$:
\begin{equation*}
	\gamma_{z,+}(\xi) = \sup_{\eta \in \Sddp} \biggl\{\xi:\eta \ \biggl\vert \ G_z : \eta \leq 1 \biggr\} = \sup_{\tau \in \Rd} \biggl\{ \xi:(\tau \otimes \tau) \ \biggl\vert \ G_z : (\tau  \otimes \tau) \leq 1 \biggr\}
\end{equation*}
where equality between the two suprema can be easily shown based on the  linearity of the two problems.

Compliance of the vault $\mathcal{S}_z$ with material distribution $\hat{\rho} = \rho \circ \hat{z}^{-1}$ and subject to the load $\hat{F}_{f,z}$ (see Section \ref{ssec:problem_formulation_plastic} for definition) can be readily defined as minus total potential energy of the system. The elastic potential energy reads $\int_{\mathcal{S}_z} j_{z,+}\bigl(\mathcal{A}_z(u,w)\bigr)\circ \hat{z}^{-1} \, \hat{\rho} \,d\Ha^2= \int_{\Omega} j_{z,+}\bigl(\mathcal{A}_z(u,w)\bigr) \, J_z \, \rho$ while the potential energy of the load can be written as $\int_{\mathcal{S}_z} \hat{v} \cdot \hat{F}_{f,z} = \int_{\Omega} w\,f$; ultimately, after a change $-\inf A = \sup (-A)$, the definition of the compliance reads
\begin{equation}
	\label{eq:comp_def}
	\Comp(z,\rho) := \sup_{\substack{u \in C^1\!(\Ob;\Rd) \\ w \in C^1\!(\Ob;\R)}} \left\{ \left. \int_{\Omega} w f - \frac{E_0}{2} \int_\Omega  \Big(  \gamma_{z,+}\bigl(\mathcal{A}_z(u,w) \bigr)\Big)^2 J_z\, \rho \ \ \right\vert \ u=0, \, w=0 \text{ on } \bO \right\}.
\end{equation}
By using classical duality arguments in the setting $Y = L^2_\rho(\Omega;\Sddp) \equiv Y^*$, with $L^2_\rho$ being the weighted Lebesgue space, one arrives at the dual definition of compliance where the complementary energy is minimized (cf. e.g. \cite{ekeland1999}):
\begin{equation}
	\label{eq:dual_comp_def}
	\Comp(z,\rho) = \inf_{\NN \in L^2_\rho(\Omega;\Sddp)} \left\{ \left. \frac{1}{2 E_0} \int_\Omega \Big(  \gamma^0_{z}\bigl(\NN \bigr)\Big)^2 J_z\, \rho  \ \ \right\vert
	\begin{array}{c}
	\DIV \, \bigl( \rho\NN J_z \bigr) = 0,\\
	-\dive \bigl( (\rho\NN J_z) \nabla{z} \bigr) = f
	\end{array}
	\text{ in } \Omega  \right\}.
\end{equation}
\begin{remark}
	The more standard form of the dual elasticity formulation may be recovered from the above: by change of variables $\eta = \rho \, \NN$ we arrive at the minimized functional $\frac{1}{2 E_0} \int_\Omega \bigl(  \gamma^0_{z}(\eta)\bigr)^2/\rho\, J_z$. However, since $\rho$ is allowed to vanish, formulation \eqref{eq:dual_comp_def} is mathematically more natural. The advantage becomes fundamental once the design problem in $\rho$ is relaxed to measures $\mu$ (see Section \ref{sec:3D}) in which case the integral with $\eta$ is ill-posed. In the relation $\eta = \rho\, \NN$ a mechanical interpretation may be found: while $\eta$ is the tensor of membrane forces (its contravariant components) the tensor $\NN$ plays the role of stress as it is referred to the material distribution $\rho$.
\end{remark}
The volume of the vault $\mathcal{S}_z$ with material distribution $\hat{\rho}:\mathcal{S}_z \to \R_+$ is (by definition) given by the integral $\int_{\mathcal{S}_z}\hat{\rho}$ \ or, through the change of variables formula, $\int_{\Omega} J_z\, \rho$ where $\rho(x) = \hat{\rho}\bigl(\hat{z}(x) \bigr)$. We formulate the design problem of \textit{Minimum Compliance Vault} under the volume constraint in the parametrized setting:
\begin{equation*}\tag*{$(\mathrm{MCV}_\Omega)$}
	\Cmin=\inf_{\substack{z \in C^1\!(\Ob;\R) \\ \rho \in L^1\!(\Omega;\R_+)}} \left\{ \Comp(z,\rho)\ \left\vert \ z = 0 \text{ on } \bO, \ \int_{\Omega} J_z\, \rho \leq V_0 \right. \right\}
\end{equation*}
where $V_0 >0$ is the prescribed upper limit for the vault's volume. Similarly as for $(\mathrm{MVV)_\Omega}$ the posed problem does not enjoy the desired mathematical properties as e.g. convexity with respect to the pair $(z,\rho)$. Again we will succeed in reducing it to the pair of convex  problems $(\mathcal{P})$,\,$(\mathcal{P}^*)$.

\subsection{Retrieving optimal elastic vault from solutions of problems $(\mathcal{P})$ and $(\mathcal{P}^*)$}
\label{ssec:recovering_dome_elastic}

Inspired by the strategy from \cite{bouchitte2001} we will establish a direct connection between problems $(\mathrm{MCV}_\Omega)$ and $(\mathcal{P}^*)$. To that aim we introduce an auxiliary function $h: \Sdd \times \Rd \to \Rb$:
\begin{equation*}
h(\eps,\vartheta) := \sup\limits_{\zeta \in \Rd} \sup\limits_{\tau \in \Rd} \biggl\{ \bigl(\eps + \zeta \, \symtens \, \vartheta\bigr):(\tau \otimes \tau) \ \, \biggl\vert\ (\mathrm{I} + \zeta \otimes \zeta):(\tau\otimes\tau) \leq 1 \biggr\} 
\end{equation*}
that for any $u,w,z$ satisfies everywhere in $\Omega$:
\begin{equation}
	\label{eq:gamma_leq_h}
	\gamma_{z,+}\bigl(\mathcal{A}_z(u,w) \bigr) \leq h\bigl(e(u),\nabla w \bigr),
\end{equation}
i.e. for given $e(u),\nabla w$ the function $h$ point-wise yields the maximal normal strain with respect to slope $\zeta = \nabla z$ of the surface $\mathcal{S}_z$. The following result sparks the idea of the link to problem $(\mathcal{P}^*)$:    
\begin{proposition}
	\label{prop:h_and_g_C}
	The function $h$ is the gauge for the closed convex set $\mathscr{C}$, namely $h = g_\mathscr{C}$. In particular for each pair $\eps \in \Sdd$, $\vartheta \in \Rd$ there holds an equivalence:
	\begin{equation}
		\label{eq:h_g_C}
		h(\eps,\vartheta) \leq 1 \qquad \Leftrightarrow \qquad \frac{1}{4} \,\vartheta \otimes \vartheta + \eps \preceq  \mathrm{I}.
	\end{equation}
\end{proposition}
\begin{proof} For each $\zeta,\tau \in \Rd$ the mapping $(\eps,\vartheta) \mapsto \bigl(\eps + \zeta \, \symtens \, \vartheta\bigr):(\tau \otimes \tau)$ is affine and zero at the origin. Therefore function $h$, as a point-wise supremum of such mappings, is a closed gauge, cf. \cite{rockafellar1970convex}. The assertion will follow once we prove that $h(\eps,\vartheta) \leq \nolinebreak 1 \ \Leftrightarrow \ (\eps,\vartheta) \in \mathscr{C}$ or, equivalently, that \eqref{eq:h_g_C} holds true. Almost directly by definition, the inequality $h(\eps,\vartheta) \leq 1$ may be rewritten as:
\begin{equation*}
	\bigl(\eps + \zeta \, \symtens \, \vartheta\bigr):(\tau \otimes \tau) \leq (\mathrm{I} + \zeta \otimes \zeta):(\tau\otimes\tau) \qquad \forall\,\zeta,\tau \in \Rd
\end{equation*}
or, by putting $\psi =\zeta \cdot \tau$,
\begin{equation*}
	\psi^2 - (\vartheta \cdot \tau)\, \psi + (\mathrm{I}-\eps):(\tau \otimes \tau) \geq 0 \qquad \forall\,\psi \in \R,\ \tau \in \Rd.
\end{equation*}
The LHS is a quadratic function in $\psi$ hence the inequality holds if and only if the discriminant is non-positive:
\begin{equation*}
	(\vartheta \cdot \tau)^2 - 4\, (\mathrm{I}-\eps):(\tau \otimes \tau) \leq 0 \qquad \forall\, \tau \in \Rd
\end{equation*}
or $\bigl(\frac{1}{4} \vartheta \otimes \vartheta + e(u)\bigr):(\tau \otimes \tau) \leq \mathrm{I}:(\tau \otimes \tau)$ for any $\tau \in \Rd$ and the proof is complete.
\end{proof}

From the substitution $\psi = \zeta\cdot\tau$ one learns about another formula for $h$ that will be useful in due course:
\begin{equation}
	\label{eq:h_with_psi}
	h(\eps,\vartheta) = \sup\limits_{\psi \in \R} \ \sup\limits_{\substack{\tau \in \Rd,\ \abs{\tau} \leq 1}} \frac{\eps:(\tau\otimes\tau)+\psi\ \vartheta\cdot \tau}{1+ \psi^2}.
\end{equation}
The link between $(\mathrm{MCV}_\Omega)$ and $(\mathcal{P}^*)$ can readily be given:

\begin{lemma}
	\label{lem:Cmin_leq_Z}
	There holds an inequality
	\begin{equation*}
		\Cmin \geq \frac{\Z^2}{2E_0 V_0}.
	\end{equation*}
\end{lemma}
\begin{proof}
By plugging the displacement based definition \eqref{eq:comp_def} of compliance $\Comp(z,\rho)$ into the problem $(\mathrm{MCV}_\Omega)$ we arrive at an $\inf$-$\sup$ problem. Due to the lack of convexity/concavity properties of the functional in $z,\rho$ and $u,w$ it is not clear if the order can be swapped to $\sup$-$\inf$ while preserving equality, but inequality as below always holds:
\begin{equation*}
	\Cmin \geq  \sup_{\substack{(u,w) \in C^1\!(\Ob;\R^3)\\(u,w)=0 \text{ on }\bO }}\ \inf\limits_{\substack{z \in C^1\!(\Ob;\R), \ z=0 \text{ on } \bO \\ \tilde{\rho} \in L^1\!(\Ob;\R_+), \ \int_\Omega \tilde{\rho}\, \leq V_0}}   \ \left\{  \int_{\Omega} w f - \frac{E_0}{2} \int_\Omega  \Big(  \gamma_{z,+}\bigl(\mathcal{A}_z(u,w) \bigl)\Big)^2 \tilde\rho \, \right\},
\end{equation*}
where a variable change $\tilde{\rho} = J_z \,\rho$ was performed and now the volume constraint reads $\int_{\Omega} \tilde{\rho} \leq \nolinebreak V_0$. For any elevation function $z \in C^1(\Ob;\R)$ and $x \in \Omega$ from inequality \eqref{eq:gamma_leq_h} follows that $ \gamma_{z,+}\bigl( \mathcal{A}_z(u,w)(x) \bigr) \leq h\bigl(e(u)(x),\nabla w(x) \bigr) \leq \norm{h\bigl(e(u),\nabla w \bigr)}_\infty$ where $\norm{k}_\infty = \sup_{x\in \Omega} \abs{k(x)}$. The infimum above can be estimated from below furnishing 

\begin{equation*}
	\Cmin \geq \sup_{\substack{(u,w) \in C^1\!(\Ob;\R^3)\\(u,w)=0 \text{ on }\bO }}
	\left\{  \int_{\Ob} w f - \frac{E_0 V_0}{2} \norm{h\bigl(e(u),\nabla w \bigr)}_\infty^2  \right\}.
\end{equation*}
In the next step we acknowledge that the pair $(u,w)$ belongs to a linear space and utilize the technique proposed in the proof of Proposition 2.1 in \cite{bouchitte2001}: each pair $(u,w)$ may be represented as $(u,w) = t\,(u_1,w_1)$ with $t \in \R_+$ and $\norm{h\bigl(e(u_1),\nabla w_1 \bigr)}_\infty \leq 1$; as a consequence
\begin{equation*}
\Cmin \geq \sup_{\substack{(u_1,w_1) \in C^1\!(\Ob;\R^3)\\(u_1,w_1)=0 \text{ on }\bO }} \  \sup_{t \geq 0}\ \ 
\left\{  t\int_{\Ob} w_1 f - t^2\frac{E_0 V_0}{2} \ \bigg\vert \ h\bigl(e(u_1),\nabla w_1 \bigr)\leq 1\ \text{ in } \Ob  \right\}.
\end{equation*}
We shall show that the RHS of the inequality above equals $\Z^2/(2E_0 V_0)$. First we solve the univariate quadratic maximization problem with respect to $t \geq 0$: the maximum is attained at $t = (\int_{\Omega} w_1\,f)/(E_0 V_0)$ (we may assume that $\int_{\Omega} w_1 \,f \geq 0$) thus the functional maximized with respect to $u_1,w_1$ reads $(\int_{\Omega} w_1\,f)^2/(2E_0 V_0)$. The assertion follows owing to Proposition \ref{prop:h_and_g_C}.
\end{proof}

We move on to give the main theorem of this section that allows to recover the vault of minimum compliance based on solutions of problems $(\mathcal{P})$ and $(\mathcal{P}^*)$:

\begin{theorem}[\textbf{Constructing elastic vault of the least compliance in the 'continuous' case}]
	\label{thm:recovring_MC_dome}
	Assume that the pairs $(\sigma,q) \in L^1(\Omega;\Sddp \times \Rd)$ and $({u},{w}) \in C^1(\Ob;\R^3)$ are solutions of problems $(\mathcal{P})$ and $(\mathcal{P}^*)$ respectively. Then the pair
	\begin{equation}
		\label{eq:optima_z_rho}
		z = \frac{1}{2}\, {w}, \qquad \rho = \frac{V_0}{\Z} \frac{1}{J_{z}}\, G_{z} : \sigma
	\end{equation} 
	solves the problem $(\mathrm{MCV}_\Omega)$ with $\Cmin = \frac{\Z^2}{2E_0 V_0}$. Moreover the displacement and stress functions
	\begin{equation*}
		(u_\e,w_\e) = \frac{\Z}{E_0 V_0}\, ({u},{w}) \qquad \text{and} \qquad \NN = \frac{1}{J_{z}\, \rho}\, \sigma \in L^\infty_{\rho}(\Omega;\Sddp)
	\end{equation*}
	solve the displacement-based and, respectively, stress-based elasticity problems \eqref{eq:comp_def} and \eqref{eq:dual_comp_def} for the optimal vault $(z,\rho)$. The functions $(u_\e,w_\e)$ and $\NN$ are linked by the constitutive law of elasticity for $\rho$-a.e. $x$:
	\begin{equation*}
		\NN(x) \in \partial j_{z,+} \Big(\mathcal{A}_{z}(u_\e,w_\e)(x)\Big).
	\end{equation*}
\end{theorem}

\begin{proof}
We start by checking whether $\rho$ satisfies the volume constraint:
\begin{equation}
	\label{eq:volume_check}
	\int_{\Omega} J_{z} \, \rho = \int_{\Omega} J_{z} \left(\frac{V_0}{\Z} \frac{1}{J_{z}}\, G_{z} : \sigma \right) = \frac{V_0}{\Z} \int_\Omega \Big(\tr\,\sigma + \sigma:\bigl(\nabla z \otimes \nabla z\bigr) \Big) = V_0
\end{equation}
where the last equality follows from Theorem \ref{thm:recovring_MV_dome}, see \eqref{eq:Z_leq_V_leq_Z} in its proof. We compute $\rho$-almost everywhere
\begin{equation*}
	\gamma_{z,+}^0(\NN) = \gamma_{z}^0(\NN) = G_z : \NN = \frac{1}{J_z\,\rho}\, G_z:\sigma =\left( \frac{\Z}{V_0} \frac{J_z}{G_z:\sigma}\right)\! \frac{1}{J_z}\, G_z:\sigma = \frac{\Z}{V_0}
\end{equation*}
and boundedness $\NN \in  L^\infty_{\rho}(\Omega;\Sddp)$ follows. Since $\rho \NN J_z = \sigma$ it is straightforward that $\NN$ is a competitor for the stress-based problem \eqref{eq:dual_comp_def}. The chain of inequalities may be written down:
\begin{equation*}
	\Cmin \leq \Comp(z,\rho) \leq  \frac{1}{2 E_0} \int_\Omega \Big(  \gamma^0_{z}\bigl(\NN \bigr)\Big)^2 J_{z}\, \rho = \frac{1}{2 E_0} \int_\Omega \biggl(  \frac{\mathcal{Z}}{V_0}\biggr)^2 J_{z}\, \rho = \frac{\Z^2}{2E_0 V_0} \leq \Cmin
\end{equation*}
where the last inequality is the assertion of Lemma \ref{lem:Cmin_leq_Z}. Ultimately, everywhere in the chain above we have equalities from which we infer that $(z,\rho)$ solves $(\mathrm{MCV}_\Omega)$ and that $\NN$ solves the stress-based elasticity problem \nolinebreak \eqref{eq:dual_comp_def}.

By manipulating with the relation $z = \frac{1}{2} \, {w}$ we compute $\rho$-a.e.
\begin{equation*}
	\mathcal{A}_{z}({u},{w}) : \NN = \bigl( e({u})+\nabla z \otimes \nabla {w} \bigr):\NN = \left(\frac{1}{4}\, \nabla {w} \otimes \nabla {w} +e({u}) \right):\NN + \bigl(\nabla z \otimes \nabla z\bigr): \NN.
\end{equation*}
Owing to definition of $\rho$ we observe that writing "$\rho$-a.e." is the same as writing "$\sigma$-a.e.". Next, since $\NN$ is point-wise proportional to $\sigma$, we infer from optimal condition (iii) in \eqref{eq:opt_cond} that 
\begin{equation*}
	\mathcal{A}_{z}({u},{w}) : \NN = \tr\,\NN +  \bigl(\nabla z \otimes \nabla z\bigr): \NN = G_{z} : \NN = \gamma_{z,+}^0(\NN) = \frac{\Z}{V_0}.
\end{equation*}
Due to Lemma \ref{lem:Cmin_leq_Z} everywhere in $\Omega$ it holds that $\gamma_{z,+}\bigl( \mathcal{A}_{z}({u},{w}) \bigr)\leq h\bigl(e({u}),\nabla {w} \bigr) \leq 1$ which renders the equation $\mathcal{A}_{z}({u},{w}) : \NN = \gamma_{z,+}^0(\NN)\ $ an extremality relation, therefore 
\begin{equation*}
	\gamma_{z,+}\bigl(\mathcal{A}_{z}({u},{w}) \bigr) = 1 \qquad \rho\,\text{-a.e.}
\end{equation*}
and consequently $\gamma_{z,+}\bigl(\mathcal{A}_{z}(u_\e,w_\e) \bigr) = \Z/(E_0 V_0)\ $ $\rho$-a.e. By optimality of $(z,\rho)$, by definition \eqref{eq:comp_def} and by the fact that $(u,w)$ solves $(\mathcal{P}^*)$ we obtain
\begin{align*}
	\Cmin = \Comp(z,\rho) \geq \int_{\Omega} w_\e f - \frac{E_0}{2} \int_\Omega  \Big(  \gamma_{z,+}\bigl(\mathcal{A}_{z}(u_\e,w_\e) \bigr)\Big)^2 J_{z} \, \rho = \frac{\Z}{E_0 V_0} \int_\Omega {w} \, f - \frac{E_0}{2} \int_\Omega  \biggl( \frac{\Z}{E_0 V_0}\biggr)^2 J_{z} \, \rho&\\
	= \frac{\Z}{E_0 V_0} \,\Z - \frac{\Z^2}{2E_0 (V_0)^2} \int_\Omega J_{z} \, \rho= \frac{\Z^2}{2E_0 V_0} = \Cmin&
\end{align*}
which ultimately is a chain of equalities proving that $(u_\e,w_\e)$ solves the displacement-based elasticity problem \eqref{eq:comp_def}. Verification of the constitutive law amounts to showing that
\begin{equation*}
	\mathcal{A}_{z}(u_\e,w_\e) : \NN = j_{z,+}\bigl( \mathcal{A}_{z}(u_\e,w_\e)\bigr) + j_{z,+}^*\bigl(\NN\bigr) \qquad \rho\text{-a.e.}
\end{equation*}
being straightforward when combining the hitherto obtained results and formulas \eqref{eq:j_gamma}. 
\end{proof}

The astonishing relation obtained deserves to be put as a separate result:
\begin{corollary}
	\label{cor:z_w}
	The elevation function $z$ of the optimal elastic vault $\mathcal{S}_z$ and its vertical displacement function $w_\e$ satisfy the relation
	\begin{equation}
		\label{eq:z_w}
		z = \frac{E_0 V_0}{2 \Z}\, w_\e.
	\end{equation}
\end{corollary}
\vspace{0.5cm}
In the course of the proof the following equalities were obtained $\rho$-a.e.
\begin{equation}
\label{eq:uniform_strain_and_stress}
\gamma_{z,+}\bigl(\mathcal{A}_{z}(u_\e,w_\e)\bigr) =  h\bigl(e(u_\e),\nabla w_\e \bigr) = \frac{\Z}{E_0 V_0}, \qquad \gamma_{z,+}^0(\NN) = \frac{\Z}{V_0}.
\end{equation}
First of all, it means that where there is material $\rho$, there the strains $\xi = \mathcal{A}_{z}(u_\e,w_\e)$ and stresses $\NN$ are uniform with respect to gauges $\gamma_{z,+}$ and $\gamma_{z,+}^0$, respectively. These features of optimal design are well known from theory of Michell structures.
In the equality $\gamma_{z,+}\bigl(\mathcal{A}_{z}(u_\e,w_\e)\bigr) =  h\bigl(e(u_\e),\nabla w_\e \bigr)$ one may seek the mystery behind the result \eqref{eq:z_w} or the recipe for the optimal elevation itself: $z = \frac{1}{2} \, {w}$. From definition \eqref{eq:comp_def} of the compliance we see that the role of $z$ as a design variable is to point-wise maximize the value $\gamma_{z,+}\bigl(\mathcal{A}_{z}(u_\e,w_\e)\bigr)$ for given functions $u,w$ and, by \eqref{eq:gamma_leq_h}, the biggest possible value is precisely $h\bigl(e(u_\e),\nabla w_\e \bigr)$. According to \eqref{eq:uniform_strain_and_stress} this upper bound is reached for every material point when $z = \frac{1}{2} {w}$ and we shall show why this is. For a fixed $x \in \Omega$ let $\eps = e(u_\mathrm{e})(x)$ and $\vartheta = \nabla w_\e(x)$. Then, let $\bar{\tau}$ be the vector that gives the maximum in \eqref{eq:h_with_psi}; we will find a solution $\bar{\psi}$ of the smooth problem 
\begin{equation*}
	h(\eps,\vartheta) = \max\limits_{\psi \in \R} \frac{\eps:(\bar{\tau}\otimes\bar{\tau})+\psi\ \vartheta\cdot \bar{\tau}}{1+ \psi^2}.
\end{equation*}
By writing the Euler-Lagrange equation one computes that
\begin{equation*}
	h(\eps,\vartheta) = \frac{\eps:(\bar{\tau}\otimes\bar{\tau})+\bar\psi\ \vartheta\cdot \bar{\tau}}{1+ \bar{\psi}^2} = \frac{1}{2} \left(\eps:(\bar{\tau}\otimes\bar{\tau}) + \sqrt{(\eps:(\bar{\tau}\otimes\bar{\tau}))^2+(\vartheta\cdot \bar{\tau})^2} \right)
\end{equation*}
where, under assumption that $(\eps,\vartheta)\neq (0,0)$
\begin{equation}
	\label{eq:optimal_psi}
	\bar{\psi} = \frac{\vartheta\cdot \bar{\tau}}{\eps:(\bar{\tau}\otimes\bar{\tau}) + \sqrt{(\eps:(\bar{\tau}\otimes\bar{\tau}))^2+(\vartheta\cdot \bar{\tau})^2}} = \frac{(\frac{1}{2} \vartheta)\cdot \bar{\tau}}{h(\eps,\vartheta)}.
\end{equation}
By optimality of $(u,w)$ in problem $(\mathcal{P}^*)$ we have $h\bigl(e(u),\nabla w \bigr) =1\ $ $\sigma$-a.e. and therefore $h\bigl(e(u_\e),\nabla w_\e \bigr) = \Z /(E_0 V_0) \ $ $\rho$-a.e. The fact that $h\bigl(e(u_\e),\nabla w_\e \bigr)$ is constant $\rho$-a.e. is fundamental since for $z =E_0 V_0/\Z\, \frac{1}{2} w_\e $ the field $\bar{\psi} = \nabla z \cdot \bar\tau$ satisfies \eqref{eq:optimal_psi} for $\eps = e(u_\e)$ and $\vartheta = \nabla w_\e \ $ $\rho$-a.e. in $\Omega$. As a result we indeed arrive at $\gamma_{z,+}\bigl(\mathcal{A}_{z}(u_\e,w_\e)\bigr) =  h\bigl(e(u_\e),\nabla w_\e \bigr) = \Z /(E_0 V_0)$ for $\rho$-a.e. point in $\Omega$.

\section{Optimal vaults and Prager structures -- the measure-theoretic setting}
\label{sec:3D}

\subsection{Discussion on constructing optimal vaults in the general case}
\label{ssec:regularity}

The hitherto found two main results on the optimal plastic and elastic design of vaults -- Theorem \ref{thm:recovring_MV_dome} and Theorem \ref{thm:recovring_MC_dome} --  rely on the two assumptions on the pairs $(\sigma,q)$ and $(u,w)$ solving problems $(\mathcal{P})$ and $(\mathcal{P}^*)$: the fields $\sigma$,\,$q$ must be integrable functions (equivalently measures absolutely continuous with respect to Lebesgue measure $\mathcal{L}^2$) and functions $u$,\,$w$ must be continuously differentiable which in general is rarely the case.
The most delicate issue concerns regularity of solutions of problem $(\mathcal{P}^*)$, more precisely: which functional spaces are suitable for functions $u$ and, independently, $w$ so that the problem $(\mathcal{P}^*)$ (or rather its relaxed variant) always attains a solution? This matter was addressed in \cite{bouchitte2020}, the result below is a part of Proposition 5.11 therein:

\begin{proposition}
	\label{prop:regularitu_u_w}
	Assume that $\Omega$ is convex. Then, the maximization problem $(\mathcal{P}^*)$ in its relaxed form attains a solution $u \in BV(\Omega;\Rd) \cap L^\infty(\Omega;\Rd)$, \  $w \in C^{0,1/2}(\Ob;\R) \cap W^{1,2}(\Omega;\R)$ such that
	\begin{equation*}
	\bigl(e(u)\bigr)_s \preceq 0, \qquad \qquad \frac{1}{4} \, \nabla w \otimes \nabla w + \bigl(e(u) \bigr)_{ac} \preceq \mathrm{I} \qquad \text{a.e. in } \Omega
	\end{equation*}
	where $e(u) = (e(u))_{ac} + (e(u))_s$ is the Lebesgue decomposition of measure $e(u) \in \Mes(\Omega;\Sdd)$ into the absolutely continuous part $(e(u))_{ac}$ and the singular part $(e(u))_s$; the derivative $\nabla w \in L^2(\Omega;\Rd)$ is intended in the weak sense. Moreover, the following uniform estimates hold:
	\begin{equation}
		\label{eq:estimates_u_w}
		\norm{u}_\infty \leq \mathrm{diam}(\Omega), \qquad \norm{w}_\infty \leq \sqrt{2}\,\mathrm{diam}(\Omega),
	\end{equation}
	where $\mathrm{diam}(\Omega) = \sup_{x,y \in \Omega}\,\abs{x-y}$ denotes the diameter of $\Omega$.
\end{proposition}

The first obstacle with handling the general case, i.e. the solutions $\sigma, q$ of $(\mathcal{P})$ being measures and solutions $u, w$ of the relaxed problem $(\mathcal{P}^*)$ as in Proposition \ref{prop:regularitu_u_w}, lies in adapting the optimality conditions in Theorem \ref{thm:opt_cond} that proved to be fundamental for recovering the optimal vault. Derivatives $\nabla w$ and $(e(u))_{ac}$ are a priori functions in Lebesgue spaces and thus are defined a.e. with respect to Lebesgue measure $\mathcal{L}^2$. Once $\sigma, q$ concentrate on lower dimensional measure (e.g. representing bars) the optimality conditions (iii), (iv) in \eqref{eq:opt_cond} require $\nabla w$ and $e(u)$ to be defined a.e. yet with respect to this lower dimensional measure and therefore more information is needed than is guaranteed by Proposition \ref{prop:regularitu_u_w}. A compromise was found in \cite[Theorem 4.1]{bouchitte2020} where Lipschitz continuous solutions $u \in \mathrm{Lip}(\Ob;\Rd)$, $w \in \mathrm{Lip}(\Ob;\R)$ were assumed, see Section \ref{ssec:3D_to_2D} and optimality conditions \eqref{eq:opt_cond_mu} below for more details.

Before we investigate an example of more general solutions of the pair $(\mathcal{P})$,\,$(\mathcal{P}^*)$ we introduce some additional objects that allow description of bar structures measure-theoretically: for arbitrary pair of distinct points $x,y \in \Rd$ let us define one dimensional matrix and vector measures
\begin{equation}
\label{eq:bar_measures}
\sigma^{x,y} := \tau^{x,y} \otimes \tau^{x,y} \, \Ha^1 \mres [x,y], \qquad q^{x,y} := \tau^{x,y} \, \Ha^1 \mres [x,y], \qquad \tau^{x,y}:= \frac{y-x}{\abs{y-x}}.
\end{equation}
Measure $\sigma^{x,y}$ ought to model a straight bar of end-points $x,y$ and of unit tensile axial force; similarly $q^{x,y}$ represents a bar with unit transverse force. For any pair of distinct points $x,y \in \R^2$ we may compute in the sense of distributions on $\Omega$: $-\DIV \, \sigma^{x,y} = \mathbbm{1}_\Omega(y)\,\tau^{x,y} \, \delta_y - \mathbbm{1}_\Omega(x)\,\tau^{x,y}\, \delta_x$ and $-\dive\, q^{x,y} = \mathbbm{1}_\Omega(y)\,\delta_y - \mathbbm{1}_\Omega(x)\,\delta_x$ where $\mathbbm{1}_\Omega$ stands the characteristic function of $\Omega$. The divergences are thus point forces (horizontal in $\Rd$ or vertical in $\R$) applied at bar's end-points as expected, unless the given end-point lies on the boundary (or more generally outside $\Omega$) -- such force disappears for it is immaterial to the equilibrium equation.

\begin{example}[\textbf{Analytical solution of the optimal vault problem for a point force over a disk domain}]
	\label{ex:cone}
	For the disk domain $\Omega = \bigl\{x \in \Rd \, \big\vert\, \abs{x} < R \bigr\}$ consider a point load $f = P \,\delta_{x_0}$ with $x_0\in \Omega$. Let us choose any positive measure on the boundary $\pi \in \Mes_+(\bO)$ that satisfies $\pi(\bO) = 1$ ($\pi$ is a probability) and
	\begin{equation}
	\label{eq:condition_p}
	x_0 = \int_\bO x \,\pi(dx),
	\end{equation}
	namely $x_0$ is the barycentre of $\pi$. We propose measures $\sigma, q$ that are measure-theoretic superposition of \textit{bar-like} measures $\sigma^{x,y}, q^{x,y}$, that is, with a slight abuse of notation we use disintegration formulas below:
	\begin{equation}
	\label{eq:fibrous_sigma_q}
	\sigma = \frac{P}{\sqrt{R^2 - \vert x_0 \vert^2}}\int_\bO \abs{x-x_0}\,\sigma^{x,x_0} \, \pi(dx), \qquad q = P \int_\bO q^{x,x_0} \, \pi(dx).
	\end{equation}
	\begin{figure}[h]
		\centering
		\subfloat[]{\includegraphics*[trim={-1cm -1cm -0cm -0cm},clip,width=0.25\textwidth]{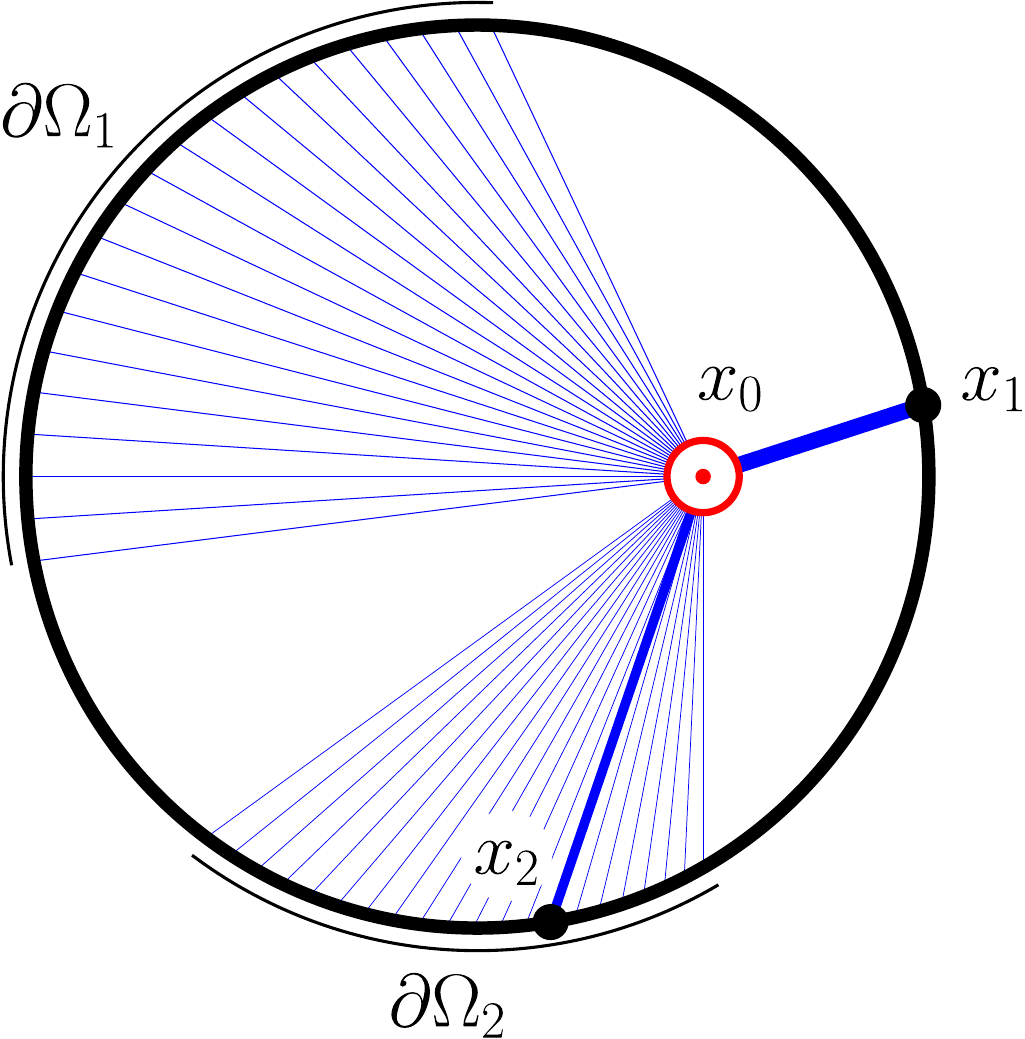}}\hspace{0.5cm}
 		\subfloat[]{\includegraphics*[trim={1.3cm -0.6cm 4cm 1.5cm},clip,width=0.33\textwidth]{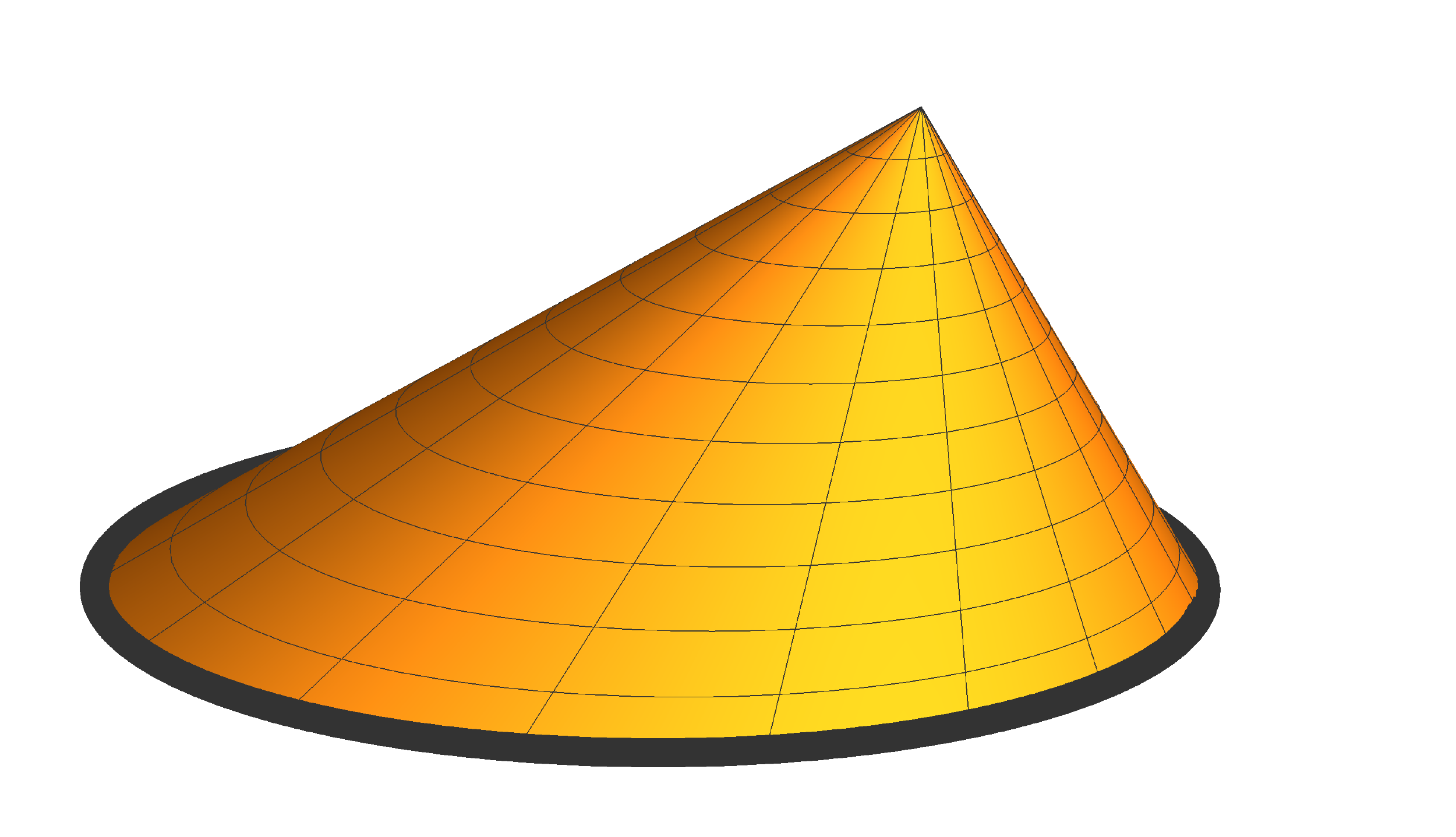}}\hspace{0.7cm}
		\subfloat[]{\includegraphics*[trim={0.8cm 0cm 2cm 1.5cm},clip,width=0.33\textwidth]{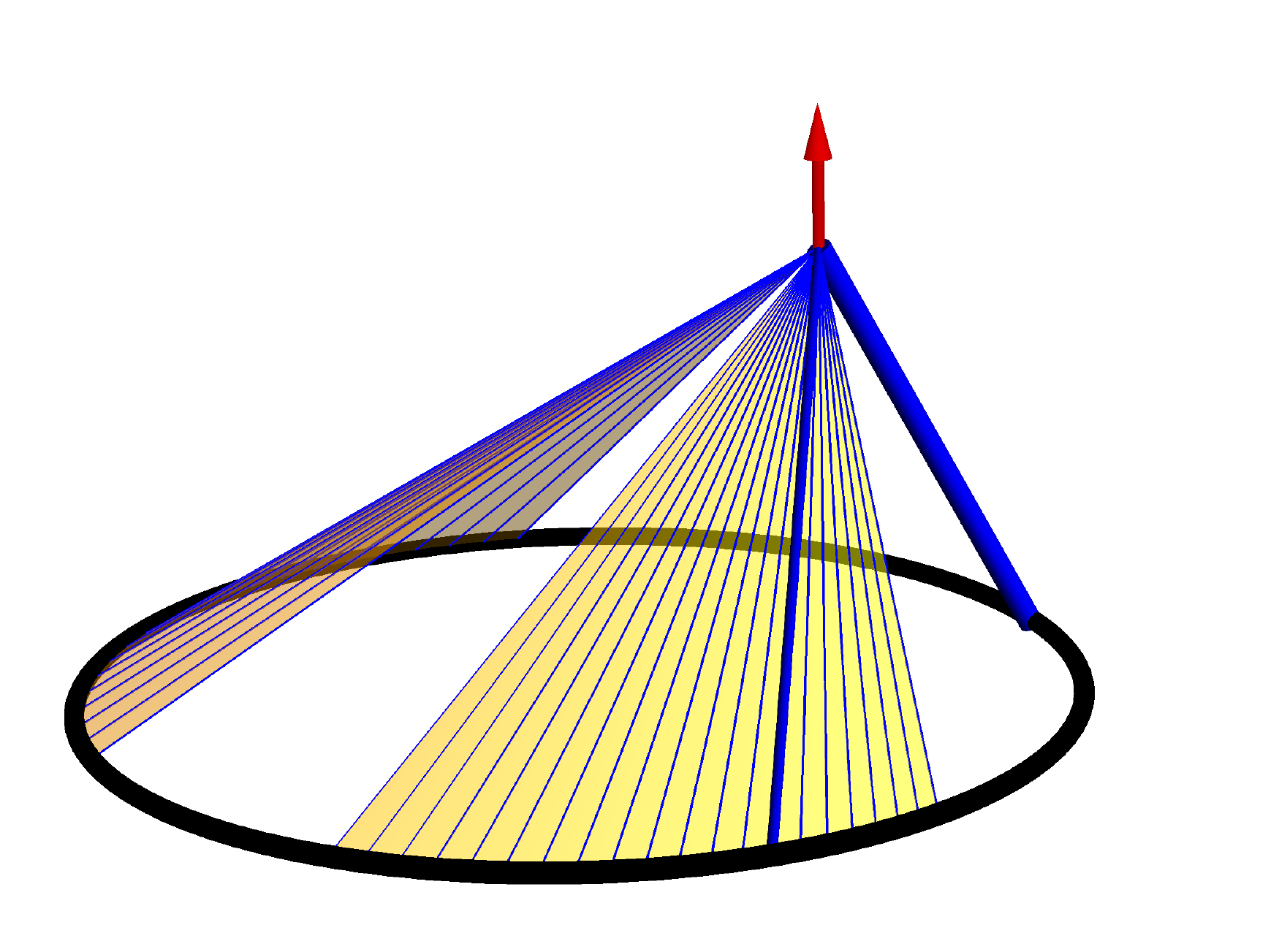}}\\
		\caption{The problem of optimal vault over a disk that is subject to a point force $P$ at $x_0$:  (a) the field $\sigma$ solving problem $(\mathcal{P})$; (b) optimal elevation function $z = \frac{1}{2}\,w$ where $w$ solves problem $(\mathcal{P}^*)$; (c) an optimal structure composed of 1D and 2D elements.}
		\label{fig:cone}
	\end{figure}
	As candidates for solutions of the relaxed problem $(\mathcal{P}^*)$ we propose
	\begin{equation}
	\label{eq:cone_uw}
	{u}(x) = 2\,h(x)\, x_0, \qquad {w}(x) = 2\sqrt{R^2 - \vert x_0 \vert^2} \, h(x).
	\end{equation}
	where $h:\Rd \rightarrow \R$ is a function of the graph $\bigl\{\bigl(x,h(x)\bigr)\bigr\}$ being a cone of vertex $\{x_0,1\}$ which passes through $\bO \times \{0\}$, see Fig. \ref{fig:cone}(b). The functions ${u}, {w}$ are Lipschitz continuous, yet not differentiable at $x_0$. The generalized variant of optimality conditions \eqref{eq:opt_cond} spoken of in this section are verified for the pairs $(\sigma,q)$ and $({u},{w})$ in \cite[Section 4.4]{bouchitte2020} where we refer for details and examples of solutions for $f = P \, \delta_{x_0}$ and different convex domains. This renders $(\sigma,q)$ and $({u},{w})$ optimal for $(\mathcal{P})$ and the relaxed version of $(\mathcal{P}^*)$, respectively, while $\Z$ may be computed as $\int_\Ob w \,f = P \, w(x_0) = 2 P \sqrt{R^2 - \vert x_0 \vert^2}$. 
	
	Condition \eqref{eq:condition_p} imposed on the probability measure $\pi \in \Mes_+(\bO)$  is very mild, offering a wide choice of solutions $(\sigma,q)$ via formulas \eqref{eq:fibrous_sigma_q}. For instance we may consider $\pi = \sum_{i=1}^2 b_i \,\Ha^1\mres\bO_i + \sum^2_{i=1} B_i \, \delta_{x_i}$, with $b_i,B_i \in R_+$ chosen so that $\pi(\bO) = 1$ and $x_0 = \int_{\bO} x \, \pi(dx)$, then the field $\sigma$ in accordance with \eqref{eq:fibrous_sigma_q} is visualized in Fig. \ref{fig:cone}(a). This optimal field $\sigma$ may be divided into three parts: a "continuous" fan radiating from $x_0$ and supported on the arc $\bO_1$; a bar $[x_1,x_0]$ of finite cross section area; the hybrid part consisting both of continuous fan supported on $\bO_2$ and a bar $[x_2,x_0]$.
\end{example}

Once the solution $\sigma$ is a general matrix-valued measure we are facing another issue when constructing the optimal vault (also in the case of differentiable $z = \frac{1}{2} w$ that will be assumed below for simplicity): the transformations $\eta = 1/J_z\, \sigma$ in Theorem \ref{thm:recovring_MV_dome} and $\rho = V_0/(\Z J_z)\, G_z:\sigma$ fail to work, even when $\eta$, $\rho$ are understood as measures. In order to explain the issue further we shall focus only on the case of the elastic design problem. If $\sigma$ is a measure we should \textit{a priori} seek elastic material distribution on the surface which is a measure $\hat{\mu} \in \Mes_+(\mathcal{S}_z)$ as well. When the material distribution is "continuous" the measure reads $\hat{\mu} = \hat{\rho} \, \Ha^2 \mres \mathcal{S}_z$ with $\hat{\rho} \in L^1(\mathcal{S}_z;\R_+)$; in that case our strategy was to find  $\rho \in L^1(\Omega;\R_+)$, i.e. a function on $\Omega$, and then send it to $\mathcal{S}_z$ simply by $\hat{\rho} = \rho \circ \hat{z}^{-1}$, where optimal $\rho$ was found via \eqref{eq:optima_z_rho} -- this strategy must be revised for general measures.

We start off with a matrix-valued measure $\sigma \in \Mes(\Ob;\Sddp)$ defined on the plane set $\Ob$ from which we must construct a suitable positive measure $\mu \in \Mes_+(\Ob)$ that shall be sent to $\ov{\mathcal{S}}_z$ to obtain $\hat{\mu} \in \Mes_+(\ov{\mathcal{S}}_z)$. This can be done by the \textit{push-forward} operation: $\hat{\mu} = \hat{z}_\# \mu$, which means that for any Borel set $\mathcal{B} \subset \ov{\mathcal{S}}_z$ we define $\hat{\mu}(\mathcal{B}) := \mu\bigl(\hat{z}^{-1}(\mathcal{B})\bigr)$.
In particular case the following relation holds by the change of variable formula (see e.g. \cite{evans1992}):
\begin{equation}
	\label{eq:change_of_variable_formula}
	\mu = \tilde{\rho} \, \mathcal{L}^2\mres\Omega \quad \text{and} \quad \hat{\mu} = \hat{z}_\# \mu \qquad \Leftrightarrow \qquad \hat{\mu} = \hat{\rho}\, \Ha^2 \mres \mathcal{S}_z \quad \text{and} \quad \hat{\rho} = \left(\frac{1}{J_z} \, \tilde{\rho}\right) \circ \hat{z}^{-1},
\end{equation}
where we recall that $J_z = \big(\mathrm{det}((\nabla \hat{z})^\top \nabla \hat{z}) \big)^{1/2} = \sqrt{1 + \vert\nabla z \vert^2}$. In the light of the hitherto used formula $\hat{\rho} = \rho \circ \hat{z}^{-1}$  the above yields $\tilde{\rho} = J_z\,\rho$. The function $\tilde{\rho} = J_z\, \rho$ turns out to be more natural than $\rho$, compare the change of variable performed in the proof of Lemma \ref{lem:Cmin_leq_Z}.

Relation \eqref{eq:change_of_variable_formula} ceases to be true for general measures -- the classical definition of Jacobian $J_z$ is specific to the Lebesgue measure only  (cf. Theorem 2.91 in \cite{ambrosio2000} for a more general change of variable formula). Nevertheless  \eqref{eq:change_of_variable_formula} helps to foresee the formula for measure $\mu \in \Mes_+(\Ob)$ that will produce an optimal distribution $\hat{\mu} = \hat{z}_\# \mu$ of the elastic material on $\ov{\mathcal{S}}_z$ in general case: when $\sigma$ is an $L^1$ function, from Theorem \ref{thm:recovring_MC_dome} we learn that $\mu = \tilde{\rho} \, \mathcal{L}^2\mres\Omega = \bigl(J_z \rho\bigr) \mathcal{L}^2\mres\Omega =  \bigl(V_0/\Z\, G_{z} : \sigma\bigr) \mathcal{L}^2\mres\Omega$ where the Jacobian $J_z$ is no longer in play. Therefore, for a general measure $\sigma \in \Mes(\Ob;\Sddp)$  and $z \in C^1(\Ob;\R)$, one may guess that an optimal distribution $\hat{\mu}$ may be recovered as follows:
\begin{equation}
	\label{eq:push_forward}
	\hat{\mu} = \hat{z}_\# \mu \ \in \ \Mes_+(\ov{\mathcal{S}}_z), \qquad \mu = \frac{V_0}{\Z}\, G_{z} : \sigma \ \in \  \Mes_+(\Ob).
\end{equation}
In Example \ref{ex:cone}, for $\sigma$,\,$z$ presented in Fig. \ref{fig:cone}(a),\,(b) the optimal elastic material distribution $\hat{\mu}$ is visualized in Fig. \ref{fig:cone}(c). It is easy to verify that $\hat{\mu}$ sharply satisfies the volume constraint since, by definition of the push-forward, $\int_{\ov{\mathcal{S}}_z} d\hat{\mu} = \int_{\Ob} d\mu =V_0/\Z \, \int_{\Ob} G_z :\sigma = V_0/\Z \, \int_{\Ob} \bigl(\tr \,\sigma  + (\nabla z \otimes \nabla z):\sigma \bigr) = V_0$, see \eqref{eq:Z_leq_V_leq_Z}. However, in order to just pose the question of optimality of the structure $\hat{\mu}$, a relaxed formulation of elastic optimal design problem from Section \ref{sec:elastic_design} must be first put forward. Such relaxation to spaces of general measures is not straightforward since the derivations in Section \ref{sec:elastic_design} has built upon classical theory of membrane shells (see \cite{green1968} or \cite{ciarlet2000}). Moreover, the process, although interesting mathematically, in author's opinion would not contribute much to mechanical aspects of the design problem, on which this work essentially focuses.

Instead, by employing the measure-theoretic setting we shall engage yet another design problem where we choose from the whole universe of three dimensional structures under pure tension (or under pure compression), including junctions of full 3D bodies, 2D shells and 1D cables or bars. Since \textit{a priori} there will be no surface $\mathcal{S}_z$ to track, the load $($still generated by $f \in \Mes(\Ob;\R))$ will be assumed to be \textit{vertically transmissible}, which, loosely speaking, means that the load $f$ can be arbitrarily distributed along vertical lines. This line of optimal design is strongly related to the concept of \textit{Prager structures} that were discussed in \cite{rozvany1982}. According to the present author's knowledge an explicit mathematical formulation of the \textit{Prager problem} has never been given except for the planar case, cf. \cite{rozvany1983}. The next subsection puts forward a proposal of such a formulation: both in the plastic setting (originally intended by William Prager and the authors of \cite{rozvany1982}) as well as in the elastic setting. The astonishing result will be that the vault-like structure $\hat{\mu}$ obtained  through the push-forward will be a solution of such 3D problem. In other words, we shall show that this very general design problem can be once more reduced to the 2D convex problems $(\mathcal{P})$,\,$(\mathcal{P}^*)$.   

\subsection{Formulation of the 3D design problem in the plastic and elastic setting -- the Prager problem}

With $\Omega \subset \R^2$ as in previous sections and for a chosen height parameter $H \in (0,\infty]$ we set our 3D design region as a closed set $\Oh := \Ob \times \h$. In the case when $H = \infty$ the design is not vertically constrained, i.e. $\Oh = \Ob \times \R$, which may be considered the most natural setting of the problem. The main advantage of the formulation with finite $H$ (no matter how big) will be the existence result in Proposition \ref{prop:existence}.
Throughout the section the hat symbol $\hat{\argu}$ shall be used to stress that the object is three dimensional or is derived with respect to 3D ambient space; in particular we will use symbols $\hDIV,\ \hat{\nabla},\ \hat{e} = \frac{1}{2}(\hat{\nabla}+\hat{\nabla}^\top)$.

We depart with a precise definition of vertically transmissible load generated by $f \in \Mes(\Ob;\R)$. Essentially, the 3D load represented by the vector-valued measure $\hat{F} \in \Mes(\Oh;\R^3)$ shall be another design variable. Since $f$ is in general a signed measure we must independently handle its positive and negative parts $f_+, f_- \in \Mes_+(\Ob)$. For each of those measures we separately define subsets $\mathscr{T}_+(\Omega,f,H)$, \ $\mathscr{T}_-(\Omega,f,H) \subset \Mes\bigl(\Oh;\R^3\bigr)$:
\begin{equation*}
\mathscr{T}_\pm(\Omega,f,H) := \left\{ \hat{F} = e_3\,\hat{f} \ \left\vert \ \hat{f}\in\Mes_+\bigl(\Oh \bigr),  \ \ \int_{B \times \h} \hat{f} = \int_B f_\pm \quad  \forall\text{  Borel set }  B \subset \Ob \right. \right\}.
\end{equation*}
The set of feasible 3D loads attainable by vertical transmission of $f$ may be readily given as 
\begin{equation}
	\label{eq:transmissible_set}
	\mathscr{T}(\Omega,f,H) := \mathscr{T}_+(\Omega,f,H) - \mathscr{T}_-(\Omega,f,H),
\end{equation}
where the difference is intended in the sense of linear space $\Mes\bigl(\Oh;\R^3\bigr)$. It can be easily checked that whenever $H$ is finite the two sets $\mathscr{T}_\pm(\Omega,f,H)$ are bounded and closed in $\Mes\bigl(\R^3;\R^3\bigr)$ and therefore weakly-* compact. As a result the same holds for $\mathscr{T}(\Omega,f,H)$. The compactness is lost for $H = \infty$, i.e. for $\Oh = \Ob \times \R$.

We move on to formulate the problem of \textit{Minimum Volume Prager Structure}. The design problem shall lie very close to the 3D setting of Michell problem up to two differences: (1) only tensile stress in the structure is permitted; (2) the load $\hat{F}$ is being designed, i.e. by choosing from  $\mathscr{T}(\Omega,f,H)$. Consequently, two mutually polar closed gauges on $\mathrm{T}^3$ shall be used $(\lambda_1(\argu)$ stands for the biggest eigenvalue$)$:
\begin{align}
	\label{eq:3D_spectral_gauges}
	\hat\gamma_+(\hat{\eps}) = \max \bigl\{\lambda_1(\hat{\eps}),0 \bigr\} = \sup_{\hat\tau \in \R^3, \ \abs{\hat{\tau}}\leq 1} \hat{\eps}:(\hat{\tau} \otimes \hat{\tau}), \qquad \hat\gamma^0_+(\hat{\sigma}) =
	\begin{cases}
		\tr \,\hat{\sigma}   &\text{if} \ \ \hat{\sigma} \in \mathrm{T}^3_+,\\
		\infty   &\text{if} \ \ \hat{\sigma} \notin \mathrm{T}^3_+
	\end{cases}
\end{align}
that generate  two mutually conjugate Michell-like elastic potentials:
\begin{equation}
	\label{eq:hat_j}
	\hat{j}_+(\hat{\eps}) = \frac{E_0}{2} \bigl(\hat{\gamma}_+(\hat{\eps}) \bigr)^2, \qquad \hat{j}^*_+(\hat{\sigma}) = \frac{1}{2 E_0} \bigl(\hat{\gamma}^0_+(\hat{\sigma}) \bigr)^2. 
\end{equation}
We put forward the plastic setting of Prager problem (that is parametrized by $H \in (0,\infty]$\,):
\begin{equation*}\tag*{$(\mathrm{MVPS}_H)$}
\hat{\mathcal{V}}^H_\mathrm{min} = \inf_{\substack{\hat{\sigma} \in \Mes(\Oh;\mathrm{T}^3_+) \\ \hat{F} \in \mathscr{T}(\Omega,f,H)}} \left\{ \left. \int_{\Oh}\! \htr \, \hat{\sigma} \ \ \right\vert \ -\hDIV\, \hat\sigma = \hat{F} \ \ \  \text{in } \ \R^3 \backslash \bigl(\bO \times \{0\}\bigr)
\right\}
\end{equation*}
The equilibrium equation $-\hDIV\, \hat\sigma = \hat{F}$ must be understood in the sense of distribution on the open set $\R^3 \backslash \bigl(\bO \times \nolinebreak \{0\}\bigr)$ which incorporates the fact that the designed structure is pinned on $\bO \times \{0\}$.

Prior to analysis of the newly posed problem we jump to formulate the elastic design problem. Here, apart from choosing optimal load $\hat{F}\in \mathscr{T}(\Omega,f,H)$, we shall search for an optimal 3D distribution of elastic material represented by a positive measure $\hat{\mu} \in \Mes_+(\Oh)$. Considering the abstract, measure-theoretic setting we draw upon the pioneering work \cite{bouchitte2001} where compliance of a 3D structure $\hat{\mu}$ subject to load $\hat{F}$ is defined via
\begin{equation}
	\label{eq:3D_comp_def}
	\hat{\Comp}(\hat{\mu},\hat{F}) = \sup_{\substack{\hat{v} \in C^1\!(\Ob\times\R;\R^3)}} \left\{ \left. \int_{\Oh}  \hat{v} \cdot \hat{F} - \frac{E_0}{2} \int_{\Oh}  \Big(  \hat{\gamma}_{+}\bigl(\hat{e}(\hat{v})\bigr) \Big)^2 d\hat\mu \ \ \right\vert \ \hat{v}=0 \text{ on } \bO \times \{0\} \right\};
\end{equation}
the function $\hat{v}$ represents the 3D vectorial displacement field, \textit{a priori} defined in the whole design space $\Ob \times \R$ (in order that the integrals are well defined compact support of $\hat{v}$ may be assumed). After \cite{bouchitte2001} the dual version reads
\begin{equation}
	\label{eq:3D_dual_comp_def}
	\hat{\Comp}(\hat{\mu},\hat{F}) = \inf_{\substack{\hat{S} \in L^2_{\hat{\mu}}(\Oh;\mathrm{T}^3_+)}} \left\{ \left. \frac{1}{2 E_0} \int_{\Oh}  \bigl(  \htr\, \hat{S}\bigr)^2 d\hat\mu \ \ \right\vert \, -\hDIV\bigl(\hat{S}\hat{\mu}\bigr) = \hat{F} \ \ \ \text{in } \R^3 \backslash \bigl(\bO \times \{0\}\bigr) \right\},
\end{equation}
where $L^2_{\hat{\mu}}\bigl(\Oh;\mathrm{T}^3_+\bigr)$ is the Lebesgue space with respect to measure $\hat{\mu}$. The formulation of the problem of \textit{Minimum Compliance Prager Structure} may be put forward:
\begin{equation*}\tag*{$(\mathrm{MCPS}_H)$}
\hat{\Comp}^H_\mathrm{min} =  \inf_{\substack{\hat{\mu} \in \Mes_+(\Oh) \\ \hat{F} \in \mathscr{T}(\Omega,f,H)}} \left\{ \hat{\Comp}(\hat{\mu},\hat{F}) \ \left\vert \ \int_{\hat{\Omega}_H} d\hat{\mu} \leq V_0 \right. \right\}       
\end{equation*}
where once again $V_0$ is the upper bound on the volume of the structure.

Problems $(\mathrm{MVPS}_H)$ and $(\mathrm{MCPS}_H)$ turn out to be the only design problems in this work that are originally well posed, namely through the direct method of the calculus of variations we infer that:
\begin{proposition}
	\label{prop:existence}
	For every finite $H>0$ there always exist solutions of convex problems $(\mathrm{MVPS}_H)$ and $(\mathrm{MCPS}_H)$.
\end{proposition}
\noindent For the proof the reader is referred to \ref{app:3D}.

\begin{remark}
	\label{rem:FMD}
	Unlike in Section \ref{ssec:problem_formulation_elastic} the somewhat unnatural choice of Michell-like elastic potential cannot be defended by results on optimal design of highly porous bodies, cf. \cite{allaire1993,bendsoe1993} -- in three dimensions thus recovered potential and the Michell's one do not coincide. Both in 2D and 3D, however, the Michell's energy can be recast by considering the \textit{Free Material Design} (FMD) in a particular setting recently put forward in \cite[Example 6.2]{bolbotowski2020a}. In the "Fibrous Material Design" problem (FibMD) point-wise we seek the 4th order elasticity tensor $\hat{\mathscr{C}}$ choosing from the set:
	\begin{equation*}
		\hat{\mathscr{H}}_{\mathrm{fib}} = \left\{ \left. \hat{\mathscr{C}} = \sum_{i=1}^n E_i\, \hat{\tau}_i \otimes \hat{\tau}_i \otimes \hat{\tau}_i \otimes \hat{\tau}_i \ \, \right\vert \ E_i \geq 0, \ \hat{\tau}_i \in \R^3, \ \abs{\hat{\tau}_i} = 1   \right\}
	\end{equation*}
	that ought to represent a fibrous material. By departing from  constitutive law of linear elasticity, i.e. $\hat{j} = \hat{j}(\hat{\mathscr{C}},\hat{\eps}) = \frac{1}{2} \bigl(\hat{\mathscr{C}}\hat{\eps}\bigr): \hat{\eps}$ we define compliance $\hat{\Comp} = \hat{\Comp}(\hat{\mathscr{C}},\hat{F})$ with the use of energy potential $\hat{j}_{+}(\hat{\mathscr{C}},\argu) = \bigl( \hat{j}^*(\hat{\mathscr{C}},\argu)  + \mathbbm{I}_{\mathrm{T}_+^3}\bigr)^*$ so that compression is ruled out. The Fibrous Material Design problem is then posed by minimizing compliance $\hat{\Comp}(\hat{\mathscr{C}},\hat{F})$ with respect to $\hat{F} \in \mathscr{T}(\Omega,f,H)$ and $\hat{\mathscr{C}} \in \Mes\bigl(\Oh;\hat{\mathscr{H}}_{\mathrm{fib}}\bigr)$ such that the integral constraint $\int_{\Oh} \tr\,\hat{\mathscr{C}} \leq E_0 V_0$ is met. By adapting computations from Examples 6.2, 6.4 in \cite{bolbotowski2020a} the FibMD problem is in one step reduced to the herein proposed design problem $(\mathrm{MCPS}_H)$. Summing up:
	
	\vspace{0.2cm}
	\centering
	\textit{The design problem $(\mathrm{MCPS}_H)$ may be reinterpreted as a problem of finding a compliance-minimizing\newline body made of linearly elastic material that point-wise is constituted by an elasticity tensor \newline optimally chosen from the fibrous-like class: $\hat{\mathscr{C}} \in \hat{\mathscr{H}}_{\mathrm{fib}}$.} 

\end{remark}

\subsection{Optimal vaults are Prager structures -- solution via the convex problems $(\mathcal{P})$,\,$(\mathcal{P}^*)$}
\label{ssec:3D_to_2D}

The two problems $(\mathrm{MVPS}_H)$ and $(\mathrm{MCPS}_H)$ seem far more general than the problems $(\mathrm{MVV}_\Omega)$ and $(\mathrm{MCV}_\Omega)$ where from the beginning a vault, being a membrane shell of middle surface $\mathcal{S}_z$, was sought. Potentially, by designing an arbitrary 3D structure we could obtain lower values of the volume and compliance, respectively. The following result shows those values satisfy the bounds already established in \eqref{eq:Z_leq_Vmin} and Lemma \ref{lem:Cmin_leq_Z} for the vault design problems: 
\begin{lemma}
	\label{lem:inequalities_3D}
	For every $H \in (0,\infty]$ there hold inequalities:
	\begin{equation*}
	\hat{\mathcal{V}}^H_\mathrm{min} \geq \mathcal{Z}, \qquad \hat{\Comp}^H_\mathrm{min} \geq \frac{\Z^2}{2E_0 V_0}.
	\end{equation*}
\end{lemma}
\noindent In Theorems \ref{thm:recovring_MV_dome} and \ref{thm:recovring_MC_dome} we have seen that whenever solutions of problems $(\mathcal{P})$,\,$(\mathcal{P}^*)$ are smooth enough the two bounds are reached for vault of continuous material distribution. In this section our goal is more ambitious: we wish to cover the case when $(\sigma,q) \in \Mes(\Ob;\Sddp) \times \Mes(\Ob;\Rd)$ and $({u},{w}) \in \mathrm{Lip}(\Ob;\Rd) \times \mathrm{Lip}(\Ob;\R)$. Such regularity of solutions of the pair $(\mathcal{P})$,\,$(\mathcal{P}^*)$ was so far observed for convex $\Omega$ (see Example \ref{ex:cross} below for the counter-example in the non-convex case). In full generality one is forced to work with regularity of $u,w$ guaranteed by Proposition \ref{prop:regularitu_u_w} and at the moment we are missing the mathematical tools for handling  this broader scenario.

Prior to proving Lemma \ref{lem:inequalities_3D} we propose another set of loads that is much larger than $\mathscr{T}(\Omega,f,H)$:
\begin{equation*}
\widetilde{\mathscr{T}}(\Omega,f) := \left\{ \hat{F} \in \Mes(\Ob\times\R;\R^3) \ \left\vert \ \int_{B \times \R} \hat{F} = \int_B e_3 f \ \text{ for any Borel set }  B \subset \Ob \right. \right\},
\end{equation*}
more precisely a proper inclusion $\mathscr{T}(\Omega,f,\infty) \subsetneq \widetilde{\mathscr{T}}(\Omega,f)$ holds true.
In particular, for given $\hat{F}_0 \in \mathscr{T}(\Omega,f,\infty)$, for any $\hat{P}\in\R^3 $ and $x_0 \in \Ob$ one has 
\begin{equation}
	\label{eq:tilde_F}
	\hat{F} = \hat{F}_0 + \hat{P} \,\delta_{(x_0,h_1)} - \hat{P}\, \delta_{(x_0,h_2)} \quad  \in \quad  \widetilde{\mathscr{T}}(\Omega,f)
\end{equation}
where $h_1,h_2 \in \R$ may be chosen arbitrarily. The minimal volume and minimal compliance problems could be modified accordingly: let us agree that the values $\widetilde{\mathcal{V}}_\mathrm{min}$ and $\widetilde{\mathcal{C}}_{\mathrm{min}}$ are the infima as in $(\mathrm{MVPS}_\infty)$ and $(\mathrm{MCPS}_\infty)$, respectively, where $\hat{F}$ is chosen from $\widetilde{\mathscr{T}}(\Omega,f)$ instead. Since $\widetilde{\mathscr{T}}(\Omega,f)$ offers much more choice, inequalities $\hat{\mathcal{V}}^H_\mathrm{min} \geq \widetilde{\mathcal{V}}_\mathrm{min}$, \  $\hat{\mathcal{C}}^H_{\mathrm{min}} \geq \widetilde{\mathcal{C}}_{\mathrm{min}}$ follow. Although this generalization looses its physical meaning, the inequalities $\widetilde{\mathcal{V}}_\mathrm{min} \geq \mathcal{Z}$ and $\hat{\Comp}^H_\mathrm{min} \geq \Z^2/(2E_0 V_0)$ will still hold true and will be easier to prove directly. 

\begin{proof}[Proof of Lemma \ref{lem:inequalities_3D}]
By acknowledging the definition of $\widetilde{\mathcal{V}}_\mathrm{min}$ above it may be written as an $\inf$-$\sup$ problem by disposing of the equilibrium equation:
\begin{align}
	\nonumber
	\hat{\mathcal{V}}^H_\mathrm{min} &\geq \widetilde{\mathcal{V}}_\mathrm{min} =  \inf_{\substack{\hat{\sigma} \in \Mes(\Ob \times \R;\mathrm{T}^3_+) \\ \hat{F} \in \widetilde{\mathscr{T}}(\Omega,f)}} \ \sup_{\substack{\hat{v} \in C^1\!(\Ob \times \R;\R^3)\\ \hat{v}=0 \text{ on }\bO \times \{0\} }} \left\{ \int_{\Ob \times  \R} \htr \, \hat{\sigma} + \left(-\int_{\Ob \times \R} \hat{e}(\hat{v}):\hat{\sigma} + \int_{\Ob \times \R} \hat{v}\cdot \hat{F}   \right) \right\}\\
	\label{eq:V_geq_Y}
	& \geq \sup_{\substack{\hat{v} \in C^1\!(\Ob \times \R;\R^3)\\ \hat{v}=0 \text{ on }\bO \times \{0\} }} \ \inf_{\substack{\hat{\sigma} \in \Mes(\Ob \times \R;\mathrm{T}^3_+) \\ \hat{F} \in \widetilde{\mathscr{T}}(\Omega,f)}}
	\left\{ \int_{\Ob \times \R} \hat{v}\cdot\hat{F} + \int_{\Ob \times \R}\Big( -\hat{e}(\hat{v}):\hat{\sigma} + \htr \, \hat{\sigma}  \Big)  \right\} =: \mathcal{Y},
\end{align}
where obtaining the bottom line is straightforward. We shall prove that the number $\mathcal{Y}$ defined above equals
\begin{equation*}
	\mathcal{Y} =\sup_{\substack{(u,w) \in C^1\!(\Ob;\R^3)\\(u,w)=0 \text{ on }\bO }} \left\{  \int_{\Ob} w \,f \ \ \bigg\vert \ \hat{\gamma}_+\bigl( \hat{e}(\hat{v}) \bigr) \leq 1 \ \text{ in } \Ob\times \R \ \text{ \ for \  }\ \hat{v}(\hat{x}) = u(x)+w(x) \,e_3   \right\} 
\end{equation*}
where by $\hat{v}(\hat{x}) = u(x)+w(x) \,e_3$ we understand that $\hat{v}$ is independent of the third, vertical coordinate. We must prove that unless $\hat{v}$ satisfies the two constraints: $\hat{v}(\hat{x}) = u(x)+w(x) \,e_3$ and $\hat{\gamma}_+\bigl( \hat{e}(\hat{v}) \bigr) \leq 1$, the infimum in the bottom line of \eqref{eq:V_geq_Y} equals $-\infty$. Assume that there exist distinct points $\hat{x}_1 =(x_0,h_1), \ \hat{x}_2=(x_0,h_2) \in \Ob\times \R$ such that $\hat{v}(x_0,h_1) \neq \hat{v}(x_0,h_2)$. Then, for any $\hat{F}_0 \in \mathscr{T}(\Omega,f,H)$ and $\hat{P} = -t\,\bigl(\hat{v}(x_0,h_1) - \hat{v}(x_0,h_2) \bigr)$ with $t \geq 0$ the load $\hat{F}$ according to \eqref{eq:tilde_F} is an element of $\widetilde{\mathscr{T}}(\Omega,f)$ for which $ \int \hat{v}\cdot\hat{F} =  \int \hat{v}\cdot\hat{F}_0\, - t\, \abs{\hat{v}(x_0,h_1) - \hat{v}(x_0,h_2)}^2$. By sending $t$ to $\infty$ we obtain $\inf_{\hat{F} \in \widetilde{\mathscr{T}}(\Omega,f)} \int \hat{v}\cdot\hat{F} = -\infty$. Next, recalling that $\tr\, \hat{\sigma} = \hat{\gamma}_+^0(\hat{\sigma})$ we infer that $\inf_{\hat{\sigma}\succeq 0} \int\bigl( -\hat{e}(\hat{v}):\nolinebreak\hat{\sigma} + \htr \, \hat{\sigma}  \bigr) = - \infty$ whenever there is a point $\hat{x}_0 \in \Ob\times \R$ such that $\hat{\gamma}_+\bigl(\hat{e}(\hat{v})\bigr)(\hat{x}_0)>1$, see the argument below \eqref{eq:weak_duality}.
Once the two constraints are enforced upon $\hat{v}$ the choice of $\hat{F}\in \widetilde{\mathscr{T}}(\Omega,f)$ is immaterial since always $\int_{\Ob \times \R} \hat{v}\cdot\hat{F} = \int_{\Ob} w \,f$, while the infimum in $\hat{\sigma}$ is reached for $\hat{\sigma}=0$. The formula for $\mathcal{Y}$ above is established. 

In the next step we will show that $\mathcal{Y} = \Z$ and to that aim the constraint $\hat{\gamma}_+\bigl( \hat{e}(\hat{v}) \bigr) \leq 1$ must be rewritten. Since $\hat{v}(\hat{x}) = u(x)+w(x) \,e_3$ the strain field reads $\hat{e}(\hat{u})(\hat{x}) = e(u)(x) + \nabla w(x) \, \symtens \, e_3$ where the plane strain $e(u) \in \Sdd$ is naturally embedded into $\mathrm{T}^3$. According to \eqref{eq:3D_spectral_gauges} we must test $ \hat{e}(\hat{v}) : (\hat{\tau} \otimes \hat{\tau})$ with $\hat{\tau}\in \R^3$, \ $\abs{\hat{\tau}}\leq 1$. Since $\hat{e}(\hat{v}):(e_3\otimes e_3) = 0$ the variables may be changed:
\begin{equation*}
	\hat{\tau} = \frac{1}{\sqrt{1+\psi^2}}\,\tau + \frac{\psi}{\sqrt{1+\psi^2}}\, e_3
\end{equation*}
where $\tau \in \R^2$, \ $\abs{\tau}\leq 1$ and $\psi \in \R$ is the slope of $\hat{\tau}$ with respect to the base plane $\R^2$. Owing to the formula \eqref{eq:h_with_psi} the gauge $ h\bigl(\nabla w, e(u) \bigr)$ reappears:
\begin{align}
	\label{eq:hat_gamma_h}
	\sup_{\hat{\tau}\in\R^3, \ \abs{\hat\tau}\leq 1} \biggl\{\hat{e}(\hat{v}):(\hat{\tau} \otimes \hat{\tau}) \biggr\} =   \sup_{\substack{\tau \in \Rd, \ \abs{\tau}\leq 1 \\ \psi \in \R}} \biggl\{\frac{1}{1+\psi^2}\, e(u):(\tau \otimes \tau)+ \frac{\psi}{1+\psi^2} \, \nabla w \cdot \tau \biggr\} = h\bigl(\nabla w, e(u) \bigr).
\end{align}
The fact that $\mathcal{Y} = \Z$ is now a direct implication of Proposition \ref{prop:h_and_g_C}. The first inequality in the assertion is proved.

The value $\hat{\mathcal{C}}^H_\mathrm{min}$ is an $\inf$-$\sup$ and so is $\widetilde{\mathcal{C}}_\mathrm{min}$; the order may be swapped to $\sup$-$\inf$ with an inequality:
\begin{align*}
	&\hat{\mathcal{C}}^H_\mathrm{min} \geq \widetilde{\mathcal{C}}_\mathrm{min} \geq \sup_{\substack{\hat{v} \in C^1\!(\Ob \times \R;\R^3)\\ \hat{v}=0 \text{ on }\bO \times \{0\} }} \ \inf_{\substack{\hat{\mu} \in \Mes_+(\Ob \times \R) \\ \hat{F} \in \widetilde{\mathscr{T}}(\Omega,f)}}
	\left\{  \int_{\Ob\times \R}  \hat{v} \cdot \hat{F} - \frac{E_0}{2} \int_{\Ob\times \R}  \Big(  \hat{\gamma}_{+}\bigl(\hat{e}(\hat{v})\bigr)\Big)^2 d\hat\mu  \biggr)  \right\}\\
	= & \sup_{\substack{(u,w) \in C^1\!(\Ob;\R^3)\\(u,w)=0 \text{ on }\bO }} \left\{\left. \inf_{\substack{\hat{\mu} \in \Mes_+(\Ob \times \R)}}
	\left\{  \int_{\Ob} w \,f  - \frac{E_0}{2} \int_{\Ob\times \R}  \Big(  \hat{\gamma}_{+}\bigl(\hat{e}(\hat{v})\bigr)\Big)^2 d\hat\mu  \biggr)  \right\} \ \right\vert \ \hat{v}(\hat{x}) = u(x)+w(x) \,e_3 \right\} = \frac{\mathcal{Y}^2}{2 E_0 V_0}.
\end{align*}
To pass to the second line the argument with taking $\hat{F} = \hat{F}(t)$ as above may be employed. The last equality is due to a variant of technique used in the proof of Lemma \ref{lem:Cmin_leq_Z}, the reader is referred to  \cite[Proposition 2]{bouchitte2007} for a more precise statement. The proof is complete by acknowledging that $\mathcal{Y} = \Z$ once again.
\end{proof}

Inequalities in Lemma \ref{lem:inequalities_3D} pave a clear way to proving that vaults are optimal in problems $(\mathrm{MVPS}_H)$ and $(\mathrm{MCPS}_H)$: based on solutions $(\sigma,q)$ and $({u},{w})$ of problems $(\mathcal{P})$ and $(\mathcal{P}^*)$ we must construct admissible 3D measures $(\hat{\sigma},\hat{F})$ and $(\hat{\mu},\hat{F})$ that will furnish $\int \tr\, \hat{\sigma} = \Z$ and $\hat{\mathcal{C}}(\hat{\mu},\hat{F}) = \Z^2/(2 E_0 V_0)$. As declared at the beginning of the present subsection we will not cover the fully general case and we will henceforward assume that solutions $({u},{w})$ are Lipschitz continuous. This extra regularity allows to introduce for any Radon measure $\mu \in \Mes_+(\Ob)$ the notion of $\mu$-tangential gradient $\nabla_\mu {w}$ and $\mu$-tangential operator $e_\mu({w})$, for details the reader is referred to \cite{bouchitte1997} and the further developments e.g in \cite{bouchitte2001,bouchitte2007}. Eventually, for $({u},{w}) \in \mathrm{Lip}(\Ob;\Rd) \times \mathrm{Lip}(\Ob;\R)$ and measures $(\sigma,q)$ the generalized optimality conditions were derived in \cite{bouchitte2020} with the use of measure-tangential calculus: the first two conditions are identical to conditions $(i)$,\,$(ii)$ from Theorem \ref{thm:opt_cond} above while with $\mu =\tr\,\sigma$ the other two read:
\begin{equation}
	\label{eq:opt_cond_mu}
	\begin{array}{ll}
		(iii)' & \  \bigl( \frac{1}{4}\, \nabla_\mu {w} \otimes \nabla_\mu {w} + e_\mu({u}) \bigr):\sigma = \tr\,\sigma;\\
		(iv)' & \   q = \frac{1}{2}\, \sigma\, \nabla_\mu {w}.
	\end{array}
\end{equation}

In the sequel $(\sigma,q) \in \Mes(\Ob;\Sddp) \times \Mes(\Ob;\Rd)$ and $({u},{w}) \in \mathrm{Lip}(\Ob;\Rd) \times \mathrm{Lip}(\Ob;\R)$ will be fixed solutions of problems $(\mathcal{P})$ and $(\mathcal{P}^*)$; in particular there hold equilibrium equations $\DIV\, \sigma=0$, $-\dive\, q = f$ in $\Omega$. For the elevation function we choose $z = \frac{1}{2} {w}$. Since $z$ is Lipschitz continuous the matrix function $(\mathrm{I}+\nabla_\mu z \otimes \nabla_\mu z)$ is uniformly bounded from each side (in the sense of operator $\preceq$) by two positive definite matrices, therefore the measure $\mu = \tr \,\sigma$ may be redefined such that
\begin{equation}
	\label{eq:opt_mu_S_2D}
	\sigma = S \mu, \qquad \mu= \frac{V_0}{\Z}\, (\mathrm{I}+\nabla_\mu z \otimes \nabla_\mu z):\sigma, \qquad  (\mathrm{I}+\nabla_\mu z \otimes \nabla_\mu z):S = \frac{\Z}{V_0} \quad \mu\text{-a.e.},
\end{equation}
which does not change the operators $\nabla_\mu$ and $e_\mu$ hence the optimality conditions \eqref{eq:opt_cond_mu} remain true. By means of the push-forward, from the plane fields above we define the 3D fields $\hat{\sigma} \in \Mes(\Ob\times\R;\mathrm{T}^3_+)$, $\hat{\mu} \in \Mes_+(\Ob\times\R)$, $\hat{S} \in L^\infty_{\hat{\mu}}(\Ob\times \R; \mathrm{T}^3_+)$, $\hat{F} \in \Mes(\Ob\times \R ; \R^3)$:
\begin{equation}
\label{eq:opt_mu_S_F}
\hat\sigma = \hat{S} \hat{\mu}, \qquad \hat{\mu} = \hat{z}_\# \mu, \qquad \hat{S} =  \left(\begin{bmatrix}
\,\mathrm{I}  &  \nabla_\mu z\, 
\end{bmatrix}^\top\!
S\, \begin{bmatrix}
\,\mathrm{I}  &  \nabla_\mu z\, 
\end{bmatrix} \right) \circ \hat{z}^{-1}, \qquad \hat{F} = e_3 \,\hat{z}_\# \,f,
\end{equation}
where by $\bigl[ \,\mathrm{I}  \ \ \   \nabla_\mu z\, \bigr] = \nabla_\mu \hat{z}$ we understand the $2 \times 3$ matrix composed of the $2 \times 2$ identity matrix $\mathrm{I}$ and the column 2D vector $\nabla_\mu z$; let us note that $\hat{F}$ is exactly the load $\hat{F}_{f,z}$ defined in Section \ref{ssec:problem_formulation_plastic}. It is straightforward to check that $\hat{F} \in \mathscr{T}(\Omega,f,\infty)$ and consequently the measure $\hat{\sigma}$ turns out to be feasible for problem 
$(\mathrm{MVPS}_\infty)$, namely:
\begin{lemma}
	\label{lem:feas_hat_sigma}
	The matrix-valued measure $\hat\sigma = \hat{S} \hat{\mu}$ in \eqref{eq:opt_mu_S_F} satisfies the 3D equilibrium equation $-\hDIV \, \hat{\sigma} = \hat{F}$ in the sense of distributions on the open set $\R^3\backslash (\bO \times \{0\})$.
\end{lemma}
\begin{proof}[Proof of Lemma \ref{lem:feas_hat_sigma}]
	We must prove that for all smooth functions $\hat{v} \in C^1(\Ob \times \R ;\R^3)$ with compact support in $\R^3\backslash (\bO \times \{0\})$ there holds $\int_{\Ob\times\R} \hat{\nabla}\hat{v} :\hat{\sigma} = \int_{\Ob \times \R} \hat{v} \cdot \hat{F}$. For convenience we decompose the virtual displacement function to $\hat{v}=(\hat{u},\hat{w})$; next we set $u = \hat{u} \circ \hat{z} \in \mathrm{Lip}(\Ob;\Rd)$, $w = \hat{w} \circ \hat{z} \in \mathrm{Lip}(\Ob;\R)$ and moreover $v = \hat{v}\circ \hat{z} = (u,w) \in \mathrm{Lip}(\Ob;\R^3)$. Due to smoothness of $\hat{v}$ the following chain rule for $\mu$-tangential differentiation holds:
	\begin{equation*}
	\nabla_\mu v (x) = \hat{\nabla} \hat{v}\bigl(\hat{z}(x) \bigr)\, \nabla_\mu {\hat{z}}(x) = \hat{\nabla} \hat{v}\bigl(\hat{z}(x) \bigr)\,\begin{bmatrix}
	\,\mathrm{I}  &  \nabla_\mu z (x)\, 
	\end{bmatrix} \qquad \text{for } \mu\text{-a.e. } x,
	\end{equation*}
	which, owing to the relation \eqref{eq:opt_mu_S_F} between $\hat{S}$ and $S$, furnishes $\mu$-a.e. in $\Ob$:
	\begin{equation}
	\label{eq:nabla_v:S}
	\hat{\nabla}\hat{v}\bigl(\hat{z}(\argu)\bigr) : \hat{S}\bigl(\hat{z}(\argu)\bigr) =  \nabla_\mu v : \begin{bmatrix}
	\,S  &  S\, \nabla_\mu z
	\end{bmatrix} = \Big( e_\mu(u):S + \nabla_\mu w \cdot \bigl(S\, \nabla_\mu z \bigr) \Big).
	\end{equation}
	By the change of variable formula in the general, push-forward setting
	\begin{align*}
	\int_{\Ob\times\R} \hat{\nabla}\hat{v} :\hat{\sigma} =\int_{\Ob\times\R} \hat{\nabla}\hat{v} :\hat{S} \, d\hat\mu =& \int_{\Ob} \hat{\nabla}\hat{v}\bigl(\hat{z}(\argu)\bigr) : \hat{S}\bigl(\hat{z}(\argu) \bigr)\, d\mu= \int_\Ob \Big(e_\mu(u):S + \nabla_\mu w \cdot \bigl(S\, \nabla_\mu z \bigr) \Big)d\mu \\ 
	= & \int_\Ob \Big(e_\mu(u):\sigma + \nabla_\mu w \cdot q \Big) = \int_\Ob w\, f= \int_{\Ob \times \R} \hat{w}\, (\hat{z}_\# f)= \int_{\Ob \times \R} \hat{v} \cdot \hat{F}.
	\end{align*}
	To pass to the second line we have explicitly used optimality condition $(iv)'$ in \eqref{eq:opt_cond_mu}, i.e. $q = \frac{1}{2}\,  \sigma \, \nabla_\mu {w} = \sigma \, \nabla_\mu z$. Equality $\int_\Ob \bigl(e_\mu(u):\sigma + \nabla_\mu w \cdot q \bigr) = \int_\Ob w\, f$ follows from equilibrium equations $\DIV \,\sigma =0$, $-\dive\,q = f$, see \cite[Corollary 3.15]{bouchitte2020} for the integration by parts formula in the framework of $\mu$-tangential calculus.
\end{proof} 
 
Together with the load $\hat{F}$, the lower dimensional measures $\hat{\sigma}$ and $\hat{\mu}$ are the candidates for solutions of problems $(\mathrm{MVPS}_H)$ and $(\mathrm{MCPS}_H)$ respectively. The necessary condition is that the domain $\Oh = \Ob \times [-H,H]$ contains the supports of $\hat{F}$, $\hat{\sigma}$ and $\hat{\mu}$. This is equivalent to enforcing that $\mathcal{S}_z \subset \Oh$ or that $\norm{z}_{\infty} \leq H$. Since $z = \frac{1}{2}\, w$ where $w$ is a solution of $(\mathcal{P}^*)$, according to estimates \eqref{eq:estimates_u_w} it is enough to guarantee that $H$ is not smaller than $\mathrm{diam}(\Omega)/\sqrt{2}$. For such $H$ optimality of vaults as a three-dimensional fibrous structures can be readily claimed; the proof of the theorem put forward below uses ideas already known from proofs of Theorems \ref{thm:recovring_MV_dome} and \ref{thm:recovring_MC_dome}, the differences are mainly technical due to the measure-tangential calculus. For those reasons the proof is moved to \ref{app:3D}.

\begin{theorem}[\textbf{Construction of a vault solving the Prager problem}]
	\label{thm:optimal_3D_structure}
	Assume that the pairs $(\sigma,q) \in \Mes(\Ob;\Sddp) \times \Mes(\Ob;\Rd)$ and  $(u,w) \in \mathrm{Lip}(\Ob;\Rd) \times \mathrm{Lip}(\Ob;\R)$ are solutions of the problems $(\mathcal{P})$ and $(\mathcal{P}^*)$ respectively. Then, for $z = \frac{1}{2}\,w$ let us choose $\hat{\mu}, \hat{S}, \hat{F}$ in accordance with \eqref{eq:opt_mu_S_F}, moreover we set
	\begin{equation*}
		{\hat{v}}(\hat{x}) =u(x) + w(x) \,e_3
	\end{equation*}
	in the sense that $\hat{v}$ does not depend on the third coordinate. Then, provided that $H \geq \mathrm{diam}(\Omega)/\sqrt{2}$, the pair
	\begin{equation*}
	\hat{\sigma} = \hat{S} \hat{\mu} \in \Mes(\ov{\mathcal{S}}_z;\mathrm{T}_+^3), \qquad \hat{F} \in \Mes(\ov{\mathcal{S}}_z;\R^3)
	\end{equation*}
	solves the minimum volume problem $(\mathrm{MVPS}_H)$ with $\hat{\mathcal{V}}^H_\mathrm{min} = \Z$, while the pair
	\begin{equation*}
	\hat{\mu}\in \Mes_+(\ov{\mathcal{S}}_z), \quad \hat{F} \in \Mes(\ov{\mathcal{S}}_z;\R^3)
	\end{equation*}
	solves the minimum compliance problem $(\mathrm{MCPS}_H)$ with $\hat{\Comp}^H_\mathrm{min} = \frac{\Z^2}{2E_0 V_0}$. Moreover, functions $\hat{v}_\e = \Z/(E_0 V_0) \, \hat{v}\in \mathrm{Lip}\bigl(\Oh;\R^3\bigr)$ and $\hat{S} \in L^\infty_{\hat{\mu}}\bigl(\Oh;\mathrm{T}^3_+\bigr)$ solve the elasticity problems \eqref{eq:3D_comp_def} (its relaxed variant) and \eqref{eq:3D_dual_comp_def} respectively. The constitutive relation holds:
	\begin{equation}
		\label{eq:const_law_mu}
		\hat{S} \in \partial \hat{j}_+\bigl(\hat{e}_{\hat{\mu}}(\hat{v}_\e) \bigr) \qquad \hat{\mu}\text{-a.e.}
	\end{equation}
\end{theorem}

\begin{example}[\textbf{Prager structures over a disk domain and subject to a single point force}]
	Let the data $x_0 \in \Omega = \bigl\{x \in \Rd \, \big\vert\, \abs{x} < R \bigr\}$ and $f = P \,\delta_{x_0}$ be as in Example \ref{ex:cone}. Let $\pi \in \Mes_+(\bO)$ be any probability measure supported on the circle $\bO$ that satisfies \eqref{eq:condition_p}. According to Example \ref{ex:cone} measures $(\sigma,q)$ in \eqref{eq:fibrous_sigma_q} and Lipschitz continuous functions $(u,w)$ in \eqref{eq:cone_uw} solve problems $(\mathcal{P})$ and $(\mathcal{P}^*)$ respectively.
	
	Then, assuming any parameter $H$ that is not smaller than $R$, the pairs $(\hat{\sigma},\hat{F})$ and $(\hat{\mu},\hat{F})$ constructed via Theorem \ref{thm:optimal_3D_structure} solve the Prager problems: the plastic setting $(\mathrm{MVPS}_H)$ and, respectively, the elastic setting $(\mathrm{MCPS}_H)$. 
	The Prager structure concentrates on the surface $\mathcal{S}_z$ being a graph of the function $z = \frac{1}{2}\, w$. The one force load is elevated to the \textit{apical point} $\hat{x}_0 \in \R^3$:
	\begin{equation*}
		\hat{F} = P \,\delta_{\hat{x}_0}, \qquad \hat{x}_0 = \left({x}_0, \sqrt{R^2 - \vert x_0 \vert^2} \right).
	\end{equation*}
	By exploiting the change of variables formula one may easily show that the push-forward operation preserves the slicing formula for $\sigma$, namely $\hat{\sigma}$ and $\hat{\mu}$ decompose to 1D straight bars, cf. Fig. \ref{fig:cone}(c):
	\begin{align*}
		\hat{\sigma} &= \frac{P}{\sqrt{R^2 - \vert x_0 \vert^2}}\int_\bO \abs{\hat{x}-\hat{x}_0}\,\hat{\sigma}^{\hat{x},\hat{x}_0} \, \pi(dx), \qquad \hat{\sig}^{\hat{x},\hat{x}_0} := \hat{\tau}^{\hat{x},\hat{x}_0} \otimes \hat{\tau}^{\hat{x},\hat{x}_0} \, \Ha^1 \mres [\hat{x},\hat{x}_0], \\
		\hat{\mu} = \tr\,\hat{\sigma} &= \frac{P}{\sqrt{R^2 - \vert x_0 \vert^2}}\int_\bO \abs{\hat{x}-\hat{x}_0}\,\hat{\mu}^{\hat{x},\hat{x}_0} \, \pi(dx), \qquad \hat{\mu}^{\hat{x},\hat{x}_0} := \Ha^1 \mres [\hat{x},\hat{x}_0]
	\end{align*}
	where we agree that $\hat{x} := (x,0) \in \R^3$ and $\hat{\tau}^{\hat{x},\hat{x}_0} := (\hat{x}_0-\hat{x})/\abs{\hat{x}_0-\hat{x}}$. The optimal objective values read: $\hat{\mathcal{V}}^H_\mathrm{min} = \Z = 2 P \sqrt{R^2 - \vert x_0 \vert^2}$ and $\hat{\Comp}^H_\mathrm{min} = \Z^2/(2E_0 V_0) = 2 P (R^2 - \vert x_0 \vert^2)/(E_0 V_0)$.
	
	For a finite number of bars, i.e. for finitely supported $\pi$, the truss relating to $\hat{\sigma}$ was for the first time considered as a candidate for the Prager structure in \cite[Proposition 3.2]{rozvany1982}, see also \cite[Section 6.3.1]{lewinski2019a}. Therein, however, global optimality was not proved -- the authors restricted the search to trusses consisting of bars connecting the apical point to the supporting circle $\bO \times \{0\}$ only. By employing Theorem \ref{thm:optimal_3D_structure} we have proved that their proposal is indeed an exact Prager structure. 

\end{example}

In the course of proving Theorem \ref{thm:optimal_3D_structure} we find a more general result that is rather unexpected. In both problems $(\mathrm{MVPS}_H)$ and $(\mathrm{MCPS}_H)$ the designed 3D structure is allowed to be pinned only on the plane curve $\bO \times \{0\}$. In the proof of inequality $\hat{\mathcal{V}}^H_\mathrm{min} \geq \Z$ in Lemma \ref{lem:inequalities_3D} (cf. the chain \eqref{eq:V_geq_Y}) this is reflected in taking the supremum among all functions $\hat{v} \in C^1\bigl(\Oh;\R^3\bigr)$ that are zero precisely on $\bO \times \{0\}$. Eventually the number $\mathcal{Y}$ at the end of the chain is recognized as supremum with respect to functions $\hat{v}$ that does not depend on the third, vertical variable at all and, as a result, they must be zero on the whole infinite cylinder $\bO \times \R$. We infer that this stronger Dirichlet condition could be imposed right from the start, in other words $\Z$ would remain a lower bound for $\hat{\mathcal{V}}^H_\mathrm{min}$ even if we allowed the structure to rest on the whole cylinder. The idea behind the proof of Theorem \ref{thm:optimal_3D_structure} was to give a recipe for an admissible structure that saturates the lower bound $\Z$, which works in this yet another, broader design problem. Similar conclusions concern the compliance minimization problem thus ultimately:
\begin{corollary}
	\label{cor:3D_on_cylinder}
	Under assumptions of Theorem \ref{thm:optimal_3D_structure} the lower dimensional pair $(\hat{\sigma},\hat{F})$ solves the minimum volume design problem where the structure may rest on the whole cylinder $\bO \times \h$, i.e. it solves the problem
	\begin{equation*}
	\inf_{\substack{\hat{\sigma} \in \Mes(\Oh;\mathrm{T}^3_+) \\ \hat{F} \in \mathscr{T}(\Omega,f,H)}} \left\{ \left. \int_{\Oh} \htr \, \hat{\sigma} \ \ \right\vert \ -\hDIV\, \hat\sigma = \hat{F} \ \  \text{in }\ \R^3 \backslash \bigl(\bO \times [-H,H]\bigr)
	\right\}.
	\end{equation*}
	An analogous result holds for the minimum compliance problem. 
\end{corollary}

This corollary settles another design problem where together with the structure we seek an optimal kinematical support being a 3D curve $\hat{\Gamma}$ obtained by elevating the plane curve $\bO \times \{0\}$. Corollary \ref{cor:3D_on_cylinder} shows that with all the cylinder $\Ob \times \R$ available the choice $\bO \times \{0\}$ is still optimal. The question of finding the optimal supporting curve $\hat{\Gamma}$ is therefore trivial. A completely different design problem would be to search for the optimal structure with a fixed, prescribed supporting curve $\hat{\Gamma}$ -- to this topic we will devote a separate Section \ref{ssec:boundary_conditions}.

\section{Discrete formulation and the conic quadratic program. The optimal grid-shell problem}
\label{sec:discrete}

After a theoretical coverage of optimal design of vaults we are addressing the matter of choosing a suitable discretization  strategy. We managed to show that at the core of every design problem covered in this work lies the pair of mutually dual convex problems $(\mathcal{P})$,\,$(\mathcal{P}^*)$ defined in the 2D reference domain $\Omega$. Since a lot of similarities between the pair $(\mathcal{P})$,\,$(\mathcal{P})^*$ and the pair of Michell problems \eqref{eq:Michell_min}, \eqref{eq:Michell_max} has been pointed out one may expect that techniques based on the concept of ground structure (cf. Section \ref{ssec:Michell}) could be employed here as well.

\subsection{Conic quadratic programming problem as a discretization of the pair $(\mathcal{P})$,\,$(\mathcal{P}^*)$}
\label{ssec:discretication}

While the ground structure approach was quite natural for Michell problem considering that Michell structures are generalized trusses, it is not obvious how to adapt the method to a rather abstract pair of problems $(\mathcal{P})$,\,$(\mathcal{P}^*)$. In Section \ref{ssec:Michell} an alternative perspective on the ground structure method was taken notice of: the constraint $\gamma\bigl(e(u)\bigr) \leq 1$ (or equivalently $-\mathrm{I} \preceq e(u) \preceq \mathrm{I}$) in problem \eqref{eq:Michell_max} was rewritten as a two-point condition \eqref{eq:two_point_Michell} that, in the context of the ground structure, bounds the virtual elongation of bars. With little intuition behind the pair $(\mathcal{P})$,\,$(\mathcal{P}^*)$ we shall look for mathematical analogies and ultimately we will rewrite the point-wise constraint $\frac{1}{4}\, \nabla w \otimes \nabla w + e(u) \preceq  \nolinebreak \mathrm{I}$ as a two-point condition as well. To that aim we introduce an auxiliary operator $\Lambda: \nolinebreak C(\Ob;\Rd) \times C(\Ob;\R) \mapsto C(\Ob \times \Ob;\R)$ which for distinct points $x,y \in \Ob$ gives
\begin{equation*}
	\Lambda(u,w)(x,y):=\frac{1}{4}\,\frac{\abs{w(y) -w(x)}^2}{\abs{y-x}} + \bigl(u(y) - u(x) \bigr)\cdot \frac{y-x}{\abs{y-x}},
\end{equation*}
while for each $x \in \Ob$ we agree that $\Lambda(u,w)(x,x) = 0$. The following result was proved in \cite[Lemma 3.5]{bouchitte2020}:
\begin{proposition}
	\label{prop:two-point_condition}
	For $\Omega \subset\Rd$ being any bounded domain (not necessarily convex) let us take functions $u \in C^1(\Ob;\Rd)$, $w \in C^1(\Ob;\R)$ with zero boundary values $u=0$, $w=0$ on $\Ob$. Then, the two constraints are equivalent:
	\begin{equation*}
		\frac{1}{4}\, \nabla w \otimes \nabla w + e(u) \preceq  \mathrm{I} \quad \ \ \text{point-wise in } \Ob \qquad \Leftrightarrow \qquad \Lambda(u,w)(x,y) \leq \abs{y-x} \quad \ \ \forall\,(x,y) \in \Ob \times \Ob.
	\end{equation*}
\end{proposition}

Writing the constraint in $(\mathcal{P}^*)$ in its two-point variant paves a way to an alternative duality scheme. For the infinite dimensional setting of this duality the reader is referred to \cite[Section 3.2]{bouchitte2020}, here we shall concentrate on the discrete version only. Henceforward $X$ we will stand for a finite subset of $\Ob$, see Fig. \ref{fig:ground_structure}(a). The only assumption imposed on $X$ will be the following:
\begin{equation}
	\label{eq:assum_on_X}
	X \backslash \bO \subset \mathrm{int} \bigl(\mathrm{conv}(X \cap \bO ) \bigr),
\end{equation}
where $\mathrm{conv}(A)$ stands for the convex hull of a set $A$, while $\mathrm{int}(A)$ for its interior. For a given load $f \in \Mes(\Ob;\R)$ we will always work with its discretization $f_X \in \Mes(X;\R)$, i.e. a load consisting of point forces applied at points in $X$ only, that weakly-* converges to $f$ with the resolution parameter of $X$ approaching zero. The load $f_X$ will be proposed individually in each example, see Section \ref{ssec:examples} below. We define a problem:
\begin{equation}
	\label{eq:Z_X_def}
	\Z_X:=\sup_{\substack{u \in C\left(\Ob;\Rd\right) \\ w \in C(\Ob;\R)}} \left\{ \left. \, \int_\Ob w\, f_X  \ \, \right\vert \, u=0, \, w=0 \text{ on } \bO, \ \ \ \Lambda(u,w)(x,y) \leq \abs{y-x} \quad  \forall\, (x,y) \in X \times X \right\}
\end{equation}
which plays the role of a discrete variant of problem $(\mathcal{P}^*)$. Indeed, by acknowledging Proposition \ref{prop:two-point_condition} in the problem above a finite subset of the continuum of  two-point constraints $\Lambda(u,w)(x,y) \leq \abs{x-y}\ $ $\forall (x,y) \in \Ob\times\Ob$ from $(\mathcal{P}^*)$ is imposed. Since $f_X$ is supported on $X$ it readily follows that only the values of $u, w$ at points of $X$ are significant. This fact encourages to pose the discrete problem in algebraic, vector-matrix formulation.

\begin{figure}[h]
	\centering
	\subfloat[]{\includegraphics*[trim={0.cm -0.65cm 0cm 0cm},clip,width=0.35\textwidth]{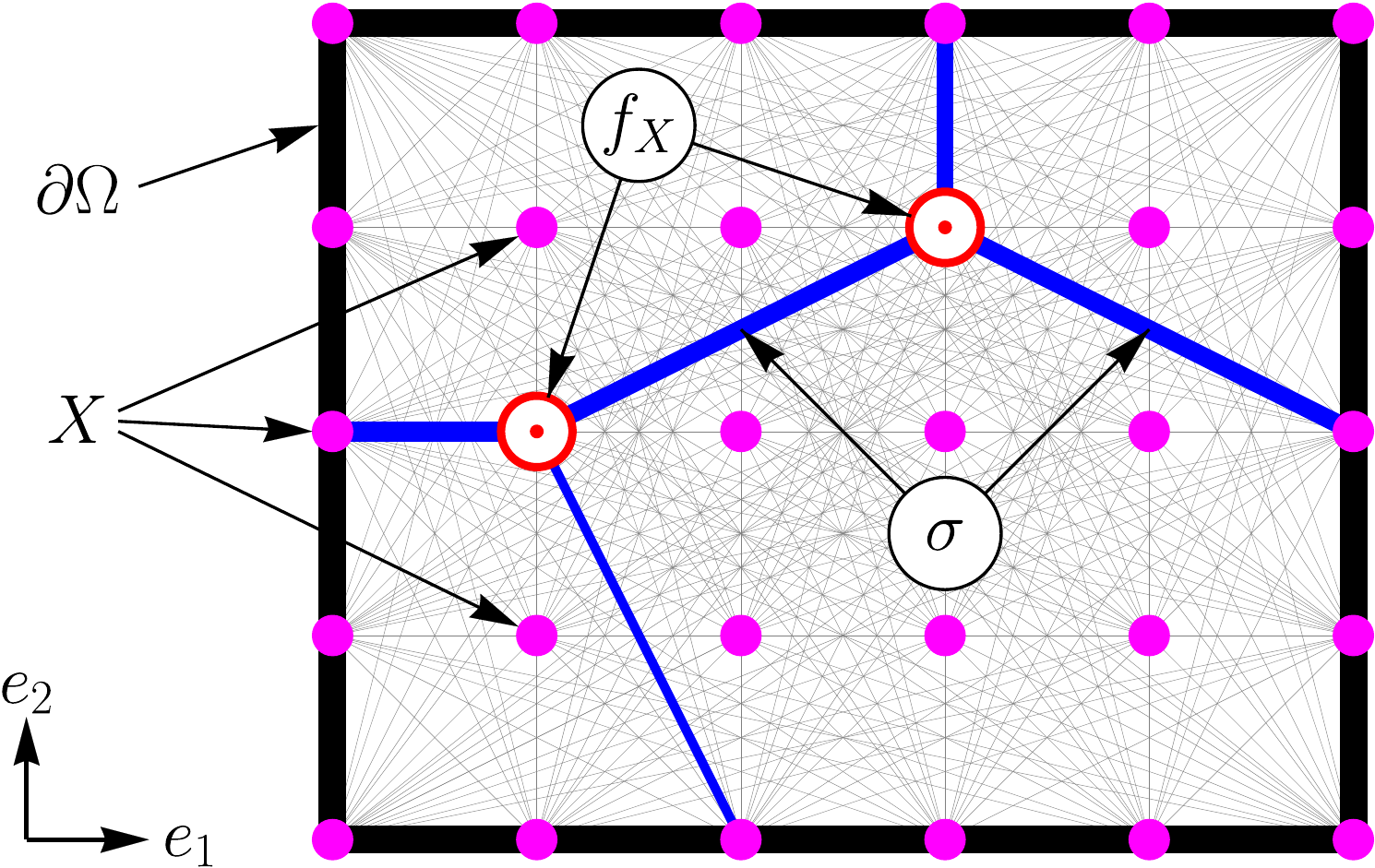}}\hspace{0.8cm}
	\subfloat[]{\includegraphics*[trim={1.cm 0cm 0.5cm 0.5cm},clip,width=0.59\textwidth]{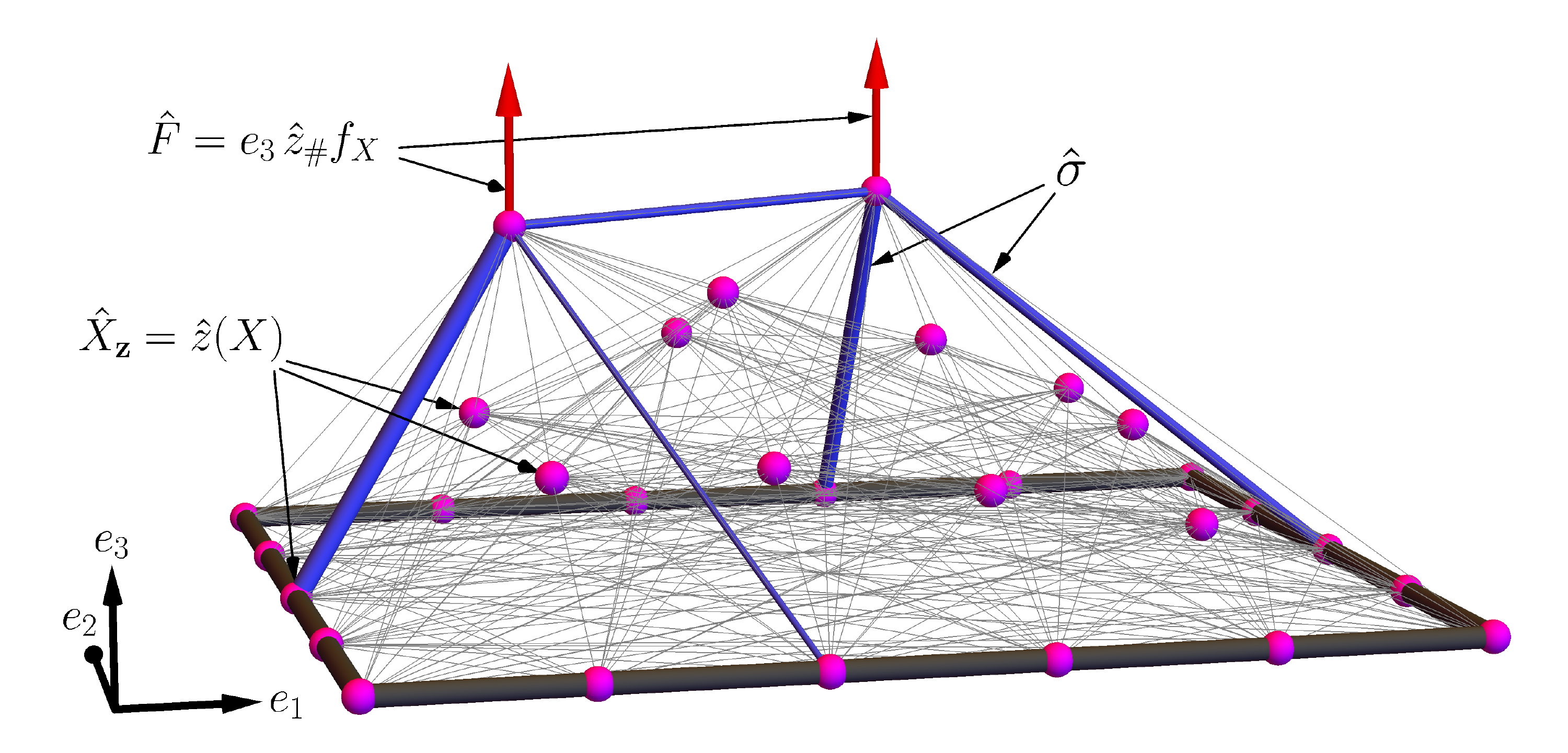}}\\
	\caption{The discrete formulation: (a) a 2D ground structure interconnecting points of the grid $X$ and an optimal 2D truss $\sigma$ according to \eqref{eq:sigma_q_discrete};\quad (b) a 3D ground structure interconnecting points of the grid $\hat{X}_{\mbf{z}}$ and an optimal \textit{grid-shell} being a 3D truss $\hat{\sigma}$ according to \eqref{eq:suboptimal_MV}.}
	\label{fig:ground_structure}
\end{figure}

Henceforward $\bar{n}$ will stand for the number of elements in the set $X$ (further referred to as \textit{nodes}), namely $\bar{n} = \# X$, while $n = \#(X\backslash \bO)$. We choose a bijection $\chi: \{1,\ldots,n\} \to X \backslash \bO$ that will identify the points of $X$ not lying on the boundary with indices $i$. Similarly, with $m= \bar{n}\,(\bar{n}-1)/2$ standing for number of all the segments (unordered pairs of distinct points, further referred to as \textit{members} or \textit{bars}) connecting the points in $X$, two mappings $\chi_-,\chi_+$ are picked so that we arrive at a bijection $[\chi_-(\argu),\chi_+(\argu)]:\{1,\ldots,m\} \to \bigl\{[x,y] \, \big\vert\, (x,y) \in X \times X, \ x \neq y \bigr\}$; the index $k$ will be used to identify members. The universe of all the segments/bars interconnecting the points in the grid $X$ shall be called the \textit{ground structure}.

Vectors and matrices, that will be used to formulate the finite dimensional program, will be displayed in bold font. We define the \textit{nodal load vector} and the \textit{member's length vector}:
\begin{equation*}
	\mbf{f} \in \R^n, \quad f_i := f_X\bigl(\{\chi(i)\}\bigr), \qquad\quad \mbf{l} \in \R^m, \quad l_k :=\abs{\chi_+(k)-\chi_-(k)}.
\end{equation*}
Upon fixing a Cartesian base $e_1, e_2$ in the plane $\R^2$, the \textit{virtual displacement vectors} $\mbf{u}_1, \mbf{u}_2, \mbf{w} \in \R^n$ will be henceforward identified with functions $u \in C(X;\Rd), w \in C(X;\R)$ which vanish on $\bO$ via the one-to-one relations $\bigl(u_{1;i}, u_{2;i}\bigr) = u\bigl( \chi(i)\bigr)$ and $w_{i} = w\bigl( \chi(i)\bigr)$. It should be stressed that the zero values of $u,w$ at points in $X \cap \bO$ are \textit{a priori} eliminated from vectors $\mbf{u}_1, \mbf{u}_2, \mbf{w}$.

The \textit{geometric matrices} $\mbf{B}_1,\mbf{B}_2, \mbf{D} \in \R^{m \times n}$ are defined through linear operations as follows:
\begin{equation}
	\label{eq:B_D_def}
	\bigl( \mbf{B}_1 \mbf{u}_1 + \mbf{B}_2 \mbf{u}_2 \bigr)_k := \Big(u\bigl(\chi_+(k) \bigr) - u\bigl(\chi_-(k) \bigr) \Big)\cdot \frac{\chi_+(k)-\chi_-(k)}{\abs{\chi_+(k)-\chi_-(k)}}, \quad 
	\bigl(\mbf{D} \mbf{w} \bigr)_k := w\bigl(\chi_+(k) \bigr) - w\bigl(\chi_-(k)\bigr),
\end{equation}
i.e. $\mbf{B}_1$ and $\mbf{B}_2$ are sparse matrices with non-zero elements being the directional cosines of the constructed ground structure being a dense truss; matrix $\mbf{D}$ consists of $0, 1$ and $-1$ only. It is clear that with non-empty set $X \cap \bO$ the matrix $\mbf{D}$ has a trivial kernel and the same holds for the operator $\R^{2n}\ni(\mbf{u}_1,\mbf{u}_2) \mapsto  \mbf{B}_1 \mbf{u}_1 + \mbf{B}_2 \mbf{u}_2 \in \R^m$ once $X \cap \bO$ contains at least three points that are not colinear -- both conditions are satisfied when \eqref{eq:assum_on_X} is assumed.

The two-point condition $\Lambda(u,w)(x,y) \leq \abs{y-x}\ $ in \eqref{eq:Z_X_def} may be readily rewritten in the algebraic format:
\begin{equation}
	\label{eq:quadratic_constraint}
	\frac{1}{4} \frac{\bigl( (\mbf{D} \mbf{w})_k \bigr)^2}{l_k} + (\mbf{B}_1 \mbf{u}_1 + \mbf{B}_2 \mbf{u}_2 )_k \leq l_k \qquad \text{for every  } k,
\end{equation}
while the integral $ \int_\Ob w\, f_X $ simply equals $\mbf{f}^\top\! \mbf{w} = \sum_{i=1}^n f_i \,w_i$. We have thus arrived at a finite dimensional convex program with a linear objective functional and a quadratic constraint \eqref{eq:quadratic_constraint}. It is well established, see e.g. \cite{andersen2003} or \cite{ben2001}, that a problem of this class may be rewritten as a \textit{conic quadratic programming problem}, i.e. a problem with linear objective, linear equality constraints and conic quadratic constraints. The advantage of the conic formulation is that it may be tackled by interior point methods of efficiency comparable to those written for linear programs, cf. \cite{andersen2003}.

In posing the conic quadratic program the so called \textit{rotated quadratic cone} in $\R^3$ will be of use:
\begin{equation*}
	\mathrm{K} := \Big\{(t_1,t_2,t_3) \in \R^3  \ \, \big\vert \ \,  t_1, t_2 \geq 0, \ \ 2\, t_1 t_2 \geq (t_3)^2\Big\}.
\end{equation*}
The cone $\mathrm{K}$ is self-dual, i.e. $\mathrm{K}^* = \mathrm{K}$ where for arbitrary convex cone $\mathcal{K}$ in a normed space $Y$ its dual cone is defined as $\mathcal{K}^* := \bigl\{ y^* \in Y^* \, \big\vert \, \pairing{y\,;y^*} \geq 0 \ \ \forall\, y\in \mathcal{K}  \bigr\}$. We put forward a pair of conic quadratic problems that will prove to be mutually dual:
\begin{tcolorbox}
\begin{equation*}\tag*{$(\mathcal{P}_X)$}
\inf\limits_{\mbf{s},\mbf{r}\in \R_+^m, \ \mbf{q} \in \R^m } 
\left\{ \, \mbf{l}^\top\! \mbf{s} + 2\,\mbf{l}^\top\! \mbf{r}   \, \left\vert \,
\begin{array}{c}
\mbf{B}_1^\top \mbf{s} = \mbf{0}\\
\mbf{B}_2^\top \mbf{s} = \mbf{0}\\
\mbf{D}^\top\! \mbf{q} = \mbf{f}
\end{array}, \,
(\mbf{r},\mbf{s},\mbf{q}) \in \mathrm{K}^m	
\right. \right\}
\end{equation*}
\vspace{0.5cm}
\begin{equation*}\tag*{$(\mathcal{P}_X^*)$}
\sup_{\substack{\mbf{u}_1,\mbf{u}_2, \mbf{w} \in \R^n \\
		\mbf{t}_1,\mbf{t}_2 \in \R^m_+, \ \mbf{t}_3\in \R^m}} 
\left\{ \, \mbf{f}^\top\!  \mbf{w} \, \left\vert \, 
\begin{array}{c}
\mbf{t}_2 +  \mbf{B}_1 \mbf{u}_1 + \mbf{B}_2 \mbf{u}_2  = \mbf{l}\\
\mbf{t}_3+\mbf{D}\mbf{w} = \mbf{0}\\
\mbf{t}_1 = 2\,\mbf{l}
\end{array}, \,
(\mbf{t}_1,\mbf{t}_2,\mbf{t}_3) \in \mathrm{K}^m	
\right. \right\}
\end{equation*}
\end{tcolorbox}
\noindent where $\mathrm{K}^m$ is the Cartesian product of $m$ cones $\mathrm{K}$, i.e. by $(\mbf{r},\mbf{s},\mbf{q}) \in \mathrm{K}^m$ we mean that $(r_k,s_k,q_k) \in \mathrm{K}$ for every $k$. 

We shall now recover the interpretation of $(\mathcal{P}_X)$ and $(\mathcal{P}^*_X)$ as discrete variants of convex problems $(\mathcal{P})$ and $(\mathcal{P}^*)$, respectively. First of all, the slack variables $\mbf{t}_1, \mbf{t}_2, \mbf{t}_3$ may be easily eliminated furnishing the quadratic constraint \eqref{eq:quadratic_constraint} therefore problem $(\mathcal{P}^*_X)$ becomes exactly the problem defined in \eqref{eq:Z_X_def} and $\Z_X = \sup \mathcal{P}^*_X$. The variable $\mbf{r}$ in the problem $(\mathcal{P}_X)$ may be eliminated as well: from the conic constraint follows inequality $2\, r_k\, s_k \geq (q_k)^2$ for every $k$ hence, considering that $2 \, l_k \, r_k$ contributes to the objective function and no other constraints on $r_k$ are present, we may always set $r_k = \bar{r}_k := \frac{1}{2}\,(q_k)^2/s_k$ whenever $s_k > 0 $ and $\bar{r}_k :=0$ otherwise when necessarily $q_k =0$ as well. After choosing $\mbf{r} = \bar{\mbf{r}}$ one arrives at the objective function
\begin{equation}
	\label{eq:elimnated_r}
	\mbf{l}^\top\! \mbf{s} + 2\,\mbf{l}^\top\! \mbf{\bar{r}} = \sum_{k=1}^m \left(l_k \, s_k + l_k \frac{(q_k)^2}{s_k} \right) = \sum_{k=1}^m \biggl(l_k \, s_k + l_k\, s_k \, (\psi_k)^2 \biggr),
\end{equation}
where $\psi_k$ is chosen so that $q_k = s_k \psi_k$. We may readily find the problem $(\mathcal{P}_X)$ as a modification of the infinite dimensional program $(\mathcal{P})$ by adding the constraint that the set of possible fields $\sigma, q$ is spanned by the ground structure $X \times X$, namely that $\sigma$ and $q$ are one dimensional measures as follows:
\begin{equation}
	\label{eq:sigma_q_discrete}
	\sigma = \sum_{k=1}^m s_k \, \sigma^{\,\chi_-(k),\,\chi_+(k)}, \qquad q = \sum_{k=1}^m q_k \, q^{\,\chi_-(k),\,\chi_+(k)},
\end{equation}
where $\sigma^{x,y}$ and $q^{x,y}$ are tensor and vector valued one dimensional measures defined in \eqref{eq:bar_measures}. It may be checked that, according to the comment below \eqref{eq:bar_measures}, distributions $-\DIV\,\sigma$ and $-\dive \, q$ are measures such that $\int_\Omega u \cdot(-\DIV\,\sigma) =  (\mbf{B}_1 \mbf{u}_1 + \mbf{B}_2 \mbf{u}_2)^\top \mbf{s} = \mbf{u}_1^\top(\mbf{B}^\top_1\!\mbf{s}) + \mbf{u}_2^\top(\mbf{B}^\top_2\!\mbf{s})$ and $\int_\Omega w\,(-\dive\,q) =  (\mbf{D} \mbf{w})^\top \mbf{q} = \mbf{w}^\top (\mbf{D}^\top\!\mbf{q})$ and the algebraic equilibrium equations in $(\mathcal{P}_X)$ are recast. Provided that the bars of non-zero $s_k, q_k$ do not overlap, one may similarly find that $\int \tr\, \sigma + (\sigma^{-1} q)\cdot q$ is equal exactly to \eqref{eq:elimnated_r}.

For a sparse grid $X$ of $6 \times 5$ nodes which generates a ground structure counting $m = 435$ members Fig. \ref{fig:ground_structure}(a) shows the one dimensional stress field $\sigma$ computed via \eqref{eq:sigma_q_discrete} for $\mbf{s} \in \R^m$ being optimal in $(\mathcal{P}_X)$.

\begin{remark}
	From the consideration above we find that variables $s_k$ satisfying the in-plane equilibrium conditions $\mbf{B}_1^\top \mbf{s} =\mbf{B}_2^\top \mbf{s} = \mbf{0}$ represents a plane truss that is pinned at points in $X \cap \bO$ and pre-stressed with tensile forces $s_k$. With $q_k$ treated as transverse forces in the bars $\mbf{D}^\top \mbf{q} = \mbf{f}$ becomes the out-of-plane equilibrium equation. Ultimately we recognize the structure described by $\mbf{s}, \mbf{q}$ as a discretized \textit{pre-stressed membrane} subject to an out of plane load $\mbf{f}$; segments $[\chi_-(k),\chi_+(k)]$ thus deserve to be called \textit{strings} instead bars. For more details on the framework of pre-stressed membrane the reader is referred to \cite{bouchitte2020}.
\end{remark}

The duality result for the pair $(\mathcal{P}_X)$,\,$(\mathcal{P}_X^*)$ along with the optimality conditions are as follows: 
\begin{theorem}
	\label{thm:duality_discrete}
	Let $X \subset \Ob \subset \Rd$ be a finite grid that satisfies condition \eqref{eq:assum_on_X}. Then  $(\mathcal{P}_X)$,\,$(\mathcal{P}^*_X)$ is a pair of mutually dual conic quadratic problems such that
	\begin{equation*}
		\Z_X = \min \mathcal{P}_X = \max \mathcal{P}^*_X<\infty,
	\end{equation*}
	in particular both problems admit their solutions. Moreover $(\mbf{s},\mbf{q})\in \R^m_+ \times \R^m$ and $(\mbf{u}_1,\mbf{u}_2,\mbf{w}) \in \R^{3 \times n}$ solve problems $(\mathcal{P}_X)$ and, respectively, $(\mathcal{P}^*_X)$ if and only if the optimality conditions below hold true:
	\begin{align}
	\label{eq:opt_cond_discrete}
	\begin{cases}
	(i)& \frac{1}{4} \bigl( (\mbf{D} \mbf{w})_k \bigr)^2/l_k + (\mbf{B}_1 \mbf{u}_1 + \mbf{B}_2 \mbf{u}_2 )_k \leq l_k \qquad \text{for every  } k;\\
	(ii) & \mbf{B}_1^\top \mbf{s} = \mbf{0}, \quad
	\mbf{B}_2^\top \mbf{s} = \mbf{0},\quad
	\mbf{D}^\top\! \mbf{q} = \mbf{f}; \\
	(iii) & \frac{1}{4} \bigl( (\mbf{D} \mbf{w})_k \bigr)^2/l_k + (\mbf{B}_1 \mbf{u}_1 + \mbf{B}_2 \mbf{u}_2 )_k = l_k \qquad \text{for each  } k \text{ such that } s_k \neq 0;\\
	(iv) &   q_k = \frac{1}{2}\, s_k\, (\mbf{D} \mbf{w})_k/l_k \qquad \text{for every  } k.
	\end{cases}
	\end{align}
\end{theorem}
\noindent The proof, relying on the well established algorithms in conic programming (cf. \cite{ben2001}), is moved to \ref{app:cone_duality}.

\subsection{The optimal grid-shell problem and the direct link to the 2D conic quadratic program}
\label{ssec:grid-shells}

\subsubsection{Discussion on constructing a suboptimal 3D structure based on solutions of $(\mathcal{P}_X)$,\,$(\mathcal{P}_X^*)$}

In the last subsection we have proposed a discrete approach to the pair of 2D problems $(\mathcal{P})$,\,$(\mathcal{P}^*)$. Our goal, however, is to develop a numerical method for optimal design of vaults being lower dimensional structures in three dimensional design space $\Ob \times \R$. After Theorem \ref{thm:optimal_3D_structure} it may seem that a way to recover the optimal 3D structure is to push-forward the one-dimensional measure $\sigma$ in \eqref{eq:sigma_q_discrete} to some surface $\ov{\mathcal{S}}_z$ furnishing $\hat{\sigma} \in \Mes(\Ob\times \R;\mathrm{T}_+^3)$. This approach would require specifying an elevation function: $\Omega$ could be triangularized using points in $X$ as nodes and function $z$ could be constructed as piece-wise affine interpolation of values $\frac{1}{2}\,\mbf{w}$ at the nodes. Such construction does not \textit{a priori} guarantee equilibrium in the sense that in general $- \hDIV \, \hat{\sigma} \neq \hat{F}$ for $\hat{F}$ being suitably elevated load $f_X$ (cf. Fig. \ref{fig:ground_structure}(b)): a simple counter-example could be given where a bar, that was straight on the base plane, becomes an unequilibrated funicular in 3D due to jumps of the slope of thus constructed $z$. 

In order to satisfy the equilibrium equations, being fundamental from the mechanical point of view, we put forward a more natural approach where a 3D truss is constructed by interconnecting points in the 3D grid that is obtained by elevating $X \subset \Ob$, see  Fig. \ref{fig:ground_structure}(b). Henceforward a function $z:X \to \R$ with zero boundary conditions $z=0$ on $\bO$ together with the function $\hat{z}: X \ni x \to \bigl(x,z(x) \bigr) \in X \times \R$ will be uniquely identified with a vector $\mbf{z} \in \R^n$ through the relation $z_i = z\bigl(\chi(i) \bigr)$ for each $i \in \{1,\ldots,n\}$, therefore also $\hat{z}\bigl(\chi(i) \bigr) = \bigl(\chi(i),z_i \bigr)$. For a given vector $\mbf{z}$ we define the elevated grid $\hat{X}_\mbf{z} = \hat{z}(X)$ while the bijection $\hat{\chi}_\mbf{z}:\{1,\ldots,n \} \to \hat{X}_\mbf{z} \backslash \bigl(\bO \times \{0\}\bigr)$ reads $\hat{\chi}_{\mbf{z}}(i) := \hat{z}\bigl(\chi(i) \bigr)$. Consequently, for each $k \in \{1, \ldots, m\}$ we put $\hat{\chi}_{\mbf{z},-}(k) := \hat{z}\bigl( \hat{\chi}_{-}(k)\bigr)$, $\hat{\chi}_{\mbf{z},+} := \hat{z}\bigl( \hat{\chi}_{+}(k)\bigr)$, namely $[\hat{\chi}_{\mbf{z},-}(k),\hat{\chi}_{\mbf{z},+}(k)]$ are the segments in the 3D ground structure interconnecting the 3D grid $\hat{X}_\mbf{z}$. Next, by $\hat{\mbf{l}}(\mbf{z}) \in \R^m$ we will denote the vector of lengths of such 3D segments which may be computed as follows:
\begin{equation*}
	\hat{\mbf{l}}(\mbf{z}) = \mbf{J}(\mbf{z})\, \mbf{l}, \quad\ \ \mbf{J}(\mbf{z}) \text{ is a diagonal matrix such that }  J_{kk}(\mbf{z}) := \sqrt{1+(\Delta_k(\mbf{z}))^2}, \quad\ \  \text{where } \Delta_k(\mbf{z}) := \frac{1}{l_k} \bigl( \mbf{D} \mbf{z} \bigr)_k.
\end{equation*}
The number $\Delta_k(\mbf{z})$ plays a role of a difference quotient for the function $z:X \to \R$ while $J_{kk}(\mbf{z})$ may be viewed as \textit{tangential Jacobian} with respect to the horizontal segment $[\chi_-(k),\chi_+(k)]$.  

Readily, for vectors $\mbf{s},\mbf{q},\mbf{r}$ and $\mbf{u}_1, \mbf{u}_2, \mbf{w}$ being solutions of conic programs $(\mathcal{P}_X)$ and, respectively, $(\mathcal{P}_X^*)$ we may construct a pair $(\hat{\sigma},\hat{F})$ that for $f= f_X$ is a competitor for the Prager problem $(\mathrm{MVPS}_H)$ for $H$ large enough:
\begin{equation}
	\label{eq:suboptimal_MV}
	\hat\sigma = \sum_{k=1}^m \hat{s}_k \, \hat\tau_k \otimes \hat\tau_k \, \Ha^1 \mres [\hat{\chi}_{\mbf{z},-}(k),\hat{\chi}_{\mbf{z},+}(k)], \quad \hat{F} = \sum_{i=1}^n e_3 f_i \,\delta_{\hat{\chi}_{\mbf{z}}(i)}, \quad \text{where} \quad \mbf{z} = \frac{1}{2}\,\mbf{w},\quad \hat{\mbf{s}} = \mbf{J}(\mbf{z}) \, \mbf{s},
\end{equation}
where by $\hat{\tau}_k$ we understand a unit vector that is tangent to the 3D segment $[\hat{\chi}_{\mbf{z},-}(k),\hat{\chi}_{\mbf{z},+}(k)]$. The proof that indeed $-\hDIV \, \hat{\sigma} = \hat{F}$ could be performed as a discrete variant of the proof of Lemma \ref{lem:feas_hat_sigma}. The one dimensional 3D stress field $\hat{\sigma}$ computed above is visualized in Fig. \ref{fig:ground_structure}(b). Similarly the pair $(\hat{\mu},\hat{F})$ as below is feasible in $(\mathrm{MCPS}_H)$: 
\begin{equation}
	\label{eq:suboptimal_MC}
	\hat\mu = \sum_{k=1}^m a_k \, \Ha^1 \mres [\hat{\chi}_{\mbf{z},-}(k),\hat{\chi}_{\mbf{z},+}(k)], \quad  \hat{F} = \sum_{i=1}^n e_3 f_i \,\delta_{\hat{\chi}_{\mbf{z}}(i)}, \quad \text{where} \quad \mbf{z} = \frac{1}{2}\,\mbf{w},\quad \mbf{a} = \frac{V_0}{\Z_X}\,\mbf{J}(\mbf{z}) \, \mbf{s}.
\end{equation}

Measures in \eqref{eq:suboptimal_MV} and \eqref{eq:suboptimal_MC} may be considered suboptimal solutions to problems $(\mathrm{MVPS}_H)$ and $(\mathrm{MCPS}_H)$, respectively. The difference with respect to the true optimal solutions given in Theorem \ref{thm:optimal_3D_structure} (aside from higher volume or compliance) is that \textit{a priori} $\hat{\sigma}, \hat{\mu}$ above may not lie on a single surface $\mathcal{S}$. It is simply because the 3D ground structure generated by the nodal grid $\hat{X}_\mbf{z}$ is truly spatial, despite the fact that $\hat{X}_\mbf{z} \subset \ov{\mathcal{S}}_{\bar{z}}$ for any Lipschitz continuous interpolation $\bar{z}:\Ob \to \R$ of $z:X \to \R$. In view of the present paper it may be considered a flaw of the proposed discrete approach -- our original goal was to optimally design a vault therefore, when looking for its discrete approximation, we should aim for 3D trusses that lie on a single Lipschitz continuous surface $\mathcal{S}$ spread over $\Omega$. Such trusses may be considered a 3D generalization of plane funiculars that will be henceforward referred to as \textit{grid-shells}.

In order to justify the proposed discrete method we shall now convey an intuition behind the fact that for optimal $\mbf{s}$ and $\mbf{w}$ the 3D trusses given via \eqref{eq:suboptimal_MV} and \eqref{eq:suboptimal_MC} are "almost grid-shells":
\begin{proposition}
	\label{prop:arg_for_GS}
	Let $x_1, x_2, x_3 \in X$ be colinear points in $\Ob$ such that $x_2 \in[x_1,x_3]$. Assume that vectors $\mbf{s},\mbf{q},\mbf{r}$ and $\mbf{u}_1, \mbf{u}_2,\mbf{w}$ are solutions of $(\mathcal{P}_X)$ and $(\mathcal{P}_X^*)$, respectively, and set $\mbf{z} = \frac{1}{2} \,\mbf{w}$. Then, for $k=3$ being the bar index such that $[\chi_-(3),\chi_+(3)] = [x_1,x_3]$, the following implication holds:
	\begin{equation*}
		s_{3} \neq 0 \qquad \Rightarrow \qquad \text{the points } \hat{z}(x_1) , \hat{z}(x_2), \hat{z}(x_3) \in \hat{X}_\mbf{z} \text{ are colinear in } \Ob \times \R,
	\end{equation*}
	where we recall the one-to-one relation $\hat{z}\bigl(\chi(i)\bigr) = \bigl(\chi(i),z_i\bigr)\ $ for each $i$.
\end{proposition}
\begin{proof}
	We assume that $x_i = \chi(i)$ and $[\chi_-(1),\chi_+(1)] = [x_1,x_2]$, $[\chi_-(2),\chi_+(2)] = [x_2,x_3]$. For each of the three members there holds inequality $(i)$ in \eqref{eq:opt_cond_discrete} while, assuming that $s_3 \neq 0$, equality $(iii)$ is true for $k = 3$. Owing to the fact that $l_3 = l_1 + l_2$ and that $\mbf{u}_1, \mbf{u}_2$ enter those inequalities/equalities linearly (and with the same directional cosines) the following may be inferred:
	\begin{equation*}
		\frac{1}{4}\,\frac{(w_3-w_1)^2}{l_3} \geq \frac{1}{4} \, \frac{(w_2 - w_1)^2}{l_1} + \frac{1}{4} \, \frac{(w_3 - w_2)^2}{l_2}.
	\end{equation*}
	On the other hand, by virtue of Schwartz inequality:
	\begin{equation*}
		(w_3  - w_1)^2 = \left(\frac{w_2-w_1}{\sqrt{l_1}}\, \sqrt{l_1} + \frac{w_3-w_2}{\sqrt{l_2}}\, \sqrt{l_2} \right)^2 \leq \left(\frac{(w_2 - w_1)^2}{l_1} + \frac{(w_3 - w_2)^2}{l_2} \right) (l_1 + l_2)
	\end{equation*}
	being precisely the opposite inequality which renders each of the above an equality. In this particular case the Schwartz inequality is an equality if and only if $(w_2 - w_1)/l_1 = (w_3-w_2)/l_2$ which furnishes the assertion.
\end{proof}

The first implication of the proposition is that for colinear $x_1, x_2, x_3$ it is impossible that the 3D truss in accordance with \eqref{eq:suboptimal_MV} or \eqref{eq:suboptimal_MC} contains a vertical triangle formed by bars $[\hat{z}(x_1),\hat{z}(x_2)]$, $[\hat{z}(x_2),\hat{z}(x_3)]$, $[\hat{z}(x_1),\hat{z}(x_3)]$: either the bar $[\hat{z}(x_1),\hat{z}(x_3)]$ has a zero cross-section area or the triangle degenerates to a single bar. Next, we consider a quadrilateral of diagonals $[x_1,x_2]$, $[x_3,x_4]$ that intersect at point $x_5$ and we assume that $x_1, \ldots, x_5 \in X$. Then, by applying Proposition \ref{prop:arg_for_GS} to triples $x_1,x_5,x_2$ and $x_3,x_5,x_4$ we find that two non-zero 3D bars $[\hat{z}(x_1),\hat{z}(x_2)]$ and $[\hat{z}(x_3),\hat{z}(x_4)]$ can occur as a part of the truss \eqref{eq:suboptimal_MV} only under the condition that they intersect  at $\hat{z}(x_5)$. Of course, the two scenarios considered are idealized, since e.g. for arbitrary quadruple $x_1,\dots,x_4 \in X$ the intersection $x_5$ usually is not a point in $X$. If, however, $X$ is very dense and the four points are far away from each other, there is a point in $X$ that is very close to the intersection in question. Thus one may hope that for such grids $X$ the truss \eqref{eq:suboptimal_MV} is "approximately a grid-shell". Clearly this kind of argumentation is vague and in order to make it rigorous one would have to investigate convergence of the discrete method (e.g. for sequence of grids $X_\eps$ being $\eps$-nets for $\Ob$) which we skip in this contribution. Nevertheless, our hopes will be met by the precise simulations carried out in Section \ref{sec:numerics} where the suboptimal 3D trusses shall numerically lie on the surface $\mathcal{S}$ interpolating points in $\hat{X}_\mbf{z}$.

Formulas for suboptimal 3D trusses \eqref{eq:suboptimal_MV} and \eqref{eq:suboptimal_MC} were proposed as a sort of "discrete push-forward" of solution $\mbf{s}$ of problem $(\mathcal{P}_X)$ to the 3D ground structure spanned by the grid $\hat{X}_\mbf{z}$ where $\mbf{z} = \frac{1}{2}\, \mbf{w}$ with $\mbf{w}$ being solution of $(\mathcal{P}^*_X)$. At this point both the formulas are just an intuitive guess inspired by the general Theorem \ref{thm:optimal_3D_structure}, namely the 3D trusses \eqref{eq:suboptimal_MV} and \eqref{eq:suboptimal_MC} are not yet proved to be solutions of any optimal design problem. In the forthcoming subsections we will show that they are indeed optimal in the class of 3D trusses extracted from the ground structure generated by $\hat{X}_\mbf{z}$, with $\mbf{z} \in \R^n$ being a design variable. After the deliberation had above the present author dares to abuse the language and call this problem an \textit{optimal grid-shell design}.

\subsubsection{Grid-shells of minimum volume -- the plastic design setting}

For a fixed finite grid $X \subset \Ob$ that satisfies \eqref{eq:assum_on_X} we consider the problem of statics of a truss -- the 3D ground structure -- that for an elevation vector $\mbf{z}$ interconnects the nodes in the 3D grid $\hat{X}_\mbf{z}$, see Fig. \ref{fig:ground_structure}(b). To each index $k \in \{1,\ldots,m\}$ we associate a bar identified with 3D segment $[\hat{\chi}_{\mbf{z},-}(k),\hat{\chi}_{\mbf{z},+}(k)] \subset \Ob \times \R$. The bars can carry tensile axial forces $\hat{s}_k \geq 0$ only. For each index $i \in \{1,\ldots,n \}$ to a node $\hat{\chi}_\mbf{z}(i) \in \hat{X}_\mbf{z} \backslash \bigl(\bO \times \{0\} \bigr)$ we apply a vertical force of magnitude $f_i$ (upwards if $f_i >0$). The are no vertical members present in the ground structure and $k$-th bar has a slope $\Delta_k(\mbf{z}) = (\mbf{D} \mbf{z})_k / l_k$ with respect to the base plane and is of length $\hat{l}_k(\mbf{z}) = J_{kk}(\mbf{z})\, l_k =  \sqrt{1+(\Delta_k(\mbf{z}))^2} \, l_k$. Up to the sign, the horizontal and vertical components of forces exerted on nodes by $k$-th bar read $s_k = 1/J_{kk}(\mbf{z}) \, \hat{s}_k$ and  $q_k = \Delta_k(\mbf{z})/J_{kk}(\mbf{z}) \, \hat{s}_k$, respectively. Readily, equations $\mbf{B}_1^\top \mbf{s} = \mbf{0}, \ \mbf{B}_2^\top \mbf{s} = \mbf{0},\ \mbf{D}^\top\! \mbf{q} = \mbf{f}$ appearing in the 2D problem $(\mathcal{P}_X)$ may be recognized as 3D equilibrium equations of the nodes in the set $\hat{X}_\mbf{z} \backslash \bigl(\bO \times \{0\} \bigr)$. Up to dividing by the constant yield stress the problem of \textit{Minimum Volume Grid-Shell} can be posed as follows: 
\begin{equation*}\tag*{$(\mathrm{MVGS}_X)$}
\V_{X,\mathrm{min}} = \inf_{\substack{\mbf{z} \in \R^n \\ \hat{\mbf{s}} \in \R_+^m}} \left\{ \sum_{k=1}^{m} \hat{l}_k(\mbf{z}) \, \hat{s}_k \ \left\vert \  \begin{array}{c}
\mbf{B}_1^\top \mbf{s} = \mbf{0}\\
\mbf{B}_2^\top \mbf{s} = \mbf{0}\\
\mbf{D}^\top\! \mbf{q} = \mbf{f}
\end{array}, \ \def\arraystretch{1.2}\begin{array}{c}
s_k = \frac{1}{J_{kk}(\mbf{z})}\, \hat{s}_k\\
q_k = \frac{\Delta_k(\mbf{z})}{J_{kk}(\mbf{z})}\, \hat{s}_k\\
\end{array}
\right.\right\}.
\end{equation*}
\begin{remark}
	We stress that, in contrast to the classical ground structure methods, the shape of the 3D ground structure from which we extract the grid-shell $\hat{\mbf{s}}$ solving $(\mathrm{MVGS}_X)$ depends on the design variable $\mbf{z}$, cf. Fig. \ref{fig:ground_structure}(b).
\end{remark}
\noindent As in the standard minimum volume formulation for trusses (cf. \cite{gilbert2003,sokol2015,zegard2014}) in $(\mathrm{MVGS}_X)$ we consider the plastic design case where, apart from the elevation vector $\mbf{z}$, we seek the axial force vector $\hat{\mbf{s}}$ that guarantees equilibrium, without minding the deformation. Similarly as in the case of $(\mathrm{MVV}_\Omega)$, the problem $(\mathrm{MVGS}_X)$ is non-convex with respect to the pair $(\mbf{z},\hat{\mbf{s}})$. Analogously, a link to the pair of convex 2D problems $(\mathcal{P}_X)$,\,$(\mathcal{P}_X^*)$ will be now established.

The natural next step is to change variables from $\hat{\mbf{s}}$ to $\mbf{s}$ via the simple relation $\hat{\mbf{s}} = \mbf{J}(\mbf{z}) \, \mbf{s}$ while recalling that $\hat{l}_k(\mbf{z}) = J_{kk}(\mbf{z})\, l_k $. We arrive at an equivalent formulation:
\begin{equation}
	\label{eq:MVGS_changed_var}
	\V_{X,\mathrm{min}}=\inf_{\substack{\mbf{z} \in \R^n \\ \mbf{s} \in \R_+^m}} \left\{ \sum_{k=1}^{m} \biggl( l_k\, s_k+ l_k\, s_k\, \bigl(\Delta_k(\mbf{z})\bigr)^2  \biggr) \, \left\vert   \begin{array}{c}
	\mbf{B}_1^\top \mbf{s} = \mbf{0}\\
	\mbf{B}_2^\top \mbf{s} = \mbf{0}\\
	\mbf{D}^\top\! \mbf{q} = \mbf{f}
	\end{array}, \ q_k = \Delta_k(\mbf{z})\, s_k
	\right.\right\}.
\end{equation}
Then we follow the idea employed in Section \ref{ssec:2D_convex_program} for minimum weight vaults. We propose a relaxation of the problem \eqref{eq:MVGS_changed_var} where instead of minimizing with respect to $\mbf{z} \in \R^n$ we switch to variable $\bm{\psi} \in \R^m$, more precisely we put $\psi_k$ in place of $\Delta_k(\mbf{z})$. Since $\IM \,\mbf{D}$ is a subset of $\R^m$ (a strict subset in general) the objective value in the relaxed problem will be not greater than $\V_{X,\mathrm{min}}$. Upon a change of variables, from $(\mbf{s}, \bm{\psi})$ to $(\mbf{s}, \mbf{q})$ where $q_k = s_k\,\psi_k$, we find that the objective function of the relaxed problem is equal exactly to \eqref{eq:elimnated_r} and ultimately the objective value in the relaxed problem is precisely $\min \mathcal{P}_X = \Z_X$, therefore
\begin{equation}
	\label{eq:ZX_leq_VminX}
	\Z_X \leq \V_{X,\mathrm{min}},
\end{equation}
which is the discrete counterpart of inequality \eqref{eq:Z_leq_Vmin} for vaults. Similarly as in Section \ref{ssec:recovering_domes} we employ the optimality conditions, this time in its discrete variant \eqref{eq:opt_cond_discrete}, to reconstruct a grid-shell of minimum volume based on the solutions of the conic quadratic program $(\mathcal{P}_X)$,\,$(\mathcal{P}_X^*)$: 
\begin{theorem}[\textbf{Construction of the least-volume grid-shell}] 
	\label{thm:recovering_MV_grid-shell}
	Assume that the triples $(\mbf{s},\mbf{q},\mbf{r}) \in \R^{3 \times m}$ and $(\mbf{u}_1,\mbf{u}_2,\mbf{w}) \in \R^{3 \times n}$ are solutions of problems $(\mathcal{P}_X)$ and $(\mathcal{P}_X^*)$ respectively. Then the pair
	\begin{equation*}
	\mbf{z} = \frac{1}{2} \, \mbf{w}, \qquad \hat{\mbf{s}} = \mbf{J}(\mbf{z}) \,\mbf{s}
	\end{equation*} 
	solves the problem $(\mathrm{MVGS}_X)$ with $\V_{X,\mathrm{min}} = \Z_X$.
\end{theorem}
\begin{proof} It suffices to show that the pair $(\mbf{z},\mbf{s})$ solves the modified problem \eqref{eq:MVGS_changed_var}. Upon putting $\psi_k = \Delta_k(\mbf{z})$ we infer from the optimality condition (iv), Eq. \eqref{eq:opt_cond_discrete} that for each $k$ there holds $q_k = \frac{1}{2}\,s_k\,\Delta_k(\mbf{w}) = s_k\,\Delta_k(\mbf{z}) = s_k \, \psi_k$. By combining the following facts: inequality \eqref{eq:ZX_leq_VminX}, feasibility of $(\mbf{z},\mbf{s})$ in \eqref{eq:MVGS_changed_var} and optimality of $\mbf{s}, \mbf{q}$ in $(\mathcal{P}_X)$ (together with $\mbf{r} = \bar{\mbf{r}}$, see \eqref{eq:elimnated_r}) we may write down the following chain:
\begin{equation}
	\label{eq:Z_X_leq_V_X_leq_Z_X}
	\Z_X \leq \V_{X,\mathrm{min}} \leq \sum_{k=1}^{m} \Big( l_k\, s_k+ l_k\, s_k\, \bigl(\Delta_k(\mbf{z})\bigr)^2 \Big) = \sum_{k=1}^m \Big(l_k \, s_k + l_k\, s_k \, (\psi_k)^2 \Big) = \Z_X
\end{equation}
which is in fact a chain of equalities rendering $(\mbf{z},\mbf{s})$ optimal for \eqref{eq:MVGS_changed_var}.
\end{proof}

\subsubsection{Grid-shells of minimum compliance -- the elastic design setting}

We jump to investigate optimal design of elastic grid-shells where, apart from elevation vector $\mbf{z} \in \R^n$, a vector of bar's cross-section areas $\mbf{a} \in \R^m$ will be chosen. The bars in the 3D ground structure spanned by the 3D grid $\hat{X}_\mbf{z}$ will be assumed to be made of homogeneous elastic material of prescribed Young's modulus $E_0$. Vectors $\mbf{u}_1, \mbf{u_2}, \mbf{w} \in \R^n$ shall describe the nodal displacements, namely for each $i$ at the node $\hat{\chi}_\mbf{z}(i) \in \hat{X}_\mbf{z} \backslash \bigl(\Ob \times \{0\} \bigr)$ we have horizontal displacements $u_{1;i}, u_{i;2}$ in two orthogonal directions and vertical displacement $w_i$.

With the use of finite difference operators defined with respect to the base plane:
\begin{equation*}
\mathrm{e}_k(\mbf{u}_1,\mbf{u}_2) := \frac{1}{l_k} \bigl( \mbf{B}_1 \mbf{u}_1 + \mbf{B}_2 \mbf{u}_2 \bigr)_k, \qquad \Delta_k(\mbf{w}) := \frac{1}{l_k} \bigl( \mbf{D} \mbf{w} \bigr)_k
\end{equation*}
through simple geometric relations we may compute the axial strain in the $k$-th bar of the 3D ground structure:
\begin{equation}
\label{eq:operator_hat_e_k}
\hat{\mathrm{e}}_k(\mbf{u}_1,\mbf{u}_2,\mbf{w};\,\mbf{z}) = \frac{\mathrm{e}_k\bigl(\mbf{u}_1,\mbf{u}_2 \bigr) + \Delta_k(\mbf{z}) \, \Delta_k(\mbf{w})}{1 + \bigl(\Delta_k(\mbf{z})\bigr)^2}. 
\end{equation}
The elastic energy stored in a single member with the strain $\epsilon_k$ may be written as $\frac{E_0}{2}\, \bigl(\check{\gamma}_+(\epsilon_k) \bigr)^2 a_k \, \hat{l}_k(\mbf{z})$ where $\check{\gamma}_+:\R \to \R_+$ is the 1D counterpart of the gauge $\hat{\gamma}_+:\mathrm{T}^3 \to \R_+$, i.e. $\check{\gamma}_+(\epsilon)$ is simply the positive part $\epsilon_+$ of $\epsilon$. The polar reads $\check{\gamma}_+^0(s) = \abs{s} + \mathbbm{I}_{\R_+}(s)$ which automatically rules out compressive axial forces $s<0$ in the dual formulation of elasticity. Compliance of the grid-shell being a 3D truss of cross sections $\mbf{a} \in \R^m$ and interconnecting nodes elevated by vector $\mbf{z} \in \R^n$ may be computed as minus infimum of the total potential energy or, equivalently, supremum of minus total potential energy:
\begin{equation}
	\label{eq:disc_comp_def}
	\Comp_X(\mbf{z},\mbf{a}) := \sup_{\mbf{u}_1,\mbf{u}_2, \mbf{w} \in \R^n} \left\{  \mbf{f}^\top\! \mbf{w}   -  \frac{E_0}{2} \sum_{k=1}^{m} \Big(\bigl(\hat{\mathrm{e}}_{k}(\mbf{u}_1,\mbf{u}_2,\mbf{w};\,\mbf{z})\bigr)_+\Big)^2 a_k\, \hat{l}_k(\mbf{z})   \right\}.
\end{equation}
By standard duality arguments we arrive at the dual, stress-based formula for the compliance:
\begin{equation}
	\label{eq:disc_dual_comp_def}
	\Comp_X(\mbf{z},\mbf{a}) = \inf_{\hat{\mbf{s}} \in \R_+^m} \left\{ \frac{1}{2E_0} \sum_{k=1}^{m} \frac{(\hat{s}_k)^2}{a_k}\, \hat{l}_k(\mbf{z}) \ \left\vert \  \begin{array}{c}
	\mbf{B}_1^\top \mbf{s} = \mbf{0}\\
	\mbf{B}_2^\top \mbf{s} = \mbf{0}\\
	\mbf{D}^\top\! \mbf{q} = \mbf{f}
	\end{array}, \ \def\arraystretch{1.2}\begin{array}{c}
	s_k = \frac{1}{J_{kk}(\mbf{z})}\, \hat{s}_k\\
	q_k = \frac{\Delta_k(\mbf{z})}{J_{kk}(\mbf{z})}\, \hat{s}_k\\
	\end{array}  \right.\right\}
\end{equation}
where we agree that for $a_k=0$ the quotient $(\hat{s}_k)^2/a_k$ equals zero if $\hat{s}_k =0$ and equals $\infty$ whenever $s_k \neq 0$.

The problem of designing the \textit{minimum compliance grid-shell} under the volume constraint reads:
\begin{equation*}\tag*{$(\mathrm{MCGS}_X)$}
	\Comp_{X,\mathrm{min}}=\inf_{\substack{\mbf{z} \in \R^n \\ \mbf{a} \in \R_+^m}} \left\{ \mathcal{C}_X(\mbf{z},\mbf{a})\ \, \left\vert \ \ \sum_{k=1}^{m}  a_k\, \hat{l}_k(\mbf{z}) \leq V_0 \right.\right\}.
\end{equation*}
The connection between problem $(\mathrm{MCGS}_X)$ and the conic program $(\mathcal{P}_X)$,\,$(\mathcal{P}_X^*)$  leads through an inequality that is the discrete counterpart of the one claimed in Lemma \ref{lem:Cmin_leq_Z}: 
\begin{lemma}
	\label{lem:C_X_geq_Z_X}
	There holds an inequality
	\begin{equation*}
	\Comp_{X,\mathrm{min}} \geq \frac{(\Z_X)^2}{2E_0 V_0}.
	\end{equation*}
\end{lemma}
\noindent The proof of this inequality, and in fact the rest of this subsection, could go in full analogy to Section \ref{ssec:recovering_dome_elastic}. It would lead through an auxiliary function $h_1: \R \times \R \to \R_+$ defined by $h_1(\epsilon,\theta) := \sup_{\psi\in \R}\, \bigl(\epsilon+ \psi \,\theta\bigr)/\bigl(1+\psi^2\bigr)$ being a one 1D counterpart of function $h$ in \eqref{eq:h_with_psi}.
However, in the discrete, grid-shell setting of the design problem the function $h_1$ would not play a role equally important to the one played by function $h$ in Sections \ref{sec:elastic_design} or \ref{sec:3D}. For this reason, and also to introduce some diversity in the text, we choose to prove Lemma \eqref{lem:C_X_geq_Z_X} by exploiting the dual definition of compliance \eqref{eq:disc_dual_comp_def} thus building upon methods used in works \cite{czarnecki2015,czarnecki2012,czarnecki2017,czubacki2015} or \cite{bendsoe1996}:

\begin{proof}[Proof of Lemma \ref{lem:C_X_geq_Z_X}]
	For any triple of vectors $\mbf{z}\in \R^n$, $\mbf{a}\in \R_+^m$, $\hat{\mbf{s}}\in \R_+^m$ such that the volume constraint $\sum_{k=1}^m a_k\, \hat{l}_k(\mbf{z}) \leq V_0$ is satisfied a chain of inequalities may be written down 
	\begin{equation*}
		 \left(\sum_{k=1}^{m} \hat{s}_k \, \hat{l}_k(\mbf{z}) \right)^2 = \left(\sum_{k=1}^{m} \frac{\hat{s}_k}{\sqrt{a_k}} \, \sqrt{a_k}\ \hat{l}_k(\mbf{z}) \right)^2 \leq \left(\sum_{k=1}^{m} \frac{(\hat{s}_k)^2}{a_k} \, \hat{l}_k(\mbf{z}) \right) \left(\sum_{k=1}^{m} a_k \, \hat{l}_k(\mbf{z}) \right) \leq \left(\sum_{k=1}^{m} \frac{(\hat{s}_k)^2}{a_k} \, \hat{l}_k(\mbf{z}) \right) V_0,
	\end{equation*}
	where the first inequality is the Schwarz inequality with respect to the following scalar product on $\R^m$: $\pairing{\mbf{x}\,;\mbf{y}} = \sum_{k=1}^m x_k \, y_k \, \hat{l}_k(\mbf{z})$ (cf. the derivation below Eq. (2.27) in \cite{czarnecki2017} where Schwarz ineq. was already utilized in a similar context). In the next step we multiply the chain above by $1/(2 E_0 V_0)$ and then we take the infimum of both its LHS and RHS with respect to the triple $(\mbf{z},\mbf{a},\hat{\mbf{s}})$ that satisfies the constraint given in $(\mathrm{MVGS}_X)$ and moreover $\sum_{k=1}^m a_k\, \hat{l}_k(\mbf{z}) \leq V_0$ in order to find that the following inequality holds:
	\begin{equation*}
		\frac{1}{2 E_0 V_0} \, \bigl( \V_{X,\mathrm{min}}\bigr)^2 \leq \Comp_{X,\mathrm{min}}.
	\end{equation*}
	Indeed, the LHS above emerged directly from definition of problem $(\mathrm{MVGS}_X)$ while $\Comp_{X,\mathrm{min}}$ may be found by plugging the dual definition \eqref{eq:disc_dual_comp_def} of compliance $\Comp_X(\mbf{z},\mbf{a})$ into $(\mathrm{MCGS}_X)$ to discover the triple infimum in $(\mbf{z},\mbf{a},\hat{\mbf{s}})$. The assertion follows by inequality \eqref{eq:ZX_leq_VminX} (or in fact equality $\V_{X,\mathrm{min}} = \Z_X$ by virtue of Theorem \ref{thm:recovering_MV_grid-shell}).
\end{proof}

The recipe for a grid-shell of minimum compliance reads as follows:

\begin{theorem}[\textbf{Constructing elastic grid-shell of the least compliance}]
\label{thm:recovering_MC_grid-shell}
Let $(\mbf{s},\mbf{q},\mbf{r}) \in \R^{3 \times m}$ and $(\mbf{u}_1,\mbf{u}_2,\mbf{w}) \in \R^{3 \times n}$ be solutions of problems $(\mathcal{P}_X)$ and $(\mathcal{P}_X^*)$ respectively. Then the pair
	\begin{equation*}
		{\mbf{z}} = \frac{1}{2} \, \mbf{w}, \qquad {\mbf{a}} = \frac{V_0}{\Z_X}\, \mbf{J}({\mbf{z}}) \,\mbf{s}
	\end{equation*} 
	solves the problem $(\mathrm{MCGS}_X)$ with $\Comp_{X,\mathrm{min}} = \frac{(\Z_X)^2}{2E_0 V_0}$. Moreover, the displacement and axial force vectors
	\begin{equation*}
	({\mbf{u}}_{1,\e},{\mbf{u}}_{2,\e},{\mbf{w}_\e}) = \frac{\Z_X}{E_0 V_0}\, ({\mbf{u}}_1,{\mbf{u}}_2,{\mbf{w}}) \qquad \text{and} \qquad \hat{\mbf{s}} = \mbf{J}({\mbf{z}}) \,\mbf{s}
	\end{equation*}
	solve, respectively, the displacement-based elasticity problem \eqref{eq:disc_comp_def} and the stress-based elasticity problem \eqref{eq:disc_dual_comp_def} for the optimal grid-shell $({\mbf{z}},{\mbf{a}})$. Each bar of non-zero cross-sectional area $a_k \neq0$ undergoes constant positive strain $\hat{\mathrm{e}}_{k}({\mbf{u}}_{1,\e},{\mbf{u}}_{2,\e},{\mbf{w}_\e};{\mbf{z}}) = \Z_X / (E_0 V_0)$. As a result the uni-axial Hooke's law holds for each bar $k$:
	\begin{equation*}
	\hat{s}_k =\bigl( E_0 \,a_k \bigr)\, \hat{\mathrm{e}}_{k}({\mbf{u}}_{1,\e},{\mbf{u}}_{2,\e},{\mbf{w}_\e};{\mbf{z}}).
	\end{equation*}
\end{theorem}
\begin{proof}
	First we compute the volume of the designed grid-shell:
	\begin{equation*}
	\sum_{k=1}^{m}  {a}_k\, \hat{l}_k(\mbf{{z}}) = \sum_{k=1}^{m} \biggl( \frac{V_0}{\Z_X} \sqrt{1+(\Delta_k({\mbf{z}}))^2}\,s_k\biggr) \biggl( \sqrt{1+(\Delta_k({\mbf{z}}))^2}\,l_k\biggr) =  \frac{V_0}{\Z_X} \sum_{k=1}^{m} \Big( l_k\, s_k+ l_k\, s_k\, \bigl(\Delta_k(\mbf{z})\bigr)^2 \Big) = V_0
	\end{equation*}
	where \eqref{eq:Z_X_leq_V_X_leq_Z_X} was acknowledged in the last equality. The pair $(\mbf{z},\mbf{a})$ is thus feasible in $(\mathrm{MCGS}_X)$ while, by virtue of optimality condition (iv) in  \eqref{eq:opt_cond_discrete} which gives $q_k = \frac{1}{2} \,s_k \Delta_k(\mbf{w}) = \Delta_k(\mbf{z})/J_{kk}(\mbf{z}) \, \hat{s}_k $, the axial force vector $\hat{\mbf{s}}$ satisfies the constraints in \eqref{eq:disc_dual_comp_def} and therefore the following chain of inequalities may be written down:
	\begin{equation*}
	\Comp_{X,\mathrm{min}} \leq \Comp_X({\mbf{z}},{\mbf{a}})\leq \frac{1}{2E_0} \sum_{k=1}^{m} \left(\frac{{\hat{s}}_k}{{a}_k} \right)^2 \! {a}_k\, \hat{l}_k({\mbf{z}}) = \frac{1}{2E_0} \sum_{k=1}^{m} \left(\frac{\Z_X}{V_0}\right)^2 \! {a}_k\, \hat{l}_k({\mbf{z}}) = \frac{(\Z_X)^2}{2E_0 V_0} \leq \Comp_{X,\mathrm{min}},
	\end{equation*}
	which ultimately is a chain of equalities and two implications follow: the grid-shell $(\mbf{z},\mbf{a})$ is optimal in $(\mathrm{MCGS}_X)$ and $\hat{\mbf{s}}$ solves the stress based elasticity problem \eqref{eq:disc_dual_comp_def}.
	
	Next, we shall compute the virtual axial strain in bars for which $a_k \neq 0$ or, equivalently, $s_k \neq 0$. By using the equality $\mbf{z} = \frac{1}{2}\, \mbf{w}$ and by a suitable grouping of the terms we obtain:
	\begin{equation}
	\label{eq:unit_strain}
	\hat{\mathrm{e}}_k(\mbf{u}_1,\mbf{u}_2,\mbf{w};{\mbf{z}}) = \frac{\mathrm{e}_k(\mbf{u}_1,\mbf{u}_2 ) + \Delta_k({\mbf{z}}) \, \Delta_k(\mbf{w})}{1 + \bigl(\Delta_k({\mbf{z}})\bigr)^2}
	= \frac{\left(\frac{1}{4}\bigl(\Delta_k(\mbf{w})\bigr)^2+\mathrm{e}_k(\mbf{u}_1,\mbf{u}_2 )\right)+\frac{1}{4}\bigl(\Delta_k(\mbf{w})\bigr)^2}{1 + \frac{1}{4}\bigl(\Delta_k(\mbf{w})\bigr)^2} = 1,
	\end{equation}
	where the last equality comes from optimality condition $(iii)$ in  \eqref{eq:opt_cond_discrete} which for $s_k\neq 0$ gives $\frac{1}{4}\bigl(\Delta_k(\mbf{w})\bigr)^2+\mathrm{e}_k(\mbf{u}_1,\mbf{u}_2 ) = \left(\frac{1}{4} \bigl( (\mbf{D} \mbf{w})_k \bigr)^2/l_k + (\mbf{B}_1 \mbf{u}_1 + \mbf{B}_2 \mbf{u}_2 )_k  \right)\!/l_k = 1$. Automatically, a constant elastic axial strain $\hat{\mathrm{e}}_{k}({\mbf{u}}_{1,\e},{\mbf{u}}_{2,\e},{\mbf{w}_\e};{\mbf{z}}) = \Z_X / (E_0 V_0)$ in bars with $a_k \neq 0$ follows. As a result the Hooke's law is established while, owing to definition of compliance \eqref{eq:disc_comp_def} and optimality of $(\mbf{z},\mbf{a})$, we may write down a chain
	\begin{align*}
	\Comp_{X,\mathrm{min}} =&\ \Comp_X(\mbf{z},\mbf{a}) \geq   \mbf{f}^\top\! {\mbf{w}_\e} - \frac{E_0}{2} \sum_{k=1}^{m} \Big( \bigl(\hat{\mathrm{e}}_{k}({\mbf{u}}_{1,\e},{\mbf{u}}_{2,\e},{\mbf{w}_\e};{\mbf{z}})\bigr)_+ \Big)^2\, {a}_k \, \hat{l}_k({\mbf{z}}) \\
	=&\  \frac{\Z_X}{E_0 V_0}\,\mbf{f}^\top\! \mbf{w} - \frac{E_0}{2} \sum_{k=1}^{m} \left(\frac{\Z_X}{E_0 V_0} \right)^2\! {a}_k \, \hat{l}_k({\mbf{z}})=  \frac{\Z_X}{E_0 V_0}\, \Z_X - \frac{(\Z_X)^2}{2E_0 (V_0)^2} \sum_{k=1}^{m} {a}_k \, \hat{l}_k({\mbf{z}}) =  \frac{(\Z_X)^2}{2E_0 V_0} = 	\Comp_{X,\mathrm{min}},
	\end{align*}
	where equality $\mbf{f}^\top\! \mbf{w} = \Z_X$ holds true by optimality of $(\mbf{u}_1,\mbf{u}_2,\mbf{w})$ for $\mathcal{P}_X^*$. The above is once again a chain of equalities rendering $({\mbf{u}}_{1,\e},{\mbf{u}}_{2,\e},{\mbf{w}_\e})$ a solution of the maximization problem \eqref{eq:disc_comp_def}.
\end{proof}

Similarly as in Section \ref{ssec:recovering_dome_elastic} we stress the astonishing relation:
\begin{corollary}
	In the optimal elastic grid-shell the elevation vector $\mbf{z} \in \R^n$ and the vertical displacement vector $\mbf{w}_\e \in \R^n$ satisfy the relation
	\begin{equation*}
		\mbf{z} = \frac{E_0 V_0}{2 \Z_X}\, \mbf{w}_\e.
	\end{equation*}
\end{corollary}

\subsection{From optimal vaults and grid-shells in pure tension to optimal designs in pure compression}
\label{ssec:compression}

In each design problem posed in this work we are looking for optimal structure that may carry tensile stresses only: either it was positive stress matrices $\eta \in \Sddp$, $\hat{\sigma} \in \mathrm{T}^3_+$ or positive axial member forces $\hat{s}_k \geq 0$. This setting is by no doubt the most natural in terms of mathematical formulation when comparing to the pure compression assumption. In practise, on the other hand, vaults are more often designed as structures that carry gravitational, downward load, i.e. $f \leq 0$, and which are spread over (not beneath) the plane region $\Omega$, namely $z \geq 0$. Under such circumstances the vault is in compression and formulations enforcing $\eta \in \mathrm{T}^2_-$, $\hat{\sigma} \in \mathrm{T}^3_-$ or negative axial member force $\hat{s}_k \leq 0$ are better suited.

Below we list the changed formulas for optimal vaults or grid-shells in case when only compression is allowed. We will not modify the plane convex problems $(\mathcal{P})$,\,$(\mathcal{P}^*)$ and their discrete counterparts $(\mathcal{P}_X)$, $(\mathcal{P}^*_X)$: for $(\sigma,q), (u,w)$ and $(\mbf{s},\mbf{q},\mbf{r})$, $(\mbf{u}_1,\mbf{u}_2,\mbf{w})$ being their solutions respectively there will still hold $\sigma \succeq 0$ and $s_k \geq 0$. Instead the formulas put forward in theorems in this work shall be modified by putting the minus sign in strategic places. Readily, Theorem \ref{thm:recovring_MV_dome} in the compression setting would furnish formulas for vault of minimum volume:
\begin{equation*}
	z = -\frac{1}{2} w,\quad \eta = - \frac{1}{J_{z}} \, \sig.
\end{equation*}
Meanwhile, with the energy potential $j_{z,-}\bigl(\xi \bigr) = \frac{E_0}{2} \bigl( \gamma_{z,-}(\xi ) \bigr)^2$ where $ 	\gamma_{z,-} = \bigl( \gamma^0_z  + \mathbbm{I}_{\mathrm{T}^2_-}\bigr)^0$, the vault in compression of minimum compliance could be constructed by the following alteration of formulas in Theorem \ref{thm:recovring_MC_dome}:
\begin{equation*}
	z = -\frac{1}{2}\, {w}, \qquad \rho = \frac{V_0}{\Z} \frac{1}{J_{z}}\, G_{z} : \sigma, \qquad (u_\e,w_\e) = \frac{\Z}{E_0 V_0}\, (-{u},{w}), \qquad \NN = -\frac{1}{J_{z}\, \rho}\, \sigma.
\end{equation*} 
Similar results apply to 3D problems posed in Section \ref{sec:3D}. Finally, Theorems \ref{thm:recovering_MV_grid-shell} and \ref{thm:recovering_MC_grid-shell} for optimal grid-shells in compression would change accordingly:
\begin{equation}
\label{eq:compression_mod}
{\mbf{z}} = - \frac{1}{2} \, \mbf{w}, \qquad \hat{\mbf{s}} = - \mbf{J}({\mbf{z}}) \,\mbf{s}, \qquad {\mbf{a}} = \frac{V_0}{\Z_X}\, \mbf{J}({\mbf{z}}) \,\mbf{s}, \qquad ({\mbf{u}}_{1,\e},{\mbf{u}}_{2,\e},{\mbf{w}_\e}) = \frac{\Z_X}{E_0 V_0}\, (-{\mbf{u}}_1,-{\mbf{u}}_2,{\mbf{w}}).
\end{equation}
The altered formulas for elastic deformations, e.g. $(u_\e,w_\e) = \Z/(E_0 V_0)\, (-{u},{w})$ are not obvious and are related to the operators $\mathcal{A}_z(u,w) := e(u) + \nabla z\, \symtens\, \nabla w$ or $\hat{\mathrm{e}}_k(\mbf{u}_1,\mbf{u}_2,\mbf{w};\mbf{z})$, see \eqref{eq:operator_hat_e_k}, that describe the elastic strains: by virtue of \eqref{eq:unit_strain}  one can check that $\hat{\mathrm{e}}_{k}({\mbf{u}}_{1,\e},{\mbf{u}}_{2,\e},{\mbf{w}_\e};{\mbf{z}}) = -\Z/(E_0 V_0)$ provided that  $a_k>0$.

\section{Numerical simulations}
\label{sec:numerics}

\subsection{The input data and software}

The discrete method put forward revolves around the pair of mutually dual conic quadratic programming problems $(\mathcal{P}_X)$,\,$(\mathcal{P}_X^*)$. On one side, the pair is the discrete counterpart of the pair $(\mathcal{P})$,\,$(\mathcal{P}^*)$ that furnishes solutions of optimal vault design problems in settings presented in Sections \ref{sec:plastic_design}, \ref{sec:elastic_design}, \ref{sec:3D}. On the other, Section \ref{ssec:grid-shells} shows that the pair $(\mathcal{P}_X)$,\,$(\mathcal{P}_X^*)$ is directly linked to to the problem of optimal grid-shell being a substructure of the 3D ground structure generated by the elevated grid $\hat{X}_{\mbf{z}}$. Therefore, given a design domain $\Omega$ and a load $f \in \Mes(\Ob;\R)$ our goal is to numerically solve the conic quadratic program $(\mathcal{P}_X)$,\,$(\mathcal{P}_X^*)$ for a fine grid $X \subset \Ob$. Then, based on solutions $(\mbf{s},\mbf{q},\mbf{r})$ and $(\mbf{u}_1,\mbf{u}_2, \mbf{w})$, the optimal grid-shell is recovered via Theorem \ref{thm:recovering_MC_grid-shell} which may be considered an approximation of the general solution given by Theorem \ref{thm:optimal_3D_structure}.

In the numerical simulations we shall choose polygonal domains $\Omega$ that may be non-convex or multi-connected. For a chosen resolution parameter $h>0$ we will use regular grids of the form
\begin{equation}
	\label{eq:X_regular}
	X = \Big(\bigl\{ (h\, j_1, h\,j_2) \, \big\vert\, j_1, j_2 \in \mathbbm{Z} \bigr\} \cap \Omega \Big) \cup X_{\bO}
\end{equation}
where $X_\bO$ is a union of sets of the form $\bigl(\{h \,j_1\}\times \R \bigr) \cap \bO $ and $\bigl(\R \times \{h \,j_2\}\bigr) \cap \bO $, $\ j_1, j_2 \in \mathbbm{Z}$, without their relative interiors (horizontal and vertical lines may intersect with whole edges of polygon $\Omega$), see Fig. \ref{fig:ground_structure}(a). Condition \eqref{eq:assum_on_X} may be easily verified. For non-convex polygons $\Omega$ the $m$-element ground structure will be constructed by choosing those segments $[x,y]$ that are contained within the closure $\Ob$. Based on the original load $f \in \Mes(\Ob;\R)$, that may be a combination of discrete point loads, loads distributed along lines or across a 2D area, the discrete load $f_X \in \Mes(X;\R)$ represented by vector $\mbf{f} \in \R^n$, will be proposed "manually", see e.g. Examples \ref{ex:diagonal_load},\,\ref{ex:pressure}. The proposed algorithm is not dedicated for handling complicated geometries of the design domains -- for that purpose the GRAND method developed in \cite{zegard2014} could be employed.

The numerical method was implemented in $\text{MATLAB}^{\tiny{\textregistered}}$ R2018b. After building the vectors $\mbf{l} \in \R^m, \mbf{f} \in \R^n$ and matrices $\mbf{B}_1, \mbf{B}_2, \mbf{D}$ (in a sparse form) the conic quadratic programs $(\mathcal{P}_X)$,\,$(\mathcal{P}_X^*)$, precisely in the form written in Section \ref{ssec:discretication}, is being solved via dedicated solver from the $\text{MOSEK}^{\tiny{\textregistered}}$ 8 toolbox, cf. \cite{mosek2019}. The solver exploits a variant of the primal-dual interior point method developed in \cite{andersen2003} which is suited for large-scale and sparse problems. By following Section 6.3 in \cite{mosek2019} the conic quadratic problem $(\mathcal{P}_X^*)$ is implemented directly (we note that the \textit{rotated} quadratic cone $\mathrm{K}$ is used for the conic constraints). As a result one obtains solutions $(\mbf{s},\mbf{q},\mbf{r})$ and $(\mbf{u}_1,\mbf{u}_2,\mbf{w},\mbf{t}_1,\mbf{t}_2,\mbf{t}_3)$ of the both programs. The post-processing part, that includes recovering the 3D grid-shell $(\mbf{z},\mbf{a})$ through Theorem \ref{thm:recovering_MC_grid-shell} and visualizations, is performed in $\text{Mathematica}^{\tiny{\textregistered}}$ 11.1.

\subsection{The adaptive approach via the member-adding technique}
\label{sssec:member_adding}

During the tests run on a personal computer, the MOSEK solver was capable of solving the conic program $(\mathcal{P}_X), (\mathcal{P}_X^*)$ for ground structures containing up to several million members, which for a square domain $\Omega$ roughly corresponds to a $50 \times 50$ grid $X$ (for a convex domain the number of members is equal to $m = \bar{n} \,(\bar{n}-1)/2$ where $\bar{n}$ is the total number of points in $X$). For comparison, a $200 \times 200$ grid generates $m \approx 800\cdot 10^6$ bars, being far beyond the reach of any interior-point method solver being run on a PC. Very similar computational challenges are faced when solving large scale linear programming problem \eqref{eq:truss_max}, \eqref{eq:truss_min} in order to produce accurate truss approximations of Michell structures. In \cite{gilbert2003} the rectangular structure of the $m \times n$ matrices $\mbf{B}_1,\mbf{B}_2$ ($m >> n$ for large $n$) was taken advantage of. The authors proposed an adaptive approach where as the point of departure one takes a small ground structure consisting of $\widetilde{m} \approx 4 \,\bar{n}$ members connecting only the neighbouring nodes in $X$; after solving the linear program for matrices  $\widetilde{\mbf{B}}_1,\widetilde{\mbf{B}}_2$ the ground structure is then updated in subsequent iterations by adding bars that violate the feasibility condition in the dual problem. The so called \textit{member-adding} adaptive technique was further developed in \cite{sokol2015} and also successfully applied to other linear programming problems: slip-line analysis in \cite{gilbert2014}, optimal design of grillages in \cite{bolbotowski2018}, optimal design of long-span bridges with gravity loading in \cite{fairclough2018}. In this work we bring the member-adding adaptive method to the pair of conic quadratic programs $(\mathcal{P}_X), (\mathcal{P}_X^*)$. The algorithm of the method is outlined in Box 7.1.
\begin{tcolorbox}[float=p!,title=\textbf{Box 7.1 The adaptive "member-adding" algorithm}]
\begin{enumerate}[label=\{\arabic*\}]
	\item Preparing the first iteration:
	\begin{enumerate}[label=\{1.\arabic*\},rightmargin=1cm]
		\item Throughout the algorithm the nodal grid $X$ according to \eqref{eq:X_regular} is fixed for a chosen resolution parameter $h$. For the number $n = \#(X \backslash \bO)$ the load vector $\mbf{f} \in \R^n$ is built and does not change in subsequent iterations. The \textit{full ground structure} will be the family of all $m$ members/segments interconnecting the nodes in $X$ that are contained in the set $\Ob$ $($each index $k \in \{1,\ldots,m\}$ identifies the segment $[\chi_-(k),\chi_+(k)] )$. 
		\item For the first iteration $iter = 1$ the initial \textit{active ground structure} is prepared that consists only of members connecting the neighbouring nodes in $X$, namely the set of member indices $k$ is restricted to
		\begin{equation*}
		I_1 = \Big\{ k\in \{1,\ldots,m\} \ \, \big\vert \ \, \abs{\chi_+(k)-\chi_-(k)}_\infty \leq h \Big\},
		\end{equation*}
		where  $\abs{y}_\infty  = \max\{\abs{y_1},\abs{y_2}\}$ for $y =(y_1,y_2) \in \Rd$.
		For $m_1 = \#(I_1)$ vector $\mbf{l}^{(1)} \in \R^{m_1}$ and matrices $\mbf{B}_1^{(1)}, \mbf{B}_2^{(1)}, \mbf{D}^{(1)} \in \R^{m_1 \times n}$ are built (in their sparse form) according to formulas in Section \ref{ssec:discretication} yet for rows referring only to indices $k \in I_1$.
	\end{enumerate}
	\item  Adaptive loop  that starts at the counter $iter = 1$. Each iteration $iter$ involves the following steps:
	\begin{enumerate}[label=\{2.\arabic*\},rightmargin=1cm]
		\item \label{item:first_step} The MOSEK solver is run for the conic quadratic program $(\mathcal{P}_X^{(iter)}), (\mathcal{P}_X^{*,(iter)})$ that is constructed for vectors and matrices $\mbf{f} \in \R^n,\ \mbf{l}^{(iter)} \in \R^{m_{iter}}$, $ \mbf{B}_1^{(iter)},\mbf{B}_2^{(iter)},\mbf{D}^{(iter)}\in \R^{m_{iter} \times n}$. A solution is returned:	
		\begin{equation*}
		\mbf{s}^{(iter)}, \mbf{q}^{(iter)}, \mbf{r}^{(iter)} \in \R^{m_{iter}}, \qquad \mbf{u}_1^{(iter)}, \mbf{u}_2^{(iter)}, \mbf{w}^{(iter)} \in \R^n.
		\end{equation*}
		\item \label{item:violation} The set of members in the full ground structure, for which the two-point condition is violated by $\mbf{u}_1^{(iter)}, \mbf{u}_2^{(iter)}, \mbf{w}^{(iter)}$ being competitors in the original problem $(\mathcal{P}^*_X)$, is identified:
		\begin{equation*}
		\qquad \delta I_{iter} = \left\{ k\in \{1,\ldots,m\} \ \left\vert \ \frac{1}{4}\,\left( \Delta_k\bigl(\mbf{w}^{(iter)}\bigr)\right)^2 + \mathrm{e}_k\bigl(\mbf{u}_1^{(iter)}, \mbf{u}_2^{(iter)} \bigr) > 1 \right.\right\};
		\end{equation*}
		it must be stressed that all the members $k$ in the \textit{full} ground structure are checked.
		\item If $\delta I_{iter} = \varnothing$ then the loop is terminated. Otherwise the active ground structure is updated:
		\begin{equation*}
		I_{iter+1} = I_{iter}\, \cup\, \delta I_{iter}
		\end{equation*}
		and the vectors/matrices are appended to $\mbf{l}^{(iter+1)} \in \R^{m_{iter+1}}$, $ \mbf{B}_1^{(iter+1)}$, $\mbf{B}_2^{(iter+1)}$, $\mbf{D}^{(iter+1)}\in \R^{m_{iter+1} \times n}$ by adding rows corresponding to every $k \in \delta I_{iter}$, where $m_{iter +1} = \#(I_{iter+1})$. Next the step \ref{item:first_step} is repeated for $iter := iter+1$.
	\end{enumerate}
	\item Constructing solutions of the original conic quadratic program $(\mathcal{P}_X), (\mathcal{P}_X^*)$:
	\begin{enumerate}[label={},,rightmargin=1cm]
		\item For $iter$ being the last counter in the adaptive loop we set $\bigl(\mbf{u}_1,\mbf{u}_2,\mbf{w} \bigr) \linebreak = \bigl(\mbf{u}_1^{(iter)}, \mbf{u}_2^{(iter)}, \mbf{w}^{(iter)} \bigr)$ while $\mbf{s}, \mbf{q}, \mbf{r} \in \R^m$ are constructed by allocating components of $\mbf{s}^{(iter)}, \mbf{q}^{(iter)}, \mbf{r}^{(iter)} \in \R^{m_{iter}}$ into zero vectors in $\R^m$. Optimality conditions \eqref{eq:opt_cond_discrete} are met for the triples $(\mbf{s}, \mbf{q}, \mbf{r})$ and $\bigl(\mbf{u}_1,\mbf{u}_2,\mbf{w} \bigr)$: conditions $(ii), (iii), (iv)$ are numerically established by MOSEK solver at each iteration, while condition $(i)$, holding for every $k \in \{1,\ldots,m\}$ (i.e. for each member in the full ground structure), is guaranteed by $\delta I_{iter} = \varnothing$.
	\end{enumerate}
\end{enumerate}
\end{tcolorbox}

A few comments on the algorithm are in order. First, we must ensure that in the first iteration $iter = 1$ the problems $(\mathcal{P}_X^{(1)}), (\mathcal{P}_X^{*,(1)})$ have solutions. It is important to note that verification of  condition \eqref{eq:assum_on_X} is \textit{a priori} insufficient to answer positively, since the set $I_1$ could potentially miss crucial members that guaranteed existence of feasible solution in the problem $(\mathcal{P}_X)$ posed for the full ground structure. However, from the prove of Theorem \ref{thm:duality_discrete} (cf. \ref{app:cone_duality}) we see that it is enough that at each point $\bar{x} =(\bar{x}_1,\bar{x}_2) \in X \cap \Omega$ we have a four bar truss formed by four chains of horizontal/vertical bars of length $h$ going from $\bar{x}$ to $(\bar{x}_1-a_1,\bar{x}_2), (\bar{x}_1+b_1,\bar{x}_2), (\bar{x}_1,\bar{x}_2-a_2), (\bar{x}_1,\bar{x}_2+b_2) \in  X_\bO$ for $a_1,a_2, b_1, b_2 > 0 $. As a result assertions of Theorem \ref{thm:duality_discrete} holds true for problems $(\mathcal{P}_X^{(1)}), (\mathcal{P}_X^{*,(1)})$ and thus also for the problems in the subsequent iterations.

The adaptive loop is by construction guaranteed to finish in a finite number of steps that does not exceed $m - m_1$ which, of course, is very pessimistic. Although no rate of convergence is examined in this work, the member-adding algorithm usually converges in no more than 10 to 15 iterations while (for large $m$ of magnitudes: $10^6$ to $10^9$) the number $m_{iter}$ is usually kept on the level between $m/1000$ and $m/100$. Exceptions may occur if in the solution of the continuous problem $(\mathcal{P}^*)$ we have subregions of $\Omega$ where $\frac{1}{4}\, \nabla w \otimes \nabla w + e(u)$ equals to identity matrix $\mathrm{I}$ exactly -- then we may expect that every single bar lying in this subregion will be added to the set of active bars $I_{iter}$ which makes $m_{iter}$ to blow up with each iteration. This phenomena is known from truss and grillage optimization and can be handled by a simple trick: from each component of the vector $\mbf{l}^{(iter)}$ we subtract a very small constant, e.g. $ 10^{-7}$.

\subsection{Examples}
\label{ssec:examples}

For several examples of domian $\Omega$ and load $f$ we will present numerical solutions of the optimal grid-shell problem, we shall focus on the case of elastic design, i.e. the problem $(\mathrm{MCGS}_X)$. Using the adaptive algorithm implemented in MATLAB the pair of conic quadratic problems $(\mathcal{P}_X), (\mathcal{P}_X^*)$ will be solved directly and the grid-shell will be recovered via Theorem \ref{thm:recovering_MC_grid-shell}. We will consider the case when only compression in the grid-shell is admissible therefore the altered formulas \eqref{eq:compression_mod} shall be employed. In each example we will use grids $X$ of three different resolutions while the figures will refer to the finest $X$. It is thus justified to treat the displayed grid-shell solutions as discrete approximations of optimal vaults or, more generally, of the 3D lower dimensional measure $\hat{\mu}\in \Mes_+(\Ob\times \R)$ furnished by Theorem \ref{thm:optimal_3D_structure}. Computations were performed on a mobile workstation with Intel Core i7-6700HQ processor, 16 GB RAM and running 64-bit Windows 7. The computational details are summarized in Table \ref{tab:miscellaneous} whose entries are described in the first Example \ref{ex:4_point_loads}. It should be stressed that symmetry of the investigated problems has not been exploited numerically.

\begin{table}[h]
	\scriptsize
	\centering
	\caption{Summary on numerical computations for the pair of problems $(\mathcal{P}_X), (\mathcal{P}_X^*)$ specified in Examples \ref{ex:4_point_loads}-\ref{ex:cross}}
	\begin{tabular}{lccccccc}
		\toprule
		Example & Grid $X$  & Full GS  & Iterations   & Active GS    & CPU time & Objective value $\Z_X$ &  Max. elevation $\norm{z}_\infty$\\
		\midrule[0.36mm]
		\ref{ex:4_point_loads}
		& $65\!\times\!65$ & $8\,923\,200$ & 8 & $20\,052$ & $42$ sec. & $2.6458 \, P a$ & $0.331 \,a $\\
		& $105\!\times\!105$ & $60\,769\,800$ & 10 & $56\,963$ & $3$ min. $35$ sec. & $2.6458 \, P a$ & $0.331 \,a $\\
		& $201\!\times\!201$ & $816\,100\,200$ & 10 & $233\,886$ & $21$ min. $08$ sec. & $2.6458 \, P a$ & $0.331 \,a $\\
		\midrule[0.1mm]
		\ref{ex:multiple_point_loads}
		& $41\!\times\!41$ & $1\,412\,040$ & 8  &$7\,484$ & $18$ sec. & $27.286 \, P a$ & $0.405 \,a $\\
		& $81\!\times\!81$ & $21\,520\,080$ & 9  &$30\,386$ & $1$ min. $40$ sec. & $27.286 \, P a$ & $0.405 \,a $\\
		& $161\!\times\!161$ & $335\,936\,160$ & 10  &$139\,823$ & $13$ min. $05$ sec. & $27.283 \, P a$ & $0.405 \,a $\\
		\midrule[0.1mm]
		\ref{ex:diagonal_load}
		& $51\!\times\!51$ & $3\,381\,300$ & 8 & $19\,618$ & $30$ sec. & $1.6565 \, t a^2$ & $0.484 \,a $\\
		& $101\!\times\!101$ & $52\,025\,100$ & 9 & $104\,148$ & $3$ min. $44$ sec. & $1.6566 \, t a^2$ & $0.484 \,a $\\
		& $201\!\times\!201$ & $816\,100\,200$ & 10 & $612\,082$ &  $35$ min. $46$ sec. & $1.6567 \, t a^2$ & $0.484 \,a $\\
		\midrule[0.1mm]
		\ref{ex:pressure}
		& $51\!\times\!51$ & $3\,381\,300$ & 8 & $13\,665$ & $24$ sec. & $0.43696 \, p a^3$ & $0.405 \,a $\\
		& $101\!\times\!101$ & $52\,025\,100$ & 9 & $71\,609$ & $2$ min. $26$ sec. & $0.43647 \, p a^3$ & $0.405 \,a $\\
		& $201\!\times\!201$ & $816\,100\,200$ & 10 & $392\,812$ &  $24$ min. $44$ sec. & $0.43615 \, p a^3$ & $0.405 \,a $\\
		\midrule[0.1mm]
		\ref{ex:cross}
		& $61\!\times\!61$ & $\approx 60 \cdot 10^6$ & 8 & $14\,905$ &  $28$ sec. & $0.22338 \, p a^3$ & $0.303 \,a $\\
		& $121\!\times\!121$ & $\approx 125 \cdot 10^6$ & 9 & $71\,313$ &  $4$ min. $10$ sec. & $0.22302 \, p a^3$ & $0.303 \,a $\\
		& $241\!\times\!241$ & $\approx 600 \cdot 10^6$ & 10 & $352\,624$ &  $29$ min. $36$ sec. & $0.22283 \, p a^3$ & $0.304 \,a $\\
		\bottomrule
	\end{tabular}
	\label{tab:miscellaneous}
\end{table}

\begin{example} \textbf{(Square domain and  four point loads)}
\label{ex:4_point_loads}
For a square domain $\Omega$ of side's length equal to $a$ we consider a load being four downward forces of magnitude $P$ symmetrically spaced as in Fig. \ref{fig:4_point_loads}(a); as a measure the load reads $f = \sum_{j=1}^4 (- P \delta_{x_j})$. The pair of conic quadratic programs $(\mathcal{P}_X)$,\,$(\mathcal{P}_X^*)$ will be solved for a regular grid $X$ in accordance with \eqref{eq:X_regular} for three resolutions: $h \in\{a/64, a/104, a/200\}$. The finest resolution gives $201 \times 201$ nodes, while $n = \#(X \backslash \bO) =199^2 = 39\,601$. The fully connected ground structure (abbrev. "Full GS" in Table \ref{tab:miscellaneous}) counts $m = \bar{n}\,(\bar{n}-1)/2 =816\,100\,200$ potential members, where $\bar{n} = 201^2$. There is no need to discretize the load, i.e. $f_X = f$ and the load vector $\mbf{f} \in \R^n$ has four non-zero components each equal to  \nolinebreak $-P$. 

\begin{figure}[h]
	\centering
	\subfloat[]{\includegraphics*[trim={0cm -1cm -0cm -0cm},clip,height=0.23\textwidth]{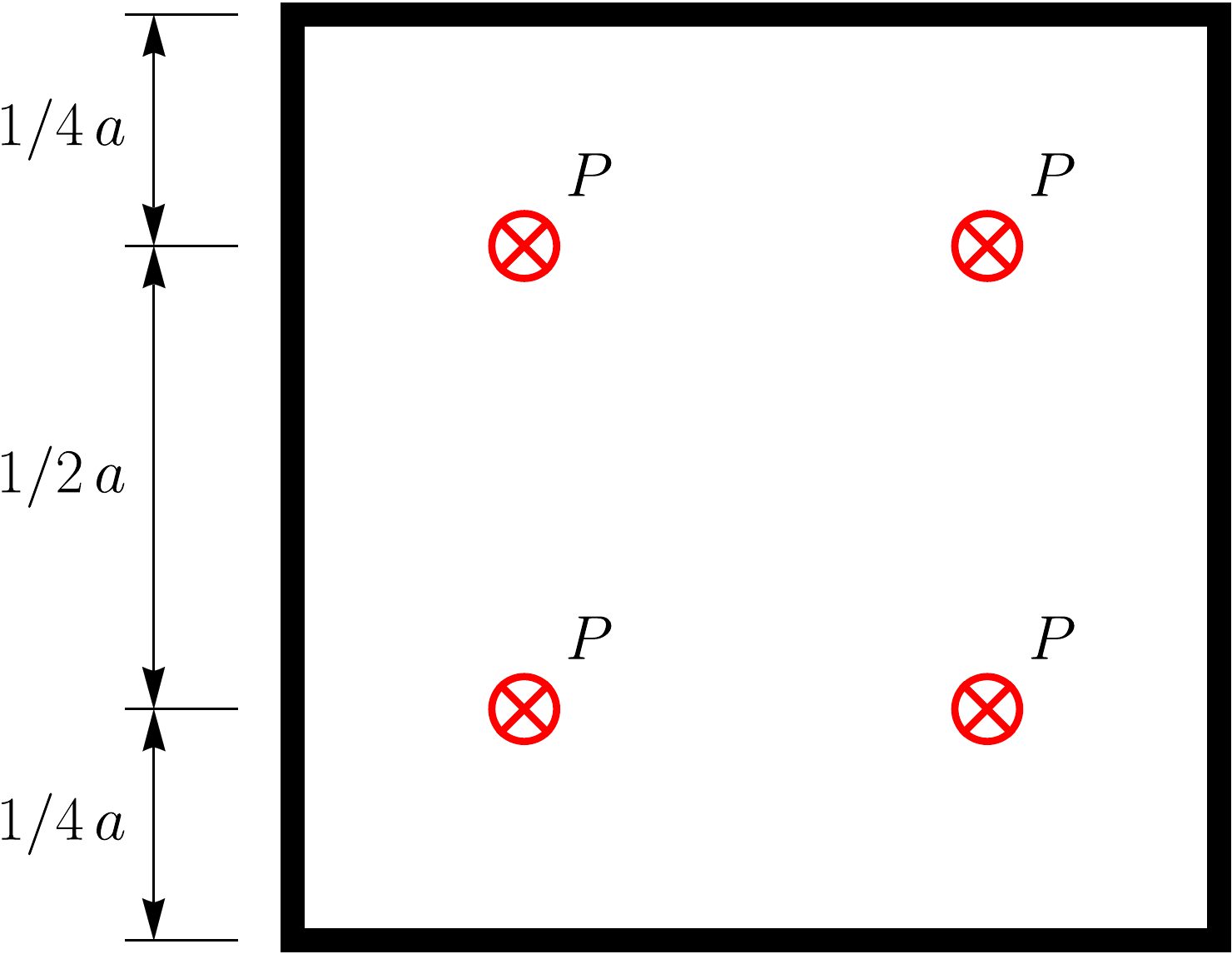}}\hspace{0.8cm}
	\subfloat[]{\includegraphics*[trim={0cm -1cm -0cm -0cm},clip,height=0.23\textwidth]{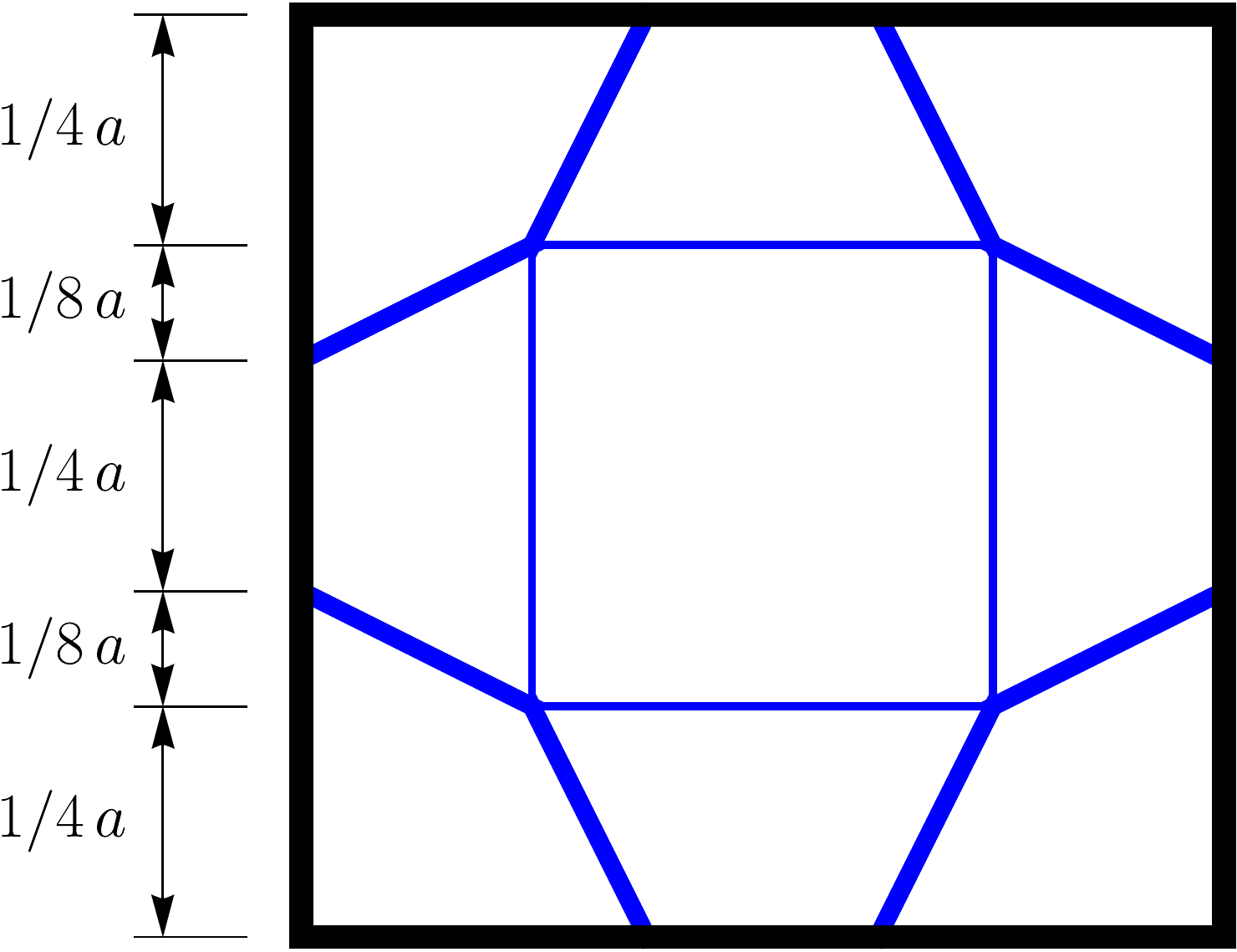}}\hspace{0.8cm}
	\subfloat[]{\includegraphics*[trim={0.5cm 0cm 0.5cm 1cm},clip,width=0.33\textwidth]{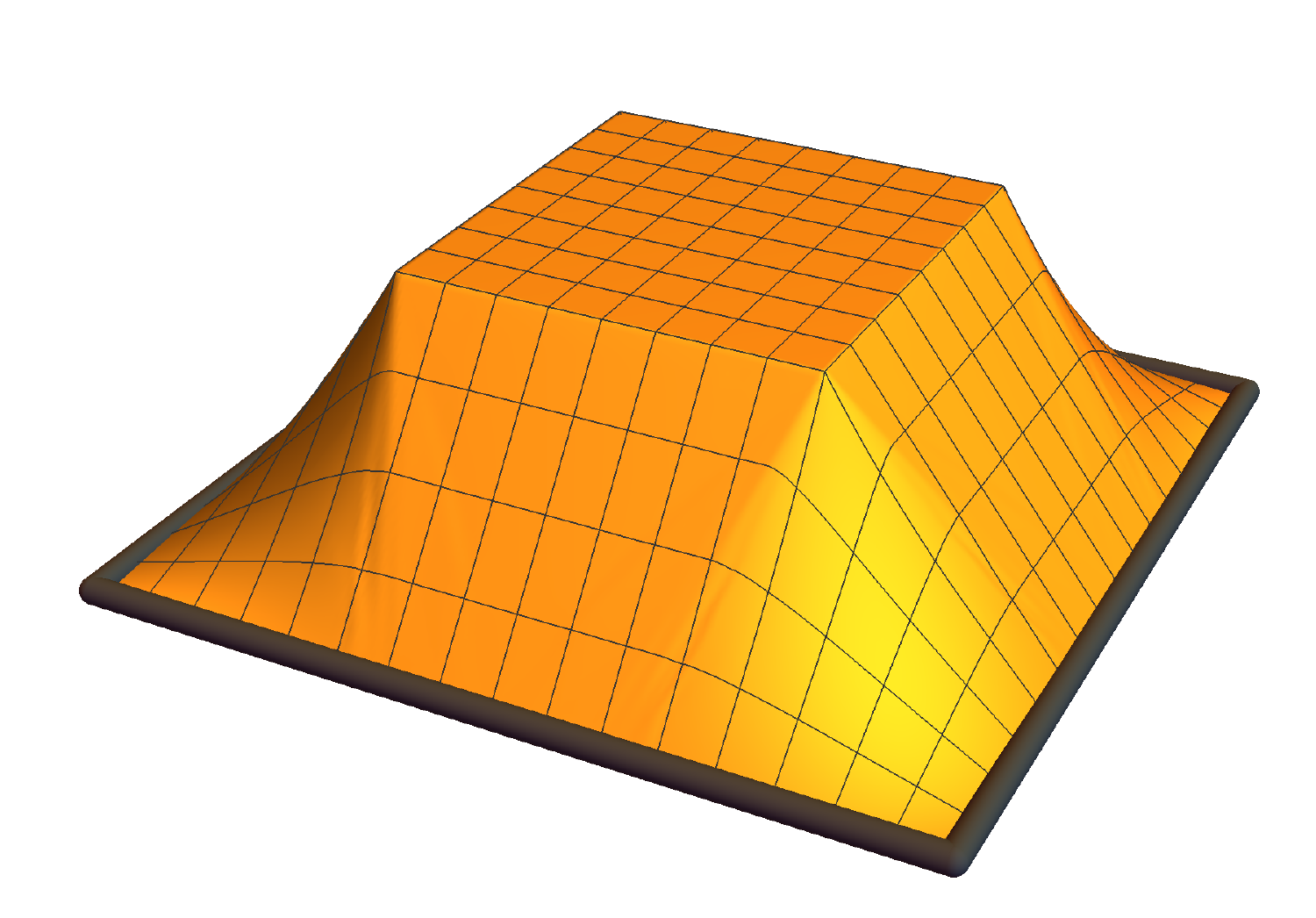}}\\
	\subfloat[]{\includegraphics*[trim={0.5cm -0cm 0.5cm 2.5cm},clip,width=0.33\textwidth]{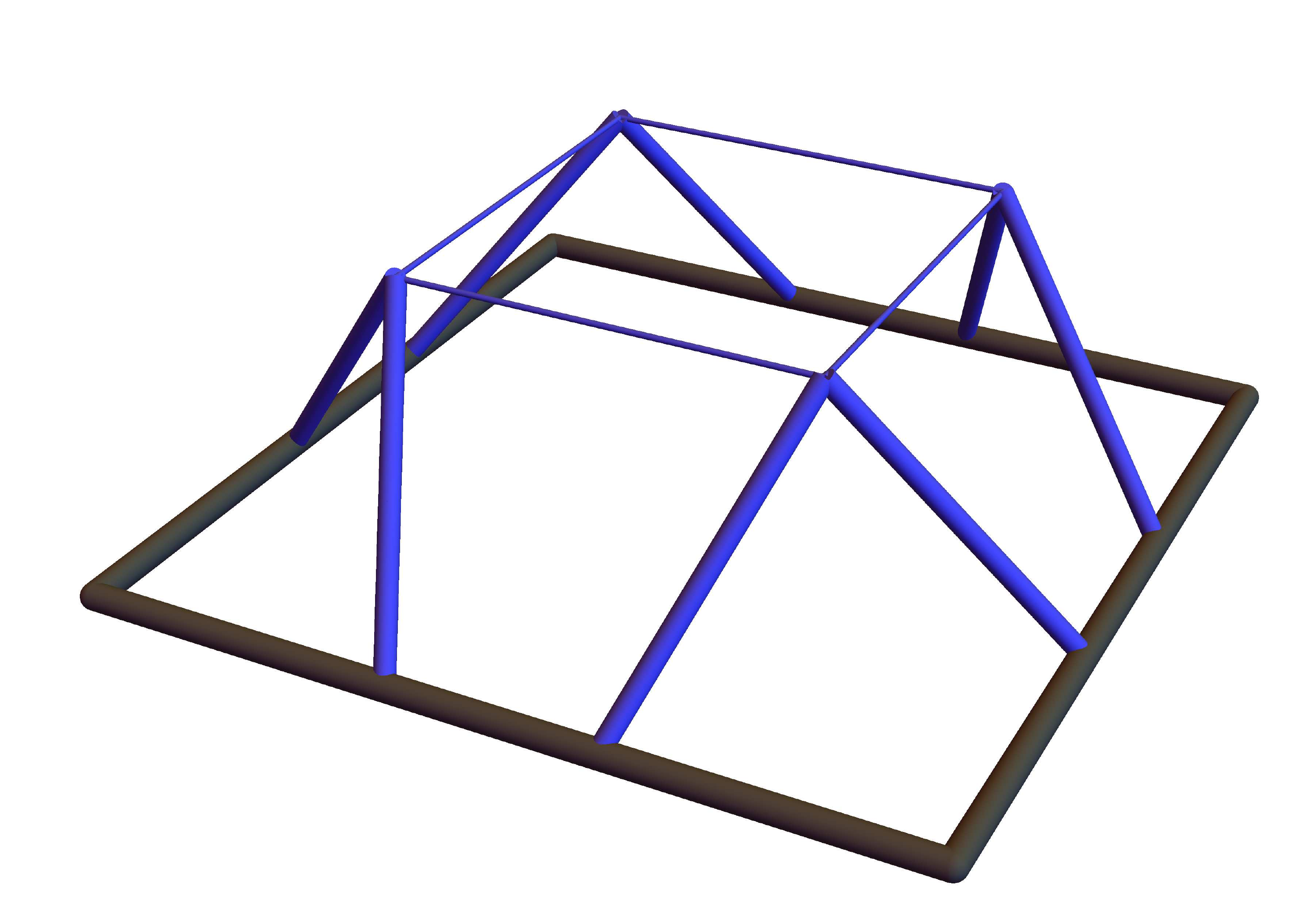}}\hspace{1cm}
	\subfloat[]{\includegraphics*[trim={0.5cm -0cm 0.5cm 2.5cm},clip,width=0.33\textwidth]{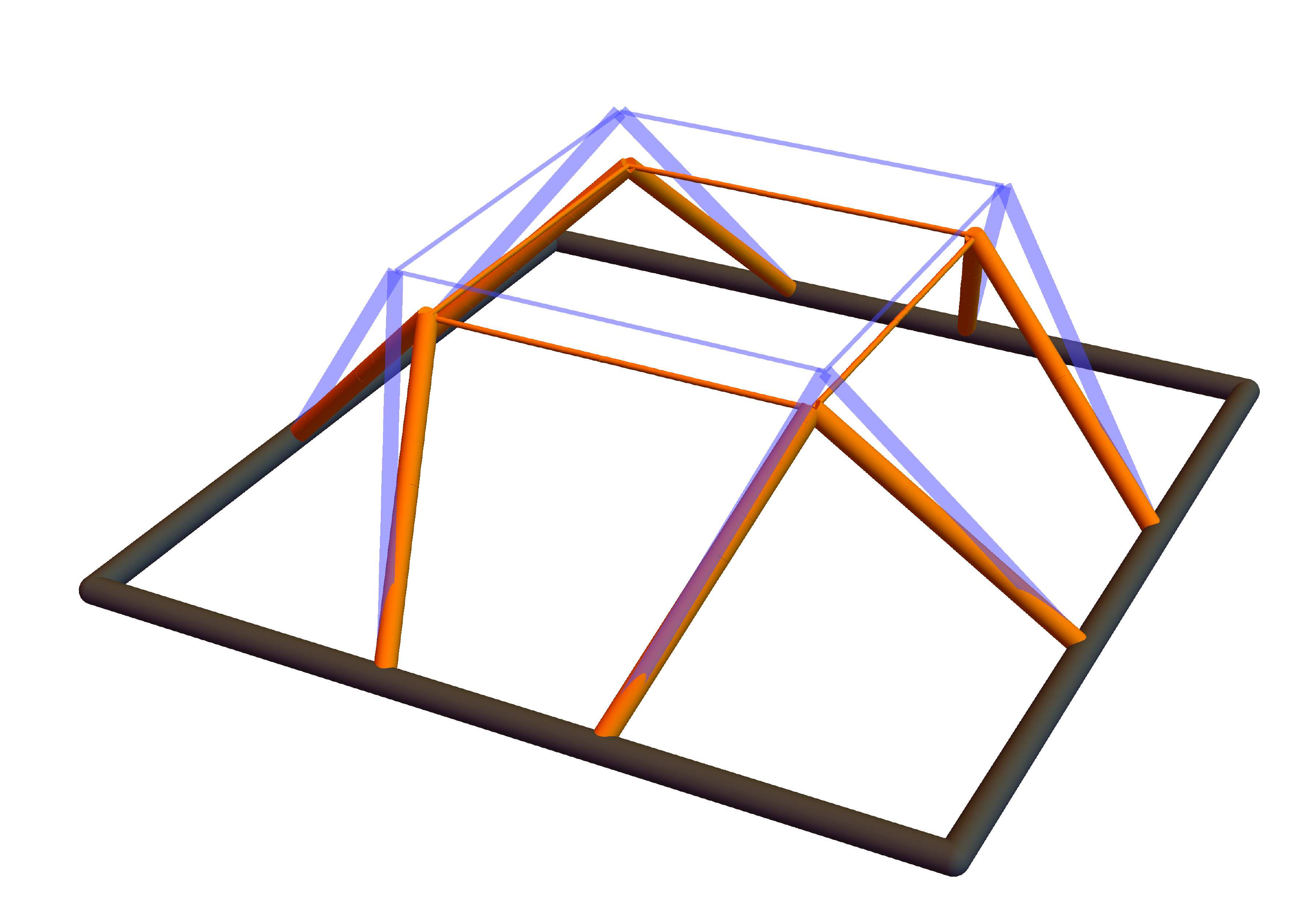}}
	\caption{The problem of optimal grid-shell in compression over a square domain and under a four-point load: (a) the loading conditions; (b) the plane truss given by $\mbf{s}$ solving $(\mathcal{P}_X)$; (c) optimal elevation function $z = -\frac{1}{2} \,w$  where $w$ solves $(\mathcal{P}^*_X)$; (d) optimal grid-shell in compression; (e) elastic deformation.}
	\label{fig:4_point_loads}
\end{figure}

For $h =a/200$ the adaptive algorithm solving the conic program  $(\mathcal{P}_X)$,\,$(\mathcal{P}_X^*)$  converged in 10 iterations and the total computational time was around 21 minutes; the number of active members (abbrev. "Active GS" in Table \ref{tab:miscellaneous}) in the last iteration was $m_{10} = 233\,886$. The objective value thus numerically obtained was $\Z_{X} = 2.6458 \, P a$. The computational details, also for other resolutions of $X$, are listed in Table \ref{tab:miscellaneous}.

The vector $\mbf{s}$ that numerically solves problem $(\mathcal{P}_X)$ represents a pre-stressed plane truss that is showed in Fig. \ref{fig:4_point_loads}(b); it turns out to be identical for all three resolutions $h$. We find that from almost one billion members available the optimal truss consists of only twelve bars: four bars parallel to the square's sides and eight bars inclined at 1:2 slope. The author suspects that Fig. \ref{fig:4_point_loads}(b) represents an exact solution of the infinite dim. problem $(\mathcal{P})$, i.e. a solution $\sigma$ of the form \eqref{eq:sigma_q_discrete}. Attempts to prove this optimality has been made, however, in order to employ optimality conditions in Theorem \ref{thm:duality} one has to propose analytical functions $u,w$ solving $(\mathcal{P}^*)$ which seems not straightforward despite numerical solutions $\mbf{u}_1, \mbf{u}_2, \mbf{w}$ of $(\mathcal{P}^*_X)$ being available. We pay attention to an important feature of solution Fig. \ref{fig:4_point_loads}(b) that will reappear in the rest of examples in this work: the optimal truss interconnects the points in the set $\mathrm{spt}\,{f} \cup \bO$ or, in the discrete setting, $\mathrm{spt}\,f_X \cup (X \cap \bO)$. This is a big difference when comparing to the Michell problem, where it is typical that even for three forces on a plane an optimal truss (optimal with respect to a given ground structure) engages intermediate points as joints where bars connect.

Based on solutions of the conic program we construct an optimal elastic grid-shell via Theorem \ref{thm:recovring_MC_dome}; for given volume upper bound $V_0$ and Young modulus $E_0$ we shall assume that the grid-shell is in pure compression instead of pure tension and thus formulas in \eqref{eq:compression_mod} will apply. Hence, for $\mbf{w}$ being numerical solution of $(\mathcal{P}^*_X)$ the optimal elevation vector reads $\mbf{z} = - \frac{1}{2}\,\mbf{w}$; in Fig. \ref{fig:4_point_loads}(c) we may see the graph being an interpolation of the induced function $z:X \to \R$. The shape of optimal grid-shell, showed in Fig. \ref{fig:4_point_loads}(d), follows by unprojecting the plane truss in Fig. \ref{fig:4_point_loads}(b) whereas formula $\mbf{a} = V_0/\Z_X \, \mbf{J}(\mbf{z})\, \mbf{s}$ yields the bar's cross section areas. It can be easily verified that the grid-shell, being a 3D truss, is geometrically unstable. Nevertheless, according to Theorem \ref{thm:recovering_MC_grid-shell} (and modification \eqref{eq:compression_mod} for the structure in compression) under the original load $\mbf{f}$ (elevated through $\mbf{z}$) the elastic deformation exists and the nodal displacements are $({\mbf{u}}_{1,\e},{\mbf{u}}_{2,\e},{\mbf{w}_\e}) = \Z_X/(E_0 V_0)\, (-{\mbf{u}}_1,-{\mbf{u}}_2,{\mbf{w}})$; the deformation is visualized in Fig. \ref{fig:4_point_loads}(e).

\end{example}

\begin{example}\textbf{(Multiple point forces evenly spaced in a square domain)}
\label{ex:multiple_point_loads}

Through the next example we investigate the phenomena noted for the case of four-force load for which the optimal truss solving the primal problem interconnects only the points in the set $\mathrm{spt}\,{f} \cup \bO$. To this aim we simply increase the number of point loads to 49, equally spaced every $a/8$ in two directions, each of magnitude $P$, cf. Fig. \ref{fig:multiple_point_loads}(a). A $161 \times 161$ nodal grid $X$ is chosen so that between every four point loads there are $20 \times 20$ nodes and thus the load $f = f_X = \sum_{j=1}^{49} (- P \delta_{x_j})$ does not simulate a continuous uniform load (cf. Example \ref{ex:pressure}). The computational details, also for other resolutions of $X$, are listed in Table \ref{tab:miscellaneous}.

\begin{figure}[h!]
		\centering
		\subfloat[]{\includegraphics*[trim={0cm -1cm -0.5cm -0cm},clip,height=0.23\textwidth]{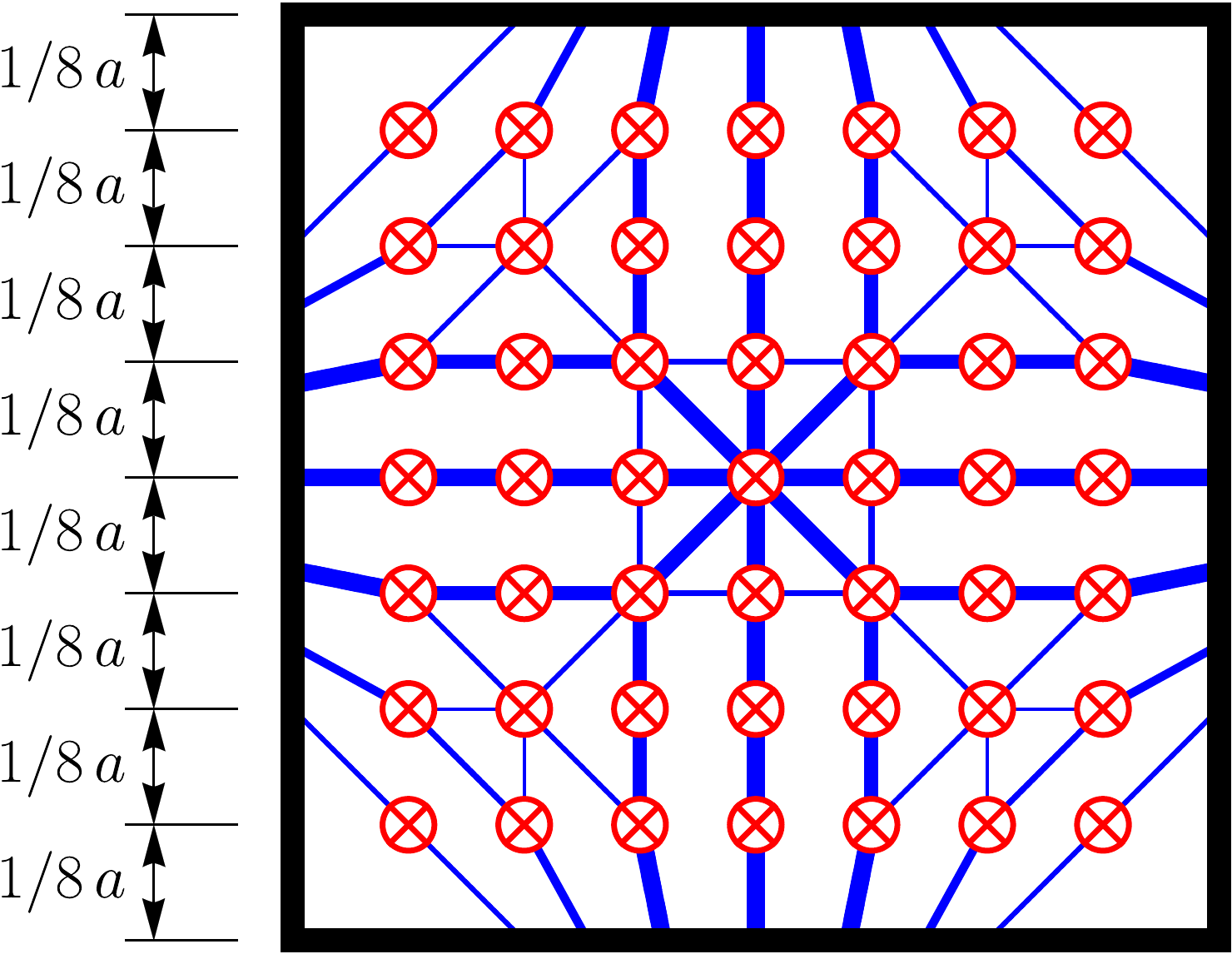}}\hspace{0.3cm}
		\subfloat[]{\includegraphics*[trim={1cm 0cm 0.5cm 2cm},clip,width=0.33\textwidth]{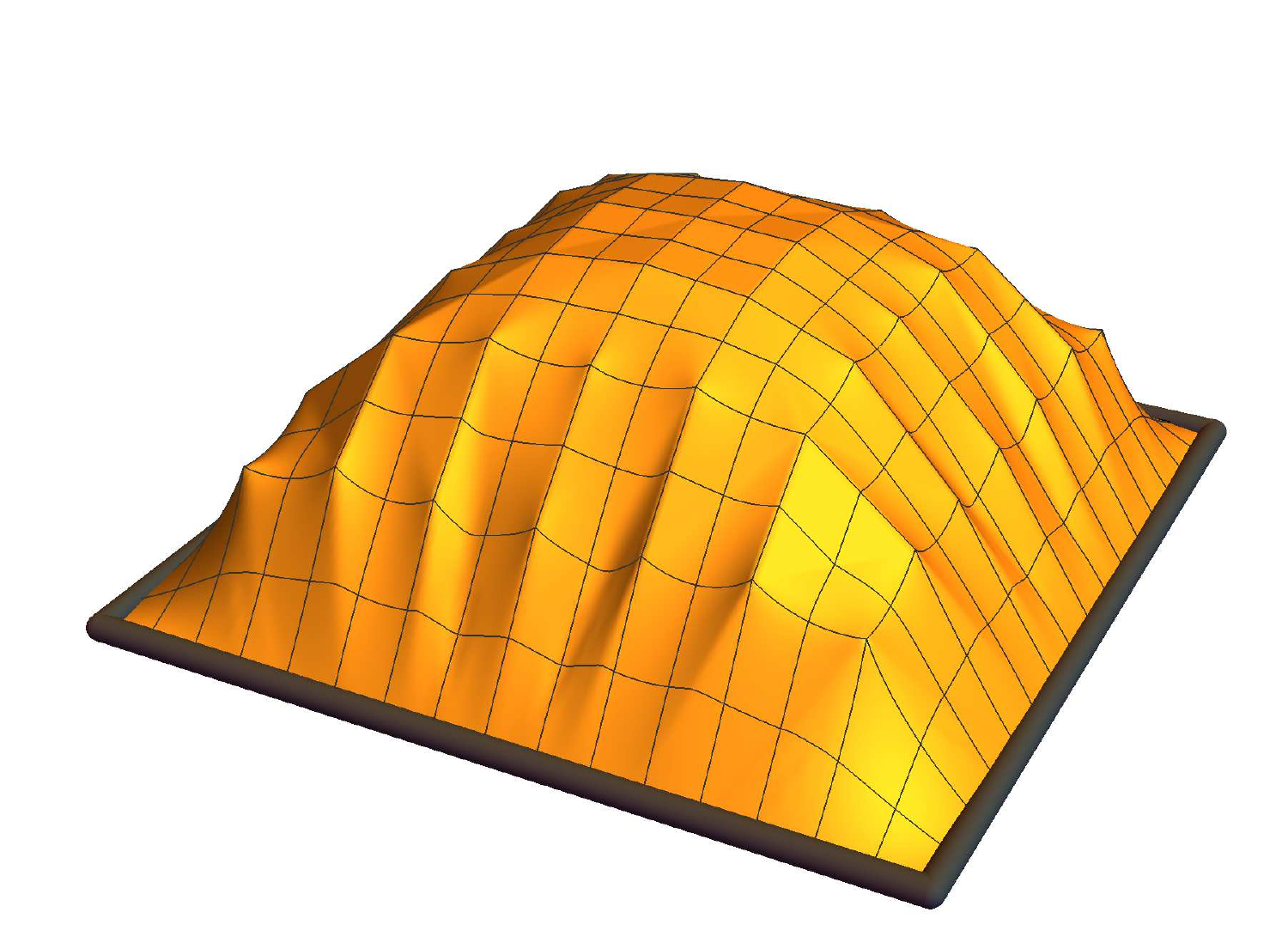}}\hspace{0.cm}
		\subfloat[]{\includegraphics*[trim={1.5cm -0cm 1.cm 4cm},clip,width=0.33\textwidth]{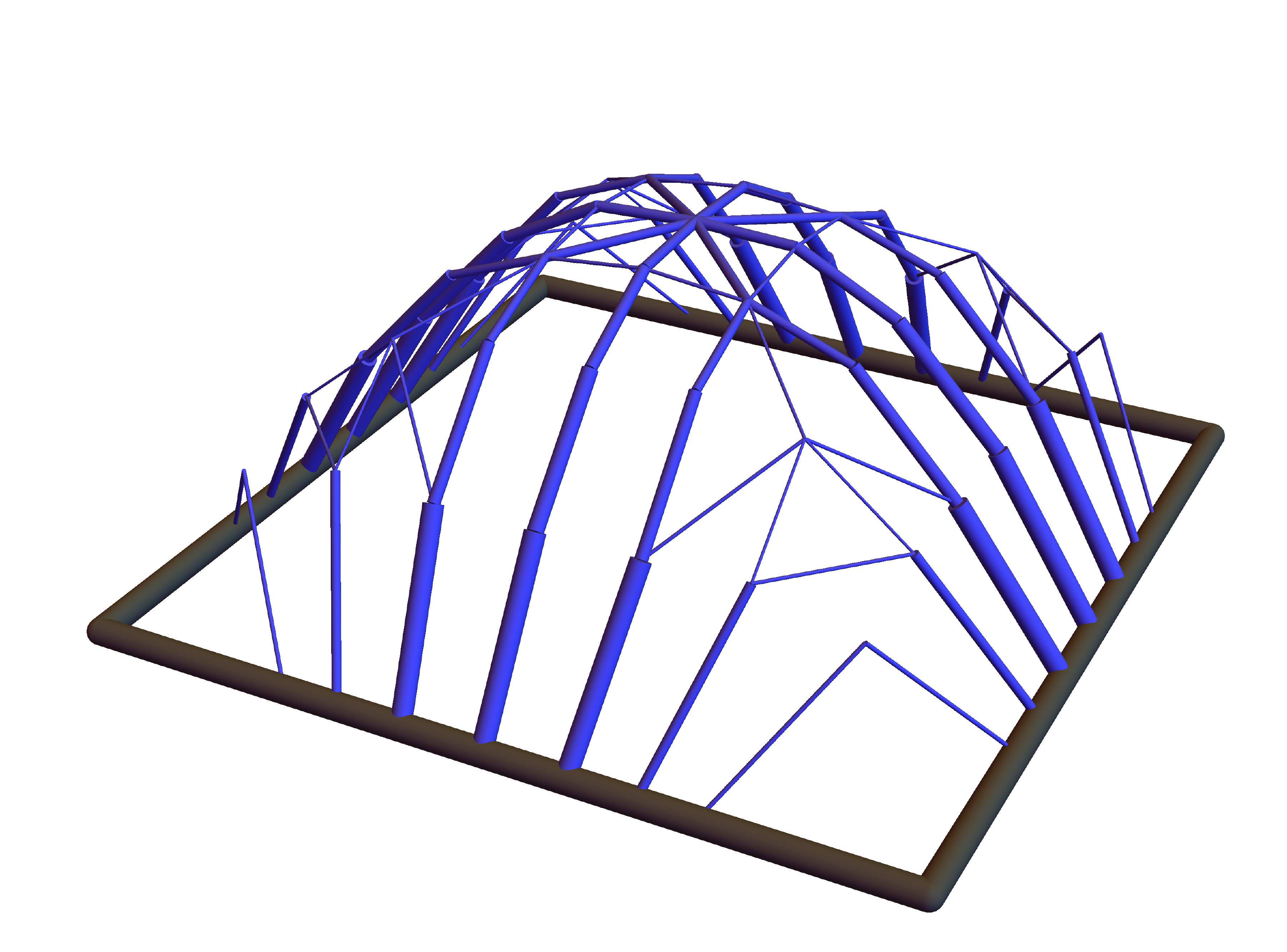}}
		\caption{The problem of optimal grid-shell in compression over a square domain and under multiple point loads: (a) plane truss given by $\mbf{s}$ solving $(\mathcal{P}_X)$; (b) optimal elevation function $z = -\frac{1}{2} \,w$  where $w$ solves $(\mathcal{P}^*_X)$; (c) optimal grid-shell in compression.}
		\label{fig:multiple_point_loads}
\end{figure}
Based on the solution $\mbf{s}$ of problem $(\mathcal{P}_X)$, visualized as a planar truss in Fig. \ref{fig:multiple_point_loads}(a) for $h = a/160$, we can readily confirm that the bars connect only the points from the set $\mathrm{spt}\,f \cup (X \cap \bO)$. Not a single node in $X$, that either does not lie on the boundary or is not carrying a load $P$, was exploited and, although a whole universe of potential members is available, the optimal structures consists of only few bars. At the same time the employed boundary points are not obviously positioned which proves that a fine discretization of the boundary is essential after all. The optimal elastic grid-shell in compression is visualized in Fig. \ref{fig:multiple_point_loads}(c).
\end{example}

\begin{example}\textbf{(Knife load distributed along diagonals of a square domain)}
\label{ex:diagonal_load}
While keeping the square domain $\Omega$ we vary the load: this time the measure $f =- t\,\Ha^1\mres[A_1,A_3] - t\,\Ha^1\mres[A_2,A_4]$ represents two knife loads of intensity $t>0$ (of units $\mathrm{N}/\mathrm{m}$) applied downwards and along diagonals of the square $\Omega$, see Fig. \ref{fig:diagonal_load}(a). For a $201 \times 201$ nodal grid the load $f$ must be discretized to $f_X \in \Mes(X;\R)$: for $h=a/200$ to every node lying on the diagonals we apply a downward point force of magnitude $t \, \sqrt{2} h$ and the vector $\mbf{f}$ is defined accordingly (except the very centre of the square where a twice bigger force is applied). The computational details, also for other resolutions of $X$, are listed in Table \ref{tab:miscellaneous}.

\begin{figure}[h]
		\centering
		\subfloat[]{\includegraphics*[trim={0cm -1cm -0cm -0cm},clip,height=0.24\textwidth]{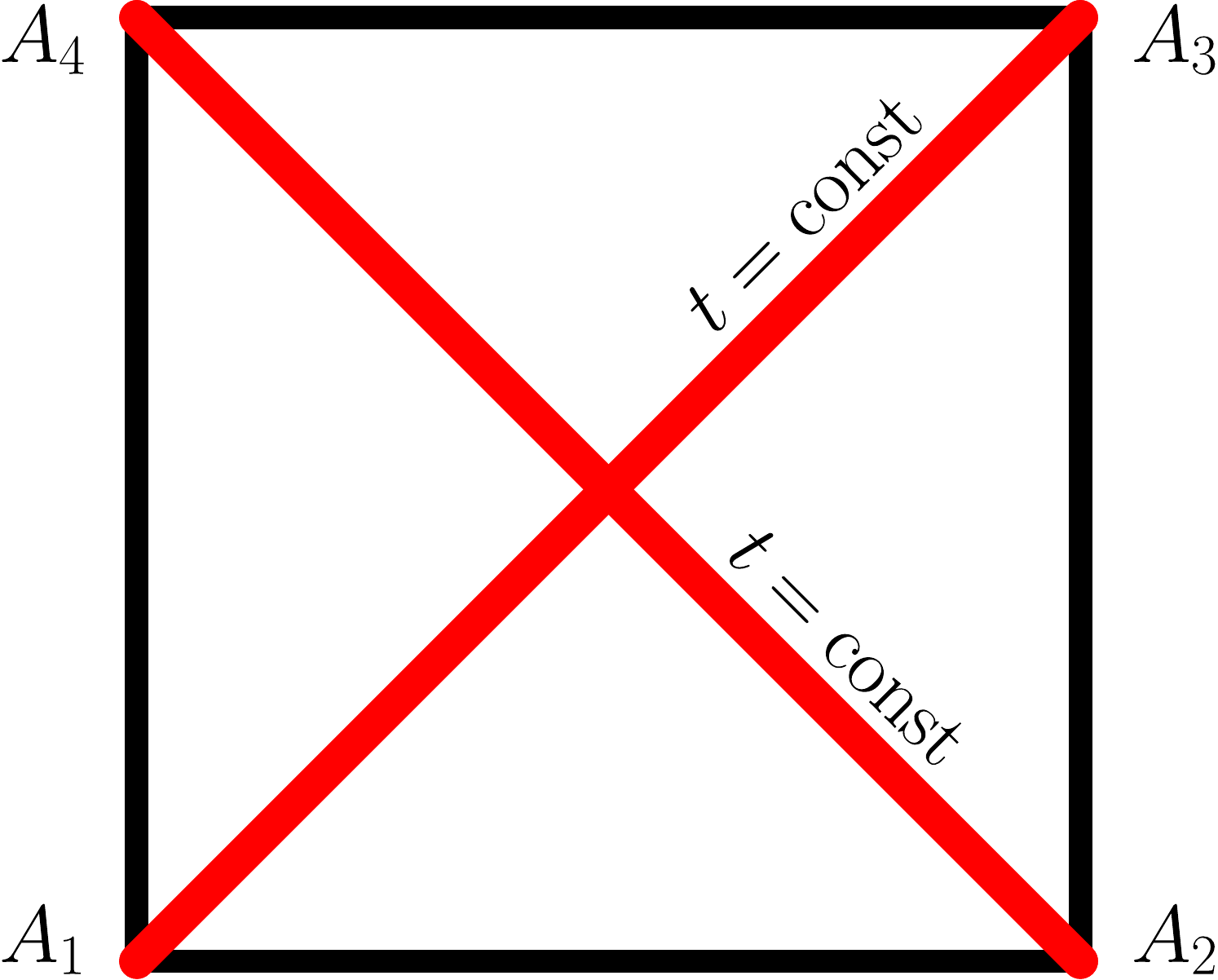}}\hspace{0.6cm}
		\subfloat[]{\includegraphics*[trim={0cm -1cm -0cm -0cm},clip,height=0.24\textwidth]{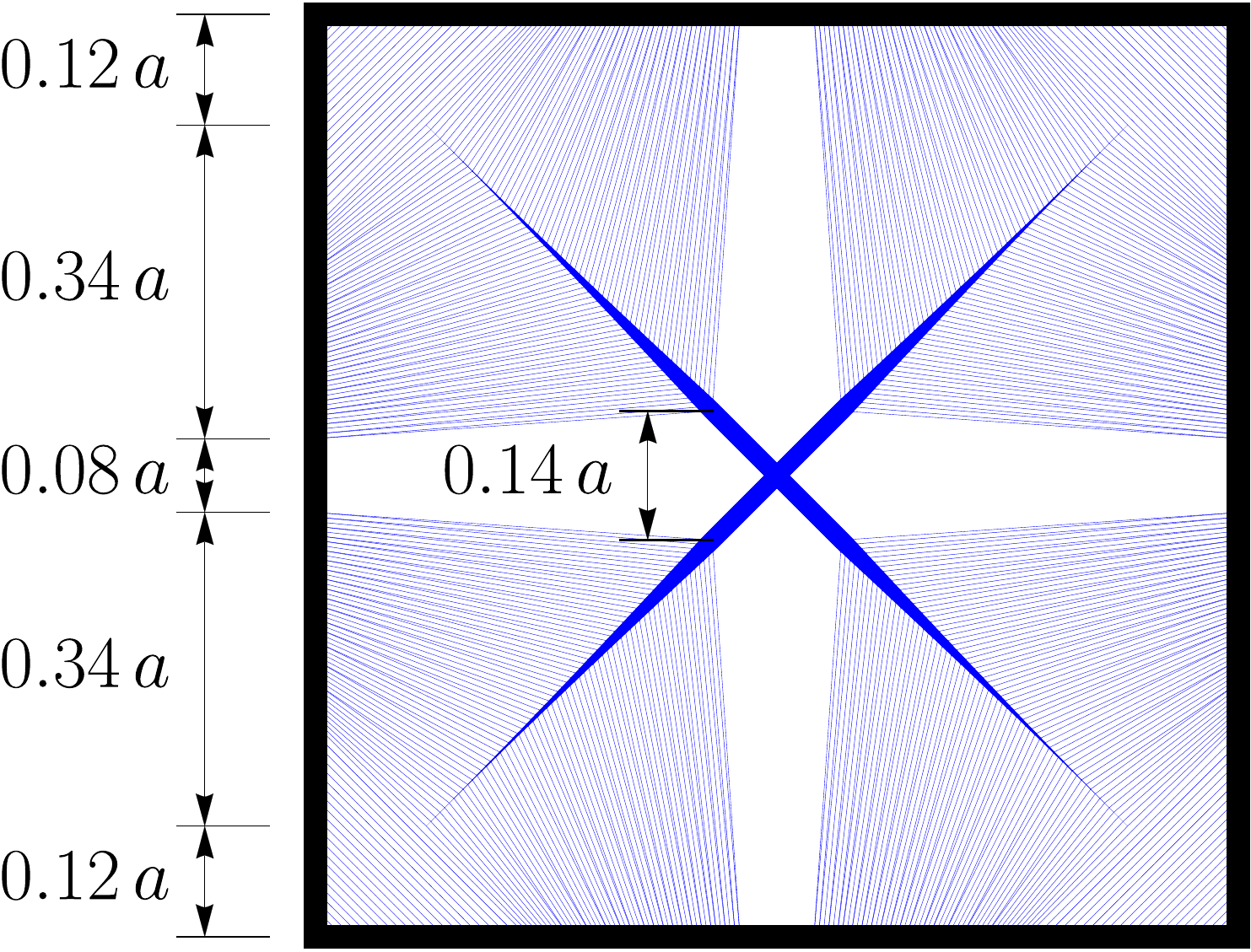}}\hspace{0.4cm}
		\subfloat[]{\includegraphics*[trim={0.5cm 0cm 0.5cm 0cm},clip,width=0.35\textwidth]{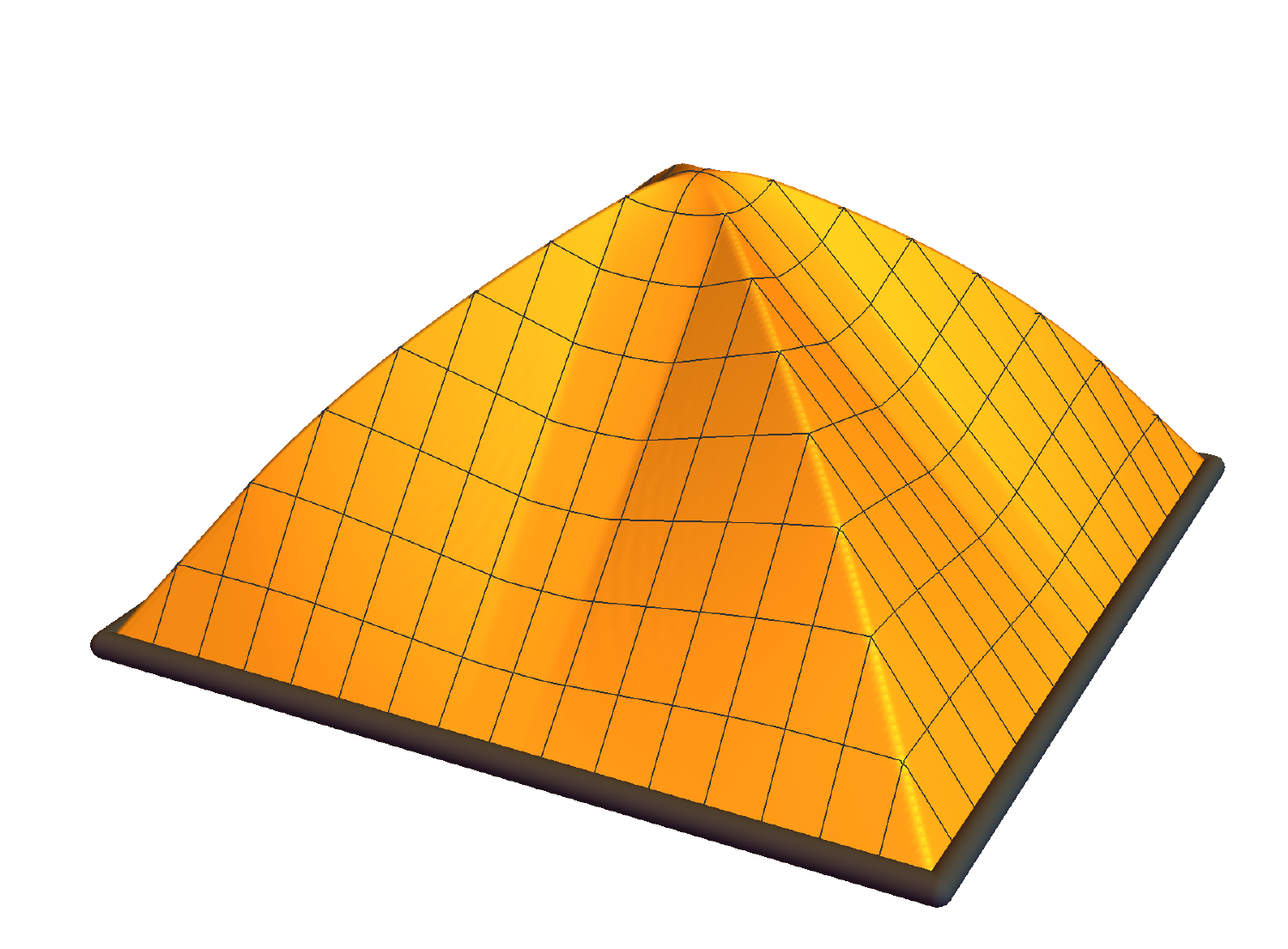}}\\
		\subfloat[]{\includegraphics*[trim={2cm -0cm 0.5cm 3cm},clip,width=0.7\textwidth]{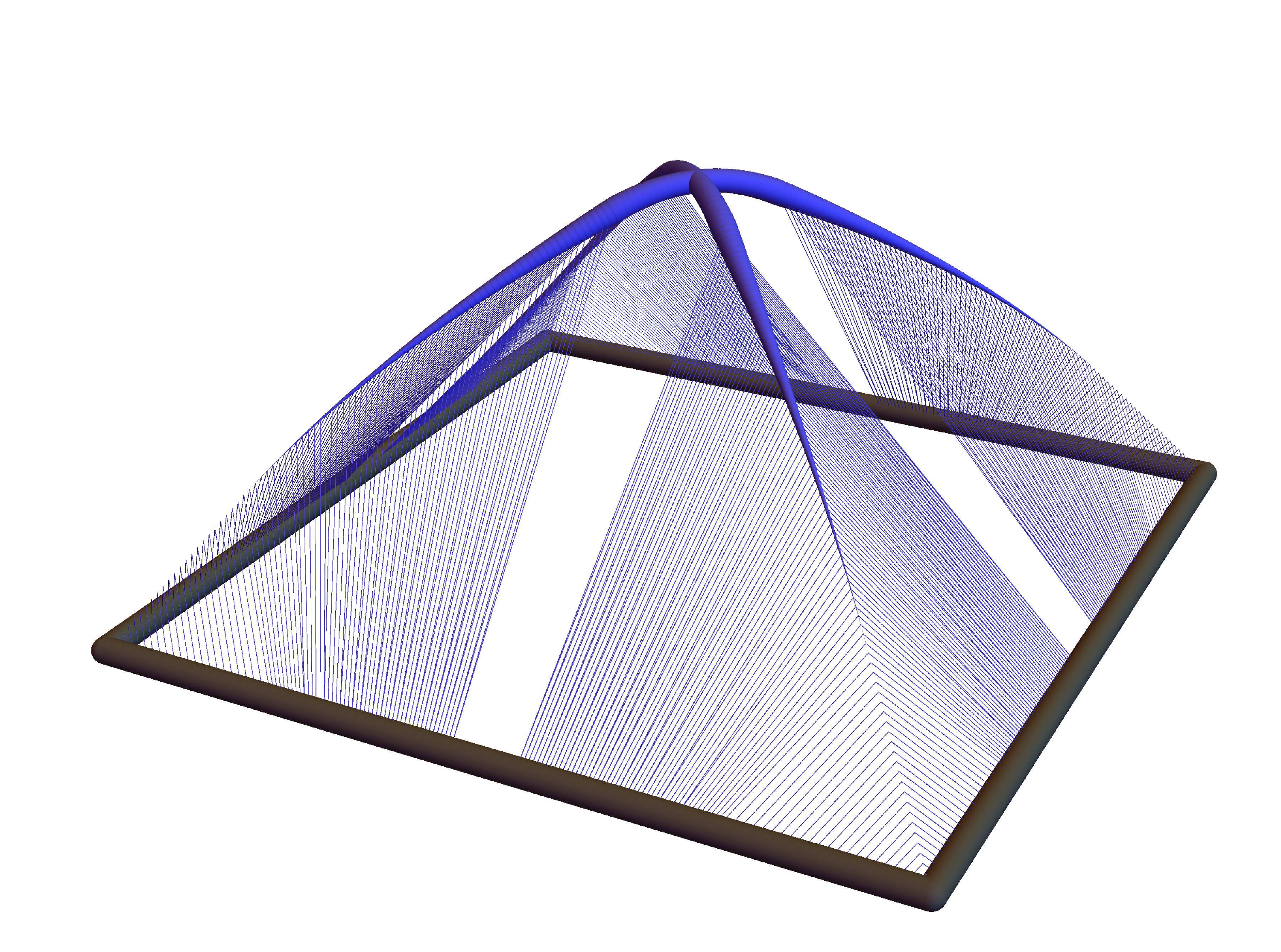}}
		\caption{The problem of optimal grid-shell in compression over a square domain and under a knife load distributed along the diagonals: (a) the loading conditions; (b) plane truss given by $\mbf{s}$ solving $(\mathcal{P}_X)$; (c) optimal elevation function $z = -\frac{1}{2} \,w$  where $w$ solves $(\mathcal{P}^*_X)$; (d) optimal grid-shell in compression.}
		\label{fig:diagonal_load}
\end{figure}

Solution $\mbf{s}$ of problem $(\mathcal{P}_X)$ is visualized as a pre-stressed truss in Fig. \ref{fig:diagonal_load}(b). It is clear that $\mbf{s}$ approximates a solution $\sigma$ of problem $(\mathcal{P})$ that has a feature which is very typical for Michell structures: $\sigma$ consists of one-dimensional part, being two bars of varying cross section, and of the continuous part, being eight fans of fibres connecting the diagonals to the boundary. The two bars do not reach the square's corners, instead they start at $0.12\, a$ distance from the boundary (numerical estimation) with zero cross sectional area. In vicinity of the corners the continuous part is a strip of parallel bars inclined at $\pm 45\degree$ angle that are not connected to the two diagonal bars. The continuous part does not fill the whole design domain -- a void subregion emerges in a shape of a cross with the arms of width linearly varying from $0.08\, a$ to $0.14\,a$ (numerical estimation); the diagonal bars are of constant cross section when crossing the void area. Unlike in many Michell structures, however, in the solution $\sigma$ predicted in Fig. \ref{fig:diagonal_load}(b) all the bars are straight, there are no curved bars or the so called \textit{Hencky nets}. Both of those structural elements are closely related to cooperation between tensile and compressive stresses which cannot occur for positive $\sigma$. Additionally, we observe that as in the previous two examples the bars constituting the field $\sigma$ connect only the points in the set $\mathrm{spt}\,f \cup \bO$.

The optimal tent-resembling elevation function, being an interpolation of data $\mbf{z} = - \frac{1}{2} \, \mbf{w}$, is showed in Fig. \ref{fig:diagonal_load}(c). The optimal elastic grid-shell in compression may be readily constructed by unprojecting onto the elevated surface (elevated ground structure to be more accurate) and by employing formula $\mbf{a} = V_0/\Z_X \, \mbf{J}(\mbf{z})\, \mbf{s}$ to compute bar's cross sectional areas, see Fig. \ref{fig:diagonal_load}(d). As there are no loads present over the subregion of $\Omega$ where continuous part of solution occurs, we deduce that the thin fibres are straight and therefore there are eight pieces of the graph $z$ that are ruled surfaces. 

One could say that the obtained grid-shell approximates a vault discussed in Sections \ref{sec:plastic_design} and \ref{sec:elastic_design}. However, there are one-dimensional structural elements present, i.e. the two arches running over parts of diagonals of the square $\Omega$, hence the truly optimal structure is predicted to be given by elastic material distribution being a 3D measure $\hat{\mu} \in \Mes_+(\Ob\times\R)$ that solves problem $(\mathrm{MCPS})$ posed in Section \ref{sec:3D}. A natural competitor for optimal solution in this loading scenario would be the pair of arches of finite cross sectional area connecting two opposite corners of the square. Such solution may be provoked by restricting the set where the structure is fixed from $\bO$ to only four vertices  $A_1,A_2,A_3,A_4$  (cf. Section \ref{ssec:support_other_than_bO}) in which case we would obtain $\widetilde{\Z}_X = 2.3097 \,t a^2$. Thus, in the context of the plastic design, a two-arch structure would require almost $39\%$ more material than the grid-shell in Fig. \ref{fig:diagonal_load}(d).
\end{example}

\begin{figure}[p!]
	\centering
	\hbox{{\vbox{\offinterlineskip\halign{#\hskip3pt&#\cr
					\centering
					\hspace{1.85cm}\subfloat[]{\includegraphics*[trim={-0cm -1cm -0cm -0cm},height=3.4cm]{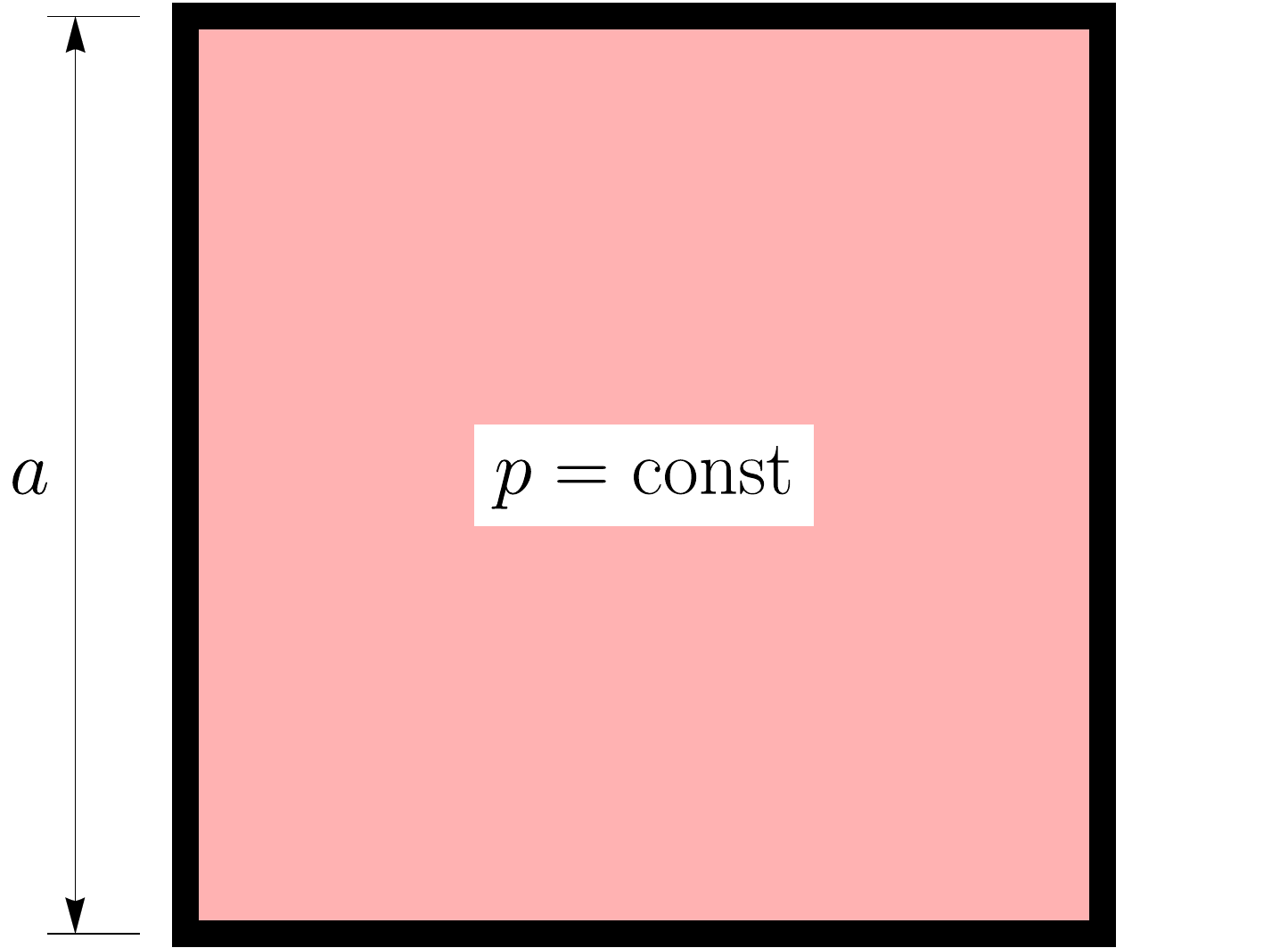}}\cr
					\noalign{\vskip3pt}
					\subfloat[]{\includegraphics*[trim={0cm -0.5cm -0cm 1.7cm},width=7cm]{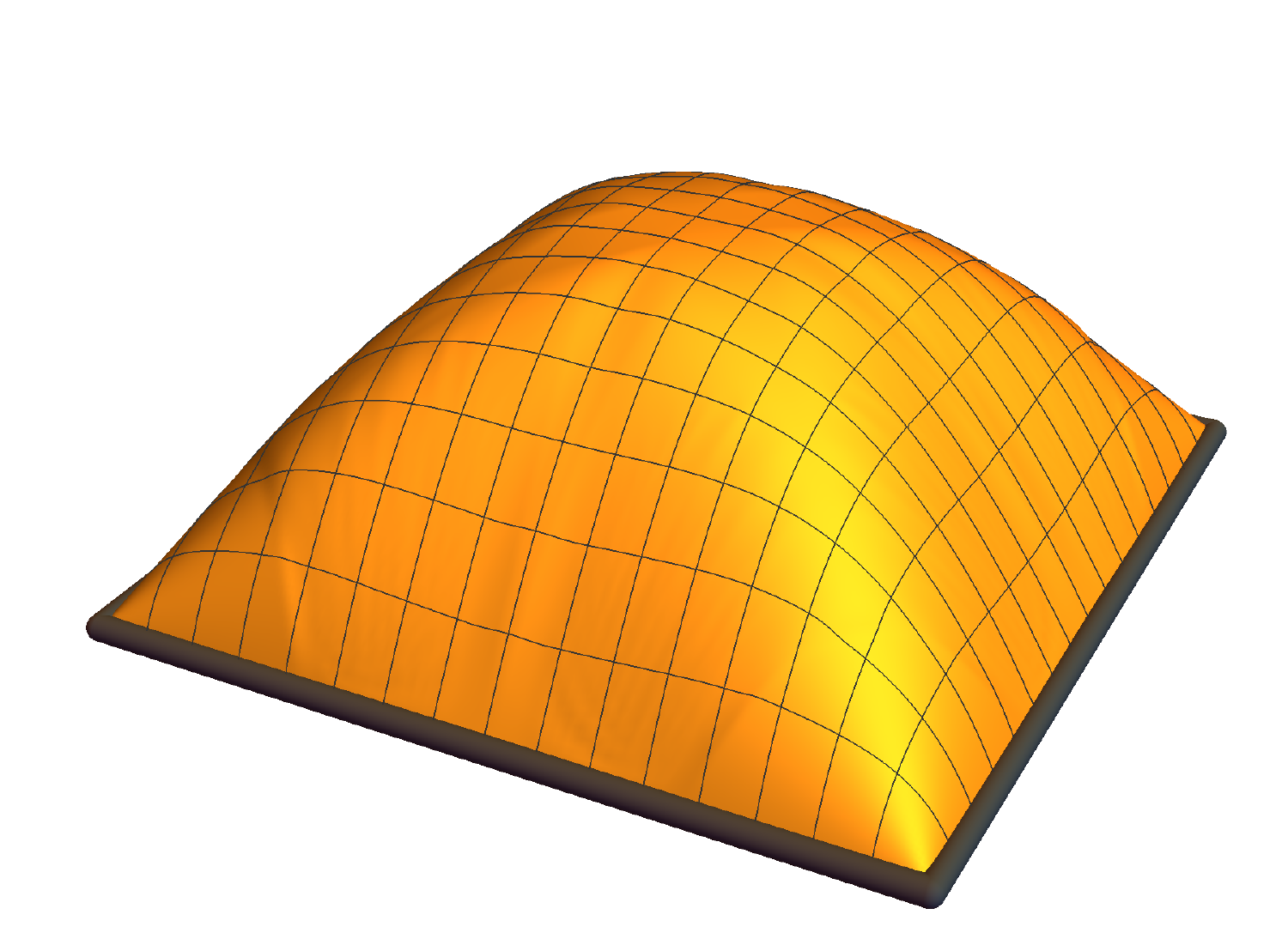}}\cr
		}}}
		\hspace{0.5cm}
		\subfloat[]{\includegraphics*[trim={0cm -0.7cm -0cm -0cm},clip,width=0.52\textwidth]{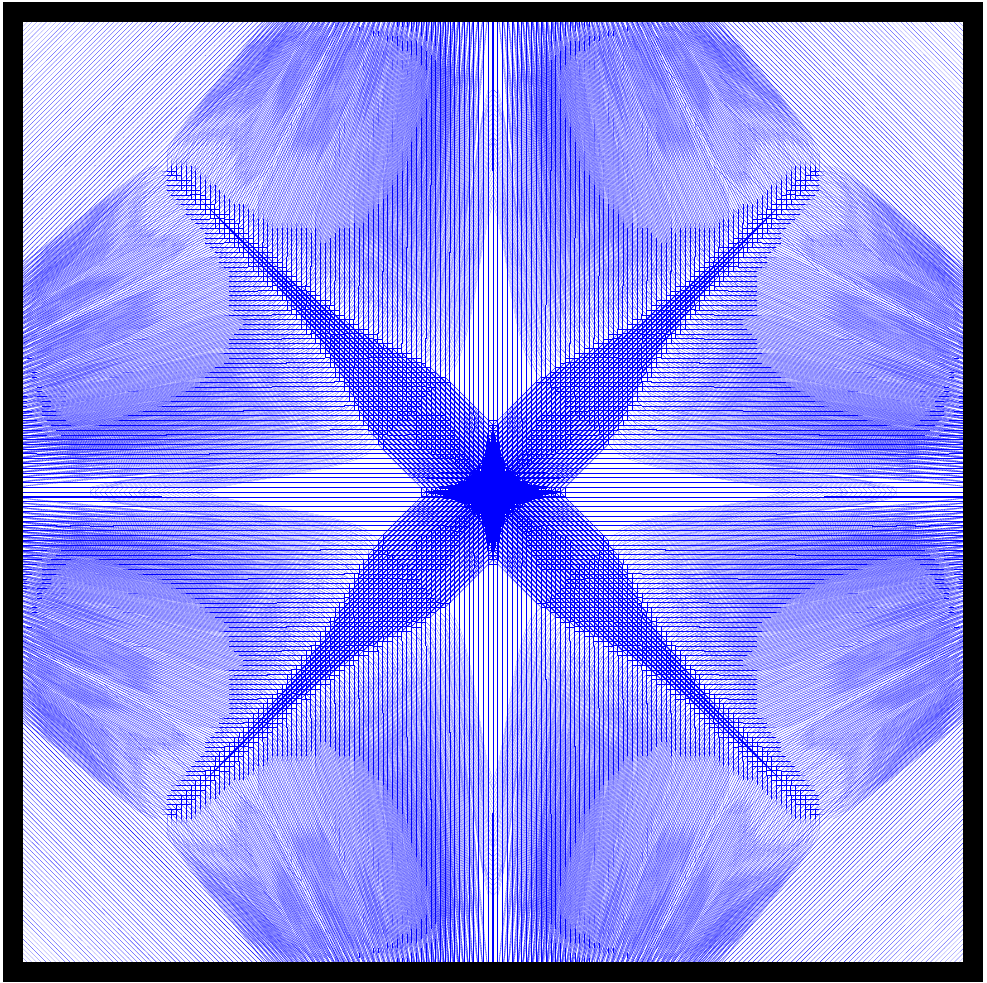}}}
	\subfloat[]{\includegraphics*[trim={2cm -0cm 0.5cm 3cm},clip,width=0.85\textwidth]{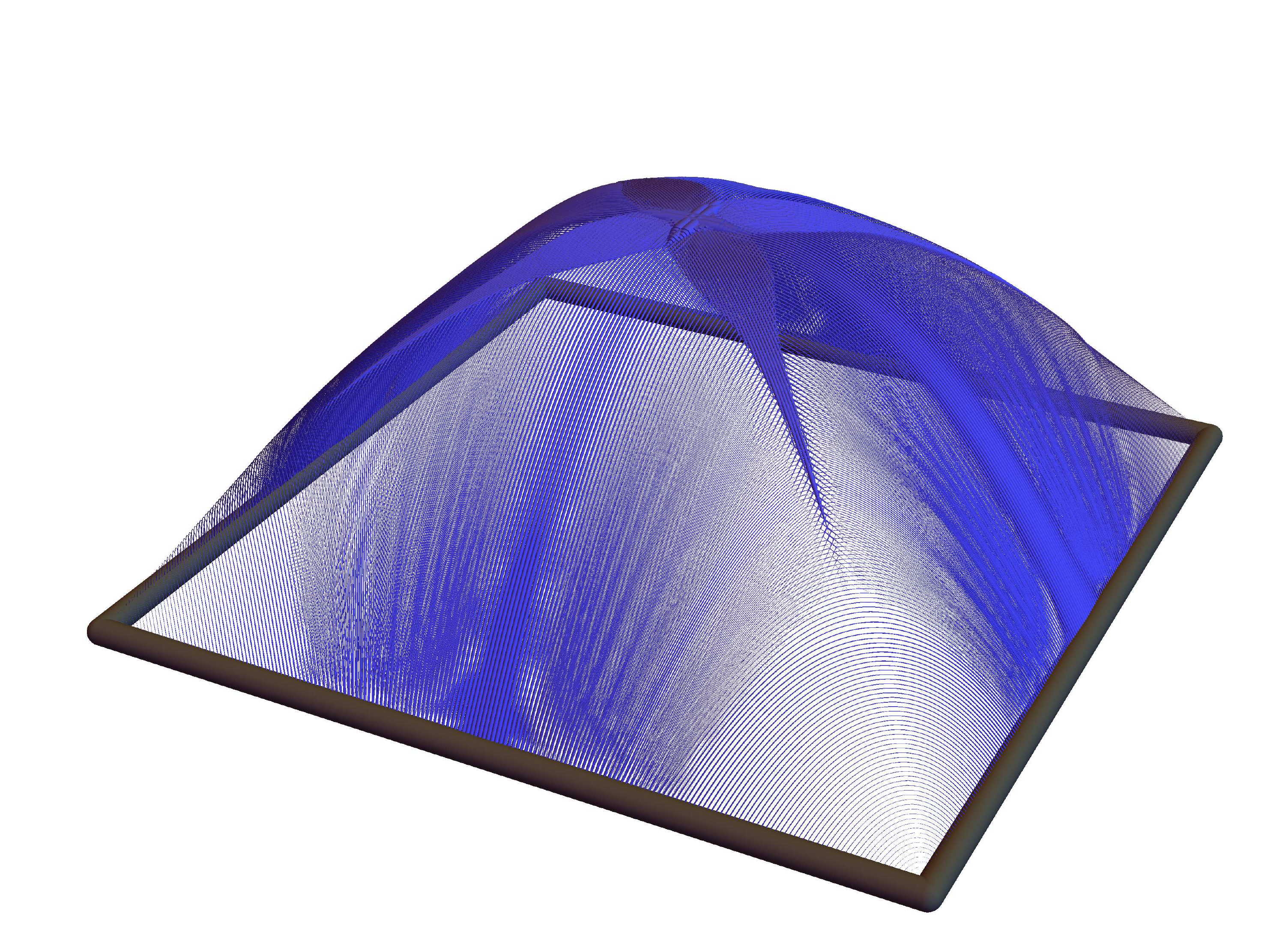}}
	\caption{The problem of optimal grid-shell in compression over a square domain and under uniformly distributed load: (a) design domain and load; (b) optimal elevation function $z = -\frac{1}{2} \,w$  where $w$ solves $(\mathcal{P}^*_X)$; (c) plane truss given by $\mbf{s}$ solving $(\mathcal{P}_X)$; (d) optimal grid-shell in compression.}
	\label{fig:pressure}
\end{figure}

\begin{example}\textbf{(Uniformly distributed load in a square domain)}
\label{ex:pressure}
We carry on with investigating the optimal grid-shell problem for square domain; the downward load is now uniformly distributed in $\Omega$, i.e. $f = -p\, \mathcal{L}^2 \mres \Omega$ where $p$ is a positive constant of units being $\mathrm{N}/\mathrm{m}^2$, see Fig. \ref{fig:pressure}(a). For a $201 \times 201$ grid we propose the discretized load $f_X \in \Mes(X;\R)$ by applying at each node a downward load of magnitude $p \, h^2$ where $h = a/200$. The computational details, also for other resolutions of $X$, are listed in Table \ref{tab:miscellaneous}.

The pre-stressed truss, that is furnished by vector $\mbf{s}$ solving problem $(\mathcal{P}_X)$, is showed in Fig. \ref{fig:pressure}(c); unprojecting $\mbf{s}$ onto the elavation function $z$ from Fig. \ref{fig:pressure}(b) yields an optimal grid-shell visualized in Fig. \ref{fig:pressure}(d). Unlike in the previous example, the structure of the exact field $\sigma\in \Mes(\Ob;\Sddp)$ solving the infinite dimensional problem $(\mathcal{P})$ is difficult to guess based on the numerical solution $\mbf{s}$. There are certain similarities between the two solutions, for instance the strip of parallel fibres near the corners. It seems that in the large part of the design domain $\sigma$ is rank-one, i.e. of the form $s \, \tau\otimes\tau$. The area adjacent to the two diagonals is an exception: a reinforcement of the structure may be observed, yet not in the form of  bars but two-dimensional caps made of significantly thicker bars running in different directions (see Fig. \ref{fig:pressure}(d) for better grasp of bar's thickness) -- therein $\sigma$ appears to be of full rank. One dimensional ribs seem to be missing; very thick bars may be found near the centre of the square but the guess of the present author is that they simulate an $L^1$ function $\sigma$ that blows up to infinity, cf. the analytical Example \ref{ex:axisymmetric}.

\end{example}

\begin{example}\textbf{(Uniformly distributed load in a cross-shaped domain)}
\label{ex:cross}
We choose a non-convex domain $\Omega$ in a shape of symmetric cross; the load is again uniformly distributed, namely $f = - p \, \mathcal{L}^2 \mres \Omega$, cf. Fig. \ref{fig:cross}(a). The nodal grid is chosen in accordance with \eqref{eq:X_regular}, for $h = a/240$ while the discretized load $f_X$ is constructed identically as in Example \ref{ex:pressure}. The computational data, also for other resolutions of $X$, is listed in Table \ref{tab:miscellaneous} (since $\Omega$ is non-convex the precise number of members in the full ground structure would require heavy computing, instead we give a rough estimate).

\begin{figure}[h!]
	\centering
	\hbox{{\vbox{\offinterlineskip\halign{#\hskip3pt&#\cr
					\centering
					\hspace{2.cm}\subfloat[]{\includegraphics*[trim={-0cm -1cm -0cm -2cm},height=4cm]{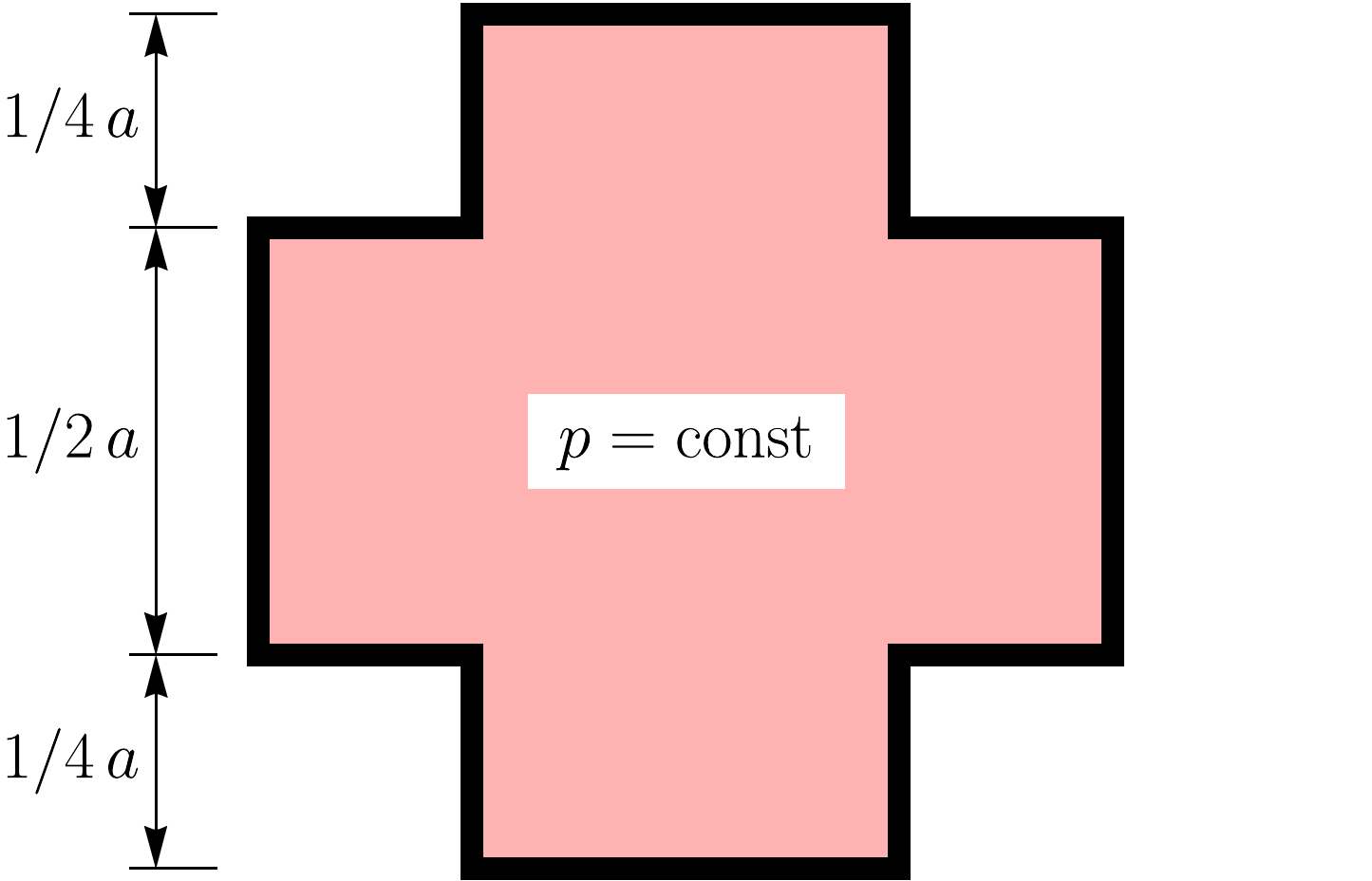}}\cr
					\noalign{\vskip3pt}
					\subfloat[]{\includegraphics*[trim={1.2cm 0.5cm 1.2cm 2cm},width=0.45\textwidth]{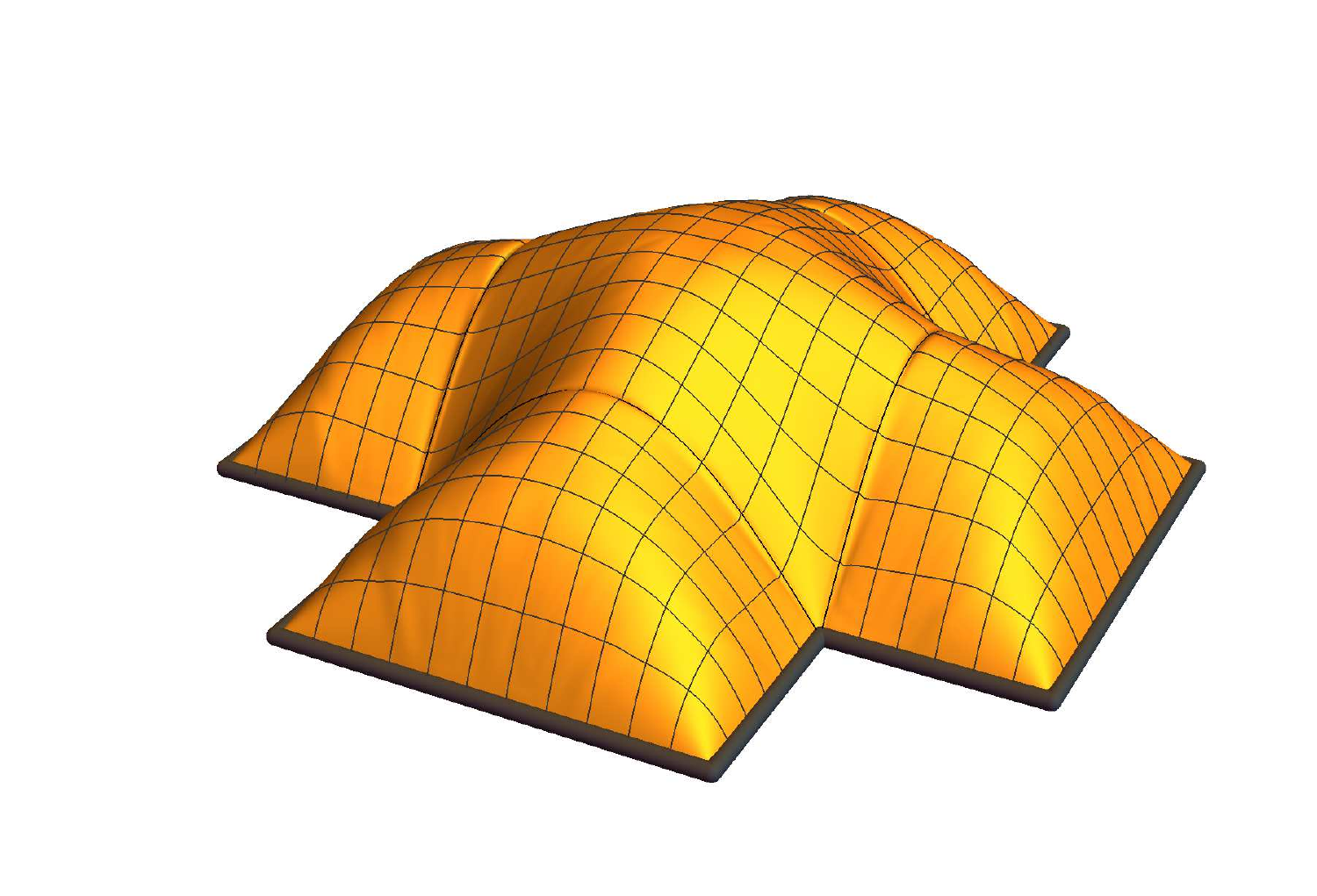}}\cr
		}}}
		\hspace{0.4cm}
		\subfloat[]{\includegraphics*[trim={0cm -1cm -0cm -0cm},clip,width=0.52\textwidth]{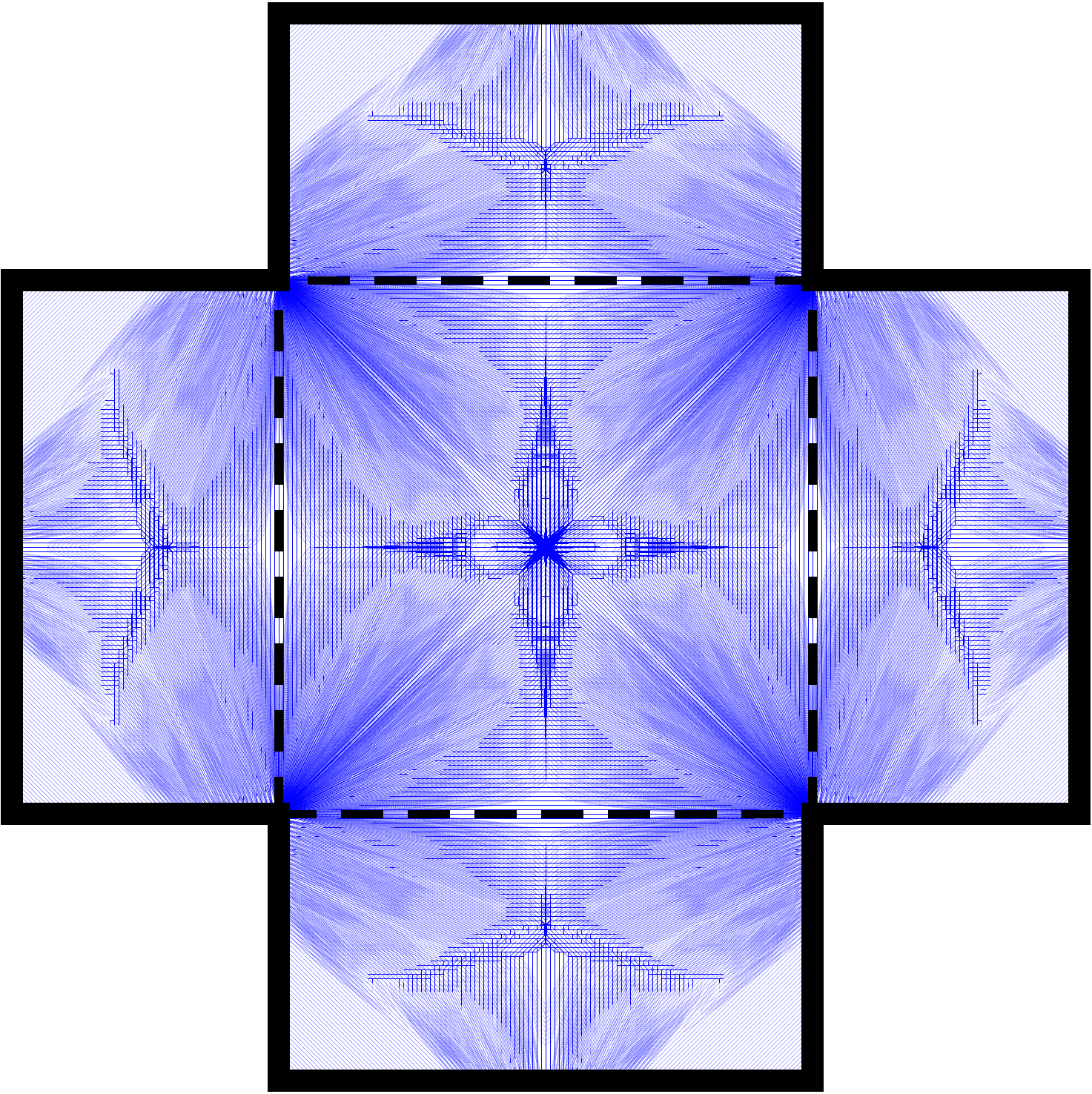}}}
	\subfloat[]{\includegraphics*[trim={5cm -0cm 2cm 4cm},clip,width=0.45\textwidth]{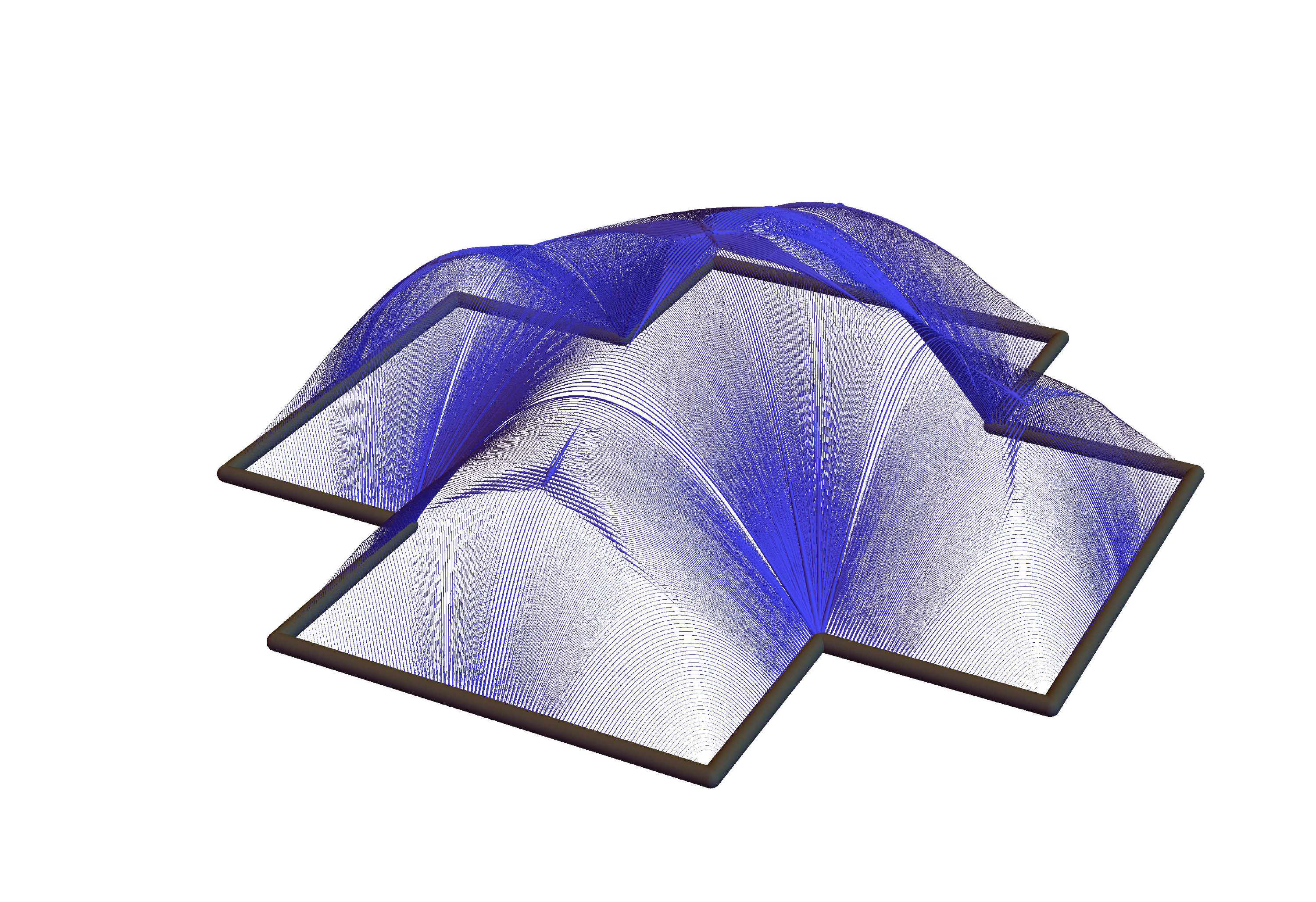}}\hspace{0.4cm}
	\subfloat[]{\includegraphics*[trim={5cm -0cm 2cm 4cm},clip,width=0.45\textwidth]{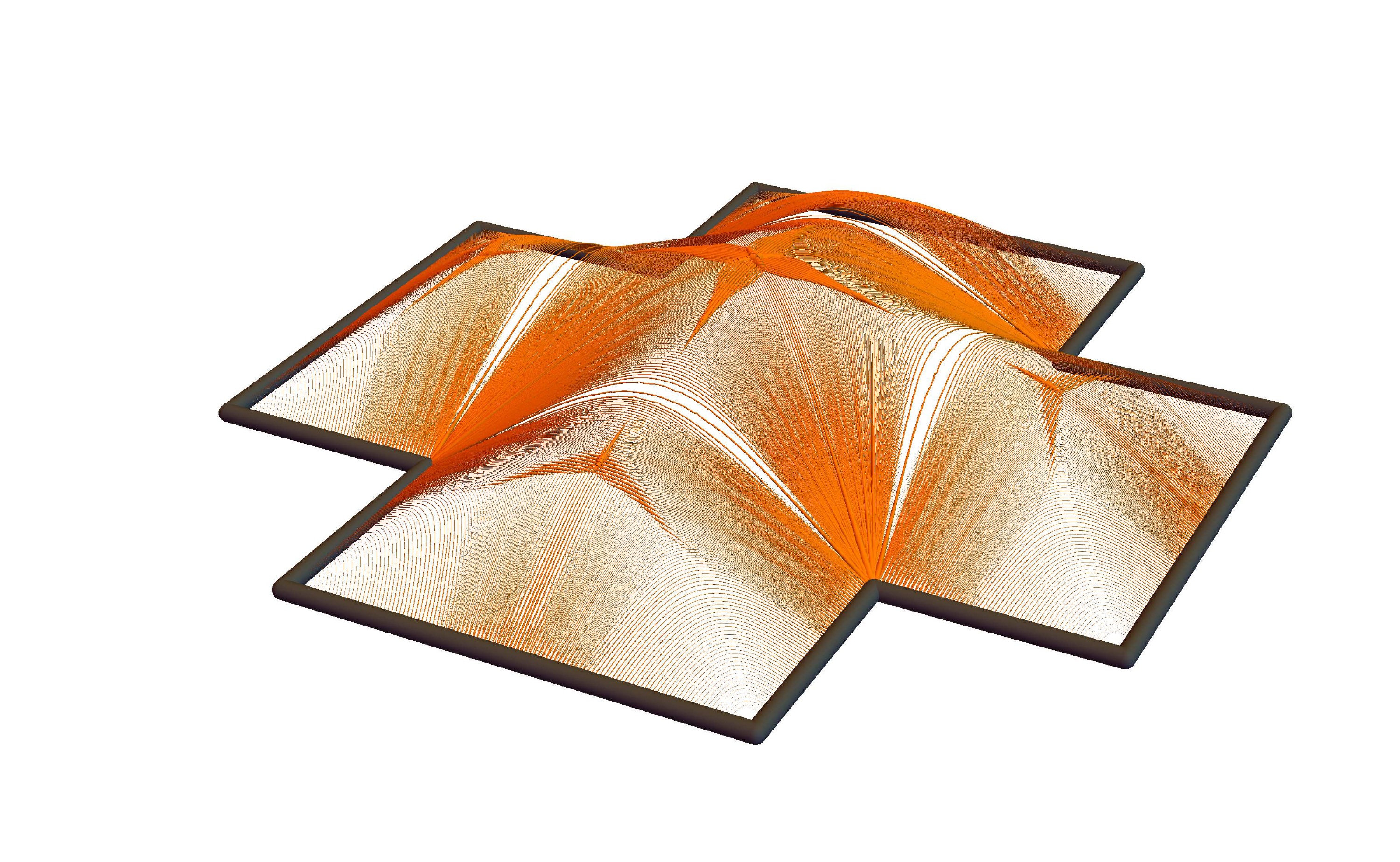}}
	\caption{The problem of optimal grid-shell in compression over a cross-shaped domain and under uniformly distributed load: (a) design domain and load; (b) optimal elevation function $z = -\frac{1}{2} \,w$  where $w$ solves $(\mathcal{P}^*_X)$; (c) plane truss given by $\mbf{s}$ solving $(\mathcal{P}_X)$; (d) optimal grid-shell in compression; (e) elastic deformation.}
	\label{fig:cross}
\end{figure}
\end{example}

The numerical solution $\mbf{s}$ of problem $(\mathcal{P}_X)$ is presented in Fig. \ref{fig:cross}(c). We observe that the truss in the central square connecting the four re-entrant corners is disconnected from the rest of the remaining four parts, namely there are no bars that interconnect the five regions, cf. the dashed lines partitioning the domain. Nevertheless, the elevation function $z$, cf. the interpolation in Fig. \ref{fig:cross}(b), is continuous as was guaranteed by Proposition \ref{prop:regularitu_u_w}. The optimal elastic grid-shell approximating a vault in compression is showed in Fig. \ref{fig:cross}(d), while in Fig. \ref{fig:cross}(e) we see its elastic deformation that admits a singularity: the displacement is clearly discontinuous along the four lines partitioning $\Omega$. More precisely it is only the component $u_{\e,\nu}$ of horizontal vector displacement $u_{\e}$ which is normal to those lines that is discontinuous. This scenario is again in agreement with Proposition \ref{prop:regularitu_u_w} where function $u$ is in general in $BV$ space while the singular part $(e(u))_s$ is negative -- since $u_\e = - \Z_X/(E_0 V_0)\, u$ (cf. \eqref{eq:compression_mod}) the singular part $(e(u_\e))_s$ is positive and therefore the grid-shell may experience a sort of cracking visible in Fig. \ref{fig:cross}(e).

\section{Variations of the form finding problem}
\label{sec:variations}

\subsection{Optimal vaults and grid-shells fixed on a plane set $\Gamma$ other than $\bO$}
\label{ssec:support_other_than_bO}

From the very beginning of this contribution we have kept the assumption that the designed vault / Prager structure / grid-shell is pinned on the boundary $\bO \times \{0\}$. It is natural to generalize the investigated problem so that the potential kinematical support of the structure is $\bO_0 \times \{0\}$ where $\Omega_0$ is a closed subset of $\bO$. Taking a step further would be to account for line and point supports in the interior of $\Omega$. Thus, in general, one could consider a problem where the structure being designed may rest on $\Gamma \times \{0\}$ where $\Gamma$ is any closed subset contained in $\Ob$. In this broader setting the infinite dimensional problem $(\mathcal{P})$ would be affected as follows: the two equilibrium equations $\DIV \, \sigma = 0 $ and $-\dive \,q = f$ would be imposed in the sense of distributions on the open set $\Rd \backslash \Gamma$ instead of $\Omega$. As a result the set $\Gamma$ has to satisfy the condition:
\begin{equation}
\label{eq:condition_on_Gamma}
\Ob \subset \mathrm{conv}(\Gamma).
\end{equation}
Indeed, according to Theorem 2.2 in \cite{bouchitte2019} a positive matrix valued measure $\sigma \in \Mes(\Rd;\Sddp)$ whose divergence $\DIV \,\sigma$ is supported on a closed set $\Gamma$ must itself be supported in the convex hull of $\Gamma$, i.e. $\mathrm{spt}\,\sigma \subset \mathrm{conv}(\Gamma)$. Once $\Omega$ is convex condition \eqref{eq:condition_on_Gamma} readily implies existence of solution $(\sigma,q) \in \Mes(\Ob;\Sddp) \times \Mes(\Ob;\Rd)$ of the altered problem $(\mathcal{P})$. If, on the other hand, $\Omega$ is non-convex the conditions \eqref{eq:condition_on_Gamma} is in general not sufficient and in fact the criteria for existence in $(\mathcal{P})$ are difficult to establish in this broader scenario.  Analysis of the dual problem $(\mathcal{P}^*)$ becomes problematic even when $\Omega$ is convex, in particular the existence and  regularity results in Proposition \ref{prop:regularitu_u_w} fail to hold. In general, rigorous mathematical results on the pair $(\mathcal{P}), (\mathcal{P}^*)$, studied thoroughly in \cite{bouchitte2020} in the case when $\Gamma = \bO$, at this point cannot be easily generalized to the case of any $\Gamma$ satisfying \eqref{eq:condition_on_Gamma} -- before the results in Sections \ref{sec:plastic_design},\,\ref{sec:elastic_design},\,\ref{sec:3D} are extended to general $\Gamma$, Theorems \ref{thm:duality} and \ref{thm:opt_cond} must be carefully revised first.

In case when $\bO \not\subset \Gamma$ an extra analysis is essential also in case of the discrete formulation, i.e. assumptions and assertions of Theorem \ref{thm:duality_discrete} need to be revised. We will skip this matter entirely although, in the example below we shall give the numerical solution of the optimal grid-shell problem for $\Gamma$ being an eight-point set  -- the solution will exist yet it will suffer from certain pathologies in the horizontal deformation vectors $\mbf{u}_1, \mbf{u}_2$.

\begin{example}\textbf{(A vault supported on eight columns)}
	\label{ex:8_columns}
	Once again we consider a square design domain together with a uniformly distributed load $f = - p \, \mathcal{L}^2 \mres \Omega$, only this time the structure is pinned only at eight points, i.e. $\Gamma = \bigl\{ A_1, \ldots, A_8 \bigr\}$, see Fig. \ref{fig:8_columns}(a). Details on computations are summed up in Table \ref{tab:8_columns}.
	
	\begin{table}[H]
		\scriptsize
		\centering
		\caption{Summary on numerical computations for the pair of problems $(\mathcal{P}_X), (\mathcal{P}_X^*)$ specified in Example \ref{ex:8_columns}.}
		\begin{tabular}{lccccccc}
			\toprule
			Example no. & Grid $X$  & Full GS  & Iterations   & Active GS    & CPU time & Objective value $\Z_X$ &  Max. elevation $\norm{z}_\infty$\\
			\midrule
			\ref{ex:8_columns}
			& $51\!\times\!51$ & $3\,381\,300$ & 8 & $16\,341$ & $25$ sec. & $0.55200 \, p a^3$ & $0.429 \,a $\\
			& $101\!\times\!101$ & $52\,025\,100$ & 9 & $78\,791$ & $3$ min. $20$ sec. & $0.54551 \, p a^3$ & $0.429 \,a $\\
			& $201\!\times\!201$ & $816\,100\,200$ & 10 & $383\,102$ &  $22$ min. $24$ sec. & $0.54228 \, p a^3$ & $0.429 \,a $\\
			\bottomrule
		\end{tabular}
		\label{tab:8_columns}
	\end{table}

	\begin{figure}[!h]
		\centering
		\hbox{{\vbox{\offinterlineskip\halign{#\hskip3pt&#\cr
						\centering
						\hspace{1.4cm}\subfloat[]{\includegraphics*[trim={-0cm -1cm -0cm -2cm},width=4.2cm]{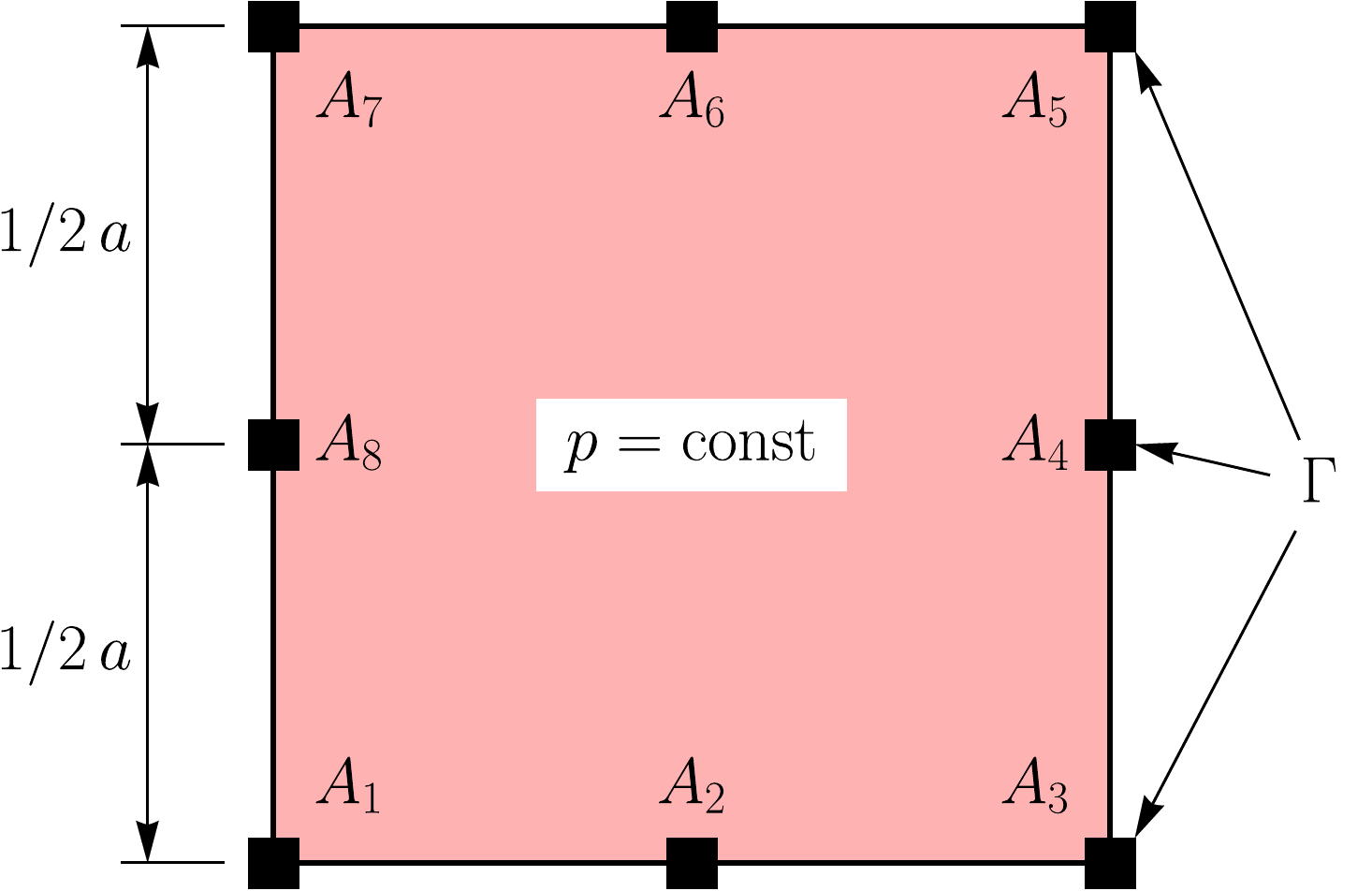}}\cr
						\noalign{\vskip3pt}
						\subfloat[]{\includegraphics*[trim={0cm -0.5cm 0cm 1.5cm},width=7cm]{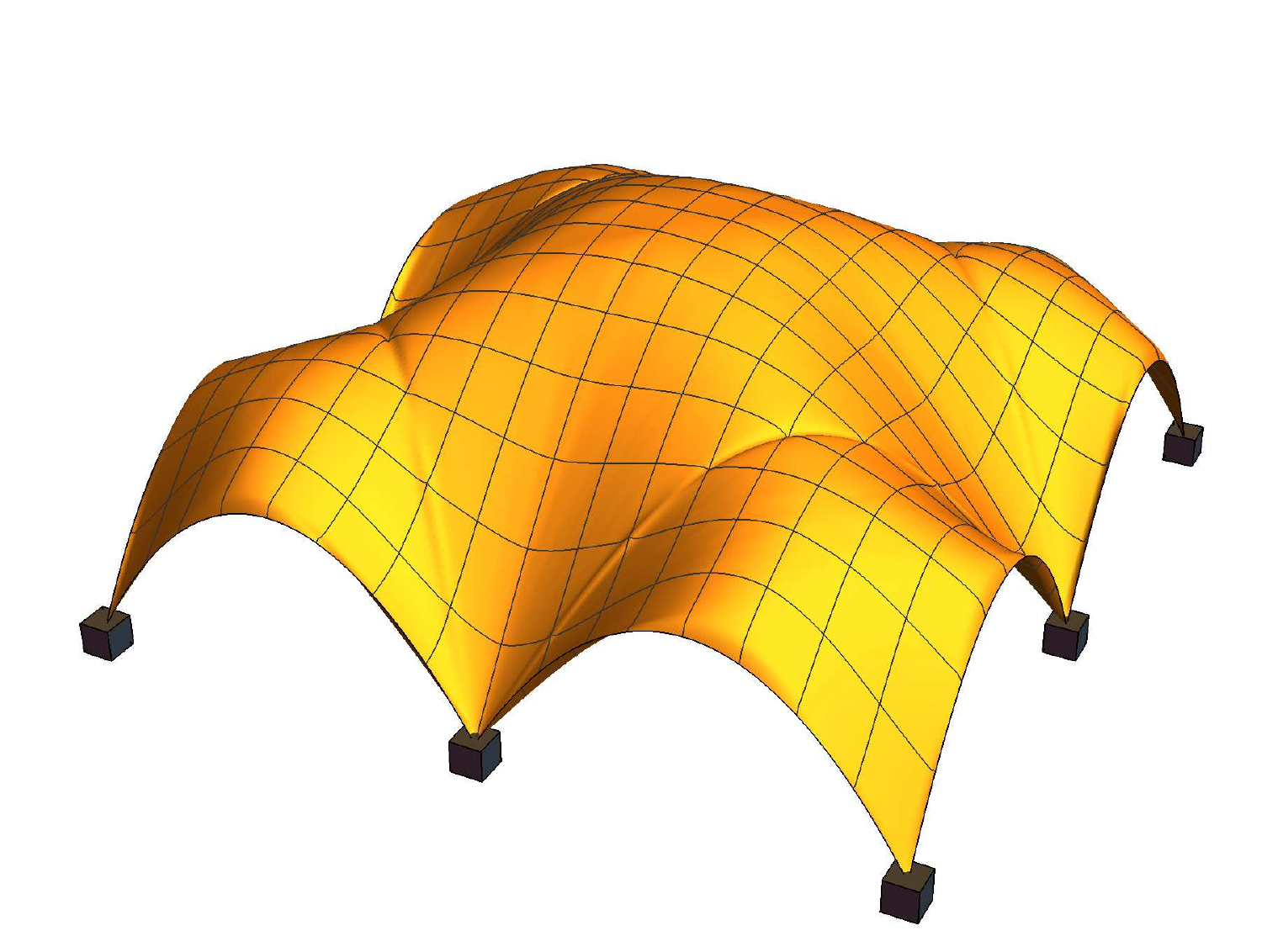}}\cr
			}}}
			\hspace{0.5cm}
			\subfloat[]{\includegraphics*[trim={0cm -1cm -0cm -0cm},clip,width=0.51\textwidth]{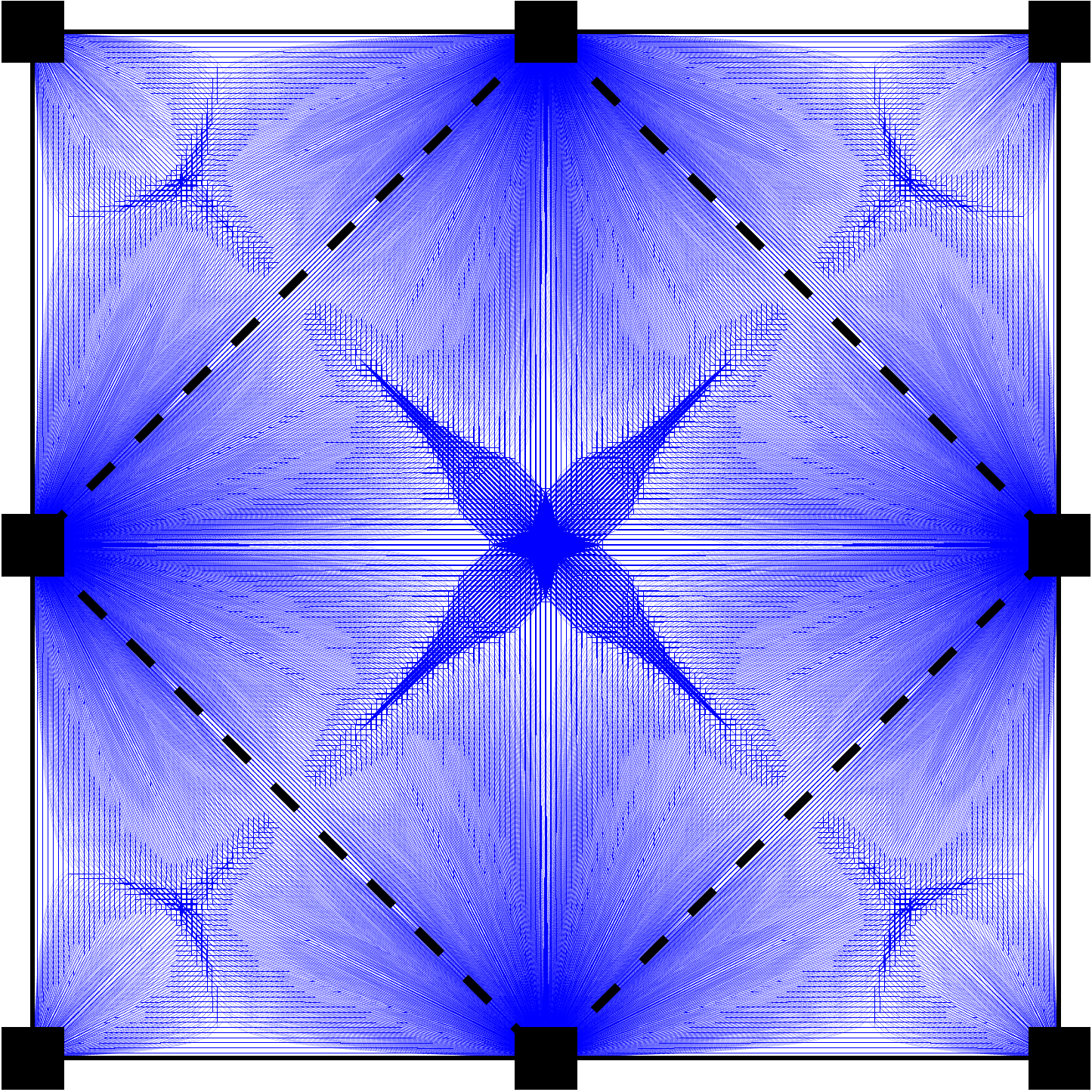}}}
		\subfloat[]{\includegraphics*[trim={2cm -0cm 1cm 4cm},clip,width=0.45\textwidth]{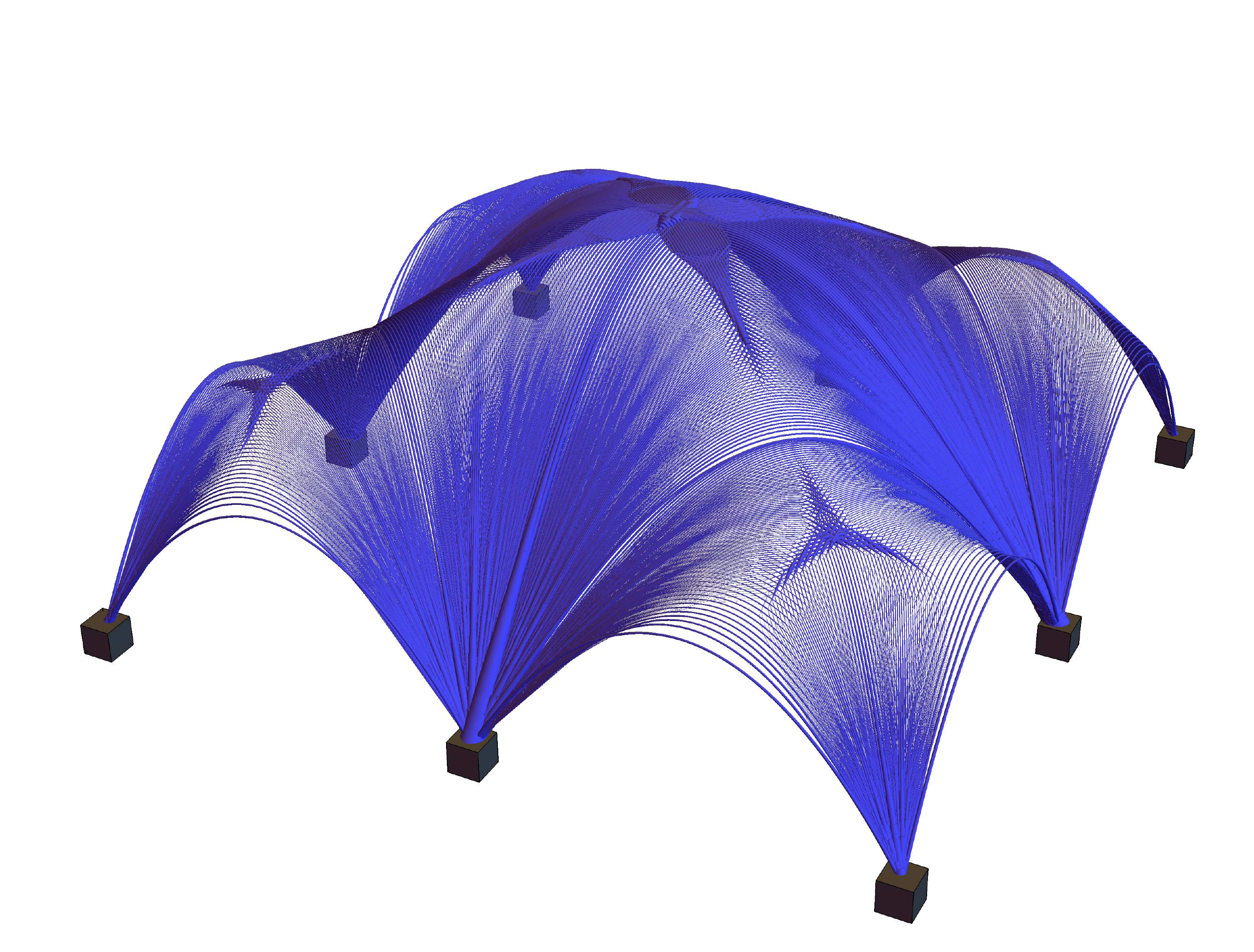}}\hspace{0.4cm}
		\subfloat[]{\includegraphics*[trim={2.5cm 1cm 1.5cm 4cm},clip,width=0.45\textwidth]{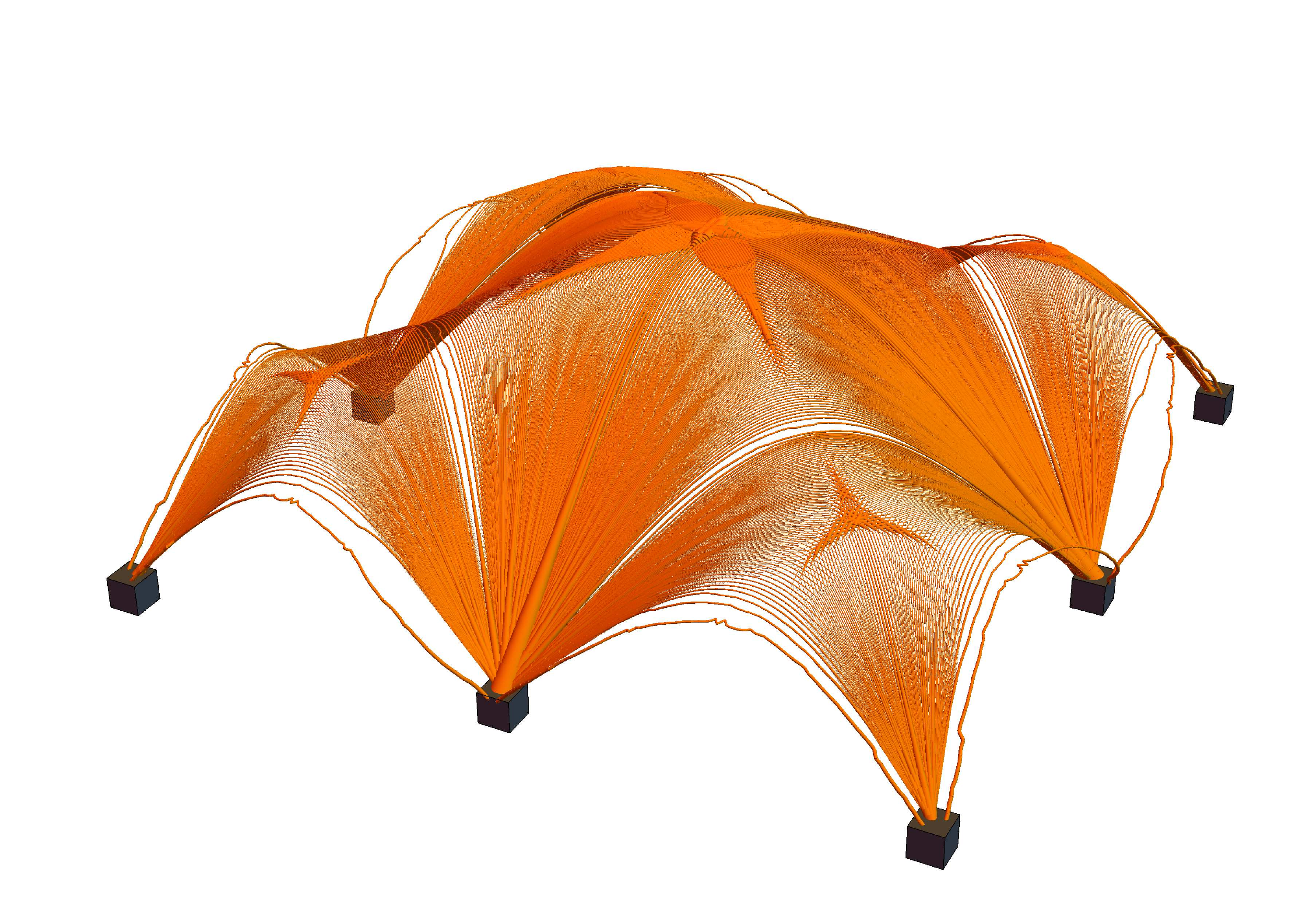}}
		\caption{The problem of optimal grid-shell in compression over a square domain with eight point supports and under uniformly distributed load: (a) design domain, load and boundary conditions; (b) optimal elevation function $z = -\frac{1}{2} \,w$  where $w$ solves $(\mathcal{P}^*_X)$; (c) plane truss given by $\mbf{s}$ solving $(\mathcal{P}_X)$; (d) optimal grid-shell in compression; (e) elastic deformation.}
		\label{fig:8_columns}
	\end{figure}
	Similarly as in Example \ref{ex:cross}, from solution $\mbf{s}$ presented in Fig. \ref{fig:8_columns}(c) we find that the optimal plane pre-stressed truss breaks along the dashed lines into five independent parts; the same concerns the optimal grid-shell in Fig. \ref{fig:8_columns}(d). Once again along those lines the vault in compression cracks as visualized in Fig. \ref{fig:8_columns}(e). In the elastic deformation we observe another singularity: the arches along the boundary lines undergo big and irregular horizontal displacements $u_\e$ that are normal and outward with respect to the boundary $\bO$. In formulation \eqref{eq:Z_X_def} of the infinite dimensional problem $\mathcal{P}^*$ we see that, unless $u = 0$ due to presence of the kinematical support, the component $u_\nu$ of $u$ that is normal respect to $\bO$ is not bounded from below by the two point condition, i.e. arbitrarily big inward normal component is admissible (it is not in contradiction with Proposition \ref{prop:regularitu_u_w} which guarantees that $\norm{u}_\infty \leq \mathrm{diam}(\Omega)$ since this result is valid only under the assumption that $\Gamma = \bO$). The analogous phenomenon may be discovered for the discrete formulation and thus the outward normal displacement on the boundary is arbitrarily big (we recall that $u_\e = - \Z_X/(E_0 V_0) \, u$ for a grid-shell in compression). Nevertheless, a solution was found by the MOSEK solver in each iteration of the adaptive loop.
\end{example}

\subsection{Optimal vaulted structures pinned on an elevated boundary}
\label{ssec:boundary_conditions}

\subsubsection{The obstacles for setting the optimum design problem in case of the boundary being a non-plane curve }
\label{sssec:any_curve}

In Sections \ref{sec:plastic_design},\,\ref{sec:elastic_design} the designed vaults are assumed to be pinned on the plane horizontal curve $\bO \times \{0\}$; similarly grid-shells in Section \ref{sec:discrete} are pinned on $(X \cap \bO) \times \{0\}$. Another natural generalization of those design problems would be to set the kinematical support on a 3D curve $\hat{\Gamma}$ such that $\hat{\Gamma} = \hat{z}_0(\bO)$ where $\hat{z}_0(x) = \bigl(x,z_0(x)\bigr)$ for some Lipschitz continuous function $z_0: \bO \to \R$.

As was pointed out in the text the design problems $(\mathrm{MVV}_\Omega)$, $(\mathrm{MCV}_\Omega)$, $(\mathrm{MVGS}_X)$, $(\mathrm{MCGS}_X)$ lack some vital mathematical properties such as convexity, e.g. with respect to the pair $(z,\eta)$ in $(\mathrm{MVV}_\Omega)$. None of those problems, however, are ever tackled directly: their solutions are reconstructed based on solutions of convex problems $(\mathcal{P})$,\,$(\mathcal{P}^*)$ or their discrete counterparts. In particular, from $w$ solving $(\mathcal{P}^*)$ follows the formula for optimal elevation:
\begin{equation*}
z = \frac{1}{2} \, w.
\end{equation*}
It is clear that such formula cannot work if the boundary is elevated -- the condition $w = 2 z_0$ on $\bO$ would have to be imposed. Although problem $(\mathcal{P}^*)$ could be modified by introducing this non-homogeneous boundary condition on $w$, the ideas from Section \ref{sec:plastic_design} that give rise to Theorem \ref{thm:recovring_MV_dome} fall apart. The obstacles are more crisp in the elastic setting where the elevation $z$ turned out to be proportional to elastic vertical displacement $w_\e$, cf. Corollary \ref{cor:z_w}. The last paragraph in Section \ref{ssec:recovering_dome_elastic} explained that this relation between $z$ and $w_\e$ is crucial for generating a constant strain in the sense that $\gamma_{z,+}\bigl(\mathcal{A}_{z}(u_\e,w_\e)\bigr) = \mathrm{const}$ and this property is intrinsic for structures of minimum compliance. If the optimal vault ought to be fixed along a non-plane curve $\hat{\Gamma}$ the deformation function $w_\e$ must still be zero on $\bO$ where necessarily $z = z_0$, then the proportionality relation between $z$ and $w_\e$ is impossible and the miraculous equalities $\gamma_{z,+}\bigl(\mathcal{A}_{z}(u_\e,w_\e)\bigr) = h(u_\e,w_\e) = \Z/(E_0 V_0)$ fails to hold. Summing up, the methods developed in this paper are of no use for problems $(\mathrm{MVV}_\Omega)$, $(\mathrm{MCV}_\Omega)$, $(\mathrm{MVGS}_X)$, $(\mathrm{MCGS}_X)$ when an arbitrary boundary curve $\hat{\Gamma}$ is considered. Nevertheless, the experiment that will be carried out in Example \ref{ex:Gamma} will indicate that in this broader setting a 2D form is not optimal.

Another story may be told about the Prager problems $(\mathrm{MVPS}_H)$, $(\mathrm{MCPS}_H)$ posed in Section \ref{sec:3D}: both are well-posed convex problems on their own, cf. Proposition \ref{prop:existence}. Although any 3D structure is feasible in those problems, from Theorem \ref{thm:optimal_3D_structure} and Corollary \ref{cor:3D_on_cylinder} we learned that a lower dimensional vault known from Sections \ref{sec:plastic_design},\ref{sec:elastic_design} solves $(\mathrm{MVPS}_H)$, $(\mathrm{MCPS}_H)$ provided that either the fixed boundary is a plane curve $\bO \times \{0\}$ or it is the whole cylinder $\bO \times \R$. In contrast with the vault problems $(\mathrm{MVV}_\Omega)$, $(\mathrm{MCV}_\Omega)$ there are no mathematical contra-indications for posing problems $(\mathrm{MVPS}_{H,\hat{\Gamma}})$, $(\mathrm{MCPS}_{H,\hat{\Gamma}})$ with the boundary being any Lipschitz curve $\hat{\Gamma}$, in particular the proof of Proposition \ref{prop:existence} remains unchanged hence there always exists  optimal 3D structure $\hat\sigma$ or $\hat\mu$ and load distribution $\hat{F}$ furnishing minimal volume $\hat{\V}_{\mathrm{min}}^{H,\hat{\Gamma}}$ or minimal compliance $\hat{\Comp}_\mathrm{min}^{H,\hat{\Gamma}}$. The big question reads:
\begin{problem}
	\label{prob:3D_is_2D?}
	For an arbitrary boundary curve $\hat{\Gamma}$ does there exist optimal solutions of $(\mathrm{MVPS}_{H,\hat{\Gamma}})$, $(\mathrm{MCPS}_{H,\hat{\Gamma}})$ that concentrate on a single surface $\mathcal{S}_z$?
\end{problem}

\noindent We will not give a definitive answer to this question although the example studied below should provide an intuition. First we note (cf. Corollary \ref{cor:3D_on_cylinder} and the comment below) that there still hold inequalities similar to those in Lemma \nolinebreak \ref{lem:inequalities_3D}:
\begin{equation}
\label{eq:ineq_3D_Gamma}
\hat{\mathcal{V}}^{H,\hat{\Gamma}}_\mathrm{min} \geq \mathcal{Z}, \qquad \hat{\Comp}^{H,\hat{\Gamma}}_{\mathrm{min}} \geq \frac{\Z^2}{2E_0 V_0}
\end{equation}
where again $\Z = \inf \mathcal{P} = \sup \mathcal{P}^*$, i.e. the value of $\Z$ is independent of $\hat{\Gamma}$. 

\begin{example} \textbf{(Load distributed along diagonals of a square with non-plane boundary curve $\hat{\Gamma}$)}
	\label{ex:Gamma}
	We consider the setting of the problem from Example \ref{ex:diagonal_load}, i.e. $\Omega$ is a square and the downward load of intensity $t>0$ is distributed along the diagonals: $f =- t\,\Ha^1\mres[A_1,A_3] - t\,\Ha^1\mres[A_2,A_4]$ where $A_i$ are the square's corners. We tackle the modified Prager problem $(\mathrm{MVPS}_{H,\hat{\Gamma}})$ (in the setting of compression) for the boundary curve $\hat{\Gamma}$ being a polygonal chain showed in Fig. \ref{fig:Gamma}(a), the jump parameter $\Delta H$ may be chosen arbitrarily. 	
	\begin{figure}[h]
		\centering
		\subfloat[]{\includegraphics*[trim={0.5cm -0.6cm 0.5cm 0cm},clip,width=0.32\textwidth]{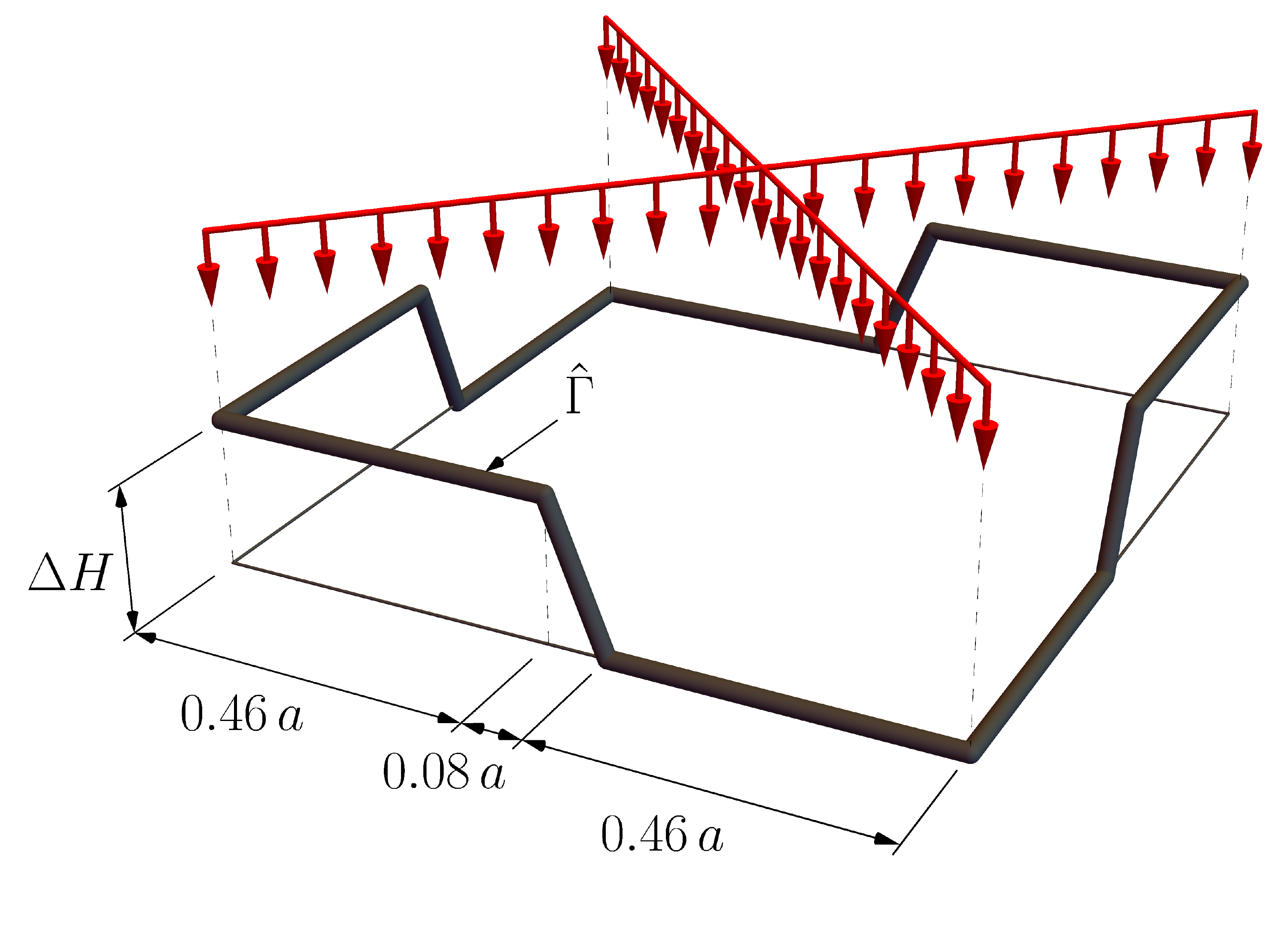}}\hspace{0.5cm}
		\subfloat[]{\includegraphics*[trim={1.5cm -2cm 0.5cm 3cm},clip,width=0.3\textwidth]{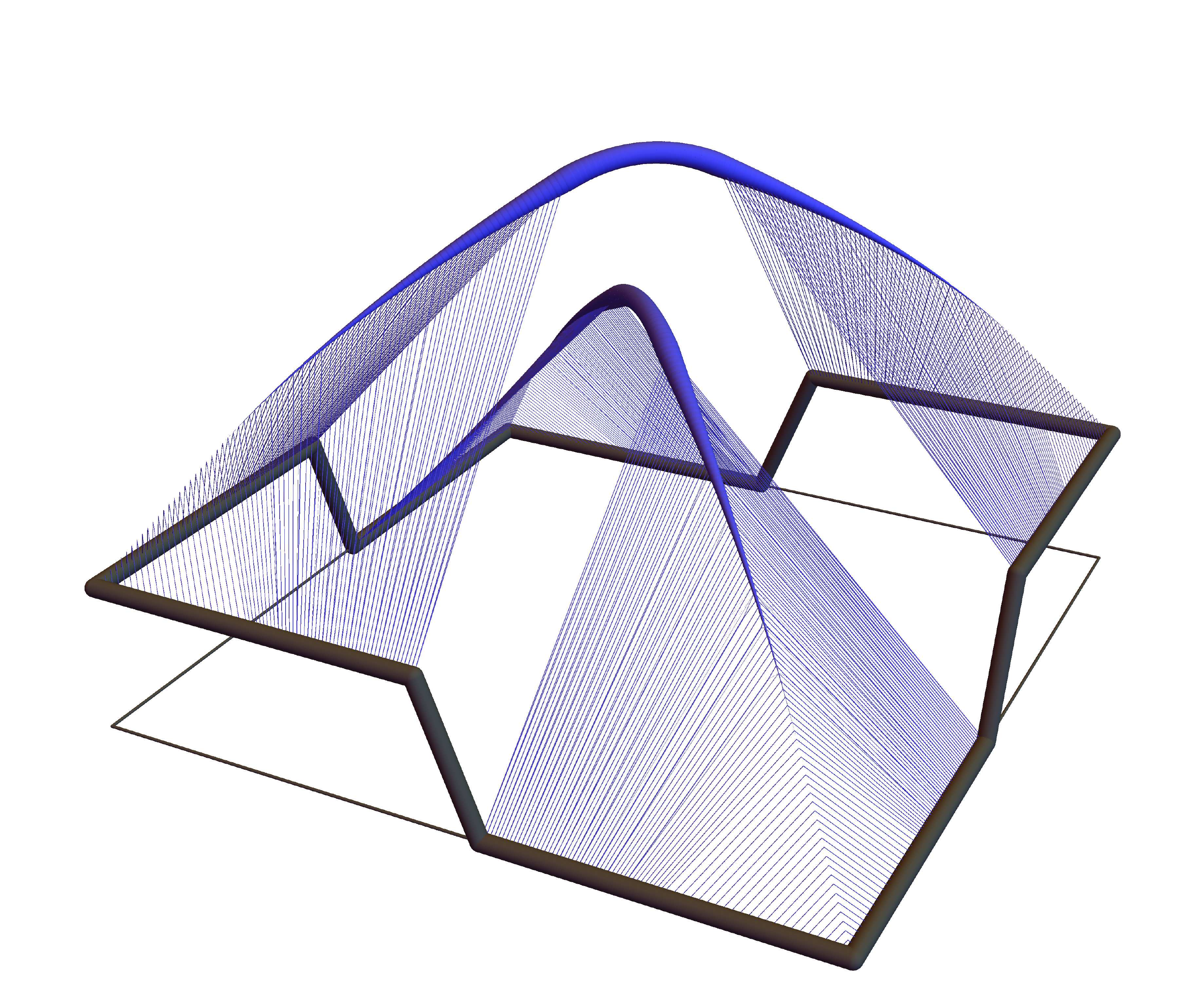}}\hspace{0.5cm}
		\subfloat[]{\includegraphics*[trim={0.cm 0cm 0cm 0cm},clip,width=0.28\textwidth]{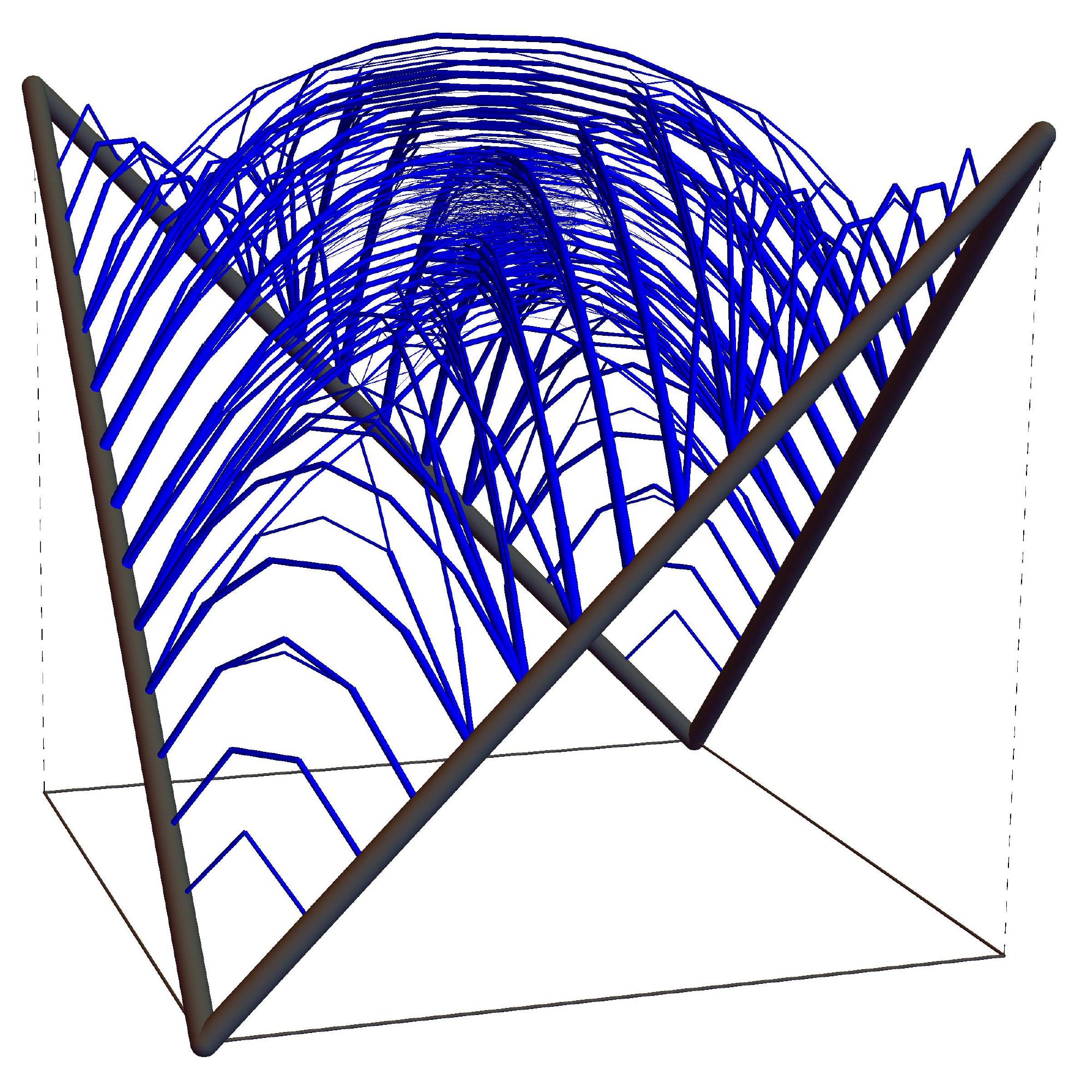}}\\
		\caption{(a) Prager problem under transmissible knife-load -- the load and boundary conditions; (b) a prediction of a Prager structure that does not constitute a 2D vault. (c) Prager problem for transmissible load that is uniformly distributed horizontally and for different boundary curve -- numerical simulation with the use of the 3D ground structure method, by courtesy of Tomasz Sok{\'o}{\l}.}
		\label{fig:Gamma}
	\end{figure}
	
	According to Theorem \ref{thm:optimal_3D_structure} the 3D structure approximated by the grid-shell from Fig. \ref{fig:diagonal_load}(d) solves the same problem yet for a plane boundary $\hat{\Gamma} = \bO \times \{0\}$, in particular the volume of the structure equals $\Z$. At the same time, in Fig. \ref{fig:diagonal_load}(d) we observe  that the loads on the two diagonals are independently carried by the two arch-like structures, i.e. if only one of the loads was applied the optimal structure would be none other but one of the arches. Since the geometry of the curve $\hat\Gamma$ is suited to solution from Example \ref{ex:diagonal_load}, see Fig. \ref{fig:diagonal_load}(b), we deduce that the 3D structure showed in Fig. \ref{fig:Gamma}(b) solves the problem $(\mathrm{MVPS}_{H,\hat{\Gamma}})$: the equilibrium is satisfied while the volume is still $\Z$ therefore optimality follows owing to the first inequality in \eqref{eq:ineq_3D_Gamma}.
\end{example}
\medskip
The optimal 3D structure in Fig. \ref{fig:Gamma}(b) is certainly not the design chased after in Problem \ref{prob:3D_is_2D?}: there is no surface $\mathcal{S}_z = \bigl\{(x,z(x))\bigr\}$ on which the two arches lie together. It does not yet prove that the answer to Problem \ref{prob:3D_is_2D?} is negative since there may be other optimal solutions that meet the required condition. However, since $\Delta H$ can be chosen arbitrarily big it is very difficult that we could point to such a 3D structure without compromising the minimum volume $\Z$. For additional consideration, in Fig. \ref{fig:Gamma}(c) we present a numerical prediction of a Prager structure for the case of a square domain $\Omega$, uniformly distributed transmissible load and a non-planar boundary curve $\hat{\Gamma}$ being a different polygonal chain. The solution, found by Tomasz Sok{\'o}{\l} through a 3D ground structure method, appears to be multi-layered instead of lying on a single surface, cf. the solution in Fig. \ref{fig:Prager_structure} for the horizontal boundary curve.

Based on the performed simulations it is the author's fear that the optimal vault problem tackled in Sections \ref{sec:plastic_design},\,\ref{sec:elastic_design} is badly posed for arbitrary curve $\hat{\Gamma}$, i.e. that its solution may not exist in general.

\subsubsection{Alteration of the convex problems $(\mathcal{P})$,\,$(\mathcal{P}^*)$ for a slanted plane boundary curve} 
\label{sssec:slanted}

In \cite{czubacki2020} the optimal arch-grid problem (cf. Section \ref{ssec:archgrids} below) was successfully tackled for a boundary curve that is not horizontal: with arches directions fixed by a Cartesian basis $e_1,e_2$ the boundary could be chosen as $\hat{\Gamma} = \hat{\Gamma}_{z_0} = \bigl\{ \bigl(x,z_0(x)\bigr) \, \big\vert \, x\in \bO \bigr\}$ where $z_0(x) = z_0(x_1,x_2) = b_0 + b_1\,x_1 +b_2\, x_2 +b_3\, x_1 x_2$. Inspired by this idea we find that the vault problem may be posed in a slightly simpler scenario where $z_0$ is affine -- henceforward we set
\begin{equation*}
\hat{\Gamma}_{z_0} = \Big\{ \bigl(x,z_0(x)\bigr) \, \big\vert \, x\in \bO \Big\}, \quad \text{where } z_0:\Rd \to \R \text{ is any affine function}, \quad G_{z_0} := \mathrm{I} + \nabla z_0 \otimes \nabla z_0 = \mathrm{const}.
\end{equation*}
We shall quickly repeat the main steps from Section \ref{sec:plastic_design} where the plastic design problem is formulated for vaults: from the problem $(\mathrm{MVV}_{\Omega})$ we jump to $(\mathrm{MVV}_{\Omega,\Gamma_{z_0}})$ where the elevation function $z \in C^1(\Ob;\R)$ must satisfy condition $z = z_0$ on $\bO$. The variables may be changed as follows: $z = \tilde{z} + z_0$ where $\tilde{z} = 0$ on $\bO$. For a function  $\sigma \in L^1(\Omega;\Sddp)$ that satisfies the equilibrium equations $\DIV\, \sigma = 0$, $-\dive (\sigma \nabla z) = f$ in $\Omega$ the normalized volume of the vaults reads
\begin{align*}
&\int_{\Omega} \Big(\tr \, \sigma + \bigl(\nabla z \otimes \nabla z \bigr) :\sigma \Big) = \int_{\Omega} \Big( \mathrm{I} : \sigma +  \bigl((\nabla \tilde{z} + \nabla z_0) \otimes  (\nabla \tilde{z} + \nabla z_0) \bigr) : \sigma \Big)\\
=& \int_{\Omega} \Big( \bigl(\mathrm{I}+\nabla z_0 \otimes \nabla z_0 \bigr):\sigma  + \bigl(\nabla \tilde{z} \otimes \nabla \tilde{z} \bigr):\sigma  \Big) + \int_{\Omega} \nabla \varphi  : \sigma = \int_{\Omega} \Big( G_{z_0}:\sigma  + \bigl(\nabla \tilde{z} \otimes \nabla \tilde{z} \bigr):\sigma  \Big)
\end{align*}
where the vector function $\varphi:=2\,\tilde{z}\,\nabla z_0  \in C^1(\Ob;\Rd)$ was introduced; since $\nabla z_0 \in \Rd$ is a constant vector we see that $\nabla \varphi = 2\, \nabla z_0 \otimes \nabla \tilde{z}$ which explains the presence of the term $\int_{\Omega} \nabla \varphi  : \sigma$ above. Since $\tilde{z}$ is zero on $\bO$ the same goes for $\varphi$ therefore  $\int_{\Omega} \nabla \varphi  : \sigma = 0$ due to the equilibrium equation $\DIV \,\sigma = 0$ in $\Omega$. Similarly one may prove that $\dive(\sigma \nabla z_0) =0$ and consequently $- \dive(\sigma \nabla \tilde{z}) = -\dive(\sigma \nabla z) = f$. 

By acknowledging the above the next step is to propose a modified pair of mutually dual convex problems $(\mathcal{P}_{z_0})$, $(\mathcal{P}^*_{z_0})$ that differs from the original one as follows:
\begin{enumerate}[label=(\Roman*)]
	\item in the problem $(\mathcal{P})$ the term $\tr\,\sigma$ is replaced by $G_{z_0}:\sigma$;
	\item in the problem $(\mathcal{P}^*)$ the constraint is replaced by $\frac{1}{4}\, \nabla w \otimes \nabla w + e(u) \preceq  G_{z_0}$ everyhwere in $\Ob$.
\end{enumerate}
The rest of the argument is almost a 1-to-1 copy of Section \ref{sec:plastic_design}, ultimately we arrive at:

\begin{proposition}
	\label{proposition:recovring_MV_dome_Gamma}
	For an affine function $z_0 \in \Rd \to \R$ assume that the pairs $(\sigma,q) \in L^1(\Omega;\Sddp) \times L^1(\Omega;\Rd)$ and $(u,w) \in C^1(\Ob;\R^2) \times C^1(\Ob;\R)$ are solutions of problems $(\mathcal{P}_{z_0})$ and $(\mathcal{P}^*_{z_0})$ respectively. Then the pair
	\begin{equation*}
	z = \frac{1}{2} w + z_0,\quad \eta = \frac{1}{J_{z}} \, \sig
	\end{equation*} 
	solves the problem $(\mathrm{MVV}_{\Omega,\Gamma_{z_0}})$ with $\mathcal{V}_{\mathrm{min},z_0} = \Z_{z_0}:= \inf \mathcal{P}_{z_0} = \sup \mathcal{P}^*_{z_0}$.
\end{proposition}

Alteration of the discrete problems $(\mathcal{P}_X)$,\,$(\mathcal{P}_X^*)$ is just as straightforward: the length vector $\mbf{l}\in \R^m$ defined by $l_k :=\abs{\chi_+(k)-\chi_-(k)}$ must be changed to:
\begin{equation*}
\mbf{l}(z_0)\in \R^m, \quad l_k(z_0) := \Big( G_{z_0}:\bigl(\chi_+(k)-\chi_-(k)\bigr)\otimes \bigl(\chi_+(k)-\chi_-(k)\bigr) \Big)^{\frac{1}{2}}
\end{equation*}
thus arriving at problems $(\mathcal{P}_{z_0,X})$, $(\mathcal{P}_{z_0,X}^*)$ with equal objective value $\Z_{z_0,X}$. To give a flavour we once more revisit the example of optimal grid-shell/vault over a square for a load distributed along the diagonals, yet for a "slanted" boundary:

\begin{example}\textbf{(Knife load along diagonals for a slanted boundary)}
	\label{ex:slanted_Gamma}
	For a square domain $\Omega$ we consider the problem of optimal grid-shell in compression that is supported on a plane slanted (non-horizontal) curve $\hat{\Gamma}_{z_0} =\bigl\{ \bigl(x,z_0(x)\bigr) \, \big\vert \, x\in \bO \bigr\}$ where $z_0$ is affine with $\nabla z_0 = -0.2 e_1 + 0.5 e_2$, see Fig. \ref{fig:slanted}(a); as in Examples \ref{ex:diagonal_load}, \ref{ex:Gamma} the downward load of intensity $t$ is distributed along the square's diagonals. In order to obtain a detailed solution we choose a very fine grid $X$ of $321 \times 321$ nodes which generated over $5.3 \cdot 10^9$ members in the full ground structure. Computational details, also for more coarse grids $X$, may be found in Table \ref{tab:slanted}.
	\begin{table}[h]
		\scriptsize
		\centering
		\caption{Summary on numerical computations for the pair of modified problems $(\mathcal{P}_{z_0,X}), (\mathcal{P}_{z_0,X}^*)$ specified in Example \ref{ex:slanted_Gamma}.}
		\begin{tabular}{lccccccc}
			\toprule
			Example no. & Slope of $\hat{\Gamma}_{z_0}$ & Grid $X$  & Full GS  & Iterations   & Active GS    & CPU time & Objective value $\Z_{z_0,X}$ \\
			\midrule
			\ref{ex:slanted_Gamma} (Fig. \ref{fig:slanted}) &  $\nabla z_0 = -0.2 e_1 + 0.5 e_2 $
			& $81\!\times\!81$  & $21\,520\,080$ & 9 & $60\,156$ & $2$ min. $20$ sec.& $1.7381 \, t a^2$\\
			&& $161\!\times\!161$  & $335\,936\,160$ & 13 & $351\,877$ & $33$ min. $15$ sec.& $1.7380 \, t a^2$\\
			&& $321\!\times\!321$  & $5\,308\,672\,320$ & 12 & $2\,187\,452$ & $5$ h $39$ min. $02$ sec.& $1.7379 \, t a^2$\\
			\bottomrule
		\end{tabular}
		\label{tab:slanted}
	\end{table}

	\begin{figure}[p!]
		\centering
		\subfloat[]{\includegraphics*[trim={0.3cm 2cm 7cm -0cm},clip,width=0.45\textwidth]{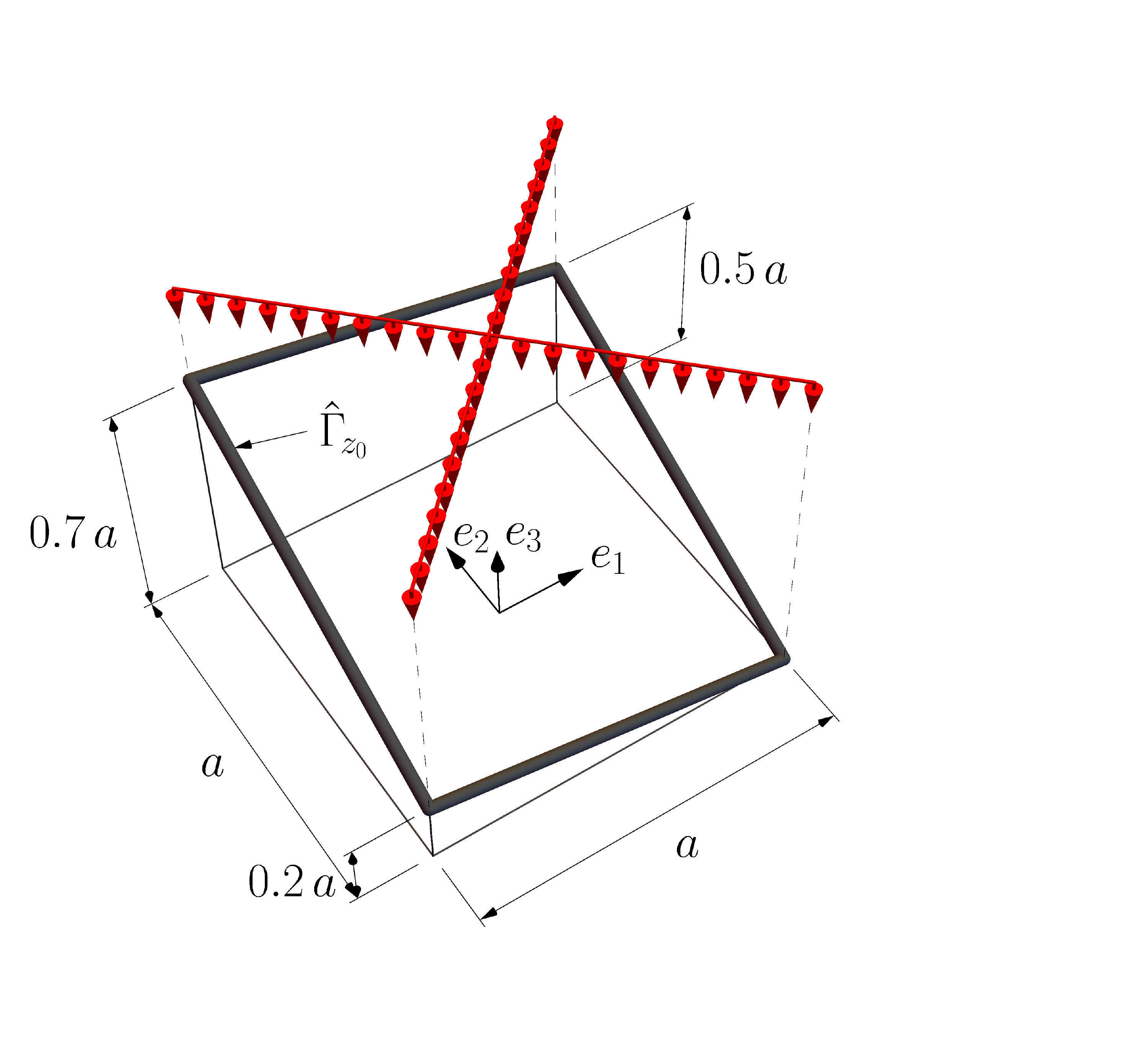}} \hspace{1.cm}
		\subfloat[]{\includegraphics*[trim={0cm -0.7cm -0cm -0cm},clip,width=0.45\textwidth]{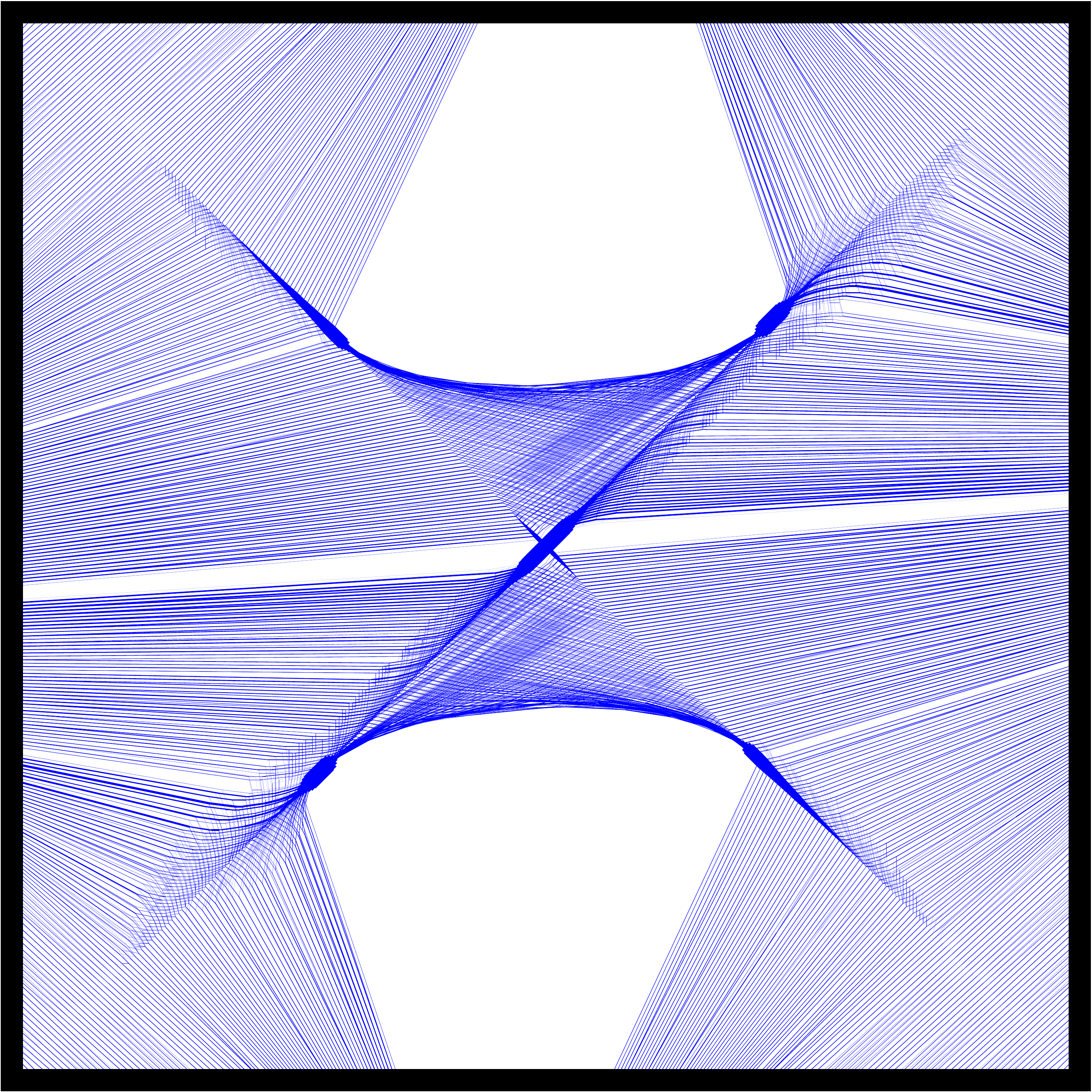}}\\
		\subfloat[]{\includegraphics*[trim={0.5cm 0.cm 1.8cm 2cm},clip,width=0.48\textwidth]{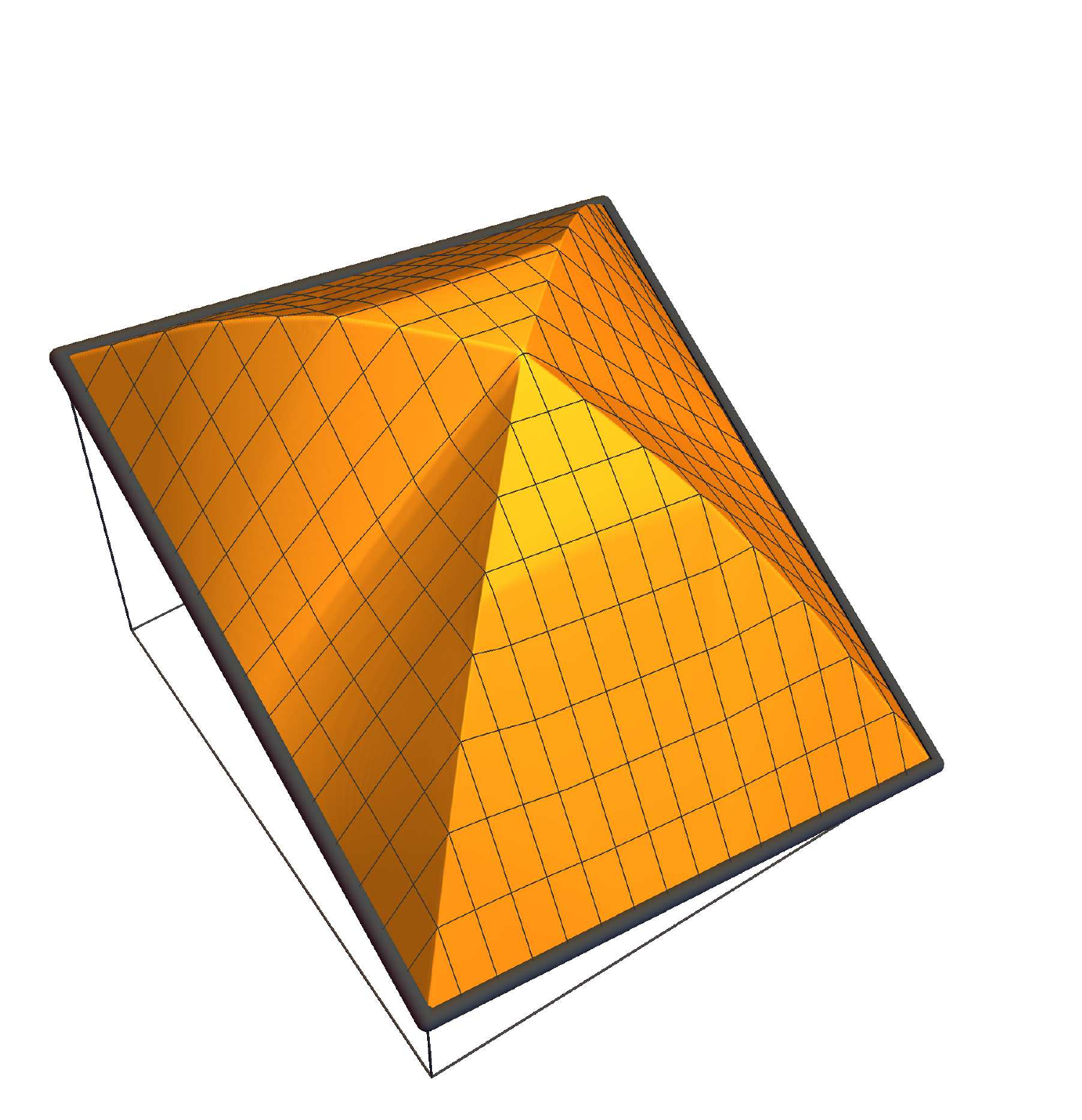}}\hspace{0.2cm}
		\subfloat[]{\includegraphics*[trim={1.3cm 0.cm 3.8cm 4cm},clip,width=0.48\textwidth]{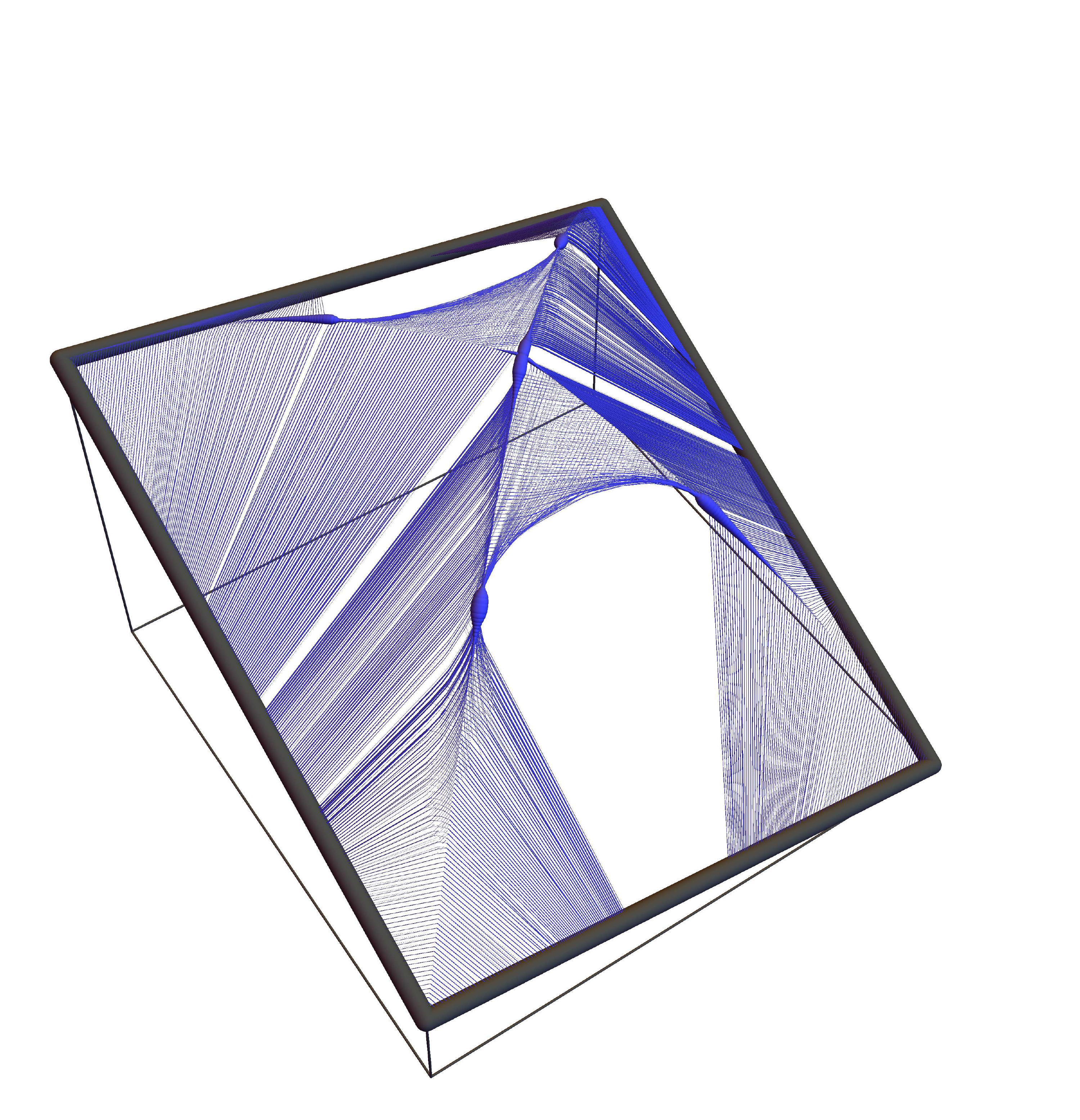}}
		\caption{The problem of optimal grid-shell in compression over a square domain for a slanted boundary and two knife loads: (a) geometry of the boundary and the load; (b) plane truss given by $\mbf{s}$ solving $(\mathcal{P}_{z_0,X})$; (c) optimal elevation function $z = -\frac{1}{2} \,w + z_0$  where $w$ solves $(\mathcal{P}^*_{z_0,X})$; (d) optimal grid-shell in compression.}
		\label{fig:slanted}
	\end{figure}
	The optimal pre-stressed truss, generated by $\mbf{s}$ solving the modified problem $(\mathcal{P}_{z_0,X})$, is showed in Fig. \ref{fig:slanted}(b). The optimal elevation of the slanted grid-shell is obtained by interpolating data $\mbf{z} = - \frac{1}{2}\,\mbf{w}+ \mbf{z}_0$  (the minus sign is due to the compression setting) where $\mbf{w}$ solves $(\mathcal{P}^*_{z_0,X})$ and  $z_{0;i} = z_0\bigl(\chi(i)\bigr)$, see Fig. \ref{fig:slanted}(c). The optimal grid-shell is presented in Fig. \ref{fig:slanted}(d). We observe that some unexpected forms emerge in comparison to the original solution from Fig. \ref{fig:diagonal_load}(d): two regions in the centre, where the field $\sigma$ is rank-two, seem to be bounded by curved bars of finite cross-sectional area. We thus learn that imposing a non-zero slope of the boundary greatly affects the optimal structural topology.
\end{example}

\subsection{A link to the optimal arch-grid problem}
\label{ssec:archgrids}

As mentioned in the introduction, this work was inspired by the problem of optimal design of arch-grids originated at the end of '70s in \cite{rozvany1979} and recently revisited in \cite{czubacki2019}, \cite{czubacki2020}. For a convex bounded domain $\Omega \subset \Rd$ and a fixed Cartesian basis $e_1,e_2$ an arch-grid is originally defined as a structure composed of arches in compression running from boundary to boundary in the two orthogonal directions $e_1,e_2$; there may be a continuum of infinitely thin arches or/and a finite number of arches of finite cross-sectional area. Similarly as in case of vaults in Sections \ref{sec:plastic_design} and \ref{sec:elastic_design} the arch-grid ought to carry the load $f$ that vertically tracks the structure.

In \cite{czubacki2020} the problem of optimal arch-grids was explicitly rewritten as a pair of mutually dual convex programs: in the primal problem  we seek a field $q \in \Mes(\Ob;\Rd)$ satisfying $-\dive\,q = f$ that minimizes certain norm; in the dual we maximize the integral $\int_\Ob w\, f $ with respect to the virtual displacement field $w$ which satisfy "line-wise" constraints $\rho_1(w;x_2) \leq \sqrt{L_1(x_2)}$ and $\rho_2(w;x_1) \leq \sqrt{L_2(x_1)}$ for each $x_2 \in \pi_2(\Ob)$ and $x_1 \in \pi_1(\Ob)$ ($\pi_i$ is the projection onto the $i$-th coordinate) where 
\begin{equation*}
	\rho_1(w;x_2) = \frac{1}{2} \left(\int_{x_1^-(x_2)}^{x_1^+(x_2)} \left(\frac{\partial w}{\partial x_1}(\xi,x_2) \right)^2\! d\xi\right)^{1/2}, \qquad \rho_2(w;x_1) = \frac{1}{2} \left(\int_{x_2^-(x_1)}^{x_2^+(x_1)} \left(\frac{\partial w}{\partial x_2}(x_1,\xi) \right)^2\! d\xi\right)^{1/2}
\end{equation*}
(in \cite{czubacki2020} definitions of $\rho_i$ differ: the factors $\frac{1}{2}$ are absent) and
\begin{equation*}
	L_1(x_2) = x_1^+(x_2) - x_1^-(x_2), \qquad L_2(x_1) = x_2^+(x_1) - x_2^-(x_1),
\end{equation*}
where for given coordinate $x_2 \in \pi_2(\Ob)$ we have $\bigl[\bigl(x_1^-(x_2)\bigr)\,,\,  \bigl(x_1^+(x_2)\bigr) \bigr] = \Ob\, \cap\, \bigl(\R \times \{x_2\}\bigr)$ and $ x_2^-(x_1),\ x_2^+(x_1)$ are similarly defined for every $x_1 \in \pi_1(\Ob)$. 

In \cite{rozvany1979} the authors proved that in the optimal arch-grid the arches lie on a single surface $\mathcal{S}_z$ for a function $z:X \to \R$ such that $z =0$ on $\bO$. One may thus say that an optimal arch-grid is a vault in which the material fibres run in two orthogonal directions $e_1,e_2$ only. Originally, the present work was aimed at relaxing this constraint. Now we shall show that the pair of problems $(\mathcal{P})$ and $(\mathcal{P}^*)$ may be altered, by adding constraints and relaxing constraints respectively, so that the pair of problems emerging in the optimal arch-grid problem is recast.

For a fixed Cartesian basis $e_1, e_2$ we consider a convex problem:
\begin{equation*}\tag*{$(\mathcal{P}_\perp)$}
\Z_\perp=\inf_{\substack{\sigma \in \Mes(\Ob;\Sddp) \\ q \in \Mes(\Ob;\Rd)}} \left\{ \int_{\Ob} \tr \, \sigma + \int_\Ob (\sigma^{-1} q)\cdot q  \, \left\vert \,  \begin{aligned}
\DIV \, \sigma &= 0,\\
-\dive\, q &= f 
\end{aligned}
\ \text{ in } \Omega, \ \ \sigma = \sum_{i=1}^{2}\sigma_i \,e_i\otimes e_i \right.\right\}
\end{equation*}
where $\sigma_i \in \Mes(\Ob;\R_+)$, and the problem which can be proved to be dual to $(\mathcal{P}_\perp)$:
\begin{equation*}\tag*{$(\mathcal{P}^*_\perp)$}
\sup_{\substack{u \in C^1\left(\Ob;\Rd\right) \\ w \in C^1(\Ob;\R)}} \left\{\, \int_\Ob w\, f  \, \left\vert \, (u,w)=0 \text{ on } \bO, \ \  \left(\frac{1}{4}\, \nabla w \otimes \nabla w +e(u)\right)\!:\!(e_i \otimes e_i) \leq 1 \ \ \text{ in } \Ob, \ i=1,2 \right.\right\}
\end{equation*}
The equality $\Z_\perp = \min \mathcal{P}_\perp = \sup \mathcal{P}_\perp^*$ could be proved similarly as in Theorem \ref{thm:duality}, cf. \cite{bouchitte2020}. Enforcing the eigenvectors of $\sigma$ to be $e_1, e_2$ everywhere in $\Ob$ together with equation $\DIV\, \sigma =0$ substantially limits the choice of admissible fields $\sigma$: each such field represents two mutually orthogonal families of parallel prismatic fibres going through $\Omega$ from boundary point to a boundary point. To make a rigorous link to the optimal arch-grid problem, however, we shall investigate the dual problem instead:
\begin{proposition}
	\label{prop:archgrid}
	Assuming that the domain $\Omega$ is convex and bounded, for any function $w \in \D(\Omega;\R)$ (smooth and with compact support in $\Omega$) the following conditions are equivalent independently for $i=1,2$: 
	\begin{enumerate}[label=(\Roman*)]
		\item There exists $u_i\, \in C^1(\Ob;\R)$ with $u_i=0$ on $\bO$ such that point-wise in $\Ob$ there holds:
		\begin{equation*}
			\frac{1}{4}\, \left(\frac{\partial w}{\partial x_i} \right)^2 + \frac{\partial u_i}{\partial x_i} \leq 1;
		\end{equation*}
		\item Denoting $j := i-1$, for each $x_j \in \pi_j(\Ob)$ there holds:
		\begin{equation}
			\label{eq:constraint_rho}
			\rho_i(w;x_j) \leq \sqrt{L_i(x_j)}.
		\end{equation}
	\end{enumerate}
\end{proposition}
\noindent Proof of the proposition is moved to \ref{app:archgird}. Up to regularity imposed on $w$ (the difference is immaterial thanks to a density result) this equivalence shows how to eliminate the vector function $u$ from the problem $(\mathcal{P}_\perp^*)$ above: function $u$ is present only in the constraint that may be replaced by the "line-wise" condition \eqref{eq:constraint_rho}. The new form of $(\mathcal{P}_\perp^*)$ becomes (up to multiplicative constant $\frac{1}{2}$ in definitions of $\rho_i$) the dual problem formulated in \cite{czubacki2020} for arch-grids, which ultimately shows that the optimal vault problem put forward in this work is in fact a generalization of the optimal arch-grid problem.

The numerical method proposed herein for the optimal vault problem may be modified in a straightforward manner to suit the optimal arch-grid problem $(\mathcal{P}_\perp), (\mathcal{P}_\perp^*)$: the ground structure must be limited only to bars going in directions $e_1, e_2$. For the regular grid $X$ as in \eqref{eq:X_regular} this can be done trivially. This way the size of thus constructed conic quadratic problems $(\mathcal{P}_{\perp,X}), (\mathcal{P}_{\perp,X}^*)$ decreases drastically as $m \approx 2 n$. The adaptive algorithm is redundant since there are no bars to add in subsequent iterations -- the problem is tackled by the MOSEK solver directly. We end this section by a demonstration:

\begin{example}\textbf{(Optimal arch-grid for uniformly distributed load on a square)}
\label{ex:archgrid}
For $\Omega$ being a square of the side's length $a$ and for the load being uniformly distributed, i.e. $f = - p\, \mathcal{L}^2\mres \Omega$ we consider the problem of optimal arch-grid for the orthogonal basis $e_1, e_2$ being parallel to the sides of the square. For a $201 \times 201$ nodal grid $X$ the ground structure consists of only $m = 80\,400$ bars parallel to $e_1, e_2$. The computational details, also for other resolutions of $X$, can be found in Table \ref{tab:archgrid}.

\begin{table}[h]
	\scriptsize
	\centering
	\caption{Summary on numerical computations for the arch-grid problem $(\mathcal{P}_{\perp,X}), (\mathcal{P}_{\perp,X}^*)$ specified in Example \ref{ex:archgrid}.}
	\begin{tabular}{lccccccc}
		\toprule
		Example no. & Grid $X$  & Full GS $=$ Active GS  & Iterations    & CPU time & Objective value $\Z_{\perp,X}$ &  Max. elevation $\norm{z}_\infty$\\
		\midrule
		\ref{ex:archgrid} (Fig. \ref{fig:archgrid})
		& $101\!\times\!101$ & $20\,200$ & 1 & $6$ sec. & $0.46005 \, p a^3$ & $0.414 \,a $\\
		& $201\!\times\!201$ & $80\,400$ & 1 &  $38$ sec. & $0.46014 \, p a^3$ & $0.414 \,a $\\
		& $401\!\times\!401$ & $320\,800$ & 1  &  $7$ min. $59$ sec. & $0.46016 \, p a^3$ & $0.414 \,a $\\
		\bottomrule
	\end{tabular}
	\label{tab:archgrid}
\end{table}

	\begin{figure}[h]
		\centering
		\subfloat[]{\includegraphics*[trim={2cm -0cm 1.5cm 3.5cm},clip,width=0.35\textwidth]{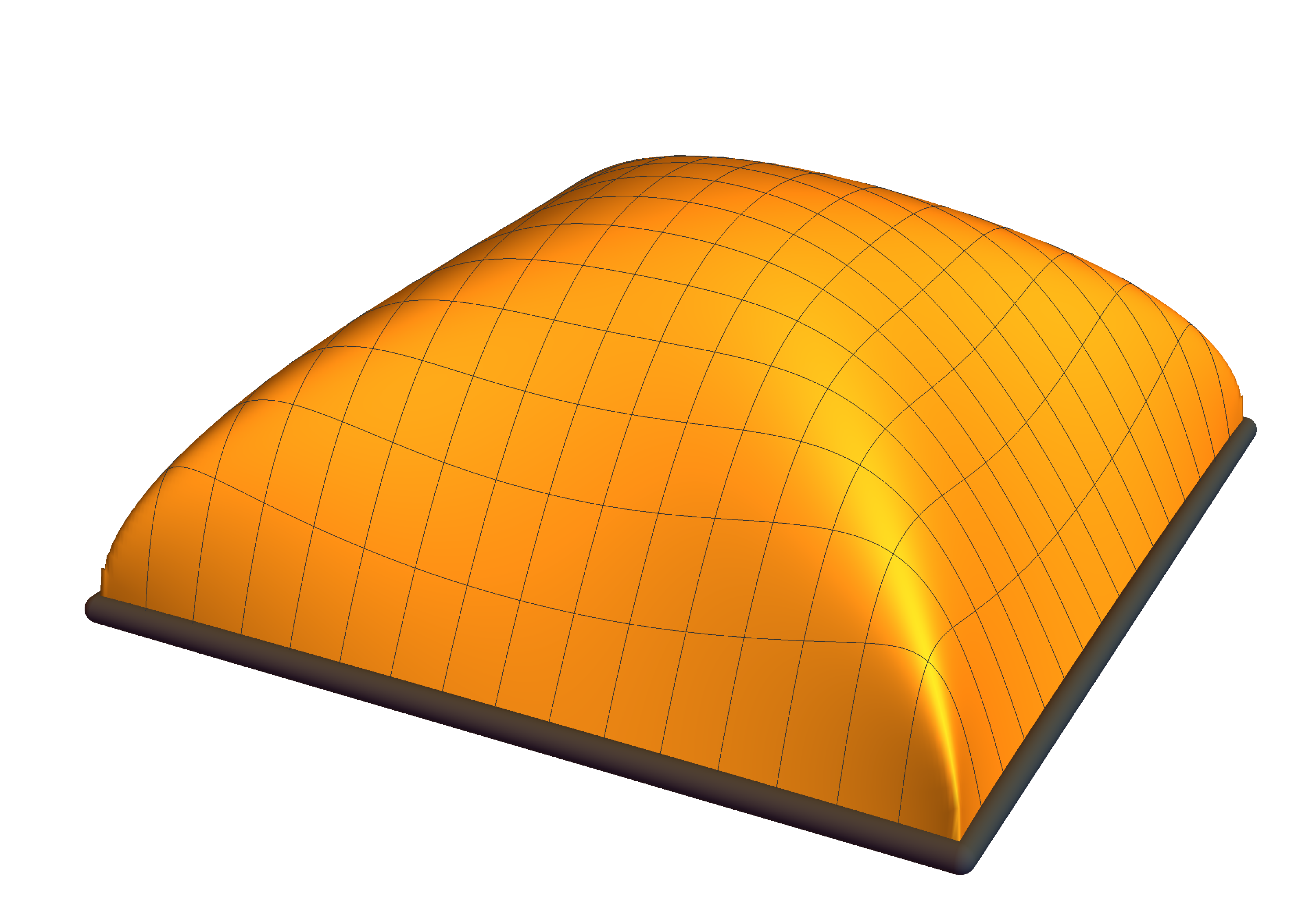}}\hspace{0.4cm}
		\subfloat[]{\includegraphics*[trim={2cm 0cm 1.5cm 3.5cm},clip,width=0.35\textwidth]{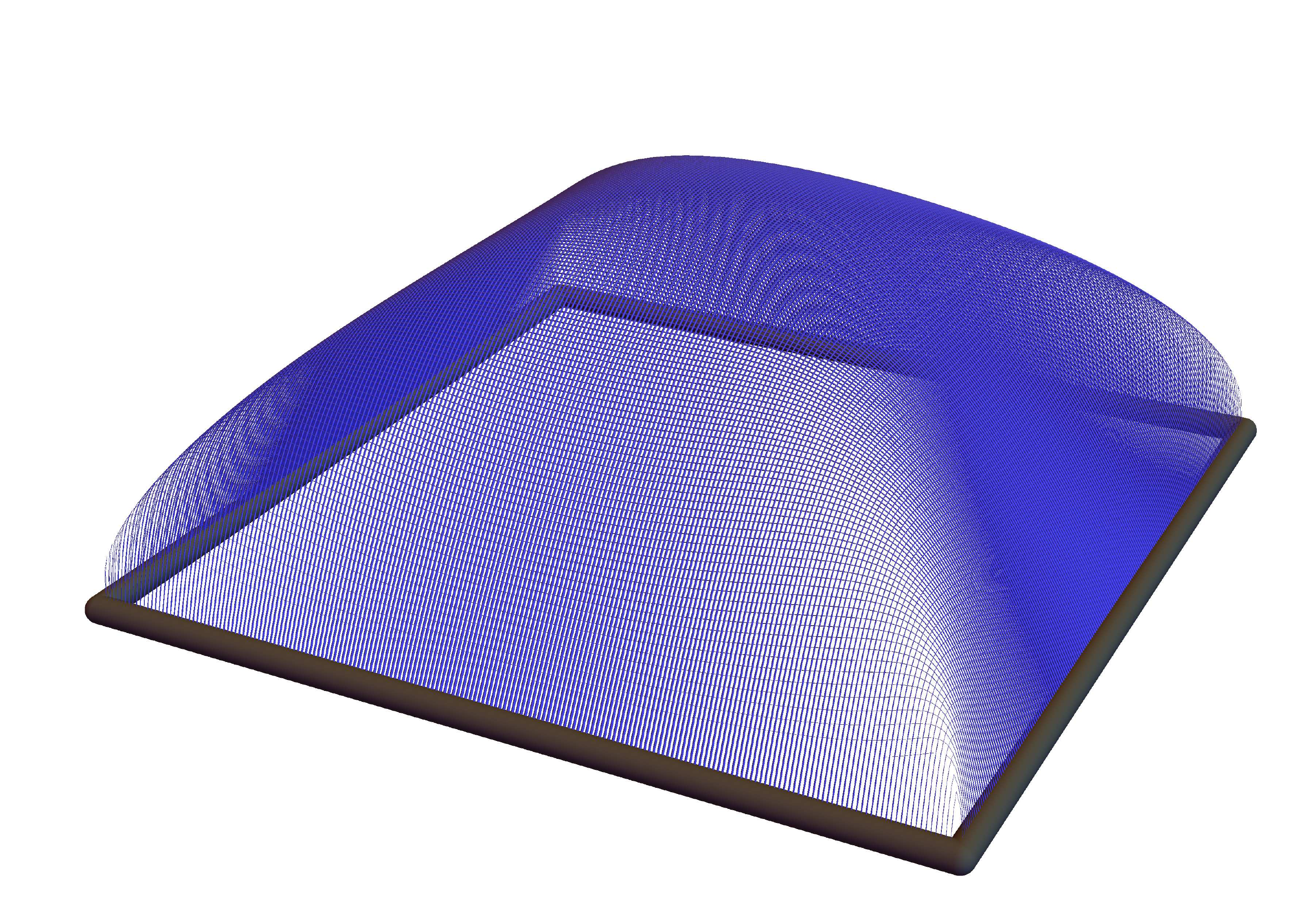}}\hspace{0.4cm}
		\subfloat[]{\includegraphics*[trim={1.5cm -1cm 1cm 2.5cm},clip,width=0.2\textwidth]{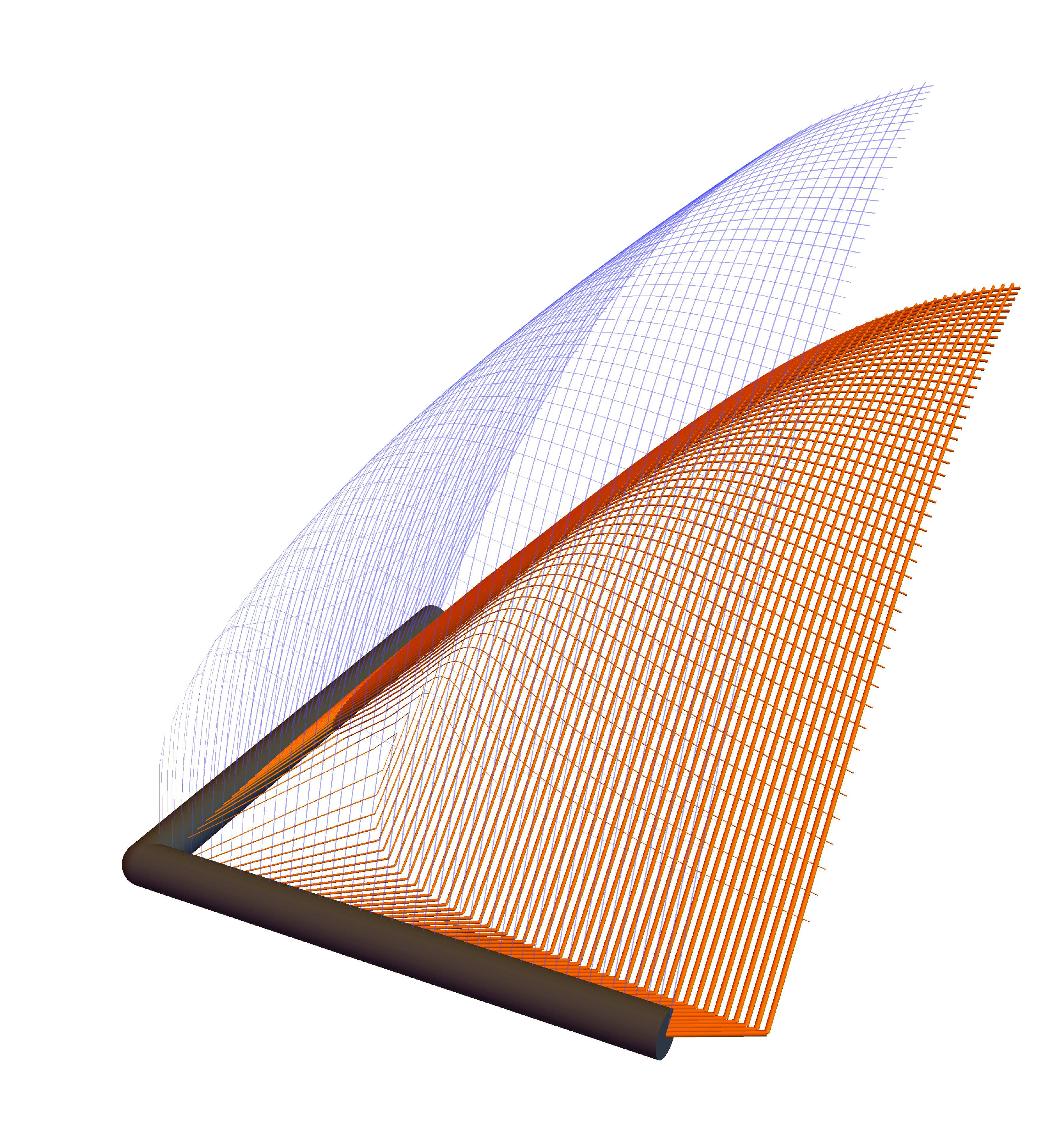}}\\
		\caption{The problem of optimal arch-grid over a square domain under uniformly distributed load: (a) optimal elevation function $z = -\frac{1}{2} \,w$  where $w$ solves $(\mathcal{P}^*_{\perp,X})$; (b) optimal arch-grid; (c) elastic deformation of the arch-grid in vicinity of a corner.}
		\label{fig:archgrid}
	\end{figure}
After an almost 1-to-1 adjustment of Theorem \ref{thm:recovering_MC_grid-shell} (and of formulas \eqref{eq:compression_mod} tailored for the compression setting) towards the problem of optimal arch-grid we may readily present the optimal elevation function $z$ in Fig. \ref{fig:archgrid}(a) and then the optimal arch-grid in Fig. \ref{fig:archgrid}(b). The results, both the visualizations and the objective value $\Z_{\perp,X}$, coincide very well with the results obtained in \cite{czubacki2019} found by alternative numerical methods. E.g. for the $201 \times 201$ nodal grid, it is natural to compare the value $\Z_{\perp,X} = 0.46014 \, p a^3$ being the minimal volume of an arch-grid with $\Z_{X} = 0.43615 \, p a^3$ in Example \ref{ex:pressure} that is the minimal volume of a grid-shell for the same boundary and load conditions: the arch-grid is $5.5\%$ heavier.

We conclude the subsection by an analysis of the elastic deformation of the generated optimal arch-grid: in Fig. \ref{fig:archgrid}(c) we focus on the structure's displacements in vicinity of one of the corners. We find that nodes in the interior of $\Omega$ that are the closest to $\bO$ undergo big horizontal displacements $\mbf{u}_{1,\e}, \mbf{u}_{2,\e}$ yet very small deflections $\mbf{w}_\e$. This picture allows to predict certain features of the exact solutions $u_\e \in BV(\Omega;\Rd)$  and $w_\e \in C(\Ob;\R)$ of problem $(\mathcal{P}^*_{\perp})$: the function $w_e$ is zero on the boundary while $u_\e$ is not (in the sense of trace in the $BV$ space). Proposition \ref{prop:regularitu_u_w} does not rule out such singularities of $u_\e, w_\e$ in the case of vaults/grid-shell either, yet in the course of numerical experiments they has not been observed so far. 
	
\end{example}

\section{Final remarks and open problems}
\label{sec:final_remarks}

\subsection{Short summary of the results and challenges of extension to full mathematical generality}

Assuming a plane supporting boundary $\bO$ in this work several design problems were put forward: 'continuous' plastic and elastic vault design problems $(\mathrm{MVV}_\Omega)$ and $(\mathrm{MCV}_\Omega)$ together with their discrete counterparts, plastic and elastic grid-shell design problems $(\mathrm{MVGS}_X)$ and $(\mathrm{MCGS}_X)$. Each of those problems is \textit{a priori} non-convex and seems difficult for tackling directly, both in therms of numerics and theoretical study, like verifying existence of solution. Instead two mutually dual convex problems $(\mathcal{P})$ and $(\mathcal{P}^*)$, naturally emerging for an optimal design plane pre-stressed membrane, were recalled after \cite{bouchitte2020}. Based on solutions $(\sigma,q)$ and $(u,w)$ of the respective problems construction of vaults that are optimal for $(\mathrm{MVV}_\Omega)$, $(\mathrm{MCV}_\Omega)$ is established. Similarly from solutions $(\mbf{s},\mbf{q})$ and $(\mbf{u},\mbf{w})$ of the discrete, ground-structure based conic quadratic programs $(\mathcal{P}_X)$ and $(\mathcal{P}_X^*)$ we were able to recast optimal grid-shells for $(\mathrm{MVGS}_X)$, $(\mathrm{MCGS}_X)$.

The 'continuous' setting of problems $(\mathrm{MVV}_\Omega)$, $(\mathrm{MCV}_\Omega)$ requires that the designed vault does not involve lower dimensional structural elements: struts or arches, thereby limiting the method to cases when $(\sigma,q)$ solving $(\mathcal{P})$ are $L^1$ functions. To cover the more general case usually occurring for practical loading data the measure theoretic formulations of the Prager problem was proposed, cf. $(\mathrm{MVPS}_H)$ and $(\mathrm{MCPS}_H)$.
From solutions of $(\mathcal{P})$,\,$(\mathcal{P}^*)$ we constructed a vaulted structure that mathematically is modelled by a 3D measure $\hat{\sigma}$  charging the single surface $\mathcal{S}_z$; mechanically $\hat{\sigma}$ may be viewed as a junction of 'continuous' membrane shell and a grid-shell. The structure is proved to solve the Prager problem, both in the plastic and elastic setting. This construction still relies on some extra regularity assumptions, yet a lot weaker ones: functions $(u,w)$ solving the dual problem $(\mathcal{P}^*)$ have to be Lipschitz continuous. For the time being this seems always to be the case when $\Omega$ is convex. However, based on numerical simulation in Example \ref{ex:cross} we predict that continuity of solution $u$ fails (note that the discrete optimal grid-shell problems $(\mathrm{MVGS}_X)$, $(\mathrm{MCGS}_X)$ are free of such regularity issues).

The first open problem is to develop the construction of optimal vault in the general regularity setting that is guaranteed by Proposition \ref{prop:regularitu_u_w} where in particular $u$ is a function of bounded variation only. The challenging part is generalization of the optimality conditions from Theorem \ref{thm:opt_cond}, in particular the product $\sigma:e(u)$ is difficult to define when both $\sigma$ and $e(u)$ are measures, see e.g. \cite{anzellotti1983}.

\subsection{Optimal elastic vaults for other constitutive laws}

The compliance minimization problems: the optimal vault problem $(\mathrm{MCV}_\Omega)$ and the Prager problem $(\mathrm{MCPS})$ are formulated under assumption of the Michell-like elastic potential. This way they are equivalent to the minimum volume problems reminiscent of the one of Michell. Moreover, the optimal structures found are a natural limits of grid-shell approximations, similarly as Michell structures are limits of families of trusses.

The Michell problem in its elastic setting, however, has a mathematical structure that is similar to other optimal design problems, including finding optimal distribution $\rho$ of a linearly elastic isotropic material, see the recent work \cite{bolbotowski2020a}. By analogy, isotropy could be enforced in the problem $(\mathrm{MCV}_\Omega)$ by a suitable adjustment of the gauge $\gamma_{z,+}$, again furnishing a non-convex optimization problem. The next step would be to modify problems $(\mathcal{P})$,\,$(\mathcal{P}^*)$: the term $\tr\, \sigma $ must be replaced by $\gamma^0_{+}(\sigma) = ((\mathscr{C}^{-1}\sigma):\sigma )^{1/2}$ with $\mathscr{C}$ being a fixed isotropic 4th order stiffness tensor, whereas the constraint in the dual problem would read $\gamma_+\bigl(1/4\, \nabla{w} \otimes \nabla{w} +e(u)\bigr)\leq 1$. Sadly, initial research shows that for such isotropic setting the passage from $(\mathcal{P})$,\,$(\mathcal{P}^*)$ to $(\mathrm{MCV}_\Omega)$ falls apart. To put it differently: the theory and numerical method developed in this work seems to be very specific to fibrous-like vaults that are mathematically characterized by the spectral norm and its modifications.

The Prager problem could be posed for the isotropic case as well; by a straightforward adaptation of Proposition \ref{prop:existence} existence of a 3D solution could be also established. Based on Example 5.1 found in \cite{bouchitte2001} one can infer that in the plane case a funicular arch would be unlikely to solve this altered Prager problem, similarly one may guess that in three dimensions a Prager structure ends up being a true 3D 'continuous' body rather than a vault concentrated on a single surface. This reasoning may raise doubts about well-posedness of the optimal isotropic elastic vault problem $(\mathrm{MCV}_\Omega)$ in the first place (recall the non-convexity of the formulation).

After restricting to boundary curves that are planar in Section \ref{ssec:boundary_conditions}, imposing the Michell-like elastic potential appears to be the second limitation of the form finding method presented in this paper. It therefore seems that the herein proposed link to a 2D convex pair of problems $(\mathcal{P})$,\,$(\mathcal{P}^*)$ cannot be treated as a general approach to form finding, instead it should rather be viewed as a mathematical passage in a very specific design problem. 

\subsection{Conjecture on optimality of a grid-shell in the case of a finite system of point loads}

In the numerical simulations in Examples \ref{ex:4_point_loads}, \ref{ex:multiple_point_loads}, \ref{ex:diagonal_load} we observed that the bars building an optimal truss $\sigma$ were connecting: either the points of load application with each other, or the loaded points  with the boundary points. So far no counter-example was found to this feature of optimal vaults and grid-shells. A conjecture can be made:

\begin{conjecture}[\textbf{A vault optimally designed for a discrete load is a grid-shell}]
	\label{conj:opt_vault=grid-shell}
	Let us assume that the load consists of a finite number of point forces, i.e. $f = \sum_{j=1}^{N} P_j \, \delta_{x_j}$. Then there exists a truss-like solution $\sigma$ of the problem $(\mathcal{P})$, namely
	\begin{equation*}
	\sigma= \sum_{k=1}^M s_k \, \sigma^{\,x^-_k,\, x^+_k} = \sum_{k=1}^M s_k \, \tau_k \otimes \tau_k \, \Ha^1 \mres [\,x^-_k,\, x^+_k], \qquad \tau_k = \frac{x^+_k - x^-_k}{\abs{x^+_k - x^-_k}},
	\end{equation*}
	where for each $k$ points $x^-_k, x^+_k$ belong to the set $\bO \cup \mathrm{spt}\,f= \bO \cup \{x_1,\ldots x_N\}$.

	As a result the optimal vault is a grid-shell composed of a finite number of bars connecting points on the boundary and the points of forces application only. 
\end{conjecture}

It is well established that this property is untrue for Michell problem: even in the case of the three force problem the Michell structure is in general a complicated framework composed of 1D curved cables and \textit{Hencky nets} being 'continuous' fibrous-like regions, cf. \cite{sokol2010}.

Apart from clear theoretical merit such a result has a tremendous consequence for the further development of the numerical method: it allows to erase from the grid $X$ all the nodes that are either not loaded or do not lie on the boundary. This idea is clearly ineffective in case of simulating a uniformly distributed load like in Example \ref{ex:pressure} but reduces the problem greatly when a discrete load or a load distributed on lines is applied. From Example \ref{ex:slanted_Gamma} it is clear that the conjecture is false if a problem of slanted boundary curve $\hat{\Gamma}$ is considered -- compare the nets bounded by curved elements in Fig. \ref{fig:slanted}(b).

In \cite[Section 5.5]{bouchitte2020} it is proved that Conjecture \ref{conj:opt_vault=grid-shell} is equivalent to a purely geometrical problem. In \cite{bouchitte2020} to a vector function $v$, being a monotone map defined over some domain $D \subset \Rd$ (see \cite{alberti1999}), we associate an intrinsic metric $d_v:D \times D \to \R_+$. By \cite[Proposition 5.25]{bouchitte2020} the foregoing conjecture will be confirmed if we manage to show that any monotone function $v :\mathrm{spt}f \cup \bO \to\R^2$ admits its maximal monotone extension $\bar{v}: \Ob \to \R^2$ so that the restriction of the metric $d_{\bar{v}}:\Ob \times \Ob \to \R_+$ to $(\mathrm{spt}f \cup \bO)^2$ coincides with $d_v$. This is a subject of an ongoing research.

\appendix
\section{Proofs of the measure-theoretic results in Section \ref{sec:3D}}
\label{app:3D}

\begin{proof}[Proof of Proposition \ref{prop:existence}]
	Problem $(\mathrm{MVPS}_H)$ is an infinite dimensional linear program on a  product of closed  sets $\Mes\bigl(\Oh ; \mathrm{T}^3_+\bigr) \times \mathscr{T}(\Omega,f,H)$. For $H < \infty$ the set $\mathscr{T}(\Omega,f,H)$ is compact with respect to the weak-* topology, see the comment below the definition \eqref{eq:transmissible_set}. The linear functional minimized is weak-* continuous and coercive with respect to the first variable $\hat{\sigma}$ and ultimately we have compactness with respect to both variables. Since the equilibrium equation is written in the distributional sense it is continuous with respect to weak-* topology. The existence of solution of $(\mathrm{MVPS}_H)$ will be thus established once we show that there exists at least one feasible pair $(\hat{\sigma},\hat{F})$. Assume first that $f = \delta_{x_0}$ for some $x_0 \in \Omega$: we will build a two bar truss that carries this load. We choose $\hat{F}_0 = e_3 \, \delta_{(x_0,h)}$ for arbitrary $h\in(0,H]$ and, for any pair of points $x_1,x_2 \in \bO$ such that $x_0 \in [x_1,x_2]$, we take $\hat{\sigma}_0 =\sum_{i=1}^2 S_i \, \hat{\tau}_i \otimes \hat{\tau}_i \, \Ha^1\mres [(x_0,h),(x_i,0)]$ where $\hat{\tau}_i\in \R^3$ is a unit vector tangent to the 3D segment $[(x_0,h),(x_i,0)]$. It is easy to check that $S_1, S_2 \geq 0$ may be chosen so that $-\DIV\,\hat{\sigma}_0 = \hat{F}_0$. For $f = -\delta_{x_0}$ we repeat the construction with $h\in[-H,0)$. For a general measure $f$ the pair $(\hat{\sigma},\hat{F})$ may be obtained by measure-theoretic superposition (with positive coefficients) of pairs $(\hat{\sigma}_0,\hat{F}_0)$. Existence of solution in $(\mathrm{MVPS}_H)$ follows.
	
	Next, for each $\hat{v} \in C^1(\Ob\times\R;\R^3)$, the functional being maximized in \eqref{eq:3D_comp_def} is linear and continuous with respect to the pair $(\hat{\mu},\hat{F})$ hence $\hat{\Comp}(\hat{\mu},\hat{F})$ is convex and weakly-* lower semi-continuous as a point-wise supremum of such functionals. The product $\bigl\{\hat{\mu}\in \Mes_+\bigl(\Oh\bigr) \, \vert\, \int d\hat{\mu} \leq V_0 \bigr\} \times \mathscr{T}(\Omega,f,H)$ is weakly-* compact therefore solution of $(\mathrm{MCPS}_H)$ exists provided that $\hat{\Comp}^H_\mathrm{min} < \infty$. This follows from the dual formula \eqref{eq:3D_dual_comp_def} where from the construction of $(\hat{\sigma},\hat{F})$ above one can take $\hat{\mu} = \tr\,\hat{\sigma}$ and $\hat{S} = \frac{d\hat{\sigma}}{d\hat{\mu}}$ in the sense of Radon-Nikodym derivative.
\end{proof}

\vspace{0.1cm}
\begin{proof}[Proof of Theorem \ref{thm:optimal_3D_structure}]
	Owing to condition $H \geq \mathrm{diam}(\Omega)/\sqrt{2}$ we have $\mathrm{spt}\,\hat{F},\,\mathrm{spt}\,\hat{\sigma},\,\mathrm{spt}\,\hat{\mu} \subset \Oh = \Ob \times [-H,H]$ owing to estimates \eqref{eq:estimates_u_w}. We compute the trace of the field $\hat{S}$; from equalities $\tr(B^\top A  B) = \tr(A B B^\top) = A^\top : (B B^\top)$ we find based on \eqref{eq:opt_mu_S_2D} that
	\begin{equation*}
	\htr \, \hat{S} =  \Big( \bigl(\tr S + S:(\nabla_\mu z \otimes \nabla_\mu z)  \bigr) \Big)\circ \hat{z}^{-1} =  \frac{\Z}{V_0}  \qquad \hat{\mu}\text{-a.e.}
	\end{equation*}
	Then we check that
	\begin{align}
	\label{eq:opt_hat_sigma}
	\int_{\Ob \times  \R} \htr\,\hat{\sigma} = \int_{\Ob \times  \R} \htr\,\hat{S}\,d\hat{\mu} = \frac{\Z}{V_0} \int_{\Ob \times  \R} d\hat{\mu} &= \frac{\Z}{V_0} \int_{\Ob} d\mu = \frac{\Z}{V_0} \int_\Ob\frac{V_0}{\Z}\, \bigl(\mathrm{I}+\nabla_\mu z \otimes \nabla_\mu z \bigr):\sigma \\
	& = \int_{\Ob} \bigl( \tr\,\sigma + (\nabla_\mu z \otimes \nabla_\mu z):\sigma \bigr) = \int_{\Ob} g_\mathscr{C}^0 (\sigma,q) = \Z, \nonumber
	\end{align}
	where the integral $ \int_{\Ob} g_\mathscr{C}^0 (\sigma,q)$ was identified based on the formula  \eqref{eq:g_C^0} and optimality condition $(iv)'$, i.e. $q = \frac{1}{2} \, \sigma \, \nabla_\mu w = \sigma \, \nabla_\mu z$. The last equality follows directly from optimality of the pair $(\sigma,q)$. Verification of $\hat{F} \in \mathscr{T}(\Omega,f,H)$ is straightforward hence, by virtue of Lemmas \ref{lem:inequalities_3D} and \ref{lem:feas_hat_sigma} optimality of the pair $(\hat{\sigma},\hat{F})$ in problem $(\mathrm{MVPS}_H)$ is established.
	
	From the chain of equalities \eqref{eq:opt_hat_sigma} information $\int_{\Ob\times\R} d \hat{\mu} = V_0$ can be extracted, i.e. the pair $(\hat{\mu},\hat{F})$ is feasible for $(\mathrm{MCPS}_H)$. Then, based on Lemma \ref{lem:feas_hat_sigma} we find that $\hat{S}$ is admissible for the stress based elasticity problem \eqref{eq:3D_dual_comp_def}. The chain of inequalities follows:
	\begin{equation*}
	\hat{\Comp}^H_\mathrm{min} \leq \hat{\Comp}\bigl(\hat{\mu},\hat{F} \bigr) \leq \frac{1}{2 E_0} \int_{\Ob\times \R}  \bigl(  \htr\, \hat{S}\bigr)^2 d\hat\mu =  \frac{1}{2 E_0} \left(\frac{\Z}{V_0} \right)^2 \int_{\Ob\times \R} d\hat\mu = \frac{\Z^2}{2E_0 V_0} \leq \hat{\Comp}^H_\mathrm{min},
	\end{equation*}
	being, in fact, a chain of equalities, which proves optimality of $(\hat{\mu},\hat{F})$ and that the stress field $\hat{S} \in L^\infty_{\hat{\mu}}(\Ob;\Sddp)$ solves the stress-based elasticity problem \eqref{eq:3D_dual_comp_def} posed for the optimal structure.
	
	Functions $({u},{w})$ are Lipschitz continuous and they solve problem $(\mathcal{P}^*)$: by combining Lemma 1 in \cite{bouchitte2007} and the fact that $P\, \xi\, P \preceq \xi$ for any $\xi \in \Sdd$ and any orthogonal projection operator $P: \Rd \to \Rd$ we may infer
	\begin{equation}
	\label{eq:mu-tangential_constraint}
	\frac{1}{4}\, \nabla_\mu {w} \otimes \nabla_\mu {w} +e_\mu({u}) \preceq  \mathrm{I} \qquad  \text{or equivalenly } \qquad  h\bigl(e_\mu({u}),\nabla_\mu {w} \bigr) \leq 1.
	\end{equation}
	With $\hat{v}(\hat{x}) = u(x) + w(x) e_3$ it is straightforward that $\hat{v} \in \mathrm{Lip}(\Ob\times \R;\R^3)$ and therefore the $\hat{\mu}$-tangential operator $\hat{e}_{\hat{\mu}} (\hat{v})$ is meaningful. From basic facts on the measure-tangential calculus (cf. \cite{bouchitte2007}) we may find that $\hat{e}_{\hat{\mu}} (\hat{v}) = P_{\hat{\mu}} \, \hat{\xi}\, P_{\hat{\mu}} $ where $\hat{\xi} = \bigl(e_\mu(u) + \nabla_\mu w \,\symtens\,e_3 \bigr) \circ \hat{z}^{-1}$. Since $P_{\hat{\mu}} \, \hat{\xi} \, P_{\hat{\mu}} \preceq \hat{\xi}$ we obtain $\hat{\mu}$-a.e. that
	\begin{equation}
	\label{eq:hat_gamma_leq_1}
	\hat{\gamma}_+\bigl( \hat{e}_{\hat{\mu}} (\hat{v}) \bigr) \leq 	\hat{\gamma}_+\bigl( \hat{\xi} \bigr) =  h\bigl(e_\mu(u),\nabla_\mu w \bigr) \leq 1,
	\end{equation}
	where the above equality was already established in \eqref{eq:hat_gamma_h}. If $u,w $ were differentiable (hence $\hat{v}$ as well) the equality \eqref{eq:nabla_v:S} would have held therefore, by the fact that $P_{\hat{\mu}}\,\hat{S}\, P_{\hat{\mu}} = \hat{S}$ (see Lemma 2 in \cite{bouchitte2007}) and through a density argument, we obtain	 
	\begin{align}
	\label{eq:extremality_mu}
	&\hat{e}_{\hat{\mu}}(\hat{v})\bigl(\hat{z}(\argu)\bigr) : \hat{S}\bigl(\hat{z}(\argu)\bigr) =  \Big( S:e_\mu(u) + S :(\nabla_\mu w \otimes \nabla_\mu z) \Big) \\
	= &\biggl( S: \biggl(\frac{1}{4} \nabla_\mu w \otimes \nabla_\mu w + e_\mu(u)  \biggr) + S:(\nabla_\mu z \otimes \nabla_\mu z)\biggr)
	= \Big( \tr \,S + S:(\nabla_\mu z \otimes \nabla_\mu z)\Big) = \htr\, \hat{S}\bigl(\hat{z}(\argu)\bigr), \nonumber
	\end{align}
	where to pass to the second line we have manipulated with the equality $z= \frac{1}{2}\,w$ and then we have employed optimality condition $(iii)'$ from \eqref{eq:opt_cond_mu}. Since $\tr\,\hat{S} = \hat{\gamma}_+^0(\hat{S})$, by combining \eqref{eq:hat_gamma_leq_1} and \eqref{eq:extremality_mu} extremality of the pair $ \hat{e}_{\hat{\mu}} (\hat{v})$, $\hat{S}$ may be inferred therefore
	\begin{equation*}
	\hat{\gamma}_+\bigl( \hat{e}_{\hat{\mu}} (\hat{v}) \bigr) = 1 \qquad \hat{\mu}\text{-a.e.}
	\end{equation*}
	and as a result $\hat{\gamma}_+\bigl( \hat{e}_{\hat{\mu}} (\hat{v}_\e) \bigr) =\Z/(E_0 V_0)$. With the supremum problem \eqref{eq:3D_comp_def} reformulated in its relaxed, measure-tangential setting (cf. \cite{bouchitte1997}) we can write down a chain
	\begin{align*}
	\hat{\Comp}^H_\mathrm{min} = \hat{\Comp}\bigl(\hat{\mu},\hat{F} \bigr) &\geq \int_{\Ob\times\R} \hat{v}_\e\cdot \hat{F} - \frac{E_0}{2} \int_{\Ob\times \R}  \Big(  \hat{\gamma}_{+}\bigl(\hat{e}_{\hat{\mu}}(\hat{v}_\e)\bigr)\Big)^2 d\hat\mu \\
	&= \frac{\Z}{E_0 V_0} \int_{\Ob} w \, f - \frac{E_0}{2} \int_{\Ob\times \R}  \biggl(\frac{\Z}{E_0 V_0}\biggr)^2 d\hat\mu 
	= \frac{\Z^2}{2 E_0 V_0} = \hat{\Comp}^H_\mathrm{min} ,
	\end{align*}
	where by optimality of $(u,w)$ in $(\mathcal{P}^*)$ we acknowledged that $\int_{\Ob} w \, f = \Z$. Again, the above chain consists of equalities only rendering $\hat{v}_\e$ a solution of the relaxed displacement-based problem \eqref{eq:3D_comp_def} for the optimal structure.
	In order to verify the constitutive relation it is enough to show that $\hat{\mu}$-a.e.
	\begin{equation}
	\hat{e}_{\hat{\mu}} (\hat{v}_\e) : \hat{S} = \hat{j}_+\bigl(\hat{e}_{\hat{\mu}} (\hat{v}_\e)\bigr) + \hat{j}_+^*\bigl( \hat{S}\bigr)
	\end{equation}
	being straightforward when using formulas \eqref{eq:hat_j}. The proof is complete.
\end{proof}

\section{Duality between the conic programs}
\label{app:cone_duality}

\begin{proof}[Proof of Theorem \ref{thm:duality_discrete}]
	In the proof we will draw upon \cite[Section 2.5]{ben2001} where a summary on conic duality may be found. Upon setting the variable vector $\mbf{x} = \bigl[\mbf{s}^\top \ \mbf{q}^\top \ \mbf{r}^\top \bigr]^\top$ problem $(\mathcal{P}_X)$ becomes the primal conic programming problem $(\mathrm{Pr})$ $($therein identified by vectors $\mbf{c}\in \R^{3m}, \mbf{p} \in \R^{3n},\mbf{b}\in \R^{3m}$ and matrices $\mbf{P}\in \R^{3m \times 3n}, \mbf{A}_k \in \R^{3m \times 3})$ provided that we put:
	\begin{equation*}
	\mbf{c} = \begin{bmatrix}
	\mbf{l} \\ \mbf{0} \\ 2\,\mbf{ \mbf{l}}
	\end{bmatrix}\!, \qquad
	\mbf{p} = \begin{bmatrix}
	\mbf{0} \\ \mbf{0} \\ \mbf{ \mbf{f}}
	\end{bmatrix}\!, \qquad
	\mbf{b} = \mbf{0}, \qquad
	\mbf{P} = \begin{bmatrix}
	\mbf{B}_1 & \mbf{B}_2 & \mbf{0} \\
	\mbf{0} & \mbf{0} & \mbf{D} \\
	\mbf{0} & \mbf{0} & \mbf{0}
	\end{bmatrix}^\top,
	\end{equation*} 
	while $\mbf{A}_k$ are matrices consisting of ones an zeros that for each $k \in \{1,\ldots,m \}$ allocate variables $s_k,q_k,r_k$ into the $k$-th conic constraint; above symbols $\mbf{0}$ stand for zero column vectors or matrices of dimensions depending on the context. According to \cite{ben2001} the conic problem $(\mathrm{Pr})$ admits its dual problem $(\mathrm{Dl})$ of variables $\bm{\nu}\in \R^{3n}, \bm{\eta}_1\in \R^3, \ldots , \bm{\eta}_m \in \R^3$. With the interpretations $\bm{\nu} = \bigl[\mbf{u}_1^\top \ \mbf{u}_2^\top \ \mbf{w}^\top \bigr]^\top $ and $\bigl[\bm{\eta}_1 \ \ldots \ \bm{\eta}_m \bigr] = \bigl[ \mbf{t}_1 \ \mbf{t}_2 \ \mbf{t}_3 \bigr]^\top$ we find that the problem $(\mathrm{Dl})$ therein is exactly $(\mathcal{P}_X^*)$.
	
	According to \cite{ben2001} in order to prove that the duality gap vanishes and moreover that problem $(\mathcal{P}_X^*)$ admits a solution it is enough to show that $(\mathcal{P}_X)$ is \textit{strictly feasible}. In the setting of Section 2.5 in \cite{ben2001} this matter boils down to showing that: (i) the rows of $\mbf{P}$ are linearly independent; (ii) there exists a feasible solution $(\mbf{s},\mbf{q},\mbf{r})$ such that for each $k$ there holds $(r_k,s_k,q_k) \in \mathrm{int}(\mathrm{K})$ or, equivalently, $2\, r_k \, s_k > q_k$. Condition (i) follows directly from the comments below \eqref{eq:B_D_def}. Since for $r_k$ we may choose an arbitrarily large number, condition (ii) will be assured if we manage to show that for each $k$ there exists a vector $\mbf{s}$ satisfying $\mbf{B}_1^\top \mbf{s} =\mbf{B}_2^\top \mbf{s} = \mbf{0}$ (a pre-stressed truss in tension pinned on $X \cap \bO$) such that $s_k >0$ for each $k$. We will sketch the structural-mechanics idea behind this fact: every bar $k$ of positive force $s_k$ either is pinned on both ends by the boundary $\bO$ or it produces point forces $F \in \Rd$ at one or two its ends. We must show that such force $F$ can always be transferred to points $X \cap \bO$ via bars in tension. To that aim it is enough that for each point $x \in X \backslash \bO$ from the ground structure $X \times X$ we can extract a three-bar truss such that: it is jointed at $x$ while the other three nodes are non-colinear points $y_1,y_2,y_3 \in \bO$ such that $x \in \mathrm{conv}\bigl(\{y_1,y_2,y_3\} \bigr)$. It may be checked that such truss always exists thanks to assumption \eqref{eq:assum_on_X}. The strict feasibility of $(\mathcal{P}_X)$ readily follows. To show that solution of $(\mathcal{P})$ exists as well we point to an explicit strictly feasible vector for $(\mathcal{P}_X^*):\ $ $\mbf{u}_1 =\mbf{u}_2 = \mbf{w} = \mbf{0}$ and $\mbf{t}_1 = 2 \,\mbf{l}$, $\mbf{t}_2 = \mbf{l}$, $\mbf{t}_3 = \mbf{0}$.
	
	According to \cite{ben2001} the optimality conditions for feasible variables in a pair of conic problems may be written in a form of complementary slackness conditions that for $(\mathcal{P}_X)$,\,$(\mathcal{P}_X^*)$ read: $t_{1;k} r_k + t_{2;k} s_k + t_{3;k} q_k = 0$ for each $k$. Feasibility in $(\mathcal{P}_X^*)$ is equivalent to satisfying condition $(i)$ in \eqref{eq:opt_cond_discrete} and moreover the equalities $t_{1;k} = 2 \,l_k$, $t_{2;k} = l_k - (\mbf{B}_1 \mbf{u}_1 + \mbf{B}_2 \mbf{u}_2)_k$, $t_{3;k} = (\mbf{D}\mbf{w})_k$; these conditions, along with $(ii)$ in \eqref{eq:opt_cond_discrete} will thus be assumed to be true in the remainder of the proof. Once $s_k=0$ the conic constraint in $(\mathcal{P}_X)$ implies that $q_k =0$ while the comp. slack. cond. gives $r_k = 0$. If $s_k > 0$ then the complementary slackness condition may be rewritten as:
	\begin{equation*}
	\left( \frac{1}{4} \frac{\bigl((\mbf{D} \mbf{w})_k \bigr)^2}{l_k} + (\mbf{B}_1 \mbf{u}_1 + \mbf{B}_2 \mbf{u}_2 )_k \right) s_k = l_k s_k + l_k \, s_k \left(\frac{1}{2} \frac{\bigl(\mbf{D} \mbf{w})_k}{l_k} - \frac{q_k}{s_k}  \right)^2 + 2 \, l_k \left( r_k - \frac{1}{2}\frac{(q_k)^2}{s_k}  \right).
	\end{equation*}
	Due to $(i)$ in \eqref{eq:opt_cond_discrete} the LHS is not greater than $l_k s_k$ whereas the RHS is not smaller than $l_k s_k$ owing to the conic constraint $(r_k,s_k,q_k) \in \mathrm{K}$. Ultimately the equality holds if and only if conditions $(iii)$, $(iv)$ in \eqref{eq:opt_cond_discrete} and $r_k = \frac{1}{2}(q_k)^2 s_k$ are fulfilled. The proof is complete.
\end{proof}

\section{Proof of the link between the optimal vault problem and optimal arch-grid problem}
\label{app:archgird}

\begin{proof}[Proof of Proposition \ref{prop:archgrid}]
	We fix $i =1$ and assume that (I) holds. Since $u_1=0$ on $\bO$ then for any $x_2 \in \pi_2(\Ob)$, the univariate function $u_1(\argu,x_2)$ is zero at ends of the section $\bigl[\bigl(x_1^-(x_2)\bigr)\,,\,  \bigl(x_1^+(x_2)\bigr) \bigr]$ and therefore
	\begin{equation*}
	\int_{x_1^-(x_2)}^{x_1^+(x_2)} \frac{\partial u_1}{\partial x_1} (\xi,x_2)\,  d\xi = 0 \qquad \forall\,x_2 \in \R.
	\end{equation*}
	As a result, inequality in (II) for every $x_2\in \pi_2(\Ob)$ follows by integrating the inequality in (I) over the segment $\bigl[\bigl(x_1^-(x_2)\bigr)\,,\,  \bigl(x_1^+(x_2)\bigr) \bigr] \times \{x_2\}$.
	
	Contrarily, again for $i=1$, assume that (II) holds for given $w \in \D(\Omega;\R)$. For any $x =(x_1,x_2) \in \Omega$ we define 
	\begin{equation*}
	u_1(x_1,x_2)= \int_{x_1^-(x_2)}^{x_1} \left(\frac{\bigl(\rho_1(w;x_2)\bigr)^2}{L_1(x_2)} - \frac{1}{4} \left(\frac{\partial w}{\partial x_1}(\xi,x_2) \right)^2 \right) d\xi,
	\end{equation*}
	which furnishes a function $u_1$ of class $C^1(\Ob;\R)$ which may be easily verified by employing the regularity $w \in \D(\Omega;\R)$. For any $x_2 \in \pi_2(\Ob)$ we readily check that $u_1\bigl(x_1^-(x_2),x_2\bigr) = u_1\bigl(x_1^+(x_2),x_2\bigr) = 0$ and we immediately infer that $u_1 = 0$ on $\bO \backslash \bO_1$ where $\bO_1$ is the part of $\bO$ that is parallel to $e_1$. Owing to the compact support of $w$ the condition $u_1=0$ on $\bO_1$ follows as well. Ultimately we compute for each $x=(x_1,x_2) \in \Omega$ that
	\begin{equation*}
	\frac{1}{4}\, \left(\frac{\partial w}{\partial x_1}(x) \right)^2 + \frac{\partial u_1}{\partial x_1}(x) =\frac{1}{4}\, \left(\frac{\partial w}{\partial x_1}(x) \right)^2 + \left(\frac{\bigl(\rho_1(w;x_2)\bigr)^2}{L_1(x_2)} - \frac{1}{4} \left(\frac{\partial w}{\partial x_1}(x) \right)^2 \right) = \frac{\bigl(\rho_1(w;x_2)\bigr)^2}{L_1(x_2)} \leq 1
	\end{equation*}
	thus furnishing (I). For $i=2$ the proof is analogous.
\end{proof}

	\footnotesize
	\noindent\textbf{Acknowledgments.}
	The paper was prepared within the Research Grant no 2015/19/N/ST8/00474 financed by the National Science Centre (Poland), entitled: Topology optimization of thin elastic shells - a method synthesizing shape and free material design.
	
	The author would like to express his appreciation to Professor Guy Bouchitt\'{e} for hosting an 8 week visit  at Universit\'{e} de Toulon in the fall of 2019. Without his mathematical assistance this work would never be possible. To Professor Tomasz Sok{\'o}{\l} the author gives thanks for providing two numerical results inserted in this paper. The author is also thankful for the invaluable guidance of his PhD supervisors Professor Tomasz Lewi\'{n}ski and Professor Piotr Rybka.

	\setlength{\bibsep}{2pt}
	\bibliographystyle{spbasic}      

\end{document}